\documentclass[11pt]{article}
\usepackage{amsmath,amssymb,amsthm,amscd, mathdots}
\usepackage{epsfig, color}
\usepackage{makeidx}
\usepackage{index}

\makeindex
\newindex{not}{ndx}{ndd}{List of Notations}

\renewcommand{\thefootnote}

\newtheorem{lemma}{Lemma}[section]
\newtheorem{proposition}{Proposition}[section]
\newtheorem{theorem}{Theorem}[section]
\newtheorem{conj}{Conjecture}[section]
\newtheorem{corollary}{Corollary}[section]
\newtheorem{definition}{Definition}[section]
\newtheorem{remark}{Remark}[section]

\newcommand{\eps}{\varepsilon} 

\newcommand{\indic}[1]{\mathbf{1}_{\{#1\}}}

\newcommand{\ds}{\displaystyle}

\DeclareMathOperator\var{var}
\DeclareMathOperator\cov{{cov}}

\def\mn{\medskip\noindent}

\def\beq{\begin{equation}}
\def\eeq{\end{equation}}
\def\beqa{\begin{eqnarray}}
\def\eeqa{\end{eqnarray}}
\def\beqax{\begin{eqnarray*}}
\def\eeqax{\end{eqnarray*}}
\def\sqz{\kern -0.2em}

\def\R{\mathbb{R}}
\def\E{\mathbb{E}}
\def\P{\mathbb{P}}
\def\N{\mathbb{N}}
\def\Z{\mathbb{Z}}
\def\T{\mathbb{T}}

\def\cB{\mathcal B}
\def\cC{\mathcal C}

\def\cF{\mathcal F}

\def\cI{\mathcal I}

\def\cL{\mathcal L}
\def\cM{\mathcal M}
\def\cN{\mathcal N}
\def\cR{\mathcal R}

\def\cL{\mathcal L}
\def\cP{\mathcal P}

\def\cS{\mathcal S}
\def\cT{\mathcal T}

\def\cZ{\mathcal{Z}}

\def\th{^{\text{th}}}
\DeclareMathOperator{\Poi}{Poisson}

\usepackage[all]{xy}

\def\rep{random exchangeable partition\ }

\def\csbp{{continuous-state branching process}}

\begin{document}

\renewcommand{\thepage}

\vskip .2in

\noindent
\includegraphics{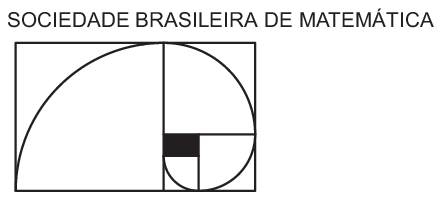}
\vskip -.8in

\rightline{\footnotesize  \sf ENSAIOS MATEM\'ATICOS}

\rightline{\footnotesize \sf 200X, Volume {\bf XX}, X--XX}

\vskip 1in

\setcounter{page}{1}

\noindent {\huge\bf {Recent Progress}}


\vskip .1in \noindent {\huge\bf {in Coalescent Theory}}



\thispagestyle{empty}

\vskip .4in
\noindent {\Large\bf {Nathana\"el Berestycki}}

\vskip .8in

\footnote{{\bf 2000 Mathematics Subject Classification: } 
60J25, 60K35, 60J80.}
\noindent {\bf Abstract.\,}
{Coalescent theory is the study of random processes where particles may join each other to form clusters as time evolves. These notes provide an introduction to some aspects of the mathematics of coalescent processes and their applications to theoretical population genetics and other fields such as spin glass models. The emphasis is on recent work concerning in particular the connection of these processes to continuum random trees and spatial models such as coalescing random walks.}

\newpage

\section*{Introduction}

The probabilistic theory of coalescence, which is the primary subject of these notes, has expanded at a quick pace over the last decade or so. I can think of three factors which have essentially contributed to this growth. On the one hand, there has been a rising demand from population geneticists to develop and analyse models which incorporate more realistic features than what Kingman's coalescent allows for. Simultaneously, the field has matured enough that a wide range of techniques from modern probability theory may be successfully applied to these questions. These tools include for instance martingale methods, renormalization and random walk arguments, combinatorial embeddings, sample path analysis of Brownian motion and L\'evy processes, and, last but not least, continuum random trees and measure-valued processes. Finally, coalescent processes arise in a natural way from spin glass models of statistical physics. The identification of the Bolthausen-Sznitman coalescent as a universal scaling limit in those models, and the connection made by Brunet and Derrida to models of population genetics, is a very exciting recent development.

\medskip The purpose of these notes is to give a quick introduction to the mathematical aspects of these various ideas, and to the biological motivations underlying them. We have tried to make these notes as self-contained as possible, but within the limits imposed by the desire to make them short and keep them accessible. Of course, the price to pay for this is a lack of mathematical rigour. Often we skip the technical parts of arguments, and instead focus on some of the key ideas that go into the proof. The level of mathematical preparation required to read these notes is roughly that of two courses in probability theory. Thus we will assume that the reader is familiar with such notions as Poisson point processes and Brownian motion.

\medskip Sadly, several important and beautiful topics are not discussed. The most obvious such topics are
the Marcus-Lushnikov processes and their relation to the Smoluchowski equations, as well as works on simultaneous multiple collisions. Also not appearing in these notes is the large body of work on random fragmentation. For all these and further omissions, I apologise in advance.


\medskip \medskip A first draft of these notes was prepared for a set of lectures at IMPA in January 2009. Many thanks to Vladas Sidoravicius and Maria Eulalia Vares for their invitation, and to Vladas in particular for arranging many details of the trip. I lectured again on this material at Eurandom on the occasion of the conference Young European Probabilists in March 2009. Thanks to Julien Berestycki and Peter M\"orters for organizing this meeting and for their invitation. I also want to thank Charline Smadi-Lasserre for a careful reading of an early draft of these notes.

\medskip Many thanks to the people with whom I learnt about coalescent processes: first and foremost, my brother Julien, and to my other collaborators on this topic: Alison Etheridge, Vlada Limic, and Jason Schweinsberg. Thanks are due to Rick Durrett and Jean-Fran\c{c}ois Le Gall for triggering my interest in this area while I was their PhD students.

\bigskip

\bigskip

\hfill N.B.

\hfill Cambridge, September 2009

\newpage

\thispagestyle{empty} \quad \thispagestyle{empty}
\tableofcontents
\newpage
\thispagestyle{empty}



\newpage

\quad
\renewcommand{\thefootnote}{\arabic{footnote}}
\setcounter{footnote}{0}
\renewcommand{\thepage}{\arabic{page}}
\setcounter{page}{7}
\renewcommand{\thesection}{\arabic{section}}
\setcounter{section}{0}


\pagestyle{myheadings} \markboth{\sf {N. Berestycki}}{\sf {Coalescent theory}}




\section{Random exchangeable partitions}

This chapter introduces the reader to the theory of exchangeable random partitions, which is a basic building block of coalescent theory. This theory is essentially due to Kingman; the basic result (essentially a variation on De Finetti's theorem) allows one to think of a random partition alternatively as a discrete object, taking values in the set $\cP$ \index[not]{$\cP$ partitions}of partitions of $\N=\{1,2,\ldots,\}$, or a continuous object, taking values in the set $\cS_0$ of tilings of the unit interval (0,1).
These two points of view are strictly equivalent, which contributes to make the theory quite elegant: sometimes, a property is better expressed on a random partition viewed as a partition of $\N$, and sometimes it is better viewed as a property of partitions of the unit interval. We then take a look at a classical example of random partitions known as the Poisson-Dirichlet family, which, as we partly show, arises in a huge variety of contexts. We then present some recent results that can be labelled as ``Tauberian theory", which takes a particularly elegant form here.

\subsection{Definitions and basic results}

We first fix some vocabulary and notation. A partition $\pi$ of $\N$ is an equivalence relation on $\N$. The blocks of the partition are the equivalence classes of this relation. We will sometime write $i \sim j$ or $i \sim_\pi j $ to denote that $i$ and $j$ are in the same block of $\pi$. Unless otherwise specified, the blocks of $\pi$ will be listed in the increasing order of their least elements: thus, $B_1$ is the block containing 1, $B_2$ is the block containing the smallest element not in $B_1$, and so on. The space of partitions of $\N$ is denoted by $\cP$. There is a natural distance on the space $\cP$, which is to take $d(\pi, \pi')$ to be equal to 1 over the largest $n$ such that the restriction of $\pi$ and $\pi'$ to $\{1, \ldots, n\}$ are identical. Equipped with this distance, $\cP$ is a Polish space. This is useful when speaking about random partitions, so that we can talk about convergence in distribution, conditional distribution, etc. We also let $[n]=\{1,\ldots, n\}$ and $\cP_n$ be the space of partitions of $[n]$.

Given a partition $\pi=(B_1, B_2, \ldots)$ and a block $B$ of that partition, we denote by $|B|$, the quantity, if it exists:
\begin{equation}\label{D:frequency}\index{Frequency of a block}
|B|:=\lim_{n \to \infty} \frac{\text{Card}(B \cap [n])}n.
\end{equation}
$|B|$ is called the asymptotic frequency of the block $B$, and is a measure of its relative size; for this reason we will often refer to it as its mass. For instance, if $\pi$ is the partition of $\N$ into odd and even integers, there are two blocks, each with mass $1/2$. The following definition is key to what follows. If $\sigma$ is a permutation of $\N$ with finite support (i.e., it actually permutes only finitely may points), and $\Pi$ is a partition, then one can define a new partition $\Pi_\sigma$ by exchanging the labels of integers according to $\sigma$. That is, $i,j$ are in the same block of $\Pi$, if and only if $\sigma(i)$ and $\sigma(j)$ are in the same block of $\Pi_\sigma$.

\begin{definition} \index{Exchangeable partition}
An exchangeable random partition $\Pi$ is a random element of $\cP$ whose law is invariant under the action of any permutation $\sigma$ of $\N$ with finite support: that is, $\Pi$ and $\Pi_\sigma$ have the same distribution for all $\sigma$.
\end{definition}

To put things into words, an exchangeable random partition is a partition which ignores the label of a particular integer. This suggests that exchangeable random partitions are only relevant when working under mean-field assumptions. However, this is slightly misleading. For instance, if one looks at the random partition obtained by first enumerating all vertices of $\Z^d$ $(v_1, v_2,, \ldots)$ in some arbitrary order, and then say that $i$ and $j$ are in the same block of $\Pi(\omega)$ if and only if $v_i$ and $v_j$ are in the same connected component in a realisation $\omega$ of bond percolation on $\Z^d$ with parameter $0<p<1$, then the resulting random partition is not exchangeable. On the other hand, if $(V_1, V_2, \ldots)$ are independent random vertices chosen according to some given distribution on $\Z^d$, then the random partition defined by putting $i$ and $j$ in the same block if $V_i$ and $V_j$ are in the same connected component, is exchangeable. Indeed, in these notes we will later see several examples where random partitions arise from a nontrivial spatial structure.


Kingman's theorem, which is the main result of this section, starts with the observation that given a tiling of the unit interval, there is always a neat way to generate an exchangeable random partition associated with this tiling. To be formal, let $\cS_0$ be the space of tilings of the unit interval $(0,1)$, that is, sequences $s=(s_0,s_1, \ldots)$ with $s_1 \ge s_2 \ge \ldots \ge 0$ and $\sum_{i=0}^\infty s_i =1$ (note that we do not require $s_0 \ge s_1$):
$$
\cS_0=\left\{s=(s_0,s_1, \ldots): s_1 \ge s_2 \ge \ldots, \sum_{i=0}^\infty s_i =1\right\}.
$$
The coordinate $s_0$ plays a special role in this sequence and this is why monotonicity is only required starting at $i=1$ in this definition. An element of $\cS_0$ may be viewed as a tiling of (0,1), where the sizes of the tiles are precisely equal to $s_0,s_1,\ldots$ the ordering of the tiles is irrelevant for now, but for the sake of simplicity we will order them from left to right: the first tile is $J_0 = (0,s_0)$, the second is $J_1=(s_0,s_0 +s_1)$, etc.
Let $s \in \cS_0$, and let $U_1, U_2,\ldots$ be i.i.d. uniform random variables on $(0,1)$. For $0<u<1$ let $I(u) \in \{0,1, \ldots\}$ denote the index of the component (tile) of $s$ which contains $u$. That is,
$$
I(u)= \inf\left\{n : \sum_{i=0}^n s_i > u \right\}.
$$
Let $\Pi$ be the random partition defined by saying $i \sim j$ if and only if $I(U_i) = I(U_j) >0$ or $i=j$ (see Figure \ref{fig:sampl}).
\begin{figure}
\begin{center}
\includegraphics[scale=.6]{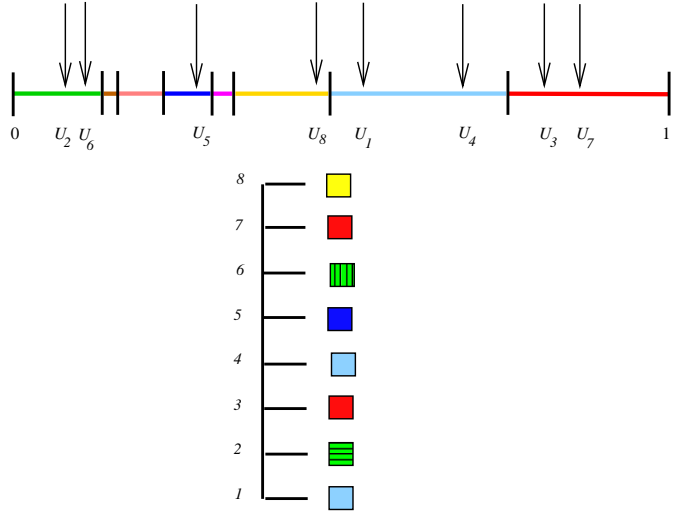}
\end{center}
\caption{The paintbox process associates a random partition $\Pi$ to any tiling of the unit interval. Here $\Pi|_{[8]} = (\{1,4\}, \{2\}, \{3,7\}, \{5\},\{6\},\{8\})$. Note how 2 and 6 form singletons.}\label{fig:sampl}
\end{figure}
Note that in this construction, if $U_i$ falls into the $0\th$ part of $s$, then $i$ is guaranteed to form a singleton in the partition $\Pi$. On the other hand, if $I(U_i)\ge 1$, then almost surely, the block containing $i$ has infinitely many members, and in fact, by the law of large numbers, the frequency of this block is well defined and strictly positive. For this reason, the part $s_0$ of $s$ is referred to as the \emph{dust} of $s$. We will say that $\Pi$ has no dust if $s_0=0$, i.e., if $\Pi$ has no singleton.

The partition $\Pi$ described by the above construction gives us an exchangeable partition, as the law of $(U_1, \ldots, U_n)$ is the same as that of $(U_{\sigma(1)}, \ldots, U_{\sigma(n)})$ for each $n\ge 1$ and for each permutation $\sigma$ with support in $[n]$.

\begin{definition}
$\Pi$ is the \emph{paintbox} partition derived from $s$.
\end{definition}
The name paintbox refers to the fact that each part of $s$ defines a colour, and we paint $i$ with the colour in which $U_i$ falls. If $U_i$ falls in $s_0$, then we paint $i$ with a unique, new, colour. The partition $\Pi$ is then obtained from identifying integers with the same colour. \index{Paintbox process}

Note that this construction still gives an exchangeable random partition if $s$ is a random element of $\cS_0$, provided that the sequence $U_i$ is chosen independently from $s$. Kingman's theorem states that this is the most general form of exchangeable random partition. For $s \in \cS_0$, let $\rho_s$ denote the law on $\cP$ of a paintbox partition derived from $s$.

\begin{theorem} \emph{(Kingman \cite{king82})}
Let $\Pi$ be any exchangeable random partition. Then there exists a probability distribution $\mu(ds)$ on $\cS_0$ such that
$$
\P(\Pi \in \cdot) = \int_{s \in \cS_0} \mu(ds) \rho_s(\cdot).
$$
\end{theorem}

\mn \emph{Sketch of proof}.
We briefly sketch Aldous' proof of this result \cite{aldous st flour}, which relies on De Finetti's theorem\index{De Finetti's theorem} on exchangeable sequences of random variables. This theorem states the following: if $(X_1, \ldots)$ is an infinite exchangeable sequence of real-valued random variables (i.e., its law is invariant under the permutation of finitely many indices), then there exists a random probability measure $\mu$ such that, conditionally given $\mu$, the $X_i$'s are i.i.d. with law $\mu$. Now, let $\Pi$ be an exchangeable partition. Define a random map $\varphi:\N \to \N$ as follows: if $i \in \N$, then $\varphi(i)$ is the smallest integer in the same block as $i$. Thus the blocks of the partition $\Pi$ may be regarded as the sets of points which share a common value under the map $\varphi$. In parallel, take an independent sequence of i.i.d. uniform random variables $(U_1, \ldots)$ on $[0,1]$, and define $X_i = U_{\varphi(i)}$. It is immediate that $(X_1, \ldots)$ are exchangeable, and so De Finetti's theorem applies. Thus there exists $\mu$ such that, conditionally given $\mu$, $(X_1, \ldots)$ is i.i.d. with law $\mu$. Note that $i$ and $j$ are in the same block of $\Pi$ if and only if $X_i = X_j$. We now work conditionally given $\mu$. Note that $(X_1, \ldots)$ has the same law as $(q(V_1), \ldots)$, where $(V_1,\ldots)$ are i.i.d. uniform on $[0,1]$, and for $x \in \R$, $q(x) = \inf\{y \in \R: F(y) >x\}$ and $F(x)$ denotes the cumulative distribution function of $\mu$. Thus we deduce that $\Pi$ has the same law as the paintbox $\rho_s(\cdot)$, where $s = (s_0, s_1, \ldots) \in \cS_0$ is such that $(s_1,\ldots)$ gives the ordered list of atoms of $\mu$ and $s_0 = 1- \sum_{i=1}^\infty s_i$. \qed

\medskip We note that Kingman's original proof relies on a martingale argument, which is in line with the modern proofs of De Finetti's theorem (see, e.g., Durrett \cite{durrett}, (6.6) in Chapter 4). The interested reader is referred to \cite{aldous st flour} and \cite{pitman stflour}, both of which contain a wealth of information about the subject.

\medskip This theorem has several interesting and immediate consequences: if $\Pi$ is any exchangeable random partition, then the only finite blocks of $\Pi$ are the singletons, almost surely. Indeed if a block is not a singleton, then it is infinite and has in fact positive, well-defined asymptotic frequency (or mass), by the law of large numbers. The (random) vector $s \in \cS_0$ can be entirely recovered from $\Pi$: if $\Pi$ has any singleton at all, then a positive proportion of integers are singletons, that proportion is equal to $s_0$. Moreover,  $(s_1, \ldots)$ is the ordered sequence of nondecreasing block masses. In particular, if $\Pi=(B_1, \ldots, )$ then
$$
|B_1| + |B_2| + \ldots = 1- s_0, \ \ a.s.
$$

There is thus a complete correspondence between the random exchangeable partition $\Pi$ and the sequence $s \in \cS_0$:
$$
\Pi \in \cP \longleftrightarrow s \in \cS_0.
$$
\begin{corollary}
This correspondence is a 1-1 map between the law of exchangeable random partitions $\Pi$ and distributions $\mu$ on $\cS_0$. This map is \emph{Kingman's correspondence}.
\end{corollary}
\index{Kingman's correspondence}

Furthermore, this correspondence is continuous when $\cS_0$ is equipped with the appropriate topology: this is the topology associated with pointwise convergence of the ``non-dust" entries: that is, $s^\eps \to s$ as $\eps \to 0$ if and only if, $s_1^\eps \to s_1, \ldots, s_k^\eps \to s_k$, for all $k \ge 1$ (but not necessarily for $k=0$).

\begin{theorem}
Convergence in distribution of the random partitions $(\Pi_\eps)_{\eps>0}$, is equivalent to the convergence in distributions of their ranked frequencies $(s_1^\eps, s_2^\eps, \ldots)_{\eps>0}.$
\end{theorem}

The proof is easy and can be found for instance in Pitman \cite{pitman stflour}, Theorem 2.3. It is easy to see that the correspondence can \emph{not} be continuous with respect to the restriction of the $\ell^1$ metric to $\cS_0$ (think about a state with many blocks of small but positive frequencies and no dust: this is ``close" to the pure dust state from the point of view of pointwise convergence, and hence from the point of view of sampling, but not at all from the point of view of the $\ell^1$ metric). \index{Topology}

\subsection{Size-biased picking}

\subsubsection{Single pick}

When given an exchangeable random partition $\Pi$, it is natural to ask what is the mass of a ``typical" block. If $\Pi$ has only a finite number of blocks, one can choose a block uniformly at random among all blocks present. But when there is an infinite number of blocks, it is not possible to do so. In that case, one may instead consider the block containing a given integer, say 1. The partition being exchangeable, this block may indeed be thought of being a generic or typical block, and the advantage is that this is possible both when there are finitely or infinitely many blocks. Its mass is then (slightly) larger than that of a typical block. When there are only a finite number of blocks, this is expressed as follows. Let $X$ be the mass of the block containing 1, and let $Y$ be the mass of a randomly chosen block of the \rep $\Pi$. Then the reader can easily verify that
\begin{equation}\label{sbp}
\P(X \in dx) = \frac{x}{\E(Y)} \P(Y \in dx), \ \ x>0.
\end{equation}
If a pair of random variables $(X,Y)$ satisfies the relation (\ref{sbp}) we say that $X$ has the size-biased distribution of $Y$. For this reason, here we say that $X$ is the mass of a size-biased picked block.\index{Size-biased! pick}

In terms of the Kingman's correspondence, $X$ has a natural interpretation when there is no dust. In that case, if $\Pi$ is viewed as a random unit partition $s \in \cS_0$, then $X$ is also the length of the segment containing a point uniformly chosen at random on the unit interval.

Not surprisingly, many of the properties of $\Pi$ can be read from the sole distribution of $X$. (Needless to say though, the law of $X$ does not characterize fully that of $\Pi$).
\begin{theorem}\label{T:sbp}
Let $\Pi$ be a random exchangeable partition with ranked frequencies $(P_i)_{i\ge 1}$. Assume that there is no dust almost surely, and let $f$ be any nonnegative function. Then:
\begin{equation}
\label{sbp formula}\E\left(\sum_{i} f(P_i)\right) = \int_0^1 \frac{f(x)}{x} \mu(dx)
\end{equation}
where $\mu$ is the law of the mass of a size-biased picked block $X$.
\end{theorem}

\begin{proof} The proof follows from looking at the function $g(x)=f(x)/x$, and observing that $\E(g(X)) = \E(\sum_{i} P_ig(P_i))$, which itself is a consequence of Kingman's correspondence, since the $P_i$ are simply equal to the coordinates $(s_1, \ldots) $ of the sequence $s \in \cS_0$, and $U_1$ falls in each of them with probability $s_i$.\end{proof}

Thus, from this it follows that the $n\th$ moment of $X$ is related to the sum of the $(n+1)\th$ moments of all frequencies:
\begin{equation}\label{sbp moments}
\E\left( \sum_i P_i^{n+1}\right) = \E(X^n).
\end{equation}
In particular, for $n=1$ we have:
$$
\E(X) = \E\left( \sum_{i} P_i^2\right).
$$
This identity is obvious when one realises that both sides of this equation can be interpreted as the probability that two randomly chosen points fall in the same component. This of course also applies to (\ref{sbp moments}), which is the probability that $n+1$ randomly chosen points are in the same component. The following identity is a useful application of Theorem \ref{T:sbp}:

\begin{theorem} \label{T:sbp bl}Let $\Pi$ be a random exchangeable partition, and let $N$ be the number of blocks of $\Pi$. Then we have the formula:
$$
\E(N) = \E(1/X).
$$
\end{theorem}

To explain the result, note that if we see that the block containing 1 has frequency $\eps>0$ small, then we can expect roughly $1/\eps$ blocks in total (since that would be the answer if all blocks had frequency exactly $\eps$).

\begin{proof} To see this, note that the result is obvious if $\Pi$ has some dust with positive probability, as both sides are then infinite. So assume that $\Pi$ has no dust almost surely, and let $N_n$ be the number of blocks of $\Pi$ restricted to $[n]$. Then by Theorem \ref{T:sbp}:
\begin{align*}
\E(N_n) &= \sum_i \P({\text{part $i$ is chosen among the first $n$ picks}})\\
&= \sum_i\E\left( 1-(1-P_i)^n\right)\\
&= \E(f_n(X)),
\end{align*}
say, where $$f_n(x) = \frac{1-(1-x)^n}x.$$
Letting $n\to \infty$, since $X>0$ almost surely because there is no dust, $f_n(X) \to 1/X$ almost surely. This convergence is also monotone, so we conclude
$$
\E(N) = \E(1/X)
$$
as required. \end{proof}

Theorem \ref{T:sbp bl} will often guide our intuition when studying the small-time behaviour of coalescent processes that come down from infinity (rigorous definitions will be given shortly). Basically, this is the study of the coalescent processes close to the time at which they experience a ``big-bang" event, going from a state of pure dust to a state made of finitely many solid blocks (i.e., with positive mass). Close to this time, we have a very large number of small blocks. Any information on $N$ can then be hoped to carry onto $X$, and conversely.


\subsubsection{Multiple picks, size-biased ordering}

Let $X=X_1$ denote the mass of a size-biased picked block. One can then define further statistics which refine our description of $\Pi$. Recall that if $\Pi= (B_1, B_2, \ldots)$ with blocks ordered according to their least elements, then $X_1= |B_1|$ is by definition the mass of a size-biased picked block. Define similarly, $X_2=|B_2|, \ldots, X_n = |B_n|$, and so on. Then $(X_1, \ldots)$ corresponds to sampling without replacement the possible blocks of $\Pi$, with a size bias at every step.

Note that if $\Pi$ has no dust, then $(X_1, \ldots,)$ is just a reordering of the sequence $(s_1, \ldots,)$ where $s$ denotes the ranked frequencies of $\Pi$, or equivalently the image of $\Pi$ by Kingman's correspondence. That is, there exists a permutation $\sigma: \N \to \N$ such that
$$
X_i = s_{\sigma(i)},\  i \ge 1.
$$
This permutation is the \emph{size-biased ordering} of $s$.\index{Size-biased! ordering} It satisfies:
$$
\P( \sigma(1)=j | s) = s_j
$$
Moreover, given $s$, and given $\sigma(1), \ldots, \sigma(i-1)$, we have:
$$
\P( \sigma(i) = j | s, \sigma(1), \ldots, \sigma(i-1)) = \frac{s_j}{1- s_{\sigma(1)} - \ldots - s_{\sigma(i-1)}}.
$$
Although slightly more complicated, the size-biased ordering of $s$, $(X_1, \ldots)$, is often more natural than the nondecreasing rearrangement which defines $s$.

As an exercise, the reader is invited to verify that Theorem \ref{T:sbp bl} can be generalised to this setup to yield: if $N$ is the number of ordered $k$-uplets of distinct blocks in the \rep $\Pi$, then
\begin{equation}
\label{kuplets}
\E(N) = \E\left(\frac1{X_1 \ldots X_k} \right).
\end{equation}
This is potentially useful to establish limit theorems for the distribution of the number of blocks in a coalescent, but this possibility has not been explored to this date.

\subsection{The Poisson-Dirichlet random partition}\index{Poisson-Dirichlet}

We are now going to spend some time to describe a particular family of random partitions called the Poisson-Dirichlet partitions. These partitions are ubiquitous in this field, playing the role of the normal random variable in standard probability theory. Hence they arise in a huge variety of contexts: not only coalescence and population genetics (which is our main reason to talk about them in these notes), but also random permutations, number theory \cite{DonellyGrimmett}, Brownian motion \cite{pitman stflour}, spin glass models \cite{bovkur}, random surfaces \cite{gamburd}... In its most general incarnation, this is a two parameter family of random partitions, and the parameters are usually denoted by $(\alpha, \theta)$. However, the most interesting cases occur when either $\alpha = 0$ or $\theta=0$, and so to keep these notes as simple as possible we will restrict our presentation to those two cases.

\subsubsection{Case $\alpha = 0$.}

We start with the case $\alpha=0, \theta >0$. We recall that a random variable $X$ has the $Beta(a,b)$ distribution (where $a,b>0$) if the density at $x$ is:
\begin{equation}\label{Betadistr0}
\frac{\P(X \in dx) }{dx}= \frac{\Gamma(a+b)}{\Gamma(a)\Gamma(b)}x^{a-1}(1-x)^{b-1}, \ \ 0<x<1.
\end{equation}
\index{Beta distribution} Thus the Beta$(1, \theta)$ distribution ($\theta>0$) is the distribution on $(0,1)$ with density $\theta(1-x)^{\theta-1}$ and this is uniform if $\theta =1$. If $a,b \in \N$ the Beta$(a,b)$ distribution has the following interpretation: take $a+b$ independent standard exponential random variables, and consider the ratio of the sum of the first $a$ of them compared to the total sum. Alternatively, drop $a+b$ random points in the unit interval and order them increasingly. Then the position of the $a\th$ point is a Beta$(a,b)$ random variable.

\begin{definition} \emph{(Stick-breaking construction, $\alpha=0$.)}
The Poisson-Dirichlet random partition is the paintbox partition associated with the nonincreasing reordering of the sequence
\begin{align}
P_1&=W_1,\nonumber \\
P_2 &= (1-P_1)W_2,\nonumber \\
\vdots & \nonumber \\
P_{n+1} &=(1-P_1 - \ldots - P_n) W_n, \label{pd-sbp}
\end{align}
where the $W_i$ are i.i.d. random variables
$$
W_i= \text{Beta}(1, \theta).
$$
We write $\Pi \sim PD(0, \theta)$.
\end{definition}
\index{Stick-breaking}
To explain the name of this construction, imagine we start with a stick of unit length. Then we break the stick in two pieces, $W_1$ and $1-W_1$. One of these two pieces ($W_1$), we put aside and will never touch again. To the other, we apply the previous construction repeatedly, each time breaking off a piece which is Beta-distributed on the current length of the stick. In particular, note that when $\theta=1$, the pieces are uniformly distributed.

While the above construction tells us what the asymptotic frequencies of the blocks are, there is a much more visual and appealing way of describing this partition, which goes by the name of ``Chinese restaurant process". Let $\Pi_n$ be the partition of $[n]$ defined inductively as follows: initially, $\Pi_1$ is the just the trivial partition $\{\{1\}\}$. Given $\Pi_{n}$, we build $\Pi_{n+1}$ as follows. The restriction of $\Pi_{n+1}$ to $[n]$ will be exactly $\Pi_n$, hence it suffices to assign a block to $n+1$. With probability $\theta/(n+\theta)$, $n+1$ starts a new block. Otherwise, $n+1$ is assigned to a block of size $m$ with probability $m/(n+\theta)$. This can be summarized as follows:
\begin{equation}\label{CRP}
\begin{cases}
\text{ start new block:} & \text{with probability } \frac\theta{n+\theta}\\
\text{ join block of size $m$:} & \text{with probability }\frac{m}{n+\theta}
\end{cases}
\end{equation}
\index{Chinese Restaurant Process}
This defines a (consistent) family of partitions $\Pi_n$, hence there is no problem in extending this definition to a random partition $\Pi$ of $\cP$ such that $\Pi|_{[n]} = \Pi_n$ for all $n\ge 1$: indeed, if $i , j \ge 1$, it suffices to say whether $i\sim j$ or not, and in order to be able to decide this, it suffices to check on $\Pi_n$ where $n = \max(i,j)$. This procedure thus uniquely specifies $\Pi$.

The name ``Chinese Restaurant Process" comes from the following interpretation in the case $\theta =1$: customers arrive one by one in an empty restaurant which has round tables. Initially, customer 1 sits by himself. When the $(n+1)\th$ customer arrives, she chooses uniformly at random between sitting at a new table or sitting directly to the right of a given individual. The partition structure obtained by identifying individuals sitted at the same table is that of the Chinese Restaurant Process.

\begin{theorem} \label{T:CRP=PD}
The random partition $\Pi$ obtained from the Chinese restaurant process (\ref{CRP}) is a Poisson-Dirichlet random partition with parameters $(0, \theta)$. In particular, $\Pi$ is exchangeable. Moreover, the size-biased ordering of the asymptotic block frequencies is the one given by the stick-breaking order (\ref{pd-sbp}).
\end{theorem}

\begin{proof}
The proof is a simple (and quite beautiful) application of P\'olya's urn theorem. In P\'olya's urn, we start with one red ball and a number $\theta$ of black balls. At each step, we choose one of the balls uniformly at random in the urn, and put it back in the urn along with one of the same colour. P\`olya's classical result says that the asymptotic proportion of red balls converges to a Beta$(1, \theta)$ random variable. Note also that this urn model may also be formally defined even when $\theta$ is not an integer, and the result stays true in this case.

Now, coming back to the Chinese Restaurant process, consider the block containing 1. Imagine that to each $1 \le i \le n$ is associated a ball in an urn, and that this ball is red if $i \sim 1$, and black otherwise, say. Note that, by construction, if at stage $n$, $B_1$ contains $r\ge 1$ integers, then as the new integer $n+1$ is added to the partition, it joins $B_1$ with probability $r/(n+\theta)$ and does not with the complementary probability. Assigning the colour red to $B_1$ and black otherwise, this is the same as thinking that there are $r$ red balls in the urn, and $n-r +\theta$ black balls, and that we pick one of the balls at random and put it back along with one of the same colour (whether or not this is to join one of the existing blocks or to create a new one!) Initially (for $n=1$), the urn contains 1 red ball and $\theta$ black balls. Thus the proportion of red balls in the urn, $X_n(1)/n$, satisfies:
$$
\frac{X_n(1)}n \underset{n \to \infty}\longrightarrow  W_1, \ \ a.s.
$$
where $W_1$ is a Beta$(1, \theta)$ random variable. (This result is usually more familiar in the case where $\theta=1$, in which case $W_1$ is simply a uniform random variable).\index{P\'olya's urn}

Now, observe that the stick breaking construction property is in fact a consequence of the Chinese restaurant process construction (\ref{CRP}). Let $i_1=1$ and let $i_2$ be the first $i$ such that $i$ is not in the same block as 1. If we ignore the block $B_1$ containing 1, and look at the next block $B_2$ (which contains $i_2$), it is easy to see by the same P\'olya urn argument that the asymptotic fraction of integers $i \in  B_2$ among those that are not in $B_1$, is a random variable $W_2$ with the $\text{Beta}(1, \theta)$ distribution. Hence $|B_2| = (1-P_1) W_2$. Arguing by induction as above, one obtains that the blocks $(B_1, B_2, \ldots)$, listed in order of appearance, satisfy the strick breaking construction (\ref{pd-sbp}).

It remains to show exchangeability of the partition, but this is a consequence of the fact that, in P\'olya's urn, given the limiting proportion $W$ of red balls, the urn can be realised as an i.i.d. coin-tossing with heads probability $W$. It is easy to see from this observation that we get exchangeability.
\end{proof}

As a consequence of this remarkable construction, there is an exact expression for the probability distribution of $\Pi_n$. As it turns out, this formula will be quite useful for us. It is known (for reasons that will become clear in the next chapter) as Ewens' sampling formula.

\begin{theorem} \label{T:esf} Let $\pi$ be any given partition of $[n]$, whose block size are $n_1, \ldots, n_k$.
$$
\P(\Pi_n = \pi) = \frac{\theta^k}{(\theta) \ldots (\theta+n-1)} \prod_{i=1}^k (n_i-1) !
$$
\end{theorem}
\index{Ewens' sampling formula}

\begin{proof}
This formula is obvious by induction on $n$ from the Chinese restaurant process construction. It could also be computed directly through some tedious integral computations (``Beta-Gamma" algebra).
\end{proof}

\subsubsection{Case $\theta=0$.}

Let $0<\alpha<1$ and let $\theta =0$.

\begin{definition} \emph{(Stick-breaking construction, $\theta=0$).} The Poisson-Dirichlet random variable with parameters $(\alpha,0)$ is the random partition obtained from the stick breaking construction, where at the $i\th$ step, the piece to be cut off from the stick has distribution $W_i \sim Beta(1-\alpha, i \alpha)$. That is,
\begin{align}
P_1&=W_1,\nonumber \\
\vdots & \nonumber \\
P_{n+1} &=(1-P_1 - \ldots - P_n) W_n, \label{pd-sbp2}
\end{align}
\end{definition}

There is also a ``Chinese restaurant process" construction in this case. The modification is as follows. If $\Pi_n$ has $k$ blocks of size $n_1, \ldots, n_k$, $\Pi_{n+1}$ is obtained by performing the following operation on $n+1$:
\begin{equation}\label{CRP2}
\begin{cases}
\text{ start new block:} & \text{with probability } \frac{k \alpha}{n}\\
\text{ join block of size $m$:} & \text{with probability }\frac{m-\alpha}{n}.
\end{cases}
\end{equation}
\index{Chinese Restaurant Process}
It can be shown, using urn techniques for instance, that this construction yields the same partition as the paintbox partition associated with the stick breaking process (\ref{pd-sbp2}).

As a result of this construction, Ewens' sampling formula can also be generalised to this setting, and becomes:
\begin{equation}\label{ESF2}
\P(\Pi_n = \pi) = \frac{\alpha^{k-1}(k-1)!}{(n-1)!} \prod_{i=1}^k (1-\alpha) \ldots (n_i-\alpha)
\end{equation}
\index{Ewens' sampling formula}
where $\pi$ is any given partition of $[n]$ into blocks of sizes $n_1, \ldots, n_k$.

\subsubsection{A Poisson construction}
\index{Poissonian construction!Poisson-Dirichlet}

At this stage, we have seen essentially two constructions of a Poisson-Dirichlet random variable with $\theta = 0$ and $0<\alpha <1$. The first one is based on the stick-breaking scheme, and the other on the Chinese Restaurant Process. Here we discuss a third construction which will come in very handy at several places in these notes, and which is based on a Poisson process. More precisely, let $0< \alpha <1$ and let $\cM$ denote the points of a Poisson random measure on $(0, \infty)$ with intensity $\mu(dx) = x^{-\alpha -1}dx$:
$$
\cM (dx) = \sum_{i\ge 1} \delta_{Y_i}(dx).
$$
In the above, we assume that the $Y_i$ are ranked in decreasing order, i.e., $Y_1$ is the largest point of $\cM$, $Y_2$ the second largest, and so on. This is possible because a.s. $\cM$ has only a finite number of points in $(\eps, \infty)$ (since $\alpha >0$). It also turns out that, almost surely,
\begin{equation}\label{sumfinite}
\sum_{i=1}^\infty Y_i < \infty.
\end{equation}
Indeed, observe that
$$
\E\left(\sum_{i=1}^\infty Y_i\indic{Y_i \le 1}\right) = \int_0^1 x \mu(dx) < \infty
$$
and so $\sum_{i=1}^\infty Y_i\indic{Y_i \le 1}< \infty$ almost surely. Since there are only a finite number of terms outside of (0,1), this proves (\ref{sumfinite}). We may now state the theorem we have in mind:
\begin{theorem}
  \label{T:PDpoisson}
  For all $n\ge 1$, let $P_n = Y_n / \sum_{i=1}^\infty Y_i $. Then the distribution of $(P_n,n\ge 1)$ is that of a Poisson-Dirichlet random variable with parameters $\alpha$ and $\theta = 0$.
\end{theorem}

The proof is somewhat technical (being based on explicit density calculations) and we do not include it in these notes. However we refer the reader to the paper of Perman, Pitman and Yor \cite{PPY} where this result is proved, and to section 4.1 of Pitman's notes \cite{pitman stflour} which contains some elements of the proof.

\medskip We also mention that there exists a similar construction in the case $\alpha = 0$ and $\theta>0$. The corresponding intensity of the Poisson point process $\cM$ should then be chosen as $\rho(dx) = \theta x^{ -1}e^{-x} dx$, which was Kingman's original definition of the Poisson-Dirichlet distribution \cite{Kingman75}. See also section 4.11 in \cite{abt} and Theorem 3.12 in \cite{pitman stflour}, where the credit is given to Ferguson \cite{Ferguson} for this result.

\subsection{Some examples}

As an illustration of the usefulness of the Poisson-Dirichlet distribution, we give two classical examples of situations in which they arise, which are on the one hand, the cycle decomposition of random permutations, and on the other hand, the factorization into primes of a ``random" large integer. A great source of information for these two examples is \cite[Chapter 1]{abt}, where much more is discussed. In the next chapter, we will focus (at length) in another incarnation of this partition, which is that of population genetics via Kingman's coalescent. In Chapter \ref{S:spin} we will encounter yet another one, which is within the physics of spin glasses.

\subsubsection{Random permutations.}

Let $\cS_n$ be the set of permutations of $S=\{1, \ldots, n\}$. If $\sigma \in \cS_n$, there is a natural action of $\sigma$ onto the set $S$, which partitions $S$ into orbits. This partition is called the cycle decomposition of $\sigma$. For instance, if
$$
\sigma =
\left(\begin{array}{ccccccc}
1 & 2 & 3 & 4 & 5 & 6 & 7\\
3 & 2 & 4 & 1 & 7 & 5 & 6
\end{array}\right)
$$
then the cycle decomposition of $\sigma$ is
\begin{equation}\label{cycle}
\sigma = (1 \ 3 \ 4)(2) (5 \ 7 \ 6).
\end{equation}
This simply means that 1 is mapped into 3, 3 into 4 and 4 back into 1, and so on for the other cycles. Cycles are the basic building blocks of permutations, much as primes are the basic building blocks of integers. This decomposition is unique, up to order of course. If we further ask the cycles to be ordered by increasing least elements (as above), then this representation is unique. Let $\sigma$ be a randomly chosen permutation (i.e., chosen uniformly at random). The following result describes the limiting behaviour of the cycle decomposition of $\sigma$. Let ${\bf L}^{(n)} = (L_1, L_2, \ldots, L_k)$ denote the cycle lengths of $\sigma$, ordered by their least elements, and let ${\bf X}^{(n)} = (L_1/n, \ldots, L_k/n)$ be the normalized vector, which tiles the unit interval $(0,1)$.

\begin{theorem} \label{T:perm}There is the following convergence in distribution:
$$
{\bf X}^{(n)} \longrightarrow_d (P_1, P_2, \ldots)
$$
where $(P_1, \dots, )$ are the asymptotic frequencies of a $PD(0,1)$ random variable in size-biased order.
\end{theorem}\index{Random permutations}

(Naturally the convergence in distribution is with respect to the topology on $\cS_0$ defined earlier, i.e., pointwise convergence of positive mass entries: in fact, this convergence also holds for the restriction of the $\ell_1$ metric).

\begin{proof}\index{Stick-breaking}
There is a very simple proof that this result holds true. The proof relies on a construction due to Feller, which shows that the stick-breaking property holds even at the discrete level. The cycle decomposition of $\sigma$ can be realised as follows. Start with the cycle containing 1. At this stage, the permutation looks like
$$
\sigma = (1
$$
and we must choose what symbol to put next. This could be any number of $\{2, \ldots, n\}$ or the symbol which closes the cycle $``)"$. Thus there are $n$ possibilities at this stage, and the Feller construction is to choose among all those uniformly at random. Say that our choice leads us to:
$$
\sigma = (1\ 5
$$
At this stage, we must choose among a number of possible symbols: every number except 1 and 5 are allowed, and we are allowed to close the cycle. Again, one must choose uniformly among those possibilities, and do so until one eventually chooses to close the cycle. Say that this happens at the fourth step:
$$
\sigma= (1\ 5\ 2)
$$
At this point, to pursue the construction we open a new cycle with the smallest unused number, in this case 3. Thus the permutation looks like:
$$
\sigma=(1\ 5\ 2)(3
$$
At each stage, we choose uniformly among all legal options, which are to close the current cycle or to put a number which doesn't appear in the previous list.

Then it is obvious that the resulting permutation is random: for instance, if $n=7$, and
$
\sigma_0 = (1 \ 3 \ 4)(2) (5 \ 7 \ 6),
$
then
\begin{align*}
\P(\sigma= \sigma_0) &= \frac17 \cdot \frac16 \cdot \ldots \cdot \frac12 \cdot \frac11  = \frac1{7!}
\end{align*}
because at the $k\th$ step of the construction, exactly $k$ numbers have already been written and thus there $n-k+1$ symbols available (the +1 is for closing the cycle). Thus the Feller construction gives us a way to generate random permutations (which is an extremely convenient algorithm from a practical point of view, too).

Now, note that $L_1$, which is the length of the first cycle, has a distribution which is uniform over $\{1, \ldots, n\}$. Indeed, $1 \le k \le n$, the probability that $L=k$ is the probability that the algorithm chooses among $n-1$ options out of $n$, and then $n-2$ out of $n-1$, etc., until finally at the $k\th$ step the algorithm chooses to close the cycle (1 option out of $n-k+1$). Cancelling terms, we get:
\begin{align*}
\P(L=k) &= \frac{n-1}n \cdot \frac{n-2}{n-1} \cdot \ldots \cdot  \frac{n-k+1}{n-k+2}\frac1{n-k+1}\\
& = \frac1n.
\end{align*}
One sees that, similarly, given $L_1$ and $\{L_1<n\}$, $L_2$ is uniform on $\{1, \ldots, n - L_1\}$, by construction. More generally, given $(L_1, \ldots, L_k)$ and given that $\{L_1+ \ldots, + L_k<n\}$, we have:
\begin{equation}
L_{k+1} \sim \text{Uniform on } \{1, \ldots, n- L_1 - \ldots - L_k\}
\end{equation}
which is exactly the analogue of (\ref{pd-sbp}). From this one deduces Theorem \ref{T:perm} easily.
\end{proof}

\subsubsection{Prime number factorisation.}

Let $n \ge 1$ be a large integer, and let $N$ be uniformly distributed on $\{1, \ldots, n\}$. What is the prime factorisation of $N$? Recall that one may write
\begin{equation}\label{prime}
N= \prod_{p \in \cP} p^{\alpha_p}
\end{equation}
where $\cP$ is the set of prime numbers and $\alpha_p $ are nonnegative integers, and that this decomposition is unique. To transfer to the language of partitions, where we want to add the parts rather than multiply them, we take the logarithms and define:
$$
L_1= \log p_1, \ldots, L_k = \log p_k.
$$
Here the $p_i$ are such that $\alpha_p>0$ in (\ref{prime}), and each prime $p$ appears $\alpha_p$ times in this list. We further assume that $L_1 \ge \ldots \ge L_k$.

\begin{theorem}\label{T:prime}
Let ${\bf X}^{(n)} = ( L_1 / \log n , \ldots, L_k / \log n)$. Then we have convergence in the sense of finite-dimensional distributions:
$$
{\bf X}^{(n)} \longrightarrow (\tilde P_1, \ldots)
$$
where $(\tilde P_1, \ldots)$ is the decreasing rearrangement of the asymptotic frequencies of a $PD(0,1)$ random variable.
\end{theorem}

In particular, large prime factors appear each with multiplicity 1 with high probability as $n \to \infty$, since the coordinates of a $PD(0,1)$ random variable are pairwise distinct almost surely.
See (1.49) in \cite{abt}, which credits Billingsley \cite{billingsley72} for this result, and \cite{DonellyGrimmett} for a different proof using size-biased ordering\index{Size-biased! ordering}.

\subsubsection{Brownian excursions.}

Let $(B_t, t\ge 0)$ be a standard Brownian motion, and consider the random partition obtained by performing the paintbox construction to the tiling of $(0,1)$ defined by $Z \cap (0,1)$, where
$$
Z=\{t \ge0: B_t =0\}
$$
is the zero-set of $B$.

\begin{figure}
\includegraphics{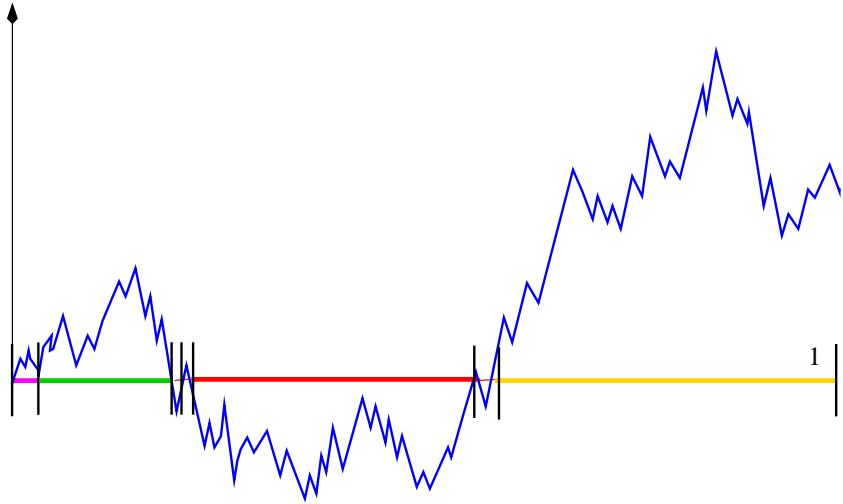}
\caption{The tiling of $(0,1)$ generated by the zeros of Brownian motion.}
\end{figure}

Let $(P_1, \ldots, )$ be the size of the tiles in size-biased order.

\begin{theorem} \label{T:bm}
$(P_1, \ldots)$ has the distribution of the asymptotic frequencies of a $PD(\frac12, 0)$ random variable.
\end{theorem}

\begin{proof}
The proof is not complicated but requires knowledge of excursion theory, which at this level we want to avoid, since this is only supposed to be an illustrating example. The main step is to observe that at the inverse local time $\tau_1= \inf\{t>0: L_t=1\}$, the excursions lengths are precisely a Poisson point process with intensity $\rho(dx) = x^{-\alpha -1}$ with $\alpha=1/2$. This is an immediate consequence It\^o's excursion theory for Brownian motion and of the fact that It\^o's measure $\nu$ gives mass
$$
\nu(e: |e| \in dt) = C t^{-3/2}
$$
for some $C>0$. From this and Theorem \ref{T:PDpoisson}, it follows that the normalized excursion lengths at time $\tau_1$ have the $PD(\frac12,0)$ distribution. One has to work slightly harder to get this at time 1 rather than at time $\tau_1$. More details and references can be found in \cite{PitmanYor}, together with a wealth of other properties of Poisson-Dirichlet distributions.
\end{proof}

\subsection{Tauberian theory of random partitions}

\subsubsection{Some general theory}

Let $\Pi$ be an exchangeable random partition with ranked frequencies $(P_1, \ldots)$, which we assume has no dust almost surely.
In applications to population genetics, we will often be interested in exact asymptotics of the following quantities:

\begin{enumerate}
\item $K_n$, which is the number of blocks of $\Pi_n$ (the restriction of $\Pi$ to $[n]$).

\item $K_{n,r}$, which is the number of blocks of size $r$, $1\le r\le n$.
\end{enumerate}

Obtaining asymptotics for $K_n$ is usually easier than for $K_{n,r}$, for instance due to monotonicity in $n$. But there is a very nice result which relates in a surprisingly precise fashion the asymptotics of $K_{n,r}$ (for any fixed $r\ge1$, as $n \to \infty$) to those of $K_{n}$. This may seem surprising at first, but we stress that this property is of course a consequence of the exchangeability of $\Pi$ and Kingman's representation. The asymptotic behaviour of these two quantities is further tied to another quantity, which is that of the asymptotic speed of decay of the frequencies towards 0. The right tool for proving these results is a variation of Tauberian theorems, which take a particularly elegant form in this context. The main result of this section (Theorem \ref{T:regvar}) is taken from \cite{GHP}, which also contains several other very nice results.

\begin{theorem}\label{T:regvar}
Let $0<\alpha<1$. There is equivalence between the following properties:
\begin{itemize}
\item[(i)] $P_j \sim Z j^{-\alpha}$ almost surely as $j \to \infty$, for some $Z>0$.

\item[(ii)] $K_n \sim D n^\alpha$ almost surely as $n \to \infty$, for some $D>0$.
\end{itemize}
Furthermore, when this happens, $Z$ and $D$ are related through
$$
Z=\left(\frac{D}{\Gamma(1-\alpha)}\right)^{1/\alpha},
$$
and we have:\\
(iii) For any $r\ge 1$, $K_{n,r} \sim \frac{\alpha(1-\alpha)\ldots (r-1-\alpha)}{r!} D n^\alpha$ as $n \to \infty$.

\end{theorem}

The result of \cite{GHP} is actually more general, and is valid if one replaces $D$ by a slowly varying sequence $\ell_n$. Recall that a function $f$ is slowly varying near $\infty$ if for every $\lambda>0$,
\begin{equation}\label{slow}\index{Regular variation} \index{Slow variation}
\lim_{x \to \infty} \frac{f(\lambda x)}{f(x)} =1.
\end{equation}
\noindent The prototypical example of a slowly varying function is the logarithm function. Any function $f$ which may be written as $f(x) = x^\alpha \ell(x)$, where $\ell(x)$ is slowly varying, is said to have regular variation with index $\alpha$. A sequence $c_n$ is regularly varying with index $\alpha$ if there exists $f(x)$ such that $c_n =f(n)$ and $f$ is regularly varying with index $\alpha$, near $\infty$.

\begin{proof} (sketch)
The main idea is to start from Kingman's representation theorem, and to imagine that the $P_j$ are given, and then see $\Pi_n$ as the partition generated by sampling with replacement from $(P_j)$. Thus in this proof, we work conditionally on $(P_j)$, and all expectations are (implicitly) conditional on these frequencies.

Rather than looking at the partition obtained after $n$ samples, it is more convenient to look at it after $N(n)$ samples, where $N(n)$ is a Poisson random variable with mean $n$. The superposition property of Poisson random variables implies that one can imagine that each block $j$ with frequencies $P_j$ is discovered (i.e., sampled) at rate $P_j$. Since we assume that there is no dust, this means $\sum_{j\ge 1} P_j =1$ almost surely, and hence the total rate of discoveries is indeed 1. Let $K(t)$ be the total number of blocks of the partition at time $t$, and let $K_r(t)$ be the total number of blocks of size $r$ at time $t$. Standard Poissonization arguments imply:
$$
\frac{K(n)}{K_n} \to 1, \ \ a.s.
$$
and
$$
\frac{K_r(n)}{K_{n,r}} \to 1, \ \ a.s.
$$
That is, we may as well look for the asymptotics in continuous Poisson time rather than in discrete time. For this we will use the following law of large numbers, proved in Proposition 2 of \cite{GHP}.
\begin{lemma}\label{L:lln}
For arbitrary $(P_j)$,
\begin{equation}\label{llnK}
\frac{K(t)}{\E (K(t)| (P_j)_{j\ge 1})} \to 1, \ \ a.s.
\end{equation}
\end{lemma}

\begin{proof}
The proof is fairly simple and we reproduce the arguments of \cite{GHP}. Recall that we work conditionally on $(P_j)$, so all the expectations in the proof of this lemma are (implicitly) conditional on these frequencies. Note first that if $\Phi(t)= \E(K(t))$, then
$$
\Phi(t) = \sum_{j\ge 1} 1- e^{-P_j t}
$$
and similarly if $V(t) = \var K(t)$, we have (since $K(t)$ is the sum of independent Bernoulli variables with parameter $1-e^{-P_j t}$):
\begin{align*}
V(t)&= \sum_{j} e^{-P_j t}(1-e^{-P_jt}) \\
&= \sum_{j \ge 1} e^{-P_j t} - e^{-2P_j t}\\
& = \Phi(2t) - \Phi(t).
\end{align*}
But note that $\Phi$ is convex: indeed, by stationarity of Poisson processes, the expected number of blocks discovered during $(t, t+s]$ is $\Phi(s)$, but some of those blocks discovered during the interval $(t, t+s]$ were in fact already known prior to $t$, and hence don't count in $K(t+s)$. Thus
$$
V(t) < \Phi(t)
$$
and by Chebyshev's inequality:
$$
\P\left(\left| \frac{K(t)}{\Phi(t)} -1 \right| >\eps\right) \le \frac1{\eps^2 \Phi(t)}.
$$
Taking a subsequence $t_m$ such that $m^2 \le \Phi(t_m) < (m +1)^2$ (which is always possible), we find:
$$
\P\left(\left| \frac{K(t_m)}{\Phi(t_m)} -1 \right| >\eps\right) \le \frac1{\eps^2 m^2}.
$$
Hence by the Borel-Cantelli lemma, ${K(t_m)}/{\Phi(t_m)} \longrightarrow 1$ almost surely as $m \to \infty$. Using monotonicity of both $\Phi(t)$ and $K(t)$, we deduce
$$
\frac{K(t_{m+1})}{\Phi(t_m)}\le \frac{K(t)}{\Phi(t)}< \frac{K(t_{m+1})}{\Phi(t_m)} .
$$
Since $\Phi(t_{m+1})/\Phi(t_m) \to 1$, this means both the left-hand side and the right-hand side of the inequality tend to 1 almost surely as $m \to \infty$. Thus (\ref{llnK}) follows.
\end{proof}

Once we know Lemma \ref{L:lln}, note that
$$
\E K(t) =: \Phi(t)= \int_0^1 (1-e^{-tx}) \nu(dx)
$$
where
$$
\nu(dx):= \sum_{j \ge 1} \delta_{P_j}(dx).
$$
Fubini's theorem implies:
\begin{equation}\label{Laplace}
\E K(t) = t \int_0^\infty e^{-tx} \bar \nu(x)dx
\end{equation}
where $\bar \nu (x) = \nu([x, \infty))$, so the equivalence between (i) and (ii) follows from classical Tauberian theory for the monotone density $\bar \nu(x)$, together with (\ref{llnK}). That this further implies (iii), is a consequence of the fact that
\begin{align}
\E K_r(t) &= \frac{t^r}{r!} \int_0^1 x^r e^{-tx} \nu(dx)\nonumber\\
&= \frac{t^r}{r!} \int_0^1 e^{-tx} \nu_r(dx),\label{Laplace2}
\end{align}
where we have denoted
$$
\nu_r(dx) = \sum_{j \ge 1} P_j^r \delta_{P_j}(dx).
$$
Integrating by parts gives us:
$$
\nu_r([0,x]) = -x^d \bar \nu(x) + r \int_0^x u^{r-1} \bar \nu(x)dx.
$$
Thus, by application of Karamata's theorem \cite{FellerII} (Theorem 1, Chapter 9, Section 8), we get that the measure $\nu_r$ is also regularly varying, with index $r-\alpha$: assuming that $\bar \nu(x) \sim \ell(x) x^{-\alpha}$ as $x \to 0$,
$$
\nu_r([0,x]) \sim \frac{\alpha}{r-\alpha}x^{r-\alpha} \ell(x),
$$
by application of a Tauberian theorem to (\ref{Laplace2}), we get that:
\begin{equation}\label{EPhi_r}
\Phi_r(t) \sim \frac{\alpha \Gamma(r-\alpha)}{r!} t^\alpha \ell(t).
\end{equation}
A refinement of the method used in Lemma \ref{L:lln} shows that
\begin{equation}\label{llnKr}
\frac{K_r(n)}{\E (K_r(n)|(P_j)_{j \ge 1})} \to 1, \ \ a.s.
\end{equation}
in that case. Putting together (\ref{EPhi_r}) and (\ref{llnKr}), we obtain (iii). \end{proof}

As an aside, note that (as pointed out in \cite{GHP}), (\ref{llnKr}) needs not hold for general $(P_j)$, as it might not even be the case that $\E (K_r(n)) \to \infty$.

\subsubsection{Example}

As a prototypical example of a partition $\Pi$ which verifies the assumptions of Theorem \ref{T:regvar}, we have the Poisson-Dirichlet$(\alpha, 0)$ partition.

\begin{theorem}\label{T:PDregvar} Let $\Pi$ be a $PD(\alpha, 0)$ random partition. Then there exists a random variable $S$ such that
$$
\frac{K_n}{n^\alpha} \longrightarrow S
$$
almost surely. Moreover $S$ has the Mittag-Leffer distribution:
$$
\P(S \in dx)= \frac1{\pi \alpha} \sum_{k=1}^\infty \frac{(-1)^{k+1}}{k!} \Gamma(\alpha k +1) s^{k-1}\sin(\pi \alpha k).
$$
\end{theorem}\index{Mittag-Leffer distribution}

\begin{proof}
We start by showing that $n^\alpha$ is the right order of magnitude for $K_n$. First, we remark that the expectation $u_n = \E(K_n)$ satisfies, by the Chinese restaurant process construction of $\Pi$, that
$$
u_{n+1} - u_n = \E\left(\frac{K_n \alpha}n\right) = \frac{\alpha u_n}{ n}.
$$
This implies, using the formula $\Gamma(x+1) = x \Gamma(x)$ (for $x>0$):
\begin{align*}
u_{n+1} & = u_n(1+ \frac\alpha{n}) \\
&=(1 + \frac{\alpha}n) (1+ \frac{\alpha}{n-1}) \ldots (1 + \frac{\alpha}1) u_1 \\
&= \frac{\Gamma(n+1+\alpha)}{\Gamma(n+1) \Gamma(1+\alpha)}.
\end{align*}
Thus, using the asymptotics $\Gamma(x+a) \sim x^a \Gamma(x)$,
$$
u_n = \frac{\Gamma(n+\alpha)}{\Gamma(n) \Gamma(1+\alpha)} \sim \frac{n^\alpha}{\Gamma(1+\alpha)}.
$$
(This appears on p.69 of \cite{pitman stflour}, but using a more combinatorial approach).

This tells us the order of magnitude for $K_n$. To conclude to the almost sure behaviour, a martingale argument is needed (note that we may not apply Lemma \ref{L:lln} as this result is only conditional on the frequencies $(P_j)_{j\ge 1}$ of $\Pi$.) This is outlined in Theorem 3.8 of \cite{pitman stflour}.
\end{proof}

Later (see, e.g., Theorem \ref{T:Betaallelic}), we will see other applications of this Tauberian theory to a concrete example arising in population genetics.

\newpage

\section{Kingman's coalescent}

In this chapter, we introduce Kingman's coalescent and study its first properties. This leads us to the notion of coming down from infinity, which is a ``big bang" like phenomenon whereby a partition consisting of pure dust coagulates instantly into solid fragments. We show the relevance of Kingman's coalescent to population models by studying its relationship to the Moran model and the Wright-Fisher diffusion and state a result of universality known as M\"ohle's lemma. We derive some theoretical and practical implications of this relationship, such as the notion of \emph{duality} between Kingman's coalescent and the Wright-Fisher diffusion. We then show that the Poisson-Dirichlet distribution describes the allelic partition associated with Kingman's coalescent. As a consequence, Ewens's sampling formula describes the typical genetic variation (or polymorphism in biological terms) of a sample of a population. This result is one of the cornerstones of mathematical population genetics, and we show a few applications.

\subsection{Definition and construction}

\subsubsection{Definition}

Kingman's coalescent is perhaps the simplest stochastic process of coalescence. It is easier to define it as a process with values in $\cP$, although by Kingman's correspondence there is an equivalent version in $\cS_0$. Let $n\ge 1$. We start by defining a process $(\Pi^n_t, t \ge 0)$ with values in the space $\cP_n$ of partitions of $[n] = \{1,\ldots, n\}$. This process is defined by saying that:
\begin{enumerate}
\item Initially $\Pi^n_0 $ is the trivial partition in singletons.

\item $\Pi^n$ is a strong Markov process in continuous time, where the transition rates $q(\pi, \pi')$ are as follow: they are positive if and only if $\pi'$ is obtained from merging two blocks of $\pi$, in which case $q(\pi, \pi') = 1$.
\end{enumerate}

To put it in words, $\Pi^n$ is a process which starts with a totally fragmented state, and which evolves with (binary) coalescences. The evolution may be described by saying that every pair of blocks merges at rate 1, independently of their size. Because of this last property, one may think of a block as a particle. Each pair of particles merges at rate 1, regardless of any additional structure. When two particles merge, the pair is replaced by a new particle which is indistinguishable from any other particle. $\Pi^n$ is sometime referred to as Kingman's $n$-coalescent or simply an $n$-coalescent (the definition of Kingman's (infinite) coalescent is delayed to Proposition \ref{P:kingmandef}).

\mn \index{Consistency}\textbf{Consistency.} A trivial but important property of Kingman's $n$-coalescent is that of consistency: if we consider the natural restriction of $\Pi^n$ to partitions in $\cP_m$, where $m<n$, we obtain a new random process $\Pi^{m,n}$. The claim is that the distribution of $\Pi^{m,n}$ exactly the law of Kingman's $m$-coalescent (and is thus independent of $n$). This is not \emph{a priori} obvious, as the projection of a Markov process needs not even stay Markov. However, it is easy and elementary to verify the claim.

One important consequence of this property is, by Kokmogrov's extension theorem, the following:
\begin{proposition}\label{P:kingmandef}There exists a unique in law process $(\Pi_t, t \ge 0)$ with values in $\cP$, such that the restriction of $\Pi$ to $\cP_n$ is an $n$-coalescent. $(\Pi_t,t\ge 0)$ is called Kingman's coalescent.
\end{proposition}

To see how this follows from Kolmogorov's extension theorem, note that a partition $\pi$ of $\N$ may be regarded as a function from $\N$ into itself: it suffices to assign to every integer $i$ the smallest integer in the same block of $\pi$ as $i$. Hence a coalescing partition process $(\Pi_t,t\ge 0)$ may formally be viewed as a process indexed by $\N$ taking its values into $E=\N^{[0, \infty)}$. The consistency property above guarantees that the cylinder restrictions (i.e., the finite-dimensional distributions) of this process are consistent, which in turn makes it possible to use Kolmogorov's extension theorem to yield Proposition \ref{P:kingmandef}.

\index{Graphical construction}Quite apart from this ``general abstract nonsense", the consistency property also suggests a simple probabilistic construction of Kingman's coalescent, which we now indicate. This construction is in the manner of graphical constructions for models such as the voter model (see, e.g., \cite{liggett} or Theorem \ref{T:dualityCRWvoter} in these notes), and serves as a model for the more sophisticated future constructions of particle systems based on Fleming-Viot processes. The idea is to label every block $B$ of the partition $\Pi(t)$ by its lowest element. That is, we construct for every $i \ge 1$, a label process $(X_t(i), t \ge 0)$, where $X_t(i)=j$ means that at time $t$, the lowest element of the block containing $i$ is equal to $j$. Thus $X_t(i)$ has the properties that $X_0(i)=i$ for every $i \ge 1$, and $X_t(i)$ can only jump downwards, at times of a coalescence event involving the block containing $i$. At each such event $X_t(i)$ jumps to the lowest element $j$ such that $j \sim_{\Pi(t^+)} i$. The point is that $(X_t(i), t \ge 0)$ can be constructed for every $i\ge 1$ simultaneously, as follows. For every $i<j$, let $\tau_{i,j}$ be an exponential random variable. To define $X_t(n)$, there is no problem in making the above informal description rigorous: indeed, to define $X_t(n)$, it suffices to look at the exponential random variables associated with $1\le i <j \le n$, as the $\tau_{i,j}$ with $n \le i < j$ cannot affect $X_t(n)$. Thus there can never be any accumulation point of the $\tau_{i,j}$ since there are only finitely many such variables to be considered.

[More formally, let $T_1 = \inf\{\tau_{i,j}, 1 \le i < j \le n\}$, and define recursively
$$
T_{k+1} = \inf\{ \tau_{i,j}: 1 \le i < j \le n, \tau_{i,j}> T_k\}.
$$
Thus $(T_1, T_2, \ldots)$ is the sequence of times at which there is a potential coalescence. Let $i_k,j_k$ be defined by $T_k=\tau_{i_k,j_k}$. Define $X_t(i)=i$ for all $t<T_1$. Inductively now, if $k \ge 1$, and $X_t(i)$ is defined for all $ 1 \le i \le n$, and all $t<T_k$. Let $I$ be the set of particles whose label changes at time $T_k$:
$$
I=\{i \in [n]: X_t(T_k^-) = j_k\}.
$$
Define $X_t(i)=X_{T_k^-}(i)$ if $i \notin I$ for all $t \in [T_k, T_{k+1})$, and put $X_t(i) = i_k$ if $i \in I$, for all $ t \in [T_k, T_{k+1})$.]

Once the label process $(X_t(i), t \ge 0)$ is defined simultaneously for all $i \ge 1$, we can define a partition $\Pi(t)$ by putting:
\begin{equation}\label{labelkingman}
i \sim_{\Pi(t)} j \text{ if and only if } X_t(i) = X_t(j).
\end{equation}
Moreover, it is obvious from the above description that the dynamics of $(\Pi(t),t\ge 0)$ restricted to $\cP_n$ is that of an $n$-coalescent. Thus (\ref{labelkingman}) is a realisation of Kingman's coalescent. Note that despite the labelling process which seems to favour lower labels rather than upper labels, the partition $\Pi(t)$ is, for every $t>0$, \emph{exchangeable}: this follows from looking at the restriction of $\Pi$ to $[n]$ for every $n\ge 1$ which contains the support of a permutation $\sigma$ with finite support. From the original description of an $n$-coalescent, it is plain that $\Pi_n(t)$ is invariant under the permutation $\sigma$. Hence $\Pi(t)$ is exchangeable.

\subsubsection{Coming down from infinity.}\index{Coming down from infinity}

We are now ready to describe what is one of Kingman's coalescent's most striking features, which is that it \emph{comes down from infinity}. As we will see, this phenomenon states that, although initially the partition is only made of singletons, after any positive amount of time, the partition contains only a finite number of blocks almost surely, which (by exchangeability) must all have positive asymptotic frequency (in particular, there is no dust almost surely anymore, as otherwise the singletons would contribute an infinite number of blocks).
Thus, let $N_t$ denote the number of blocks of $\Pi(t)$.

\begin{theorem}\label{T:Kcdi}
Let $E$ be the event that for all $t>0$, $N_t< \infty$. Then $\P(E)=1$.
\end{theorem}

In words, coalescence is so strong that all dust has coagulated into a finite number of solid blocks. We say that Kingman's coalescent \emph{comes down from infinity}. This is a big--bang--like event, which is indeed reminiscent of models in astrophysics.

\begin{proof}
The proof of this result is quite easy, but we prefer to first give an intuitive explanation for why the result holds true. Note that the time it takes to go from $n$ blocks to $n-1$ blocks is just an exponential random variable with rate $n(n-1)/2$. When $n$ is large, this is approximately $n^2/2$, so we can expect the number of blocks to approximately solve the differential equation:
\begin{equation}\index{Differential equation}\label{ode}
\begin{cases}
u'(t) &= - \ds \frac{u(t)^2}2\\
u(0)&=+\infty.
\end{cases}
\end{equation}
(\ref{ode}) has a well-defined solution $u(t) = 2/t$, which is finite for all $t>0$ but infinite for $t=0$. This explains why $N_t$ is finite almost surely for all $t>0$. in fact, one guesses from the ODE approximation:
\begin{equation}\index{Small-time behaviour}\label{K:smalltimes0}
N_t \sim \frac2t, \ \ t \to 0
\end{equation}
almost surely. This statement is correct indeed, but unfortunately it is tedious to make the ODE approximation rigorous. Instead, to show Theorem \ref{T:Kcdi}, we use the following simple argument. It is enough to show that, for every $\eps>0$, there exists $M>0$ such that $\P(N_t>M) \le \eps.$ For this, it suffices to look at the restrictions $\Pi^n$ of $\Pi$ to $[n]$, and show that
\begin{equation}\label{cdi0}
\limsup_{n\to \infty}\P( N^n_t >M) \le \eps.
\end{equation}
Here we used the notation $N^n_t $ for the number of blocks of $\Pi^n_t$. For every $n \ge 1$, let $E_n$ be an exponential random variable with rate $n(n-1)/2$. Then note that, by Markov's inequality:
\begin{align*}
\P( N_t^n >M) &= \P\left( \sum_{k=M}^n E_k >t\right) \\
& \le  \frac1t \E\left( \sum_{k=M}^n E_k \right) \\
& \le \frac1t \sum_{k=M}^\infty \frac2{k(k-1)}.
\end{align*}
The right-hand side of the above inequality is independent of $n$, and can be made as small as desired provided $M$ is chosen large enough. Thus (\ref{cdi0}) follows.
\end{proof}

\subsubsection{Aldous' construction}

We now provide two different constructions of Kingman's coalescent which have some interesting consequences. The first one is due to Aldous (section 4.2 in \cite{aldous}). Let $(U_j)_{j=1}^\infty$ be a collection of i.i.d. uniform random variables on $(0,1)$. Let $E_j$ be a collection of independent exponential random variables with rate $j(j-1)/2$, and let
$$
\tau_j = \sum_{k=j+1}^\infty E_k < \infty.
$$
Define a function $f:(0,1)\to \R$ by saying $f(U_j)=\tau_j$ for all $j\ge 1$, and $f(u)=0$ if $u $ is not one of the $U_j$'s. Define a tiling $S(t)$ of $(0,1)$ by looking at the open connected components of $\{u \in (0,1): f(u)>t\}$. See figure \ref{Fig:aldous} for an illustration.
\begin{figure}
\begin{center}
\includegraphics[scale=.7]{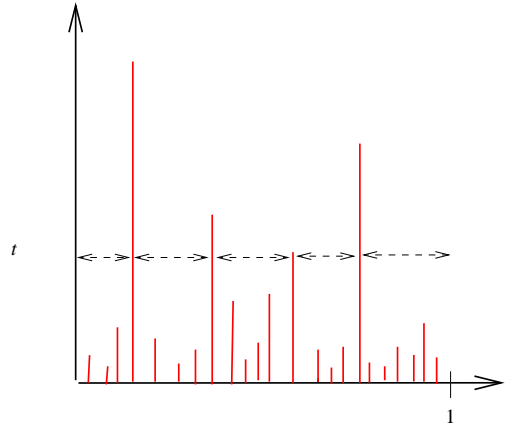}
\caption{Aldous' construction. The vertical sticks are located at uniform random points on $(0,1)$. The stick at $U_j$ has height $\tau_j$. These define a tiling of $(0,1)$ as shown in the picture. The tiles coalesce as $t$ increases from 0 to $\infty$.}
\end{center}
\label{Fig:aldous}
\end{figure}

\begin{theorem}\label{T:Kconstr1}
$(S(t),t\ge 0)$ has the distribution of the asymptotic frequencies of Kingman's coalescent.
\end{theorem}\index{Aldous' construction of Kingman's coalescent}

\begin{proof}
We offer two different proofs, which are both instructive in their own ways. The first one is straightforward: in a first step, note that the transitions of $S(t)$ are correct: when $S(t)$ has $n$ fragments, one has to wait an exponential amount of time with rate $n(n-1)/2$ before the next coalescence occurs, and when it does, given $S(t)$, the pair of blocks which coalesces is uniformly chosen. (This follows from the fact that, given $S(t)$, their linear order is uniform). Once this has been observed, the second step is to argue that the asymptotic frequencies of Kingman's coalescent forms a Feller process with an entrance law given by the ``pure dust" state $S(0)=(1,0, \ldots) \in \cS_0$. (Naturally, this Feller property is meant in the sense of the usual topology on $\cS_0$, i.e., not the restriction of the $\ell^1$ metric, but that determined by pointwise convergence of the non-dust entries.) This argumentation can be found for instance in \cite[Appendix 10.5]{aldous}. Since it is obvious that $S(t) \to (1,0, \ldots)$ in that topology as $t \to 0$, we obtain the claim that $S(t)$ has the distribution of the asymptotic frequencies of Kingman's coalescent.

The second proof if quite different, and less straightforward, but more instructive. Start with the observation that, for the finite $n$-coalescent, the set of successive states visited by the process, say $(\Pi_n, \Pi_{n-1},\ldots, \Pi_1)$ (where for each $1\le i \le n$, $\Pi_i$ has exactly $i$ blocks), is independent from the holding times $(H_n, H_{n-1}, \ldots, H_2)$ (this is, of course, \emph{not} true of a general Markov chain, but holds here because the holding time $H_k$ is an exponential random variable with rate $k(k-1)/2$ independent from $\Pi_k$.) Letting $n\to \infty$ and considering these two processes backward in time, we obtain that for Kingman's coalescent the reverse chain $(\Pi_1, \Pi_2, \ldots)$ is independent from the holding times $(H_2, H_3, \ldots)$. It is obvious in the construction of $S(t)$ that the holding times $(H_2, \ldots)$ have the correct distribution, hence it suffices to show that $(\Pi_1, \ldots, )$ has the correct distribution, where $\Pi_k$ is the random partition generated from $S(T_k)$ by sampling at uniform random variables $(U_j)$ independent of the time $k\ge 1$ (here $T_k$ is a time at which $S(t)$ has $k$ blocks).

To this end, we introduce the notion of \emph{rooted segments}. A rooted segment on $k$ points $i_1, \ldots, i_k$ is one of the possible $k!$ linear orderings of these $k$ points. We think of them as being oriented from left to right, the leftmost point being the root of the segment. If $n \ge 1$ and $1 \le k \le n$, consider the set $\cR_{n,k}$ of all rooted segments on $\{1, \ldots, n\}$ with exactly $k$ distinct connected components (the order of these $k$ segments is irrelevant). We call such an element a broken rooted segment. \index{Rooted segments}

\begin{lemma}\label{L:rootedsegments}
The random partition associated with a uniform element of $\cR_{n,k}$ has the same distribution as $\Pi_k^n$, where $(\Pi_k^n)_{n \ge k \ge 1}$ is the set of successive states visited by Kingman's $n$-coalescent.
\end{lemma}

\begin{proof}
The proof is modeled after \cite{gibbs}, but goes back to at least Kingman \cite{king82}. It is obvious that the partition associated with $\Xi_n$, a random element of $\cR_{n,n}$, has the same structure as $\Pi_n^n$ (as both these are singletons almost surely). Now, let $k \le n$ and let $\Xi$ be a randomly chosen element of $\cR_{n,k}$, and let $\Xi'$ be obtained from $\Xi$ by merging a random pair of clusters and choosing one of the two orders for the merged linear segment at random. Then we claim that $\Xi'$ is uniform on $\cR_{n,k-1}$. Indeed, if $\xi \preceq \xi'$ denotes the relation that $\xi'$ can be obtained from $\xi$ by merging two parts, we get:
\begin{align*}
\P(\Xi'= \xi') & = \sum_{\xi \in \cR_{n,k}: \xi \preceq \xi'} \P(\Xi= \xi) \P(\Xi'= \xi'| \Xi = \xi)\\
&= \sum_{\xi \in \cR_{n,k}: \xi \preceq \xi'}  \frac1{|\cR_{n,k}|} \frac12 \frac2{k(k-1)}\\
& =\frac1{|\cR_{n,k}|} \frac1{k(k-1)} | \{\xi \in \cR_{n,k}: \xi \preceq \xi'\}|.
\end{align*}
The point is that, given $\xi' \in \cR_{n,k-1}$, there are exactly $n-k+1$ ways to cut a link from it and obtained a $\xi \in \cR_{n,k}$ such that $\xi \preceq \xi'$. Note that there can be no repeat in this construction, and hence, $| \{\xi \in \cR_{n,k}: \xi \preceq \xi'\}| = n-k+1$, which \emph{does not} depend on $\xi'$. In particular,
\begin{equation}
\P(\Xi'=\xi') =  \frac{n-k+1}{k(k-1)|\cR_{n,k}|}
\end{equation}
and thus $\Xi'$ is uniform on $\cR_{n,k-1}$. \end{proof}

\begin{figure}[ht]
\begin{center}
\epsfig{file=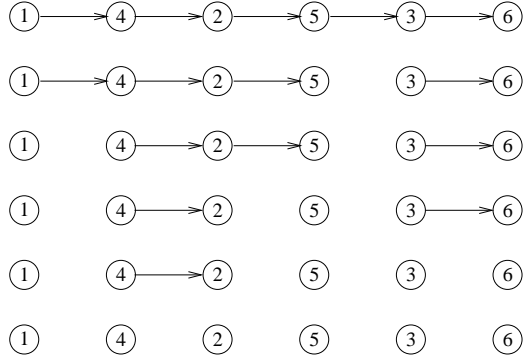, width=.6\textwidth,
height=.413\textwidth} \end{center} \caption{Cutting a rooted
random segment. }
\label{figcutseg}
\end{figure}

The lemma has the following consequence. It is easy to see that a random element of $\cR_{n,k}$ may be obtained by choosing a random rooted segment on $[n]$, and breaking it at $k-1$ uniformly chosen links. Rescaling the interval $[0,n]$ to the interval $(0,1)$ and letting $n\to \infty$, it follows from this argument that $\Pi_k$, which is the infinite partition of Kingman's coalescent when it has $k$ blocks, has the same distribution as the unit interval cut at $k-1$ uniform random points. This finishes the proof of Theorem \ref{T:Kconstr1}. \end{proof}

This theorem, and the discrete argument given in the second proof, have a number of useful consequences, which we now detail.

\begin{corollary}\label{C:simplex}
Let $T_k$ be the first time that Kingman's coalescent has $k$ blocks, and let $S(T_k)$ denote the asymptotic frequencies at this time, ranked in nonincreasing order. Then $S(T_k)$ is distributed uniformly over the $(k-1)$-dimensional simplex:
$$
\Delta_k= \left\{x_1 \ge \ldots \ge x_k\ge 0: \sum_{i=1}^k x_i=1\right\}.
$$
\end{corollary} \index{Simplex}

We also emphasize that the discrete argument given in the second proof of Theorem \ref{T:Kconstr1}, has the following nontrivial consequence for the time-reversal of Kingman's $n$-coalescent: it can be constructed as a Markov chain with ``nice", i.e., explicit, transitions. Let $(\Xi_1, \ldots, \Xi_n)$ be a process such that  $\Xi_k \in \cR_{n,k}$ for all $1 \le k\le n$, and defined as follows: $\Xi_1$ is a uniform rooted segment on $[n]$. Given $\Xi_i$ with $1\le i \le n-1$, define $\Xi_{i+1}$ by cutting a randomly chosen link from $\Xi_i$. (See Figure \ref{figcutseg}).

\begin{corollary}
The time-reversal of $\Xi$, that is, $(\Xi_n, \Xi_{n-1}, \ldots, \Xi_1)$, has the same distribution as Kingman's $n$-coalescent in discrete time.
\end{corollary} \index{Time-reversal}

As a further consequence of this link, we get an interesting formula for the probability distribution of Kingman's coalescent:

\begin{corollary}
\label{C:K1dim}
Let $1\le k \le n$. Then for any partition of $[n]$ with exactly $k$ blocks, say $\pi = (B_1, B_2, \ldots, B_k)$, we have:
\begin{equation}\label{K1dim}
\P(\Pi_k^n = \pi) = \frac{(n-k)!k!(k-1)!}{n!(n-1)!} \prod_{i=1}^k |B_i|!
\end{equation}
\end{corollary} \index{Marginals! Kingman's coalescent}

\begin{proof}
The number of elements in $\cR_{n,k}$ is easily seen to be
\begin{equation}\label{enumeration}
|\cR_{n,k}|={n-1 \choose k-1} \frac{n!}{k!}.
\end{equation}\index{Rooted segments}
Indeed it suffices to choose $k-1$ links to break out of $n-1$, after having chosen one of $n!$ rooted segments on $[n]$. Ignoring the order of the clusters gives us (\ref{enumeration}). Since the same partition is obtained by permuting the elements in a cluster of the broken rooted segment, we obtain immediately (\ref{K1dim}).
\end{proof}

It is possible to prove (\ref{K1dim}) directly on Kingman's coalescent by induction, which is the one chosen by Kingman \cite{king82} (see also Proposition 2.1 of Bertoin \cite{bertoin}). However this approach requires to guess the formula beforehand, which is really not that obvious! Induction works, but doesn't explain at all why such a formula should hold true. In fact, miraculous cancellations take place and (\ref{K1dim}) may seem quite mysterious. Fortunately, the connection with rooted segments explains why this formula holds.

Alternatively, we note that, given Corollary \ref{C:simplex}, (\ref{K1dim}) can be obtained by conditioning on the frequencies of $\Pi_k$, which are obtained by breaking the unit interval $(0,1)$, at $k-1$ uniform independent random points, and then sampling from this partition as in Kingman's representation theorem. This has a Dirichlet density with $k-1$ parameters, so such integrals can be computed explicitly, and one finds (\ref{K1dim}). 

\medskip Later, we will describe a construction of Kingman's coalescent in terms of a Brownian excursion (or, equivalently, of a Brownian continuum random tree), which is seemingly quite different. Both these constructions can be used to study some of the fine properties of Kingman's coalescent: see \cite{aldous} and \cite{kingBM}.

\subsection{The genealogy of populations}

We now approach a theme which is a main motivation for the study of coalescence. We will see how, in a variety of simple population models, the genealogy of a sample from that population can be approximated by Kingman's coalescent. This will usually be formalized by taking a \emph{scaling limit} as the population size $N$ tends to infinity, while the sample size $n$ is fixed but arbitrarily large. A striking feature of these results is that the limiting process, Kingman's coalescent, is to some degree \emph{universal}, as shown in the upcoming Theorem \ref{L:Mohlecriterion}. That is, its occurrence is little sensitive to the microscopic details of the underlying probability model, much like Brownian motion is a universal scaling limit of random walks, or SLE is a universal\index{Universality} scaling limit of a variety of critical planar models from statistical physics.

However, there are a number of important assumptions that must be made in order for this approximation to work. Loosely speaking, those are usually of the following kind:

\begin{enumerate}

\item[(1)] Population of constant size, and individuals typically have few offsprings.

\item[(2)] Population is well-mixed (or mean-field): everybody is liable to interact with anybody.

\item[(3)] No selection acts on the population.

\end{enumerate}

We will see how each of these assumptions is implemented in a model. For instance, a typical assumption corresponding to (1) is that the population size is constant and the number of offsprings of a random individual has finite variance. Changing other parameters of the model (e.g., such as overlapping generations or not) will not make any macroscopic difference, but changing any of those 3 points will usually affect the genealogy in essential ways. Indeed, much of the rest of the volume is devoted to studying coalescent processes in which some or all of those assumptions are invalidated. This will lead us in general to coalescent with multiple mergers, taking place in some physical space modeled by a graph. But we are jumping ahead of ourselves, and for now we first expose the basic theory of Kingman's coalescent.

\subsubsection{A word of vocabulary}

Before we explain the Moran model in next paragraph, we briefly explain a few notions from biology. From the point of view of applications, the samples concern not the individuals themselves, but usually some of their genetic material. Suppose one is interested in some specific gene (that is to say, a piece of DNA which codes for a certain protein, to simplify). Suppose we sample $n$ individuals from a population of size $N \gg n$. We will be interested in describing the \emph{genetic variation} in this sample corresponding to this gene, that is, in quantifying how much diversity there is in the sample at this gene. Indeed, what typically happens is that several individuals share the exact same gene and others have different variations. Different versions of a same gene are called \emph{alleles}. Here we will implicitly assume that all alleles are selectively equivalent, i.e., natural selection doesn't favour a particular kind of allele (or rather, the individual which carries that allele). \index{Genetic variation}\index{Allele}\index{Selection}

To understand what we can expect of this variation, it turns out that the relevant thing to analyse is the \emph{ancestry} of the genes we sampled, and, more precisely, the genealogical relationships between these genes. To explain why this is so, imagine that all genes are very closely related, say our sample comes from members of one family. Then we expect little variation as there is a common ancestor to these individuals going back not too far away in the past. Genes may have evolved from this ancestor, due to mutations, but since this ancestor is recent, we can expect these changes to be not very many. On the contrary, if our sample comes from individuals that are very distantly related (perhaps coming from different countries), then we expect a much larger variation.\index{Ancestral lineages}

\medskip \textbf{Ancestral partition.} It thus makes sense to desire to analyse the \emph{genealogical tree} of our sample. We usually do so by observing the \emph{ancestral partition process}.\index{Ancestral partition} Suppose that we have a certain population model of constant size $N$ which is defined on some interval of time $I = [-T,0]$ where $T$ will usually be $\infty$. Then we can sample without replacement $n$ individuals from the population at time 0, say $x_1, \ldots, x_n$, with $n<N$, and consider the random partition $\Pi^n_t$ such that $i \sim j$ if and only if $x_i$ and $x_j$ share the same ancestor  at time $-t$. The process $(\Pi^n_t,0\le t \le T)$ is then a coalescent process. It is very important to realise that the direction of time for the coalescent process is the \emph{opposite} of the direction of time for the ``natural" evolution of the population.

Recalling that we only want the ancestry of the gene we are looking at, rather than that of the individual which carries it, simplifies greatly matters. Indeed, in diploid populations like humans (i.e., populations whose genome is made of a number of pairs of homologous chromosomes, 23 for humans), each gene comes from a single parent, as opposed to individuals, who come from two parents. Thus in our sample, we have a number of $n$ genes, and we can go back one generation in the past and ask who were the ``parents" (i.e., the parent gene) of each of those $n$ genes. It may be that some of these genes share the same parent, e.g., in the case of siblings. In that case, the ancestral lineages corresponding to these genes have \emph{coalesced}. Eventually, if we go far enough back into the past, all lineages from our initial $n$ genes, will have coalesced to a most recent common ancestor, which we can call the \emph{ancestral Eve} of our sample. Note that if we sample $n$ individuals from a diploid population such as humans, we actually have $2n$ genes each with their genealogical lineage. Thus from our point of view, there won't be any difference between haploid and diploid populations, except that the population size is in effect doubled. From now on, we will thus make no distinction between a \emph{gene} and an \emph{individual}.\index{Genealogical tree}\index{Ancestral Eve}\index{Diploid}

\begin{figure}
\begin{center}
\includegraphics[scale=.5]{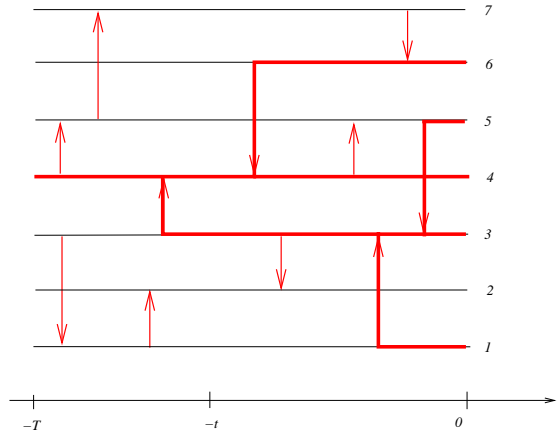}
\end{center}
\caption{Moran model and associated ancestral partition process. An arrow indicates a replacement, the direction shows where the lineage comes from. Here $N=7$ and the sample consists of individuals 1,3,4,5,6. At time $t$, $\Pi_t = \{1,3,5\},\{4,6\}$, while at time $T$, $\Pi_T=\{1,3,4,5,6\}$.}
\label{Fig:moran}
\end{figure}

\subsubsection{The Moran and the Wright-Fisher models}

The Moran model is perhaps the simplest model which satisfies assumption (1), (2) and (3). In it, there are a constant number of individuals in the population, $N$. Time is continuous, and every individual lives an exponential amount of time with rate 1. When an individual dies, it is simultaneously replaced by an offspring of another individual in the population, which is uniformly chosen from the population. This keeps the population size constant equal to $N$. This model is defined for all $t \in \R$. See the accompanying Figure \ref{Fig:moran} for an illustration. Note that all three assumptions are satisfied here, so it is no surprise that we have:

\begin{theorem}\label{T:moran} Let $n\ge 1$ be fixed, and let $x_1, \ldots, x_n$ be $n$ individuals sampled without replacement from the population at time $t=0$. For every $N\ge n$, let $\Pi^{N,n}_t$ be the ancestral partition obtained by declaring $i \sim j$ if and only if $x_i$ and $x_j$ have a common ancestor at time $-t$. Then, speeding up time by $(N-1)/2$, we find:
$$
(\Pi^{N,n}_{(N-1) t/2}, t \ge 0) \text{ is an $n$-coalescent}.
$$
\end{theorem}\index{Moran model}

\begin{proof} The model may for instance be constructed by considering $N$ independent stationary Poisson processes with rate 1 $(Z_t(i), - \infty < t  < \infty)_{i=1}^N$. Each time $Z_t(i)$ rings, we declare that the $i\th$ individual in the population dies, and is replaced by an offspring from a randomly chosen individual in the rest of the population. Since the time-reversal of a stationary Poisson process is still a stationary Poisson process, we see that while there are $k \le n$ lineages that have not coalesced by time $-t$, each of them experiences what was a death-and-substitution in the opposite direction of time, with rate 1. At any such event, the corresponding lineage jumps to a randomly chosen other individual. With probability $(k-1)/(N-1)$, this individual is one of the other $k-1$ lineages, in which case there is a coalescence. Thus the total rate at which there is a coalescence is $k(k-1)/(N-1)$. 
Hence speeding time by $(N-1)/2$ gives us a total coalescence rate of $k(k-1)/2$, as it should be for an $n$-coalescent with $k$ blocks.
\end{proof}

In the Wright-Fisher model, the situation is similar, but the model is slightly different. The main difference is that generations are discrete and non-overlapping (as opposed to the Moran model, where different generations overlap). To describe this model, assume that the population at time $t \in \Z$ is made up of individuals $x_1, \ldots, x_N$. The population at time $t+1$ may be defined as $y_1, \ldots, y_N$, where for each $1\le i \le N$, the parent of $y_i$ is randomly chosen among $x_1, \ldots, x_N$. Again, the model may be constructed for all $t \in \Z$. As above, all three conditions are intuitively satisfied, so we expect to get Kingman's coalescent as an approximation of the genealogy of a sample.

\begin{theorem} \label{T:wf}\index{Wright-Fisher}
Fix $n\ge 1$, and let $\Pi^{N,n}_t$ denote the ancestral partition at time $t$ of $n$ randomly chosen individuals from the population at time $t=0$. That is, $i \sim j$ if and only if $x_i$ and $x_j$ share the same ancestor at time $-t$. Then as $N \to \infty$, and keeping $n$ fixed, speeding up time by a factor $N$:
$$
(\Pi^{N,n}_{Nt}, t \ge 0) \longrightarrow_d (\Pi^n_t, t \ge 0)
$$
where $\longrightarrow_d$ indicates convergence in distribution under the Skorokhod topology of $\mathbb{D}([0, \infty), \cP_n)$, and $(\Pi^n_t,t\ge 0)$ is Kingman's $n$-coalescent.
\end{theorem}

\begin{proof} (sketch) Consider two randomly chosen individuals $x,y$. Then the time it takes for them to coalesce is Geometric with success probability $p = 1/N$: indeed, at each new generation, the probability that the two genes go back to the same ancestor is $1/N$ since every gene chooses its parent uniformly at random and independently of one another. Let $T_N$ be a geometric random variable with parameter $1/N$.
Since
$$
\frac1N T_N \longrightarrow_d E,
$$
an exponential random variable with parameter 1, we see that the pair $(x,y)$ coalesces at rate approximately 1 once time is speed up by $N$. This is true for every pair, hence we get Kingman's $n$-coalescent.
\end{proof}

We briefly comment that this is the general structure of limiting theorems on the genealogy of populations: $n$ is fixed but arbitrary, $N$ is going to infinity, and after speeding up time by a suitable factor, we get convergence towards the restriction of a nice coalescing process on $n$ particles.

Despite their simplicity, the Wright-Fisher or the Moran model have proved extremely useful to understand some theoretical properties of Kingman's coalescent, such as the \emph{duality relation} which will be discussed in the subsequent sections of this chapter.
However, before that, we will discuss an important result, due to M\"ohle,
which gives convergence towards Kingman's coalescent in the above sense, for a wide class of population models known as Cannings models and may thus be viewed as a result of \emph{universality}.

\subsubsection{M\"ohle's lemma}

We now describe the general class of population models which is the framework of M\"ohle's lemma, and which are known as Cannings models\index{Cannings models} (after the work of Cannings \cite{cannings1, cannings2}). As the reader has surely guessed, we will first impose that the population size stays constant equal to $N\ge 1$, and we label the individuals of this population $1, \ldots, N$. To define this model, consider a sequence of exchangeable integer-valued random variables $(\nu_1, \ldots, \nu_N)$, which have the property
that
\begin{equation}\label{canningsfixed}
\sum_{i=1}^N \nu_i = N.
\end{equation}
The $\nu_i$ have the following interpretation: at every generation, all $N$ individuals reproduce and leave a certain number of offsprings in the next generation. We call $\nu_i$ the number of offsprings of individual $i$. Note that once a distribution is specified for the law of $(\nu_i)_{i=1}^N$, a population model may be defined on a bi-infinite set of times $t \in \Z$ by using i.i.d. copies $\{(\nu_{i}(t))_{i=1}^N, t\in \Z\}$. The requirement (\ref{canningsfixed}) corresponds to the fact that the total population size stays constant, and the requirement that for every $t \in \Z$, $(\nu_i(t))_{i=1}^N$ forms an exchangeable vector corresponds to the fact that there are no spatial effects or selection: every individual is treated equally.

Having defined this population dynamics, we consider again the coalescing process obtained by sampling $n<N$ individuals from the population at time $0$, and considering their ancestral lineages: that is, let $(\Pi_{t}^{n,N}, t =0, 1, \ldots)$ be the $\cP_n$-valued process defined by putting
$i \sim j $ if and only if individuals $i$ and $j$ share the same ancestor at generation $-t$. This is the ancestral partition process\index{Ancestral partition} already considered in the Moran model and the Wright-Fisher diffusion.

Before stating the result for the genealogy of this process, which is due to M\"ohle \cite{mohle00}, we make one further definition: let
\begin{equation}\label{cNcannings}
c_N = \E\left(\frac{\nu_1(\nu_1 -1)}{N-1}\right).
\end{equation}
Note that $c_N$ is the probability that two individuals sampled randomly (without replacement) from generation 0 have the same parent at generation $-1$. Indeed, this probability $p$ may be computed by summing over the possible parent of one of those lineages and is thus equal to
$$p = \E\left(\sum_{i=1}^N \frac{\nu_i}{N}\frac{\nu_i -1}{N-1}\right) = c_N
$$
since $\E(\nu_i) = \E(\nu_1)$ by exchangeability.
Thus $c_N$ is the probability of coalescence of any two lineages in a given generation. Note that if we wish to show convergence to a continuous coalescent process, $c_N$ (or rather $1/c_N$) gives us the correct time-scale, as any two lineages will coalescence in a time of order 1 after speeding up by $1/c_N$. We may now state the main result of this section:

\begin{theorem}\label{L:Mohlecriterion}\emph{(M\"ohle's Lemma.)}\index{Mohle's Lemma@M\"ohle's Lemma} Consider a Cannings model defined by i.i.d. copies $\{(\nu_{i}(t))_{i=1}^N,t \in \Z\}$. If
\begin{equation}
\label{Mohlecriterion}
\frac{\E(\nu_1(\nu_1-1)(\nu_1-2))}{N^2 c_N} \underset{N \to \infty}\longrightarrow 0,
\end{equation}
then $c_N \to 0$ and the genealogy converges to Kingman's coalescent.
\end{theorem}

The formal statement which is contained in the informal wording of the theorem is that $(\Pi^{n,N}_{t/c_N}, t\ge 0)$, converges to Kingman's $n$-coalescent for every $n\ge 1$.

Although the proof is not particularly difficult, we do not include it in these notes, and refer the interested reader to \cite{mohle00}. However, we do note that the left hand side of (\ref{Mohlecriterion}) is, up to a scaling, equal to the probability that three lineages merge in a given generation. Thus the purpose of (\ref{Mohlecriterion}) is to demand that the rate at which three  or more lineages coalesce is negligible compared to the rate of pairwise mergers: this property is indeed necessary if we are to expect Kingman's coalescent in the limit. See M\"ohle \cite{mohle98} for other criterions similar to (\ref{Mohlecriterion}).

\subsubsection{Diffusion approximation and duality}

Consider the Moran model discussed in Theorem \ref{T:moran}, and assume that at some time $t$, say $t=0$ without loss of generality, the population consists of exactly two types of individuals: those which carry allele $a$, say, and those which carry allele $A$. For instance, one may think that allele $a$ is a mutation which affects a fraction $0 < p <1$ of individuals. How does this proportion evolve with time? What is the chance it will eventually invade the whole population?

From the description of the Moran model itself, it is easy to see that, in the next $dt$ units of time (with $dt$ infinitesimally small), if $X_t$ is the fraction of individuals with allele $a$, we have, if $X_t =x$:
$$
X_{t+dt} =
\begin{cases}
x+\frac1N & \text{ with probability } Nx(1-x)dt +o(dt)\\
x-\frac1N & \text{ with probability } Nx(1-x)dt +o(dt)\\
x & \text{ with probability } 1- 2Nx(1-x)dt +o(dt).
\end{cases}
$$
Indeed, $X_t$ may only change by $+1/N$ if an individual from the $A$ population dies (which happens at rate $N(1-x)$) and is replaced with an individual from the $a$ population (which happens with probability $x$). Hence the total rate at which $X_t$ increases by $1/N$ is $Nx(1-x)$. Similarly, the total rate of decrease by $1/N$ is $x(1-x)$, since for this to occur, an individual from the $a$ population must die and be replaced by an individual from the $A$ population.

Thus we see that the expect drift is
$$
\E(dX_t | \sigma(X_s, 0 \le s \le t )) = 0
$$
and that
$$
\var(dX_t |\sigma(X_s, 0 \le s \le t )) = \frac2{N}X_t(1-X_t)dt +o(dt).
$$
By routine arguments of martingale methods (such as in \cite{ethier-kurtz}), it is easy to conclude that, speeding time by $N/2$, $X_t$ converges to a nondegenerate diffusion:

\begin{theorem} \label{T:diffusionapprox} Let $(X_t^N, t\ge 0)$ be the fraction of individuals carrying the $a$ allele at time $t$ in the Moran model, started from $X_0^N=p \in (0,1)$. Then
$$
(X_{Nt/2}^N, t \ge 0) \longrightarrow_d (X_t, t\ge 0)
$$
in the Skorokhod topology of $\mathbb{D}(\R_+, \R)$, where $X$ solves the stochastic differential equation:
\begin{equation}\label{WFdiffusion}
dX_t = \sqrt{X_t(1-X_t)}dW_t ; \ \ X_0 = p
\end{equation}
and $W$ is a standard Brownian motion. (\ref{WFdiffusion}) is called the Wright-Fisher diffusion.
\end{theorem}\index{Diffusion approximation}

Note that the Wright-Fisher diffusion (\ref{WFdiffusion}) has infinitesimal generator
\begin{equation}
Lf(x) = \frac12 x(1-x) \frac{d^2}{dx^2}.
\end{equation}

\begin{remark}
In some texts different scalings are sometimes considered, usually due to the fact that the ``real" population size for humans (or any diploid population) is $2N$ when the number of individuals in the population is $N$. These texts sometime don't slow down accordingly the scaling of time, in which case the limiting diffusion is:
$$
dX_t = \sqrt{\frac12 X_t(1-X_t)}dW_t ; \ \ X_0 = p
$$
which is then called the Wright-Fisher diffusion. This unimportant change of constant explains discrepancies with other texts.
\end{remark}

As we will see, this diffusion approximation has many consequences for questions of practical importance, as several quantities of interest have exact formulae in this approximation (while in the discrete model, these quantities would often be hard or impossible to compute exactly). There are also some theoretical implications, of which the following is perhaps the most important. This is a relation of \emph{duality}, in the sense used in the  particle systems literature (\cite{liggett}), between Kingman's coalescent and the Wright-Fisher diffusion. Intuitively, the Wright-Fisher diffusion describes the evolution of a subpopulation forward in time, while Kingman's coalescent describes ancestral lineages backward in time, so this relation is akin to a change of direction of time. The precise result is as follows:

\begin{theorem} \label{T:dualityKWF}
Let $\E^\rightarrow$ and $\E^\leftarrow$ denote respectively the laws of a Wright-Fisher diffusion and of Kingman's coalescent. Then, for all $0<p<1$, and for all $n\ge 1$, we have:
\begin{equation}\label{dualityKWF}
\E^\rightarrow_p((X_t)^n) = \E^\leftarrow_n\left( p^{|\Pi_t|}\right)
\end{equation}
where $|\Pi_t|$ denotes the number of blocks of the random partition $\Pi_t$.
\end{theorem}\index{Duality}

\begin{proof} (sketch)\index{Moran model}
Consider a Moran model with total population size $N \ge 1$, and consider a subpopulation of allele $a$ individuals obtained by flipping a coin for every individual with success probability $p$. Choose $n$ individuals at random out of the total population at time $Nt/2$. What is the chance of the event $E$ that these $n$ individuals carry the $a$ allele? On the one hand, this can be computed by going backward in time $Nt/2$ units of time: by Theorem \ref{T:moran}, there are then approximately $|\Pi_t|$ ancestral lineages, where $\Pi$ is Kingman's $n$-coalescent, and each of them carries the $a$ allele with probability $p$. If each of them carries the $a$ allele, then their descendant also carries the allele $a$, so
$$
\P(E) \approx \E^\leftarrow_n(p^{|\Pi_t|}).
$$
On the other hand, by Theorem \ref{T:diffusionapprox}, at time $tN/2$ we know that the proportion of $a$ individuals in the population is approximately $X_t$. Thus the probability of the event $E$ is, as $N \to \infty$, approximately
$$
\P(E) \approx \E^\rightarrow_p(X_t^n).
$$
Equating the two approximations yields the result.
\end{proof}

The relationship (\ref{dualityKWF}) is called a \emph{duality relation}. In general, two processes $X$ and $Y$ with respective laws $\E^\leftarrow$ and $\E^\rightarrow$ are said to be dual if there exists a function $\psi(x,y)$ such that for all $x,y$:
\begin{equation}\label{duality}
\E^\rightarrow_x(\psi(X_t,y)) = \E^\leftarrow_y(\psi(x,Y_t)).
\end{equation}
In our case $X_t$ is the Wright-Fisher diffusion and $N_t=Y_t$ is the number of blocks of Kingman's coalescent, and $\psi(x,n) = x^n$. In particular, as $n$ varies, (\ref{dualityKWF}) fully characterizes the law of $X_t$, as it characterizes all its moments.

As an aside, this is a general feature of duality relations: as $y$ varies, the $\E^\rightarrow_x(\psi(X_t,y))$ characterizes the law of $X_t$ started from $x$. In particular, relations such as (\ref{duality}) are extremely useful to prove uniqueness results for martingale problems. This method, called the duality method, has been extremely successful in the literature of interacting particle systems and superprocesses, where it is often relatively simple to guess what martingale problem a certain measure-valued diffusion should satisfy, but much more complex to prove that there is a unique in law solution to this problem. Having uniqueness usually proves convergence in distribution of a certain discrete model towards the continuum limit specified by the martingale problem, so it is easy to see why duality can be so useful. For more about this, we refer the reader to the relevant discussion in Etheridge \cite{etheridge}.

We stress that duals are not necessarily unique: for instance, Kingman's coalescent is also dual to a process known as the Fleming-Viot diffusion, which will be discussed in a later section as it will have important consequences for us.

\medskip We now illustrate Theorems \ref{T:diffusionapprox} and \ref{T:dualityKWF} with some of the promised applications to questions of practical interest. Consider the Moran model of Theorem \ref{T:moran}. The most obvious question pertains to the following: if the $a$ population is thought of as a mutant from the $A$ population, what is the chance it will survive forever? It is easy to see this can only occur if the $a$ population invades the whole population and all the \emph{residents} (i.e., the $A$ individuals) die out. If so, how long does it take?

Let $X_t^N$ denote the number of $a$ individuals in a Wright-Fisher model with total population size $N$ and initial $a$ population $X_0^N = pN$, where $0<p<1$. Note that $X_t^N$ is a finite Markov chain with only two traps, 0 and $N$. Thus $X^N:=\lim_{t \to \infty} X_t^N$ exists almost surely and is equal to 0 or $N$. Let $D$ be the event that $X^N=0$ (this is event that the $a$ alleles die out), and let $S$ be the complementary event (this is the event that $a$ wipes out $A$). In biological terms, the time at which this occurs, say $T_N$, is known as the \emph{fixation time}\index{Fixation time}.

\begin{theorem} \label{T:absorption} We have:
\begin{equation}\label{absorption1}
\P(S)=p ; \ \ \P(D)=1-p.
\end{equation}
Moreover, the fixation time $T_N$ satisfies:
\begin{equation}\label{absorption2}
\E(T_N) \sim -N(p \log p + (1-p) \log(1-p)).
\end{equation}
\end{theorem}

\begin{proof} The first part of the result follow directly from the observation that $X_t^N$ is a bounded martingale and the optional stopping theorem at time $T$. For the second part, we use the diffusion approximation of Theorem \ref{T:diffusionapprox} and content ourselves with verifying that for the limiting Wright-Fisher diffusion, the expected time $T$ to absorption at 0 or 1 is
\begin{equation}\label{absorption3}
\E(T) = -2(p \log p + (1-p) \log(1-p)).
\end{equation}
Technically speaking, there is some further work to do such a checking that $2T_N/N \to T$ in distribution and in expectation, but we are not interested in this technical point here. Note that for a diffusion, if $f(p)$ is the expected value of $T$ starting from $p$, then $f(p)$ satisfies:
\begin{equation}\label{absorption4}
\begin{cases}
Lf(p)=-1;\\
f(0)=f(1)=0,
\end{cases}
\end{equation}
where $Lf(x) = \frac12 x(1-x)f''(x)$ is the generator of the Wright-Fisher diffusion. To see where (\ref{absorption4}) comes from, observe that, for any $\eps>0$, and for all $0 < x <1$,
$$
f(x) = \E_x(T) = \E_x ( \E_x(T | \cF_\eps)),
$$
where $\cF_\eps = \sigma(X_s, s \le \eps)$. Thus by the Markov property, letting $Pf(x)$ be the semigroup of the diffusion $X$:
\begin{align*}
f(x) &= \E_x(\eps + E_{X_\eps}(T))\\
& = \eps + \E_x( f(X_\eps)) \\
& = \eps + P_\eps f(x) \\
& = \eps + f(x) +\eps Lf(x) +o(\eps)
\end{align*}
where the last equality holds since $Lf(x) = \lim_{\eps \to 0} \frac{P_\eps f - f}{\eps}$. Since this last equality must be satisfied for all $\eps>0$, we conclude, after canceling the terms $f(x)$ on both sides of this equation and dividing by $\eps$:
$$
1+ Lf(x)=0
$$
for all $x \in (0,1)$, which, together with the obvious boundary conditions $f(0)=f(1)=1$, is precisely (\ref{absorption4}).

Now, (\ref{absorption4}) can be solved explicitly and the solution is indeed (\ref{absorption3}), hence the result.
\end{proof}

For $p = 1/2$, we get from Theorem \ref{T:absorption} that
$$
\E(T_N) \sim 1.38 N
$$
or, for diploid populations with $N$ individuals, $\E(T_N) \sim 2.56 N$. As Ewens \cite[Section 3.2]{ewensbook} notes, this long mean time is related to the fact that the spectral gap of the chain is small.

In practice, it is often more interesting to compute the expected fixation time (and other quantities) given that the $a$ allele succeeded in invading the population. In that case it is possible to show:
$$
\E(T_N | S) \sim -\frac{N(1-p)}p \log(1-p).
$$
See, e.g., \cite[Theorem 1.32]{durrettbookDNA}.

\subsection{Ewens' sampling formula}

We now come to one of the true cornerstones of mathematical population genetics, which is Ewens' sampling formula for the allelic partition of Kingman's coalescent (these terms will be defined in a moment). Basically, this is an exact formula which governs the patterns of genetic variation within a population satisfying all three basic assumptions leading to Kingman's coalescent. As a result, this formula is widely used in population genetics and in practical studies; its importance and impact are hard to overstate.

\subsubsection{Infinite alleles model}

We now define one of the basic objects of this study, which is the allelic partition.\index{Allelic partition} It is based on a model called the infinite alleles model\index{Infinite alleles model} which we now describe. Imagine that, together with the evolution of a population forwards in time (such as considered in the Moran model or in the Wright-Fisher model, say), there exists a process by which mutations occur and which induces differences between the allele observed in a child from that of his parent. If we consider a large gene, (i.e., one which consists of a fairly large DNA sequence), it is reasonable to assume that the mutation will make a change never seen before and never to be reproduced again by a future mutation, that is, \emph{every mutation generates a new, unique allele}. To simplify extremely, imagine we are looking at the genealogy of a gene coding for, say, eye colour. We may imagine that, initially, all individuals carry the same allele, i.e., have the same eye colour (say brown). Then as time goes by, a mutation occurs, and the child of a certain individual carries a new colour, maybe blue. His descendants will also all have blue eyes, and descendants of other individuals will carry brown eyes, until one of them gets a new mutation, giving him say green eyes, which he will in turn transmit to his children, and so on and so forth. The \emph{allelic partition}\index{Allelic partition} is the one that results when we identify individuals carrying the same eye colour (or, more generally, the same allele at the observed gene). We describe this partition through a vector, called the \emph{allele frequency spectrum}\index{Allele frequency spectrum}, which simply counts the number of different alleles with a given multiplicity: that is, $a_i$ is the number of distinct alleles which are shared by exactly $i$ individuals. See Figure \ref{Fig:ESF} for an illustration of the allelic partition.

For instance, (the two following datasets are taken from \cite{durrettbookDNA}, and were gathered respectively by \cite{Croyne} and \cite{underhill}), in a study of $n=60$ drosophilae (\emph{D. persimilis}), the allelic partition was represented by,
$$
a_1=18, a_2 = 3, a_4=1, a_{32}=1.
$$
That is, 1 allele was shared by 32 individuals, 1 allele was shared by 4 individuals, 3 alleles were found in pairs of individuals, and 18 individuals had a unique allele. Thus the associated partition had $18+3+1+1=23$ blocks.
In another, larger study of Drosophila\index{Drosophila} (\emph{D. pseudobscura}), on $n=718$ individuals:
\begin{align*}
&a_1=7, a_2=a_3=a_5= a_6=a_8=a_9=a_{26}=a_{36}=a_{37}=1, \\
& a_{82}=2, a_{149}=1, a_{266}=1.
\end{align*}

\subsubsection{Ewens sampling formula}

Given the apparent difference between these two data sets, what can we expect from a \emph{typical} sample?
It is natural to assume that mutations arrive at constant rate in time, and, that, for many populations, assumptions (1), (2) and (3) are satisfied so that the genealogy of a sample is defined by Kingman's coalescent. Thus, we will assume that, given the coalescence tree of a sample (obtained from Kingman's coalescent), mutations fall on the coalescence tree according to a Poisson point process with constant intensity per unit length, which we define to be $\theta/2$ for some $\theta>0$, for reasons that will be clear in a moment. We may then look at the (random) allelic partition that this model generates, and ask what does this partition typically look like. See Figure \ref{Fig:ESF} for an illustration.

\begin{figure}
  \begin{center}
    \includegraphics[width=5cm]{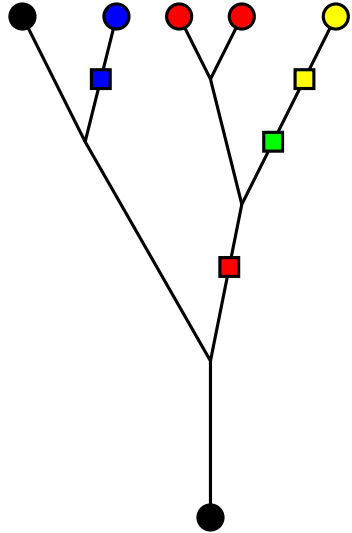}
    \caption{The allelic partition generated by mutations (squares)\index{Allelic partition}.}
    \label{Fig:ESF}
  \end{center}
\end{figure}

Note that, $\Pi^n$ defines a consistent family of partitions as $n$ increases, so one may define a random partition $\Pi$, called the allelic partition, such that $\Pi^n = \Pi|_{[n]}$ for all $n\ge 1$.

\begin{theorem}\label{T:KingmanESF} Let $\Pi$ be the allelic partition obtained from Kingman's coalescent and the infinite alleles model with mutation rate $\theta/2$. Then $\Pi$ has the law of a Poisson-Dirichlet random partition with parameter $\theta$. In particular, the probability that $A_1 =a_1$, $A_2 =a_2, \ldots, A_n = a_n$, is given by:
\begin{equation}\label{ESF3}
p(a_1, \ldots, a_n) = \frac{n!}{\theta(\theta+1)\ldots (\theta +n-1)}\prod_{j=1}^n \frac{(\theta/j)^{a_j}}{a_j!}.
\end{equation}
\end{theorem}

\begin{proof} We first note that (\ref{ESF3}) is simply a reformulation of Ewens' sampling formula\index{Ewens' sampling formula} in Theorem \ref{T:esf}, taking into account the combinatorial factors needed when describing $\Pi$ indirectly through the summary statistics $(a_1, \ldots, a_n)$, which is the traditional data recorded. Note in particular that (\ref{ESF3}) shows that
the vector $(A_1, \ldots, A_n)$ has the distribution of independent Poisson random variables $Z_1, \ldots, Z_n$ with parameters $\theta/j$, conditioned on the event that $\sum_{j=1}^n j Z_j = n$.

\medskip A particularly simple and elegant proof that $\Pi$ has the Poisson-Dirichlet $(0,\theta)$ consists in showing directly that $\Pi$ can be constructed as a Chinese Restaurant Process\index{Chinese Restaurant Process} with parameter $\theta$. Suppose the coalescence tree is drawn with the root at the bottom and the $n$ leaves at the top (like a real-life tree!) To every leaf of the tree, there is a unique path between it and the root, and there is a unique first mutation mark on this path, which we run from the leaf to the root, i.e. top to bottom (if no such mark exists, we may consider extending the coalescence tree by adding an infinite branch from the root, on which such a mark will always exist almost surely). Note that to describe the allelic partition $\Pi^n$, it suffices to know to know the portion of the coalescence tree between all the leaves and their first marks, and do not care about later coalescence events or mutations (here, time is also running in the coalescence direction, i.e., from top to bottom). This consideration leads us to associate to the marked coalescence tree a certain forest, i.e., a collection of trees, which are subtrees from the coalescence tree and contain all the leaves and their nearest marks.

Define $0<T_n< T_{n-1}< \ldots T_2 <T_1$ to be the times at which there is an event which reduces the number of branches in the forest described above. These events may be of two types: either a coalescence or a mutation (which ``kills" the corresponding branch). In either case, the number of branches decreases by 1, so there are exactly $n$ such times until there are no branches left. The Chinese Restaurant Process structure of the partition is revealed when we try to describe this forest from bottom to top, by looking at what happens at the reverse times $T_1, T_2, \ldots, T_n$. At each time $T_m$, $1\le m \le n$, we may consider the partition $\Pi^n_m$, which is the partition defined by the $m$ lineages uncovered by time $T_m$. At time $T_1$, we are adding a lineage which may be considered the root of this forest. Naturally, the partition it defines is just $\{1\}$. Suppose now that $m \ge 1$, and we are looking at the distribution of the partition $\Pi_{m+1}^n$, given $\Pi_m^n$. This $(m+1)\th$ lineage disappeared from the forest for either of two reasons: either by coalescence, or by mutation. If it was a mutation, then we know this lineage will open a new block of the partition $\Pi_{m+1}^n$. If, however, it was by a coalescence, then our $(m+1)\th$ customer joins one of the existing blocks. It remains to compute the probabilities of these events. Note that between times $T_m$ and $T_{m+1}$ there are precisely $m+1$ other lineages. Suppose the cluster sizes of $\Pi_m^n$ are respectively $n_1, \ldots, n_k$. The total rate of coalescence when there are $m+1$ lineages is $m(m+1)/2$, and the total mutation rate is $(m+1) \theta/2$. It follows that,
\begin{align*}
\P(\text{new block}|\Pi_m^n) &= \frac{\theta (m+1)/2}{m(m+1)/2 + \theta(m+1)/2}\\
&= \frac{\theta}{m+\theta}.
\end{align*}
Indeed, the event that there is a coalescence or a mutation at time $T_m$ rather than a mutation, is independent of $\Pi_m^n$ by the strong Markov property of Kingman's coalescent at time $T_{m+1}$.
Similarly,
\begin{align}
\P(\text{join block of size $n_i$}|\Pi_m^n) &= \frac{m(m+1)/2}{m(m+1)/2+ \theta(m+1)/2}\frac{n_i}{m} \label{esf4}\\
&= \frac{n_i}{m+\theta}.\nonumber
\end{align}
Indeed, in order to join a table of size $n_i$, first a coalescence must have occurred (this is the first term in the right-hand side of (\ref{esf4})), then we note that conditionally on this event, the new lineage coalesced with a randomly chosen lineage, and thus a particular group of size $n_i$ was chosen with probability $n_j/m$.
We recognize here the transitions of the Chinese Restaurant Process, so by Theorem \ref{T:CRP=PD} we get the desired result.
\end{proof}

Remarkably, when W. Ewens \cite{ewens} found this formula (\ref{ESF3}), this was without the help of Kingman's coalescent, although clearly this viewpoint is a big help. There exist numerous proofs of Ewens sampling formula for the infinite alleles model, of which several can be found in \cite{esf proofs}, together with some interesting extensions.

\subsubsection{Some applications: the mutation rate}

We briefly survey some applications of the above result. One chief application of Theorem \ref{T:KingmanESF} is the estimation of the mutation rate\index{Mutation rate} $\theta/2$. Note that, by Ewens' sampling formula, the conditional distribution of $\pi$ given the number of blocks $K_n$ is
\begin{equation}\label{esfcond1}
\P(\Pi_n = \pi | K_n = k) = c_{n,k} \prod_{i=1}^k n_i !
\end{equation}
where $c_{n,k} = \sum_{\star} \prod_{i=1}^k n_i! $ where the sum is over all partitions with $k$ blocks. Equivalently, since $\sum_{j=1}^n a_j = k$,
\begin{equation}\label{esfcond2}
p(a_1, \ldots, a_n |k) = c'_{n,k} \prod_{j=1}^n \frac{(1/j)^{a_j}}{a_j!}
\end{equation}
for a different normalizing constant $c'_{n,k}$. The striking feature of (\ref{esfcond1}) and (\ref{esfcond2}) is that both right-hand sides \emph{do not depend on $\theta$}. In particular, we can not learn anything about $\theta$ beyond what we can tell from simply looking at the number of blocks. In statistical terms, $K_n$ is a sufficient statistics\index{Sufficient statistics} for $\theta$. This raises the question: how to estimate $\theta$ from the number of blocks?

\begin{theorem}\label{T:PDnblocks} let $\Pi$ be a $PD(\theta)$ random partition, and let $\Pi_n$ be its restriction to $[n]$, with $K_n$ blocks. Then
\begin{equation}\label{PDnblocksLLN}
\frac{K_n}{\log n} \longrightarrow \theta, \ \ a.s.
\end{equation}
as $n \to \infty$. Moreover,
\begin{equation}\label{PDnblocksCLT}
\frac{K_n - \theta\log n}{\sqrt{\theta \log n}} \longrightarrow_d \cN(0,1).
\end{equation}
\end{theorem}

\begin{proof}
(\ref{PDnblocksLLN}) is an easy consequence of the Chinese Restaurant Process construction of a $PD(\theta)$ random partition. Indeed, let $I_i$ be the indicator random variable that customer $i$ opens a new block. Then $K_n = \sum_{i=1}^n I_i$ and the random variables $I_i$ are independent, with
$$
\P(I_i = 1) = \frac{\theta}{\theta + i-1}.
$$
Thus $\E(K_n) \sim \theta \log n,$, and $\var (K_n) \le \E(K_n)$. Find a subsequence $n_k$ such that $ k^2 \le \E(K_{n_k}) < (k+1)^2$. (Thus if $\theta =1$, $n_k =e^{2k}$ works.) By Chebyshev's inequality:
$$
\P\left(\left|\frac{K_{n_k}}{\log (n_k)} - \theta\right| > \eps\right) \le  \frac{\var(K_{n_k})}{\eps^2 (\log n_k)^2} \le \frac{2\theta}{\eps^2 k^2}
$$
for every $\eps>0$. By the Borel-Cantelli lemma, there is almost sure convergence along the subsequence $n_k$. Now, using monotonicity of $K_n$ and the sandwich theorem, we see that for $n$ such that $n_k \le n < n_{k+1}$, we get
$$
\frac{K_{n_{k}}}{\log (n_{k+1})} \le \frac{K_n}{\log n} \le \frac{K_{n_{k+1}}}{\log (n_{k})}
$$
but since $1\le \frac{\log (n_{k+1})}{\log (n_{k})} \le \frac{(k+1)^2}{k^2} \to 1$, we see that both left- and right-hand sides of the inequalities converge to $\theta$ almost surely. This proves (\ref{PDnblocksLLN}).
The central limit theorem (\ref{PDnblocksCLT}) follows from a similar application of the Lindeberg-Feller theorem for triangular arrays of independent random variables (see Theorem 4.5 in \cite{durrett}).
\end{proof}

As Durrett \cite{durrettbookDNA} observes, while Theorem \ref{T:PDnblocks} is satisfactory from a theoretical point of view, as it provides us with a way to estimate the mutation rate $\theta$, in practice the convergence rate of $1/{\sqrt{\log n}}$ is very slow: it takes $n=e^{100}$ to get a standard deviation of about $0.1$.

In fact, it can be shown using Ewens' sampling formula that the maximum-likelihood estimator of $\theta$ is the value $\hat \theta$ which makes $\E(K_n)$ equal to the observed number of blocks. In that case, the variance of that estimator is also roughly the same, as it can be shown using general theory of maximum likelihood (see below Theorem 1.13 in \cite{durrettbookDNA})
$$
\var(\hat \theta) = \left(\frac1{\theta^2} \var(K_n)\right)^{-1} \sim \frac{\theta }{\log n}
$$
which is the same order of magnitude as before, i.e., very slow. Other ideas to estimate $\theta$ can be used, such as the sample homozigosity\index{Homozigosity}: this is the proportion of pairs of individuals who share the same allele in our sample. See \cite[Section 1.3]{durrettbookDNA} for an analysis of that functional.

\subsubsection{The site frequency spectrum}

Along with the infinite alleles model, there is a different model for genetic variation, called the infinite sites model\index{Infinite sites model}. For arbitrary reasons, we have decided to spend less time on this model than on the infinite alleles model, and the reader wanting to move on to the next subject is invited to do so. However, this model will come back as a very useful theoretical tool in later models of $\Lambda$-coalescents, as it turns out to be closely related to the infinite alleles model but is partly easier to analyse.

The infinite sites model looks at a type of data which is altogether different from the one we were trying to model with the infinite alleles model. Suppose we look at a fixed chromosome (rather than a fixed locus\index{Locus} on that chromosome). We are interested in seeing which sites of the chromosome are subject to variation. Suppose, for instance that the chromosome is made of 10 nucleotides, i.e., is a word of 10 letters in the alphabet {\tt A,T,C,G}. In a sample of $n$ individuals, we can observe simultaneously these 10 nucleotides. It then makes sense to ask which of those show variation: for instance, we may observe that nucleotide 5 is the same in all individuals, while number 3 has different variants present in a number of individuals, say $k$. These $k$ individuals don't necessarily have the same nucleotide 3 or at other places, and so it would seem that by observing this data we gain more insight about the genealogical relationships of the individuals in our sample (as we can observe several loci at once!) but it turns out that this is not the case.

The infinite sites model starts with the assumption that each new mutation affects a new, never touched before or after, site (locus) of the chromosome. This mutation is transmitted unchanged to all the descendants of the corresponding individual, and will be visible forever. Hence mutations just accumulate, instead of erasing each other. Our assumptions will still be that the mutation rate is constant, and that, given the coalescence tree on $n$ individuals (a realisation of Kingman's $n$-coalescent), mutations fall on it as a Poisson point process with constant intensity $\theta/2$ per unit length. In this model, there is no natural partition to define, but it makes sense to ask:

\begin{enumerate}

\item what is the total number of sites $S_n$ at which there is some variation?

\item What is the number $M_j(n)$ of sites at which exactly $j$ individuals have a mutation?

\end{enumerate}

$S_n$ is called the number of \emph{segregating sites}\index{Segregating sites} and is often referred to in the biological literature as SNPs for Single Nucleotide Polymorphism\index{SNP (single nucleotide polymorphism)}. The statistics $(M_1(n), M_2(n), \ldots, M_n(n))$ is called the site frequency spectrum\index{Site frequency spectrum} of the sample. Note that $\sum_{j=1}^n M_j (n) = S_n$. Note also that $S_n$ may be constructed simultaneously for all $n\ge 1$ by enlarging the sample and using the consistency of Kingman's coalescent.

\begin{theorem}\label{T:sitesK}
We have:
\begin{equation}\label{sitesllnK}
\frac{S_n}{\log n} \longrightarrow \theta, \ \ a.s.
\end{equation}
as $n \to \infty$, and furthermore for all $j \ge 1$:
\begin{equation}\label{sitefreqK}
\E(M_j(n)) = \frac{\theta}{j}.
\end{equation}
\end{theorem}

\begin{proof}
Naturally, the reader is invited to make a parallel with Theorem \ref{T:PDnblocks}. The result (\ref{sitesllnK}) is conceptually slightly simpler, as given the coalescence tree, $S_n$ is simply a Poisson random variable with mean $\theta L_n/2$, where $L_n$ is the total length of the tree, i.e., the sum of the lengths of all the branches in the tree. But observe that while there are $k$ lineages, the time it takes to get a coalescence is exponential with mean $2/(k(k-1))$. Thus since there are exactly $k$ branches during this interval of times, we get:
$$
\E(L_n) = \sum_{k=2}^n k\frac{2}{k(k-1)} = 2\sum_{j=1}^{n-1} \frac1j = 2h_{n-1}
$$
where $h_n$ is the harmonic series. Thus
\begin{align*}
\E(S_n) &= \E(\E(S_n | L_n)) = \frac{\theta}2\E(L_n) = \theta h_{n-1}  \\
& \sim \theta \log n.
\end{align*}
Easy large deviations for sum of exponential random variables and Poisson random variables, together with the Borel-Cantelli lemma, show the strong law of large numbers (\ref{sitesllnK}).

For the site frequency spectrum, we use the Moran model embedding of Kingman's coalescent (Theorem \ref{T:moran}). Let $N$ denote the population size, and let $n=N$ be the sample size: that is, the sample is the whole population. Recall that the genealogy of these $n$ individuals, sped up by a factor $(N-1)/2$, is a realisation of Kingman's $n$-coalescent. Note that if a mutation appears at time $-t$, the probability that it affects exactly $k$ individuals is precisely $p_t(1,k)$ where $p_t(x,y)$ denotes the transition probabilities of the discrete Markov chain which counts the number of individuals carrying that mutation. Since mutations only accumulate, we get by integration over $t:$
\begin{equation} \label{Green0}
\E(M_j(n)) = \int_0^\infty p_t(1,j) \theta dt = \theta G(1,k)
\end{equation}
where $G(x,y)$ is the Green function\index{Green function} of the Markov chain, which computes the total expected number of visits to $y$ started from $x$. (Since this chain gets absorbed in finite time (and finite expectation) to the state $N$ or 0, the Green function is finite. It is moreover easy to compute this Green function explicitly. Let $\hat G$ denote the Green function for the discrete time chain, $\hat X$. Note that
\begin{equation}\label{Green1}
G(x,y) = \frac1{q(y)} \hat G(x,y)
\end{equation}
where $q(y)$ is the total rate at which the chain leaves state $y$. In our case, that is $2y(N-y)/N$.
Now, note that $\hat G(k,k)$ is $1/p$ where $p$ is the probability of never coming back to $k$ starting from $k$. Indeed, the number of visits to $k$ started from $k$ is geometric with this parameter. When the chain leaves $k$, it is equally likely to go up or down. Using the fact that $\hat X$ is a martingale, and the optional stopping theorem such as in Theorem \ref{T:absorption}, we get:
$$
p= \frac12 \frac1k + \frac12 \frac1{N-k} = \frac{2N}{k(N-k)}.
$$
Thus
\begin{equation}
\hat G(k,k) = \frac{k(N-k)}{2N}.
\end{equation}
Now, note that by the Markov property,
$$
\hat G(1,k) = \P(T_k<T_0) \hat G(k,k) = \frac1k \hat G(k,k)
$$
and thus we get:
$$
\hat G(1,k) = \frac{2(N-k)}{N}.
$$
Remembering (\ref{Green1}):
\begin{equation}
G(1,k)= \frac1k.
\end{equation}
Using (\ref{Green0}), this completes the result.
\end{proof}

\newpage

\section{$\Lambda$-coalescents}

In this chapter we introduce the $\Lambda$-coalescent processes, also known as coalescents with multiple collisions. We show a useful and intuitive construction of these processes in terms of a certain Poisson point process, and analyse the phenomenon of coming down from infinity for these processes. We explain the relevance of these processes to the genealogy of populations through two models, one due to Schweinsberg, and another one due to Durrett and Schweinsberg: as we will see, these processes describe the genealogy of a population either when there is a very high variability in the offspring distribution, or if we take in to account a form of selection and recombination (these terms will be defined below). Finally, we give a brief introduction to the work of Bertoin and Le Gall about these processes.

\subsection{Definition and construction}

\subsubsection{Motivation}
We saw in the previous chapter how Kingman's coalescent is a suitable approximation for the genealogy of a sample from a population which satisfies a certain number of conditions such as constant population size, mean-field interactions and neutral selection, as well as low offspring variability. When these assumptions are not satisfied, i.e., for instance, when size fluctuations cannot be neglected, or when selection cannot be ignored, we need some different kind of coalescence process to model the genealogy. Assume for instance that the population size has important fluctuations. E.g., assume that from time to time there are ``bottlenecks" in which the population size is very small, due to periodical environmental conditions for instance. In our genealogy, this will correspond to times at which a large proportion of the lineages will coalesce. Similarly, if there is a large impact of selection, individuals who get a beneficial mutation will quickly recolonize an important fraction of the population, hence we will observe multiple collision when tracing the ancestral lineages at this time. Large variability in offspring distribution such as many coastal marine species also leads to the same property that many lineages may coalesce at once. In those situations, Kingman's coalescent is clearly not a suitable approximation and one must look for coalescent processes which allow for multiple mergers.

\subsubsection{Definition}

Of course, from the point of view of sampling, it still makes sense to require that our coalescing process be Markovian and exchangeable, and also that it defines a \emph{consistent process}:\index{Consistency} that is, there exists an array of numbers $(\lambda_{b,k})_{2\le k \le b}$ which gives us the rate at which any fixed $k$-tuple of blocks merges when there are $b$ blocks in total. To ask that it is consistent is to ask that these numbers do not depend on $n$ (the sample size), and that the numbers $\lambda_{b,k}$ satisfy the recursion:
\begin{equation}\label{consistency}
\lambda_{b,k} = \lambda_{b+1,k} + \lambda_{b+1, k+1}.
\end{equation}
Indeed, a given group of $k$ blocks among $b$ may coalesce in two ways when reveal an extra block $b+1$: either these $k$ coalesce by themselves without the extra block, or they coalesce together with it. (For reasons that will be clear later, at the moment we do not allow for more than 1 merger at a time: that is, several blocks are allowed to merge into 1 at a given time, but there cannot be more than 1 such merger at any given time). We will refer to coalescent processes with no simultaneous mergers as \emph{simple}\index{Simple coalescence}\index{Exchangeable partition}.

\begin{definition} A family of $n$-coalescents is any family of simple, Markovian, $\cP_n$-valued coalescing processes $(\Pi^n_t, t\ge 0)$, such that $\Pi^n_t$ is exchangeable for any $t\ge 0$ and consistent in the sense that the law of $\Pi^n$ restricted to $[m]$ is that of $\Pi^m$, for every $1 \le m \le n$. It is uniquely specified by an array of numbers satisfying the consistency condition $(\ref{consistency})$, where for $2\le k \le b$:
$$
\lambda_{b,k} = \text{ merger rate of any given $k$-tuple of blocks among $b$ blocks}.
$$
\end{definition}

Naturally, given any consistent family of $n$-coalescents, there exists a unique in law Markovian process $\Pi$ with values in $\cP$ such that the restriction of $\Pi$ to $\cP_n$ has the law of $\Pi$.

\begin{definition} \label{D:Lambda}The process $(\Pi_t, t\ge 0)$, is a $\Lambda$-coalescent, or coalescent with multiple collisions.
\end{definition}
\index{Lambda-coalescent@$\Lambda$-coalescents}
\index{Multiple collisions}

\subsubsection{Pitman's structure theorems}

Multiple collisions, of course, refer to the fact there are times at which more than 2 blocks may merge, but also implicitly to the fact that such mergers do \emph{not} occur simultaneously. (Processes without this last restriction have also been studied, most notably in \cite{sch2}. They have enjoyed renewed interest in recent years, see, e.g., Taylor and Veber \cite{TaylorVeber} or Birkner et al. \cite{BirknerXi}).

The name of $\Lambda$-coalescent, however, may seem mysterious to the reader at this point. It comes from the following beautiful characterisation of coalescents with multiple collisions, which is due to Pitman \cite{pit99}.

\begin{theorem} \label{T:Pitman}Let $\Pi$ be a coalescent with multiple collisions associated with the array of numbers $(\lambda_{b,k})_{2 \le k \le b}$. Then there exists a finite measure $\Lambda$ on the interval $[0,1]$, such that:
\begin{equation}\label{rates}
\lambda_{b,k}= \int_0^1 x^{k-2} (1-x)^{b-k} \Lambda(dx) \ \ (2 \le k \le b).
\end{equation}
The measure $\Lambda$ uniquely characterizes the law of $\Pi$, which is then called a $\Lambda$-coalescent.
\end{theorem}

\begin{proof} (sketch) Pitman's proof is based on De Finetti's theorem for exchangeable sequences of 0's and 1's: it turns out that (\ref{consistency}) is precisely the necessary and sufficient condition to have:
$$
\mu_{i,j} = \E(X^i(1-X)^j), \ \ i,j \ge 0.
$$
for some random variable $0\le X \le 1$, where $\mu_{i,j} = \lambda_{i+j+2, j+2}$. (See (23)-(25) in \cite{pit99}).
\end{proof}

This proof is clean but not very intuitive. We will launch below into a long digression about this result, which we hope has the merit of explaining \emph{why} this result is true, even though, unfortunately this heuristics does not seem to yield a rigorous proof. However, along the way we will also uncover a useful probabilistic structure beneath (\ref{rates}), which will then be used to produce an elegant construction of $\Lambda$-coalescents.

\medskip The bottom line of this explanation is that Theorem \ref{T:Pitman} should be regarded as a L\'evy-It\^o\index{Levy-Ito decomposition@L\'evy-It\^o decomposition} decomposition for the process $\Pi$. The main reason for this as follows. Because we treat blocks as exchangeable particles (in particular, we do not differentiate between an infinite block and a block of size 1), it is easy to convince oneself that a coalescent with multiple collisions $(\Pi_t, t\ge 0)$ (in the sense of definition \ref{D:Lambda}), is a L\'evy process, in the sense that for every $t, s \ge 0$ we may write, given $\cF_t =\sigma(\Pi_s, s\le t)$:
\begin{equation}\label{levy}
\Pi_{t+s} = \Pi_t \star \Pi'_s
\end{equation}
where $\Pi'_s$ is independent from $\cF_t$ and has the same distribution as $\Pi_s$. Here, the $\star$ operation is defined as follows: for a partition $\pi=(B_1, \ldots )$ and a partition $\pi'=(B'_1, \ldots)$, the partition $\rho = \pi \star \pi'$ is defined by saying that we coagulate all the blocks of $\pi$ whose labels are in the same block of $\pi'$: for instance if $i$ and $j$ are in the same block of $\pi'$, then $B_i$ and $B_j$ will be subsets of a single block of $\rho$. The operation $\star$ is noncommutative and does not turn $\cP$ into a group. However, it does turn it into what is known in abstract algebra as a \emph{monoid} (i.e., the operation is associative and has a neutral element which is the trivial partition into singletons).

The identity (\ref{levy}) says that $(\Pi_t, t\ge 0)$ may be considered a L\'evy process\index{Levy process@L\'evy process} in the monoid $(\cP, \star)$. At this point, it is useful to remind the reader what is a L\'evy processes: a real-valued process $(X_t, t\ge 0)$ is called L\'evy if it has independent and stationary increments: for every $t\ge 0$, the process $(X_{t+s} - X_t, s \ge 0)$ is independent from $\cF_t$ and has same distribution as the original process $X$. The simplest example of L\'evy processes are of course Brownian motion and the simple Poisson process (an excellent introduction to L\'evy processes can be found in \cite{bertoin-levy}). The most fundamental result about L\'evy processes is the L\'evy-It\^o decomposition, which says that any real-valued L\'evy process can be decomposed as a sum of a Brownian motion, a deterministic drift, and compensated Poisson jumps. The simplest way to express this decomposition is to say that the characteristic function of $X_t$ may written as:
$$
\E(e^{{\bf i}qX_t}) = \exp( t \psi(q))
$$
where
\begin{equation}\label{LevyKhintchin}
\psi(u) = c_1 {\bf i}  q - c_2 \frac{q^2}2 + \int_{-\infty}^\infty e^{{\bf i} q x} - 1 -qx \indic{|x| <1} \nu(dx).
\end{equation}
Here, $c_1 \in \R, c_2 \ge 0$ and $\nu(dx)$ is any measure on $\R$ such that
\begin{equation}\label{Levymeasure}
\int_{\R} (|x|^2 \wedge 1 )\nu(dx) < \infty.
\end{equation}
A measure which satisfies (\ref{Levymeasure}) is called a \emph{L\'evy measure}\index{Levy measure@L\'evy measure}, and the formula (\ref{LevyKhintchin}) is called the \emph{L\'evy-Khintchin formula}\index{Levy-Khintchin@L\'evy-Khintchin formula}. It says in particular that the evolution of $X$ is determined by a Brownian evolution together jumps and deterministic drift, where the rate at which the process makes jumps of size $x$ is precisely $\nu(dx)$. The integrability condition (\ref{Levymeasure}) is precisely what must be satisfied in order to make rigorous sense of that description through a system of compensation of these jumps by a suitable drift.

The notion of L\'evy process can be extended to a group $G$, where here we require only the process $X$ to satisfy that $X(t)^{-1}X(t+s)$ is independent from $\cF_t$, and has the same distribution as $X(s)$, for every $s,t \ge 0$. Without entering into any detail, when the group $G$ is (locally compact) abelian, the Fourier analysis approach to the L\'evy-Khintchin formula (\ref{LevyKhintchin}) is easy to extend (via the characters of $G$, which are then themselves a locally compact abelian group) and yields a formula similar to (\ref{LevyKhintchin}). In noncommutative setups, this approach is more difficult but nonetheless there exist some important results such as a result of Hunt for Lie groups \cite{Hunt}. (I learnt of this in a short but very informative account \cite{applebaum}).

I am not aware of any result in the case where the group $G$ is replaced by a non-abelian monoid such as $\cP$, but it is easy to imagine that any L\'evy process in $\cP$ (i.e., a process which satisfies (\ref{levy}) may be described by a measure $\nu$ on the space $\cP$, which specifies the infinitesimal rate at which we multiply the current partition $\Pi_t$ by $\pi$:
\begin{equation}\label{levymeasureP}
\nu(d\pi) = \text{ rate: } \Pi_t \to \Pi_t \star \pi.
\end{equation}
Now, note that in our situation, we have some further information available: we need our process to be exchangeable, and to not have more than 1 merger at a time, i.e., to be what we called \emph{simple}. Thus $\nu$ must be supported on measures with only one nontrivial block, and must be exchangeable. By De Finetti's theorem, the only possible way to do that is to have a (possibly random) number $0<p<1$, and have every integer $i$ take part into that block by tossing a coin with success probability $p$. Let $\kappa_p$ denote this random partition.

\begin{definition}The operation $\pi \mapsto \pi \star \kappa_p$ is called a $p$-merger\index{p-merger@$p$-merger} of the partition $\pi$.\end{definition}

In words, for every block of the partition $\pi$, we toss a coin whose probability of heads is $p$. We then merge all the blocks that come up heads. That is, we coalesce a fraction $p$ of all blocks through independent coin toss. Therefore, we see that (\ref{levymeasureP}) transforms into: given a coalescent with multiple collision $(\Pi_t, t\ge 0)$, there exists a measure $\nu$ on $[0,1]$, such that
$$
\text{at rate } \nu(dp): \text{ perform a $p$-merger}.
$$
If this is indeed the case, then note that the numbers $\lambda_{b,k}$ satisfy:
$$
\lambda_{b,k}= \int_0^1 p^k(1-p)^{b-k} \nu(dp).
$$
This is now looking very close to the statement of Theorem \ref{T:Pitman}. It remains to see why $\nu(dp)$ may be written as $p^{-2} \Lambda(dp)$ where $\Lambda$ is a finite measure. (This is, naturally, the equivalent of the integrability condition (\ref{Levymeasure}) in this setup). However, this is easy to see: imagine that there are currently $n$ blocks. If $p$ is very small, in fact small enough that only  1 block takes part in the $p$-merger, then we may as well ignore this event since it has no effect on the process. In order to have a well-defined process, thus suffices that the rate at which at least two blocks merge is finite (when there is a finite number of blocks), and thus this condition reads:
\begin{equation}\label{levymeasureP2}
\int_0^1 {n \choose 2} p^2 \nu(dp) < \infty.
\end{equation}
Naturally, this is the same as asking that $\nu(dp)$ can be written as $p^{-2} \Lambda(dp)$ for some finite measure $\Lambda$. Thus
\begin{equation}
\lambda_{b,k} = \int_0^1 p^{k-2}(1-p)^{b-k} \Lambda(dp)
\end{equation}
for some finite measure $\Lambda$ on [0,1]. This is precisely the content of Theorem \ref{T:Pitman}.\qed

\medskip I do not know whether this approach has ever been made rigorous (or has ever been attempted). The problem, of course, is that the L\'evy-It\^o decomposition (\ref{levymeasureP}) is not known a priori for general monoids. This is a pity, as I think this approach is more satisfying.

However, all this digression has not been in vain, as on the way we have discovered the probabilistic structure of a general $\Lambda$-coalescent. It also gives us a nice construction of such processes in terms of a Poisson point process. This construction is often referred to as the \emph{Poissonian construction}\index{Poissonian construction!$\Lambda$-coalescent}. (As explained above in length, this should really be regarded as the L\'evy-It\^o decomposition of the L\'evy process $(\Pi_t, t\ge 0)$. We summarise it below. The construction is easier when $\Lambda$ has no mass at 0: $\Lambda(\{0\}) = 0.$

\begin{theorem}\label{T:Poissonconstruction} Let $\Lambda$ be a measure on $[0,1]$ such that $\Lambda(\{0\}) = 0$. Let $(p_i, t_i)_{i \ge 1}$ be the points of a Poisson point process on $(0,1] \times \R_+$ with intensity $p^{-2}\Lambda(dp) \otimes dt$. The process $(\Pi_t, t\ge 0)$ may be constructed by saying that, for each point $(p_i, t_i)$ of the point process, we perform a $p_i$-merger at time $t_i$.
\end{theorem}

Recall that a $p$-merger is simply defined by saying that we merge a proportion $p$ of all blocks by independent coin-toss. This construction is well-defined because the restriction $\Pi^n$ of $\Pi$ to $\cP_n$ is well-defined for every $n\ge 1$, thanks to the remark that the total rate at which pairs coalesce is finite (\ref{levymeasureP2}). As usual, since the restrictions $\Pi^n$ are consistent, this uniquely defines a process $\Pi$ on $\cP$, which has the property that $\Pi|_{[n]} = \Pi^n$.

\medskip We stress that this structure theorem is often the key to proving results about $\Lambda$-coalescents, so we advise the reader to make sure this sinks through before proceeding further. At the risk of boring the reader, here it is again in simple words: a coalescent process with multiple collisions is entirely specified by a finite measure $\Lambda$ on (0,1): $p^{-2} \Lambda(dp)$ gives us the rate at which a fraction $p$ of all blocks coalesces (at least when $\Lambda(\{0\})=0$).

The case where $\Lambda$ has an atom at 0, say $\Lambda(\{0\}) = \rho $ for some $\rho>0$, is not much different. It can be seen from (\ref{rates}) that this number comes into play only if $k=2$: that is, for binary mergers. It is easy to see what happens: decomposing
\begin{equation}\label{decompdirac}
\Lambda(dp) = \rho \delta_{0}(dp) + \hat \Lambda(dp)
\end{equation}
where $\hat \Lambda$ has no mass at 0, the dynamics of $(\Pi_t, t\ge 0)$ can be described by saying that, in addition to the Poisson point process of $p$-mergers governed by $p^{-2}\hat \Lambda(dp)$, every pair of blocks merges at rate $\rho$. More formally:

\begin{corollary}
\label{C:Poissonconstruction} Let $\Lambda$ be a measure on (0,1) and let $\rho:= \Lambda(\{0\})$. Let $(p_i, t_i)_{i \ge 1}$ be the points of a Poisson point process on $(0,1] \times \R_+$ with intensity $p^{-2}\hat \Lambda(dp) \otimes dt$, where $\hat \Lambda$ is defined by (\ref{decompdirac}). The process $(\Pi_t, t\ge 0)$ may be constructed by saying that, for each point $(p_i, t_i)$ of the point process, we perform a $p_i$-merger at time $t_i$, and, in addition, every pair of blocks merges at rate $\rho$.
\end{corollary}

Thus, the presence of an atom at zero adds a ``Kingman component" to the $\hat \Lambda$-coalescent. We will see below that, indeed, when $\Lambda$ is purely a Dirac mass at 0, the corresponding $\Lambda$-coalescent is Kingman's coalescent (sped up by an appropriate factor corresponding to the mass of this atom).

To conclude this section on Pitman's structure theorems, we give the following additional interpretation for the significance of the measure $\Lambda$.

\begin{theorem}\label{T:Lambdainterp}
Let $\Lambda$ be a finite measure on $[0,1]$ with no mass at zero. Let $(\Pi_t, t\ge 0)$ be a $\Lambda$-coalescent. Let $T$ be the first time that 1 and 2 are in the same block. Then $T$ is an exponential random variable with parameter $\Lambda([0,1])$. Moreover, if $F$ is the fraction of blocks that take part in the merger occurring at time $T$, then $F$ is a random variable in $(0,1)$, with law:
\begin{equation}
\P(F \in dx) = \frac{\Lambda(dx)}{\Lambda([0,1])}.
\end{equation}
\end{theorem}

In other words, the finite measure $\Lambda$, normalised to be a probability measure, gives us the law of the fraction of blocks that are coalescing, when two given integers first become part of the same block.

\begin{proof} The proof is obvious from Theorem \ref{T:Poissonconstruction}. Until 1 and 2 coalesce, their respective blocks are always $B_1$ and $B_2$ because of our convention to label blocks by increasing order of their least elements. Given an atom of size $p$, the probability that 1 and 2 coalesce is precisely $p^2$. Thus, define a thinning of the Poisson point process $(t'_i, p'_i)$, where each mark $(t_i,p_i)$ is kept with probability $p_i^2$. By classical theory of Poisson point processes, the resulting point process is also Poisson, but with intensity measure $ \Lambda(dp) \otimes dt$. Thus the rate at which they coalesce is precisely $\Lambda([0,1])$, and the first point of this process has a distribution which is proportional to $\Lambda$.
\end{proof}

\subsubsection{Examples}

It is high time to give a few examples of $\Lambda$-coalescents. Naturally, the main example of a $\Lambda$-coalescent is that of Kingman's coalescent:

\mn \textbf{Example 1.} Let $\Lambda$ be the unit Dirac mass at 0:
$$
\Lambda(dx)= \delta_0(dx).
$$
In that case, (\ref{rates}) translates into $\lambda_{b,k}= 0$ except if $k=2$, in which case $\lambda_{b,2}= 1.$ Thus the corresponding $\Lambda$-coalescent is nothing but Kingman's coalescent (every pair of blocks is merging at rate 1).

\mn \textbf{Example 2.} Another measure which will play an important role towards the end of these notes will be the case where $\Lambda(dx)= dx$ is the uniform measure on $(0,1)$. In this case the $\Lambda$-coalescent is known as the \emph{Bolthausen-Sznitman coalescent}\index{Bolthausen-Sznitman coalescent}, and the transition rates $\lambda_{b,k}$ can be computed more explicitly as
\begin{equation}\label{ratesBS}
\lambda_{b,k} = \frac{(k-2)!(b-k)!}{(b-1)!} = \left[(b-1) {b-2 \choose k-2}\right]^{-1}, \ \ 2 \le k \le b.
\end{equation}
The Bolthausen-Sznitman coalescent first arose in connection with the physics of \emph{spin glass}\index{Spin glass}, an area about which we will say a few words at the end of these notes. But this is not the only area for which this process is relevant: for instance, we will see that it describes the statistics of a certain combinatorial model of random trees and is thought to be a universal\index{Universality} scaling limit in a wide variety of models which can be described by ``random travelling waves"\index{Travelling waves}: all those topics will be (briefly) discussed in that last chapter.

\mn \textbf{Example 3.} Let $0<\alpha <2$. Assume that $\Lambda(dx)$ is the Beta$(2-\alpha, \alpha)$ distribution\index{Beta distribution}:
\begin{equation}\label{Betadistr}
\Lambda(dx) = \frac1{\Gamma(2-\alpha) \Gamma(\alpha)}x^{1-\alpha}(1-x)^{\alpha -1}dx.
\end{equation}
The resulting coalescent is simply called the \emph{Beta-coalescent}\index{Beta-coalescent} with parameter $\alpha$. It is an especially important family of coalescent processes, for both theoretical and practical reasons: on the one hand we will see that they are related to the genealogy of populations with large variation in the offspring distribution, and on the other hand, they are intimately connected with the properties of an object known as the stable Continuum Random Tree\index{Continuum Random Tree}. This correspondence and its consequences will be discussed in the next chapter.

\mn \textbf{Example 4.} A peculiar coalescent arises if $\Lambda$ is simply taken to be a Dirac mass at $p=1$. In that case, nothing happens for an exponential amount of time with mean 1, at which point \emph{all} blocks coalesce into 1. The corresponding coalescent tree is then \emph{star-shaped}: there is one root and an infinite amount of leaves connected to it. For this reason some authors call this process the star-shaped coalescent\index{Star-shaped}

\medskip Let $0<\alpha <2$, and consider the Beta-coalescent $(\Pi_t, t\ge 0)$ defined by (\ref{Betadistr}) in Example 3. Note that when $\alpha =1$, this is just the Bolthausen-Sznitman coalescent, while when $\alpha \to 2^-$,  the Beta distribution is an approximation of the Dirac mass, and hence if $\mu_\alpha$ denotes the distribution (\ref{Betadistr}), we have:
$$
\mu_\alpha \Rightarrow \delta_0, \ \ (\alpha \to 2).
$$
where $\Rightarrow$ is the vague convergence (convergence in distribution). Thus one should think of a Beta-coalescent with $1< \alpha <2$ as some kind of interpolation between Kingman's coalescent and the Bolthausen-Sznitman coalescent. In fact, we have:

\begin{theorem}\label{T:BetaKingman} Let $(\Pi^{(\alpha)}_t, t\ge 0)$ denote a Beta-coalescent with $1< \alpha <2$. Then as $\alpha \to 2$ from below, we have:
$$
(\Pi^{(\alpha)}_t, t\ge 0) \longrightarrow_d (\Pi_t, t\ge 0),\ \ \
$$
where $\Pi$ is Kingman's coalescent, and $\longrightarrow_d$ stands for convergence in distribution in the Skorokhod space $\mathbb{D}(\R_+, \cP)$.
\end{theorem}

An illustration of this result is given in Figure \ref{Fig:betakingman}, which was generated by Emilia Huerta-Sanchez, whom I thank very much for allowing me to use this picture.

\begin{figure}
  \begin{center}
  \includegraphics[scale=.7]{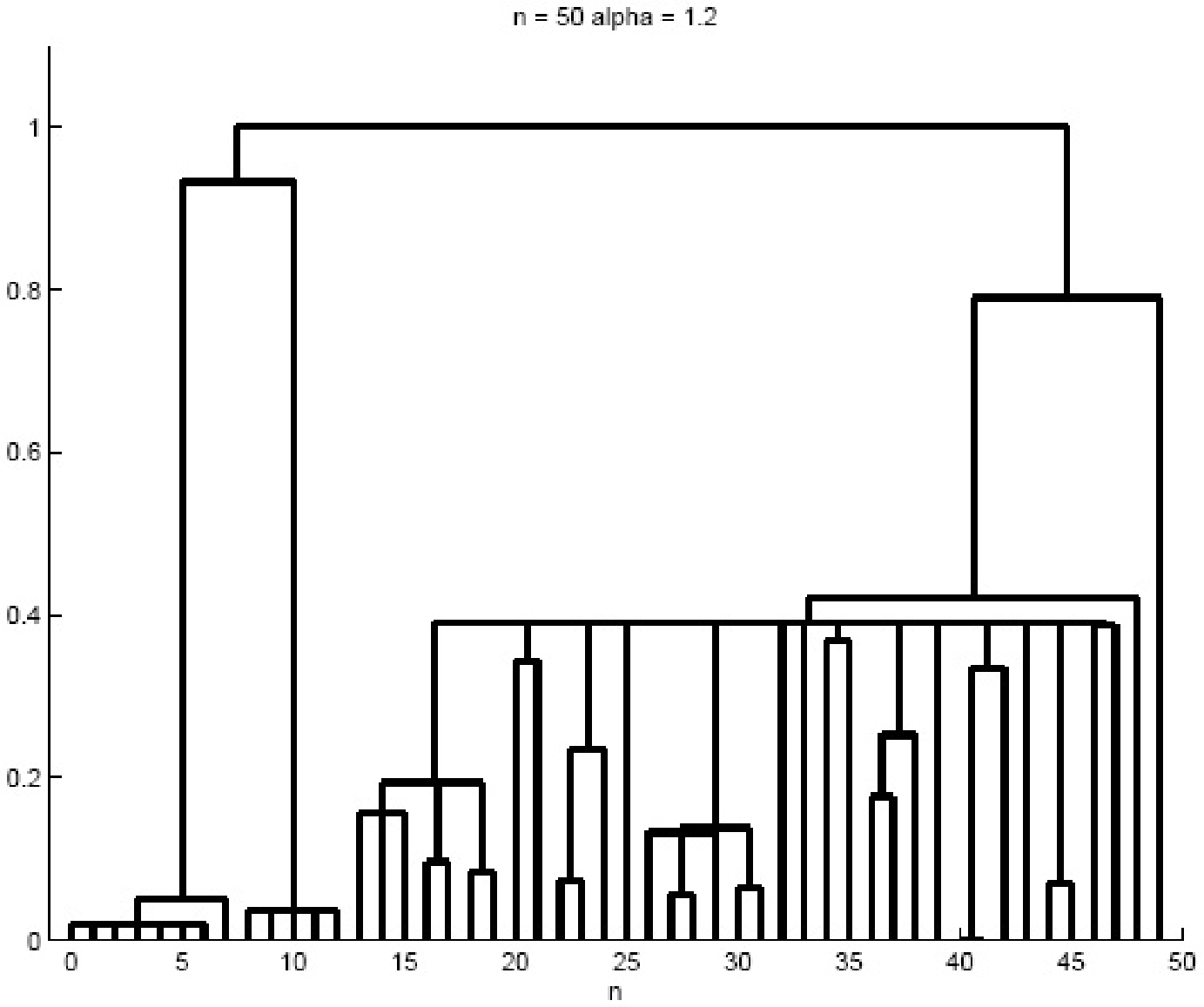}
  \includegraphics[scale=.7]{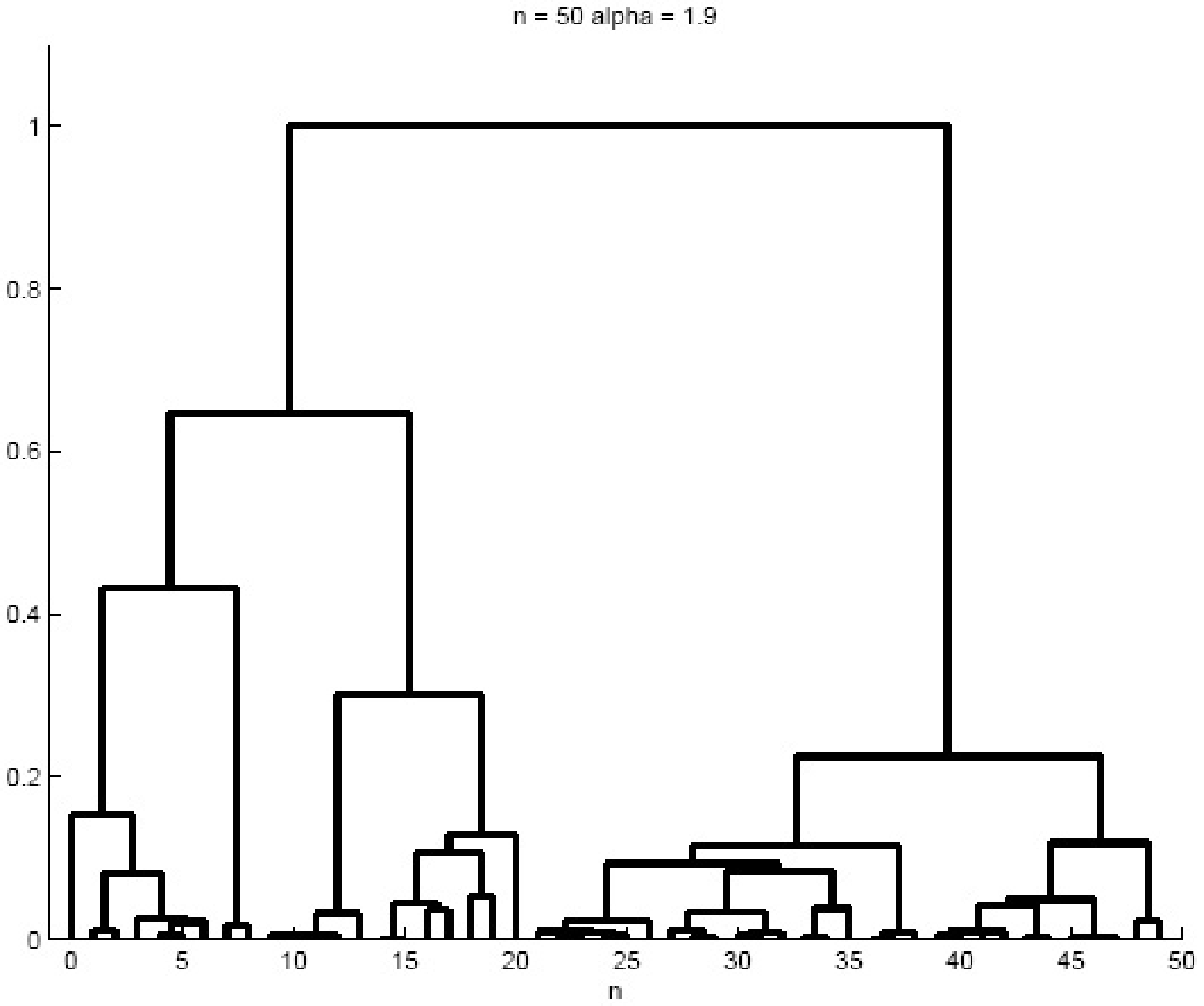}
  \end{center}
  \caption{The Beta-coalescent for two different values of the parameter $\alpha$: top, $\alpha =1.2$; bottom $\alpha =1.9$. Courtesy of Emilia Huerta-Sanchez.}
  \label{Fig:betakingman}
\end{figure}

\subsubsection{Coming down from infinity}\index{Coming down from infinity}

Fix a finite measure $\Lambda$ on $[0,1]$, and consider a $\Lambda$-coalescent $(\Pi_t, t\ge 0)$. One of the first things we saw for Kingman's coalescent is that it comes down from infinity, meaning that almost surely after any positive amount of time, the total number of blocks has been reduced to a finite number (Theorem \ref{T:Kcdi}). Given that in Kingman's coalescent only binary mergers are possible, and that here we may have many more more blocks merging at once, one may naively think that this should be also true for every $\Lambda$-coalescent. In fact, such is not the case, and whether or not a given $\Lambda$-coalescent comes down from infinity depends on the measure $\Lambda$. This is part of a more general paradox, which we will describe in more details later, that Kingman's coalescent is actually the one in which coalescence is \emph{strongest} (see Corollary \ref{C:Kfast}).

It is easy to see that there exists some $\Lambda$-coalescents which do not come down from infinity. Indeed, for some measures $\Lambda$, $\Pi_t$ has a positive fraction of dust for every $t>0$ almost surely (and hence an infinite number of blocks):

\begin{theorem}\label{T:dustCNS}
Let $D$ be the event that for every $t>0$, $\Pi_t$ has some singletons. Then $\P(D)=1$ if and only if
\begin{equation}\label{dustCNS}
\int_0^1 x^{-1} \Lambda(dx) < \infty.
\end{equation}
In the opposite case, $\P(D)=0$.
\end{theorem}

\begin{proof} (sketch) If this integral is finite, then the rate at which a given block takes part in a merger is finite, and so after any given amount of time, there remains a positive fraction of singletons that have never taken part in a merger. The converse uses a zero-one law for $\Pi$ along the lines of Blumenthal's zero-one law (details in \cite{pit99}).
\end{proof}

For instance, if $\alpha <1$, then the Beta-coalescent has a positive fraction of singletons at all times, while this fails if $\alpha \ge 1$. In particular, the Bolthausen-Sznitman coalescent does not have any dust. We will see below that the Bolthausen-Sznitman coalescent (which, we remind the reader, corresponds to the case $\alpha=1$ of the Beta-coalescent) is a two-sided borderline case, in the sense that it does not come down from infinity but has no dust. However if $\alpha$ is larger than 1, then the corresponding coalescent comes down from infinity (Corollary \ref{C:Betacdi}), and if it is smaller then the coalescent has dust with probability 1.

What are the conditions on $\Lambda$ to ensure coming down from infinity? The first thing that is needed is to say that, if the number of blocks becomes finite, this can only happen instantly near time zero, except in the case of the star-shaped coalescent.

\begin{theorem}
\label{T:cdi} Assume that $\Lambda(\{1\})=0$. Let $E$ be the event that for every $t>0$, $\Pi_t$ has only finitely many blocks. Let $F$ be the event that $\Pi_t$ has infinitely many blocks for all $t>0$. Then
$$
\P(E) =1 \text{ or } \P(F)=1.
$$
If $\P(E)=1$ we say that the coalescent comes down from infinity, and if $\P(F)=1$ we say that the process stays infinite.
\end{theorem}

Clearly, the assumption $\Lambda(\{1\})=0$ is essential, or otherwise we get a coalescence of all blocks in finite positive time. This possibility being taken away, the argument is a (fairly simple) application of the zero-one law mentioned above together with the consistency property. See \cite{pit99} for details.

\medskip
\cite{pit99} left open the question of finding a practical criterion for deciding whether a given $\Lambda$-coalescent comes down from infinity. The first answer came from the PhD thesis of Jason Schweinsberg, who proved a necessary and sufficient condition for this. We describe his result now. Given a finite measure $\Lambda$, let $\lambda_{b,k}$ denote the rate (\ref{rates}) and let
\begin{equation}\label{totalrate}
\lambda_b = \sum_{k=2}^b {b \choose k}\lambda_{b,k}.
\end{equation}
Note that $\lambda_b$ is the total coalescence rate when there are $b$ blocks. It turns out that the relevant quantity is the number
\begin{equation}
\gamma_b = \sum_{k=2}^b (k-1) {b \choose k} \lambda_{b,k}.
\end{equation}
To explain the relevance of this quantity, note that if there are currently $N_t=b$ blocks, then after $dt$ units of time
\begin{equation}\label{ode1}
\E(N_{t+dt} | N_t =b) = b - \gamma_b dt
\end{equation}
since if a $k$-tuple of blocks merges, then this corresponds to a decrease of $N_t$ by $(k-1)$. Define a function $\gamma:\R_+ \to \R$ by putting $\gamma(b) := \gamma_{\lfloor b \rfloor}$ for all $b \in \R_+$. Following the differential equation heuristics (\ref{ode}) already used for Kingman's coalescent, we see that if $u(t) =\E(N_t)$, from (\ref{ode1}) we expect $u(t)$ to approximately solve the differential equation\index{Differential equation}:
\begin{equation}\label{ode2}
\begin{cases}
u'(t) & = -\gamma(u(t)); \\
u(0)& = \infty.
\end{cases}
\end{equation}
Forgetting about problems such as discontinuities of $\gamma$ and rigour in general, we get by solving formally the differential equation (\ref{ode2}):
\begin{equation}
\int_{0}^t \frac{u'(s)}{\gamma(u(s))}ds  = -t
\end{equation}
so that making the change of variable $x=u(s)$,
\begin{equation}
\int_{u(t)}^\infty \frac{dx}{\gamma(x)} = t.
\end{equation}
We see hence that $u(t)$ is finite if and only if $\int^\infty \frac{dx}{\gamma(x)} < \infty$. Remembering that $\gamma(x) = \gamma_{\lfloor x \rfloor}$ leads us to Schweinsberg's criterion\index{Schweinsberg's criterion} \cite{sch1}:

\begin{theorem}\label{T:cdiCNS} Let $\Lambda$ be a finite measure on $[0,1]$. The associated $\Lambda$-coalescent comes down from infinity if and only if
\begin{equation}\label{CNS}
\sum_{b=2}^\infty \gamma_b^{-1} < \infty.
\end{equation}
\end{theorem}

\begin{proof}(sketch) We will sketch the important steps that lead to Theorem \ref{T:cdiCNS}. The first one is to define $T_n$ which is the time it takes to coalesce all $n$ first integers. Then we have naturally, $0=T_1\le T_2 \le \ldots \le T_n$, and note that the coalescent comes down from infinity if and only if $T_\infty:=\lim_{n\to \infty} T_n < \infty$ almost surely.

Assume that (\ref{CNS}) holds. Fix $n\ge1$ and consider the restriction $\Pi^n$ of $\Pi$ to $[n]$. Let $R_0=0$, and define $R_i$ to be sequence of times at which $\Pi^n$ loses at least one block, and if there is only one block left then define $R_i= R_{i-1}$. Thus $R_{n-1} = T_n$ as after $n-1$ coalescences we are sure to be done. Thus if $L_i= R_i - R_{i-1}$ we have
$
\E(T_n) = \sum_{i=1}^{n-1} \E(L_i).
$
Now, conditioning upon $N_{i-1}:= N_{T_{i-1}}$ the number of blocks at time $T_{i-1}$, we see that $L_i$ is exponentially distributed with rate $\lambda_{N_{i-1}}$ so long as $N_{i-1} >1$. Thus
\begin{equation}\label{Tn}
\E(T_n) = \sum_{i=1}^{n-1} \E( \lambda_{N_{i-1}^{-1}} \indic{N_{i-1}>1}).
\end{equation}
Observe that, if $J_i=N_{i-1} - N_i$ is the decrease of $N$ at this collision, we have
$$
\P(J_i = k-1 | N_{i-1} = b) = {b \choose k} \frac{\lambda_{b,k}}{\lambda_b}
$$
and thus $\E(J_i | N_{i-1} = b ) = \gamma_b / \lambda_b$. Plugging this into (\ref{Tn}) yields:
$$
\E(T_n) = \sum_{i=1}^{n-1} \E(\gamma_{N_{i-1}}^{-1} \E(J_i | {N_{i-1}}))
$$
since when $N_{i-1} = 1, J_i =0$ anyway. It follows that
\begin{equation}
\label{Tn2}
\E(T_n) =  \E(\sum_{i=1}^{n-1}\gamma_{N_{i-1}}^{-1} J_i ).
\end{equation}
Looking at the random variable in the expectation of the right-hand side in (\ref{Tn2}), $X= \sum_{i=1}^n \gamma_{N_{i-1}}^{-1} J_i$, we see that, intuitively speaking this random variable is very close to
$ \sum_{b=2}^n \gamma_b^{-1}$, as $\gamma_{N_{i-1}}^{-1}$ will be repeated exactly $J_i$ times. Thus if $J_i$ isn't too big and if $\gamma_b$ doesn't have too wild a behaviour, it is easy to understand how this yields the desired result. For instance, we get an easy upper-bound by monotonicity: some simple convexity arguments show that $\gamma_b$ is nondecreasing with $b$, and hence
$$
X=\gamma_{N_{i-1}}^{-1} J_i \le \sum_{j={N_{i-1}}}^{N_i} \gamma_j^{-1}
$$
Thus under the assumption (\ref{CNS}), we get by the monotone convergence theorem $\E(T_\infty) < \infty$ and thus the coalescent comes down from infinity.

The other direction is a little more delicate, and the main thing to be proved is that if the coalescent comes down from infinity, i.e., if $T_\infty < \infty$, then in fact this random variable must have finite expectation. Granted that, a dyadic argument applied to (\ref{Tn2}) does the trick. Thus we content ourselves with verifying:

\begin{lemma}\label{L:CNS}
The coalescent comes down form infinity if and only if $\E(T_\infty)< \infty$.
\end{lemma}

\begin{proof}
Let $A_m$ be the event that $T_m > T_{m-1}$, that is, at time $T_{m-1}$, $\Pi^m$ still has two blocks. Then the expected time it takes for these two blocks to coalesce is just $\lambda_{2,2}^{-1} =:\rho$. Thus
$$
\E(T_\infty) = \sum_{m=2}^\infty \E(T_m - T_{m-1}) = \rho \sum_{m=2}^\infty \P(A_m).
$$
Hence, assuming $\E(T_\infty) = \infty$, we get $\sum_{m=2}^\infty \P(A_m) = \infty$. An application of the martingale version of the Borel-Cantelli lemma then shows that $A_m$ occurs infinitely often almost surely. When this is so, $T_\infty$ is greater than an infinite number of i.i.d. nonzero exponential random variables, and hence $T_\infty = \infty$ almost surely. This finishes the proof of Lemma \ref{L:CNS}.
\end{proof}
The reader is referred to the original paper \cite{sch1} for more details about the rest of the proof.
\end{proof}

As an application of this criterion, it is easy to conclude:

\begin{corollary}\index{Coming down from infinity}\label{C:Betacdi}
Let $0<\alpha <2$. The Beta-coalescent with parameter $\alpha$ comes down from infinity if and only if $\alpha >1$. In particular, the Bolthausen-Sznitman coalescent does not come down from infinity.
\end{corollary}

We will see later that another (equivalent) criterion for coming down from infinity is that
$$
\int_1^\infty\frac{dq}{\psi(q)} <\infty
$$
where $\psi(q) = \int_0^1 (e^{-qx}-1 +qx )x^{-2} \Lambda(dx)$ is the Laplace exponent of a certain L\'evy process. As we will see, this criterion is related to critical properties of continuous-state branching processes. There is in fact a strong connection between $\Lambda$-coalescents and these branching processes; this connection will be explored in the next section, and hence this will give a different proof of Theorem \ref{T:cdiCNS}. Along the way, we will be able to make rigorous the limit theorem which is suggested by the heuristic approach outlined before this theorem: that is, for small times $t>0$:
\begin{equation}\label{smalltimesheuristics}
N_t \approx u(t), \text{ where } \int_{u(t)}^\infty \frac{db}{\gamma(b)} = t.
\end{equation}

\subsection{A Hitchhiker's guide to the genealogy}

This section is devoted to the study of a few simple models where the genealogy is well-approximated by $\Lambda$-coalescents. There are a number of models where such convergence is discussed. For instance, Sagitov \cite{sag99} gave a simple model which is closely related in spirit to the first one we will be studying. (Remarkably, that paper was published simultaneously to that of Pitman \cite{pit99} and, although independent, it also contained a definition of $\Lambda$-coalescents, so that both Pitman and Sagitov share the credit for the discovery of this process). We have chosen to discuss two main models. These are:

\begin{enumerate}
\item  a Galton-Watson model due to Schweinsberg \cite{schweinsberg} where the offspring distribution is allowed to have heavy tails,

\item A model with selection and recombination (also known as \emph{hitchhiking}), studied by Durrett and Schweinsberg \cite{DurrettSchweinsbergSPA}.

\end{enumerate}

We also note that recently, Eldon and Wakeley \cite{eldwak} came up with a model which illustrates further the impact of offspring variability and gives rise to $\Lambda$-coalescents for the genealogies. Some biological and statistical implications of these findings are discussed in \cite{eldwak} and \cite{eldwak2}.

\subsubsection{A Galton-Watson model}

We now describe the population model that we will work with in this section. This is a model derived from the well-known Galton-Watson branching process, but, unlike these processes, the population size is kept constant by a sampling mechanism: we assume that the offspring distribution of an individual has mean $1<\mu < \infty$, so that by the law of large numbers, if there are $N$ individuals in the population at some time $t$, then the next generation consists of approximately $N \mu > N$ individuals. Instead of keeping all those $N\mu$ individuals alive, we declare that only $N$ of them survive, and they are chosen at random among the $N\mu$ individuals of that generation. Thus the population size is constant equal to $N$. Formally, the model is defined generation by generation, in terms of i.i.d. offsprings $X_1, \ldots , X_N$ (where the distribution of $X$ allows for heavy tails and is specified later), and from random variables $(\nu_i)_{i=1}^N$ which are exchangeable and have the property that $\sum_{i=1}^N \nu_i = N$. The variable $\nu_i$ corresponds to the actual offspring number of individual $i$ after the selection step. Note that this population model may be extended to a bi-infinite set of times $\Z$ by using i.i.d. copies $\{(\nu_{i}(t))_{i=1}^N, t\in \Z\}$: thus this model belongs to the class of Cannings populations models\index{Cannings models} discussed in Theorem \ref{L:Mohlecriterion}.

Having defined this population dynamics, we consider as usual the coalescing process obtained by sampling $n<N$ individuals from the population at time $0$, and considering their ancestral lineages: that is, let $(\Pi_{t}^{n,N}, t =0, 1, \ldots)$ be the $\cP_n$-valued process defined by putting
$i \sim j $ if and only if individuals $i$ and $j$ share the same ancestor at generation $-t$. This is the by-now familiar ancestral partition\index{Ancestral partition}. We now specify the kind of offspring distribution we have in mind for the Galton-Watson process, which allows for heavy-tails\index{Heavy tails}. We assume that there exists $\alpha \ge 1$ and $C>0$ such that for all $x \ge 0$:
\begin{equation}\label{heavytails0}
\P(X >x) \sim C x^{-\alpha}.
\end{equation}
One can also think of the case
\begin{equation}
\label{heavytails1}
\P(X=x) \sim C' x^{-\alpha -1},
\end{equation}
although (\ref{heavytails0}) is a slightly weaker assumption and so we prefer to work with it. When $\alpha >1$, $\mu:=\E(X)<\infty$ and we further assume that $\E(X) >1$, so that the underlying Galton-Watson mechanism is \emph{sueprcritical}. Recall that in Cannings models, the correct time scale is given by the inverse coalescence probability $c_N^{-1}$, where:
\begin{equation}\label{cN}
c_N = \E\left(\frac{\nu_1(\nu_1 -1)}{N-1}\right).
\end{equation}
As was already discussed, $c_N$ is the probability that two randomly sampled without replacement at random from generation 0 have the same parent at generation $-1$, and thus it is the probability of coalescence of any two lineages. Schweinsberg's result states that there is a phase transition at $\alpha =2$ for the behaviour of the genealogies.

\begin{theorem}\label{T:schweinsbergGW} Assume (\ref{heavytails0}) and
$\mu >1$. For any $n\ge 1$, as $N\to \infty$:
\begin{enumerate}
\item If $\alpha \ge 2$, the genealogy converges to Kingman's coalescent.

\item If $1\le \alpha <2$, the genealogy converges a Beta-coalescent with parameter $\alpha$.
\end{enumerate}
\end{theorem}

As usual, the formal statement which is contained in the informal wording of the theorem is that, in the case $\alpha \ge 2$, $(\Pi^{n,N}_{t/c_N}, t\ge 0)$, converges to Kingman's $n$-coalescent for every $n\ge 1$, while it converges to the restriction of a Beta-coalescent to $[n]$ if $\alpha \in [1,2)$.

\begin{remark}
Note that $\alpha =2$ is precisely the critical value which delimitates the convergence of the rescaled random walks $(S_n:= \sum_{i=1}^n X_i)$ towards a Brownian motion or a L\'evy process with jumps. As we will see in the next chapter and in the appendix, this is not a coincidence: Galton-Watson trees can be described in terms of processes known as \emph{height functions}\index{Height process}, or \emph{contour processes}\index{Contour process}, which are close relative of random walks with step distribution $X$. If this step distribution is in the domain of attraction of a normal random variable, we thus expect a tree which is close to the Brownian continuous random tree\index{Continuum Random Tree}, for which the genealogy is closely related to Kingman's coalescent, as is proved in $\cite{kingBM}$ and will be shown later in these notes. On the contrary, if the step distribution is in the domain of a stable random variable with index $1<\alpha <2$, then the limiting tree is called the stable continuum random tree\index{Continuum Random Tree} and its genealogy is known to be given by Beta-coalescents $(\cite{bbs2})$, as will be discussed in more details later on.
\end{remark}

\begin{remark}
When $\alpha < 1$, the coalescent obtained from the ancestral partitions converges to a coalescent with simultaneous multiple collisions\index{Simultaneous multiple collisions}. As we do not enter in the detail of these processes in these notes, we only refer the reader to part $(d)$ of Theorem 4 in $\cite{schweinsberg}$.
\end{remark}

\begin{proof} We will go over a few of the important steps of the proof of Theorem \ref{T:schweinsbergGW}, leaving as usual the more difficult details for the interested reader to find out in the original paper \cite{schweinsberg}.

\emph{Case 1.} Let $\alpha \ge 2$. The main idea is to use M\"ohle's lemma (Theorem \ref{L:Mohlecriterion}).\index{Mohle's Lemma@M\"ohle's Lemma}
Thus it suffices to check that (\ref{Mohlecriterion}) holds. Recall that this condition states that
$$
\frac{\E(\nu_1(\nu_1-1)(\nu_1-2))}{N^2 c_N} \to 0
$$
as $N \to \infty$.
It is easy to see that this can be rephrased as:
\begin{equation}
\label{Mohlecriterion0}
\frac{N}{c_N}\E\left(\frac{X_1(X_1-1)(X_1-2)}{S_N^3} \indic{S_N \ge N}\right) \underset{N \to \infty}\longrightarrow 0,
\end{equation}
where $X_1$ is the offspring number of individual 1 before selection, and $S_N = X_1 + \ldots + X_N$.
Now, it is easy to see, when $\alpha >2$, that $c_N = \E(X_1(X_1 -1) / S_N^2 ) \sim c/N$ for some $c>0$. Thus (\ref{Mohlecriterion0}) reduces to showing that
$$
N^2 \E\left(\frac{X_1(X_1-1)(X_1-2)}{S_N^3} \indic{S_N \ge N}\right) \longrightarrow 0.
$$
However,
$$
\frac{X_1(X_1-1)(X_1-2)}{S_N^3} \indic{S_N \ge N} \le \frac{X_1^3}{\max( X_1^3, N^3)}
$$
and thus
$$
\E\left(\frac{X_1(X_1\!-\!1)(X_1\!-\!2)}{S_N^3} \indic{S_N \ge N}\right) \le \sum_{k=0}^N \frac{k^3}{N^3} \P(X\!=\!k) + \P(X> N)
$$
Multiplying by $N^2$ and using (\ref{heavytails0}) it is easy to see that this converges to 0, and hence the condition in M\"ohle's lemma (\ref{Mohlecriterion}) is verified.

\emph{Case 2.} Assume that $1< \alpha <2$. There are two steps to verify. The first one is to compute the asymptotics of $c_N$, the scale parameter.

\begin{lemma}\label{L:cN-asym} We have the asymptotics, as $N\to \infty$:
\begin{equation}\label{cN-asym}
c_N \sim C  N^{1-\alpha}\alpha \mu^{-\alpha} B(2-\alpha, \alpha)
\end{equation}
where $B(2-\alpha, \alpha):= \Gamma(\alpha) \Gamma(2-\alpha)$.
\end{lemma}

\begin{proof} (sketch) Note that $$c_N \sim \E\left( \frac{X_1 (X_1 -1)}{S_N^2} \indic{S_n \ge N}\right).$$ Write $S_N = X_1 + S'_N$, where $S'_N = X_2 + \ldots X_N$. By the law of large numbers, $S'_N \approx N \mu$, so $$c_N \approx \E\left(\frac{X (X-1)}{(X +M)^2}\right)$$ with $M= N \mu$. Thus (\ref{cN-asym}) follows from the statement
\begin{equation}
\lim_{M \to \infty} M^{\alpha -1}\E\left( \frac{X(X-1)}{(X+M)^2}\right) = C\alpha B(\alpha, 2-\alpha).
\end{equation}
This is purely a statement about the distribution of $X$, which is shown by tedious but elementary manipulations.
\end{proof}

The second ingredient of the proof is to show a limit theorem for the probability that there is a $p$-merger for some $0<p<1$ at a given generation. Note that this is essentially the same as asking that $X_1 /S_N \ge p$.
\begin{lemma}\label{L:GWpmerger}
\begin{equation}
\label{GWpmerger}
\lim_{N\to \infty} \frac{N}{c_N} \P(\frac{X_1}{S_N} \ge p)  = \frac1{B(2-\alpha, \alpha)}\int_p^1 y^{1-\alpha}(1-y)^{\alpha-1} y^{-2} dy.
\end{equation}
\end{lemma}
\begin{proof}(sketch)
To explain how this comes about, we follow the same heuristic as above, and write:
$$
\frac{X_1}{S_N} \approx\frac{X_1}{X_1+ N\mu}
$$
so that:
\begin{align*}
\P\left(\frac{X_1}{S_N} \ge p\right)   &\approx \P\left(\frac{X_1}{X_1+ N\mu} \ge p\right)  \\
&= \P\left(X_1 \ge \frac{p}{1-p} \mu N\right).
\end{align*}
Using the assumption (\ref{heavytails0}), we deduce, using Lemma \ref{L:cN-asym}:
$$
\frac{N}{c_N} \P\left(\frac{X_1}{S_N} \ge p\right)  \approx \frac1{\alpha B(2-\alpha, \alpha)} \left(\frac{1-p}p\right)^\alpha.
$$
Using the substitution $z=(1-y)/y$ in the integral of (\ref{GWpmerger}), the right-hand side can be rewritten as a Beta-integral, so this is precisely what was requested.
\end{proof}

The last lemma shows that the infinitesimal rate of a $p$-merger is, for any $0<p<1$, approximately what it would be if this was a Beta-coalescent. From there, is not hard to conclude to the case 2 of Theorem \ref{T:schweinsbergGW} (the i.i.d. structure of the generations gives the asymptotic Markov property of the coalescent). Thus the proof is complete.
\end{proof}

\begin{remark} \label{R:Sagitov}Theorem \ref{T:schweinsbergGW} should be compared to the earlier paper of Sagitov $\cite{sag99}$: in that paper, general Cannings model are considered and it is shown that the genealogy could converge to any $\Lambda$-coalescent under the appropriate assumptions. (This is what led him to define $\Lambda$-coalescents in the first place, as opposed to the more ``abstract" route based on consistency and exchangeability which was followed by Pitman and in these notes). His main assumption is that $N^2 \sigma^2(N) \P(\nu_1 >Nx) \to \int_{x}^1 y^{-2}\Lambda(dy)$, together with some additional moment assumptions. The model of Theorem \ref{T:schweinsbergGW} is of course a particular case of the Cannings model, however checking this main assumption is where all the work lies.
\end{remark}
\subsubsection{Selective sweeps}

In this section we describe the effect of a phenomenon called \emph{selective sweeps}\index{Selective sweep} on the genealogy of a population. As usual we will start by explaining what we are trying to model (that will lead us to the notion of recombination, hitchhiking and selective sweeps, all these being fundamental concepts in population genetics) and then explain the mathematical model and associated results, which are due to Durrett and Schweinsberg \cite{DurrettSchweinsbergSPA}.

\medskip This model is our first which doesn't ignore selection. When a favourable allele is generated through mutation, it quickly spreads out to the whole population (this is easy to see with variations of the Moran model\index{Moran model}: suppose at each step we have a higher chance to kill an $A$ individual than an $a$ individual: such a selective advantage\index{Selective advantage} quickly drives the $A$ population out with positive probability). When we look at the ancestral lineages\index{Ancestral lineages}, what happens is that all lineages quickly coalesce into one, which is the lineage corresponding to the individual that got the mutation. Thus we have approximately a star coalescent at this time, which isn't so interesting. However, some interesting things occur when we look at another location of the genome (one says \emph{locus}\index{Locus}). The reason why this is interesting is that there are some nontrivial correlations between the genotypes\index{Genotype} of an individual at different locations.

The main mechanism which gives rise to this correlation is called \emph{recombination}\index{Recombination}. This is a type of mutation which rearranges large portions of one's genetic material: more precisely, it causes two homologous chromosomes to exchange genetic material. As a result, a chromosome that is transmitted to a recombinant's offspring contains genetic material both from the mother and the father (whereas normally, it is only that of one of the two parents). Recombination is a truly fundamental process of life, as it guarantees a mixing of the genetic material.

\medskip Suppose a selective sweep occurred at some locus $\alpha$, where the allele $a$, being favourable, drove out the resident $A$ population, and consider a different locus $\beta$ along the same chromosome, this one being selectively neutral. Now, in a sample of the population, after the sweep, most people descend from the initial mutant that got the favourable $a$ allele at locus $\alpha$. On the face of it, one would thus expect that at locus $\beta$, everybody should get the same allele as the one that this individual had at locus $\beta$ (say $b$). However, because of recombination, some individuals got their genetic material at the locus $\beta$ from individuals which may not have been a descendant from the original mutant. As a consequence, a fraction of individuals ``escape" the selective sweep. Let $\Theta$ be the random ancestral partition\index{Ancestral partition}, which (as usual) tells us which individuals from the sample of size $n$ have the same ancestors at the time of the advantageous mutation. Then this elementary reasoning shows that this random partition is likely to be ``close" to the partition $\kappa_p$ which defines $p$-mergers\index{p-merger@$p$-merger}: that is, with one nontrivial block $B$ which contains a positive fraction $p$ of all integers, selected by independent coin-tossing (in our description, $1-p$ is the probability to escape the sweep). Of course, this demands some care, as there can be several sources of error in this reasoning (for instance, once a lineage escapes the sweep by recombination, the parent of the recombinant could himself be a descendant of the initial mutant, or individuals who escape the sweep may coalesce together - however, all those things are unlikely if the sweep occurs rapidly compared to the time scale of Kingman's coalescent).

\medskip The fact that different loci are not independent is called \emph{Linkage desequilibrium}\index{Linkage desequilibrium} so we have seen how recombination is a (main) contributor of this desequilibrium. That a selectively neutral allele can quickly invade a large part of the population due to linkage desequilibrium (for instance through recombination with a favourable allele) is known as \emph{Genetic Hitchhiking}\index{Genetic Hitchhiking}\index{Hitchhiking}. I believe that the first rigorous investigation of this phenomenon goes back to Maynard-Smith and Haigh \cite{Maynard-Smith} in 1974, which was another cornerstone of theoretical population genetics.

\medskip
\textbf{Model.} It is time to define a model and state a first theorem. First, let $0<s<1$ be the selective advantage of the allele $a$: we work with a Moran model with selection, which says that every time an $a$ is replaced by an $A$ individual, this change is rejected with probability $s$. Let $0<r<1$ be the recombination probability at locus $\beta$: in our setting, this means that when a new individual is born, it adopts the genetic material of his parent at both loci most of the time, but with probability $r$, the allele at locus $\beta$ comes from a different parent who is selected uniformly at random in the population (this is because we are treating the two parents as two separate members of the population).

\begin{theorem}\label{T:DS1} \emph{(Durrett and Schweinsberg \cite{DurrettSchweinsbergTPB}, \cite{DurrettSchweinsbergSPA})} \\
Let $p=\exp(- r (\log N) /s)$. Assume that there exists $C_1$ such that $r< C_1 / \log N$. Then there exists $C>0$ such that, conditionally given that allele $a$ eventually invaded the whole population:
\begin{equation}
d_{TV}( \Theta, \kappa_p) < \frac{C}{\log N}.
\end{equation}
\end{theorem}

Here $d_{TV}$ denote the total variation distance between the law of the random partition $\Theta$ and that of the $p$-merger partition $\kappa_p$:
$$
d_{TV}(\Theta, \kappa_p)= \sup_{\pi \in \cP_n} |\P(\Theta = \pi) - \P(\kappa_p = \pi)|.
$$
We note that martingale arguments imply that the probability that allele $a$ eventually invades the whole population (and thus that a selective sweep occurs) is
$$
\frac{s}{1-(1-s)^N},
$$
which is approximately $s$ if $s$ is large compared to $1/N$, or approximately $1/N$ if $s$ is smaller.

\medskip While Theorem \ref{T:DS1} tells us what the genealogies look like between the beginning of the selective sweep and its end, in reality that is not what we care about: we do not simply wish to trace ancestral lineages back to the most recent selective sweep, but we wish to describe the entire genealogical tree of the sample of the population we are looking at. In that case, it is more likely that the genealogy has been affected by a series of selective sweeps that have occurred at various portions of the genome. We still assume that the locus we are considering is neutral, but study the combined effects of recombination after a series of selective sweeps. Of course, we cannot expect the selective advantage $s$ and recombination probability $r$ to be the same during all those events: this depends on the type of mutation, but also on the position of this advantageous mutation with respect to the observed locus: the further away this advantageous mutation occurs, the smaller the recombination probability $r$. This led Gillepsie \cite{gillespie} to propose the following:

\medskip \textbf{Model.} We run the usual Moran model dynamics. In addition, the chromosome is identified with the interval $I=[-L,L]$ and we observe the locus at position $x=0$. Mutations occur as a point process
$$
\cP= \sum_{i \ge 1} \delta_{(t_i,x_i,s_i)}
$$
on the state space $[0,\infty) \times [-L,L] \times (0,1)$. The first coordinate stands for time, the second for the position on the chromosome, and the third coordinate $s$ is the selective advantage associated with this mutation. We assume that $\cP$ is a Poisson point process, whose intensity measure $K$ is given by:
$$
K(dt,dx,ds) = dt \otimes \mu(dx, ds)
$$
where the measure $\mu(dx, ds)$ governs the rate of beneficial mutations with selective advantage $s$ occurring at position $x$. We also assume given a function $r:[-L,L] \to (0,1)$, which tells us what is the recombination probability $r$ when there is a mutation at position $x$ along the chromosome. The function $r$ that we have in mind is something like $r(x) =r|x|$ (i.e., the recombination probability is proportional to the distance), but we will simply assume that:
\begin{enumerate}
\item $r(0)=0$;
\item $r$ is decreasing on $[-L,0]$ and increasing on $[0,L]$.
\end{enumerate}
In general, we will work with a Poisson point process $\cP=\cP_N$ where the subscript indicates a possible dependence on the total population size $N$, and will do so consistently throughout the rest of this section. Strictly speaking, one must also specify if a selective sweep starts when a previous one hasn't already been completed. Here we will simply reject this possibility (in the regime we will study, this possibility is too infrequent anyway).

We may now state Durrett and Schweinsberg's key approximation for this model (Theorem 2.2 in \cite{DurrettSchweinsbergSPA}):

\begin{theorem}
\label{T:DS2} Assume that the functions $r_N$ is such that $(\log N) r_N$ converges uniformly towards a function $R:[-L,L]\to (0,\infty)$ satisfying (i) and (ii). Suppose also that $N\mu_N$ converges weakly to a measure $\mu$. Then the genealogies, sped up by a factor of $N$, converge (for finite-dimensional distributions) to a $\Lambda$-coalescent, where $\Lambda= \delta_0 + x^2 \eta(dx)$, where
\begin{equation}\label{DS2formula}
\eta([y,1]) = \int_{-L}^L \int_0^1 s\indic{e^{-r(x)/s} \ge y}\mu(dx, ds).
\end{equation}
\end{theorem}

The term ``$s$" in the integrand corresponds to requiring that the sweep is successful, and the other term comes directly from Theorem \ref{T:DS1}.  Note that (as noted in \cite{DurrettSchweinsbergSPA}) the finite-dimensional distribution convergence may not be strengthened to a Skorkohod-type convergence, as there are in reality several transitions occurring ``simultaneously" when there is a single selective sweep.

\medskip To get some intuition for (\ref{DS2formula}), it helps to consider a few examples. If all mutations have the same selective advantage $s$, and if $r_N(x) = r/ \log N$ for some fixed $r>0$, with all mutations occurring at rate $\alpha/N$ for some $\alpha >0$, then the measure $\eta$ which appears in Theorem \ref{T:DS2} is a point mass at $p=e^{-r/s}$ of mass $s\alpha$.

If now $\mu(dx,ds) = \alpha dx \otimes \delta_s$ (that is, the selective advantage is still constant and the mutation rate is constant along the chromosome, with total rate $\alpha L/N$), and if $r$ is constant, then $\Lambda = \delta_0 + \Lambda_0$ and $\Lambda_0$ has density $cy$ for $e^{-rL/s}<y <1$ and 0 otherwise, with $c =e^{2\alpha s^2/r}$. In particular, as $L\to \infty$ (infinitely long chromosomes) this is the measure $\Lambda_0(dy)=cydy$ for all $0<y<1$.

Finally, note that any measure $\Lambda$ which contains a unit mass at 0 may arise in Theorem \ref{T:DS2} (see example 2.5 in \cite{DurrettSchweinsbergSPA}).

\medskip \textbf{Comments}

\mn 1. The upper-bound in Theorem \ref{T:DS1} is of size $1/\log N$, which, in practice, is not that small. Durrett and Schweinsberg prove that a better approximation can be obtained by using a coalescent with \emph{simultaneous} multiple collisions\index{Simultaneous multiple collisions}.

\mn 2. Etheridge, Pfaffelhuber and Wakolbinger \cite{EPW} independently (and simultaneously) obtained some equivalent approximations but using a quite different route.

\subsection{Some results by Bertoin and Le Gall}

In this section, we briefly go over some of the results proved by Bertoin and Le Gall in their papers \cite{blg1} and \cite{blg2}. This section is intended to give a bird's eye view on this part of their work, which would take more time to cover properly. This section does not cover the work of \cite{blg0} and \cite{blg3} (the former will be discussed towards the end of the notes in connection with the Bothausen-Sznitman coalescent, while the latter will be discussed in the next section with the fine asymptotics of $\Lambda$-coalescents).

The first observation of Bertoin and Le Gall is that any $\Lambda$-coalescent process may be realised as a stochastic flow in the classical sense of Harris \cite{harris}. The state space of the flow is the so-called space of bridges\index{Bridge}, that is, c\`adl\`ag nondecreasing functions going from 0 to 1 on the interval $(0,1)$:
\begin{align}
\cB&=\{f:[0,1] \to [0,1] \text{ cadlag nondecreasing }; \nonumber \\
&\ \ \  \ \ f(0)=0 \text{ and } f(1)=1\}. \label{D:bridges}
\end{align}
This point of view allows to define a measure-valued process called the (generalized) Fleming-Viot process\index{Fleming-Viot!process}, which is a neat way of generalising the notion of \emph{duality}\index{Duality} to $\Lambda$-coalescents. The stochastic differential equations which describe the generalized Fleming-Viot process (i.e., the equivalent of the Wright-Fisher diffusion in this context) is then studied, and finally this is used to come back to Kingman's coalescent and show a surprising connection to a coalescing flow\index{Coalescing flow} of particles on the circle. We will follow a somewhat different order of presentation, starting with Fleming-Viot processes, partly because of their importance in what follows.

\subsubsection{Fleming-Viot processes}

The idea behind Fleming-Viot process is quite simple, but unfortunately things often look messy when written down. We will try to stay as informal as possible, following for instance Etheridge's excellent discussion of the subject \cite{etheridge}.

Suppose we consider the population dynamics given by the Moran model, with a total population size equal to $N$, run for an undetermined but finite amount of time. Suppose that in addition to the data of the ancestral lineages, we add another information, which is the allelic type carried by each individual. Of course, this depends on the initial allelic type of every individual in the population at the beginning of times. To make things simple, we imagine that, initially, all the $N$ individuals carry different types. We label them, e.g., $A_1(0), A_2(0), \ldots, A_N(0)$, and note that it absolutely doesn't matter what is the state space of the variables $A_i(0)$. For the sake of convenience we choose them to be independent random variables $U_1, \ldots, U_N$, uniformly distributed on (0,1). These types are then transferred to the offsprings of individuals according to the population dynamics, and this gives us for every time $t$ a collection $\{A_i(t)\}_{i=1}^N$. Each of the $A_i(t)$ is thus an element of the initial collection $A_1(0), \ldots, A_N(0)$.

Consider the distribution of allelic types at some time $t>0$: how does this look like? The first thing to note is that the vector $(A_1(t), A_2(t), \ldots, A_N(t))$ is exchangeable. This suggests considering the measure
\begin{equation}\label{FVempiricalD}
\mu_{N}(t) = \frac1N\sum_{i=1}^N \delta_{A_i(t)}
\end{equation}
and taking a limit as $N \to \infty$. Indeed, if we speed up time by a factor $N$, nothing prevents us from defining directly a sequence of labels $\{A_i(t)\}_{i=1}^\infty$ for all $t>0$ which has the following dynamics:

\begin{enumerate}
\item Initially $A_i(0)=U_i$ is a collection of i.i.d. uniform random variables $U_i$ on $(0,1)$.

\item At rate 1, for each $i<j$, $A_j(t)$ becomes equal to $A_i(t)$.

\end{enumerate}

Note then that the sequence $A_i(t)$ is an \emph{infinite} exchangeable sequence, so we can apply De Finetti's theorem, which tells us that for each fixed $t>0$, the empirical distribution of labels $\mu_N$ defined in (\ref{FVempiricalD}) has a weak limit almost surely. In fact, the next result gives a stronger statement. To state it we need the following notations: if $n \ge 1$, and $f(x_1, \ldots, x_n)$ is any continuous function on $[0,1]^n$, one may define a function $F$ on measures $\mu$ on [0,1] by saying
\begin{equation}
\label{formfunction}
F(\mu) = \int \ldots \int f(x_1, \ldots, x_n) \mu(dx_1) \ldots \mu(dx_n).
\end{equation}
The function $F$ may be interpreted as follows: given a measure $\mu$ on $[0,1]$, sample $n$ points $(x_1, \ldots, x_n)$ distributed according to $\mu$ and evaluate $f(x_1, \ldots, x_n)$. The expectation of this random variable (conditionally given $\mu$) is equal to $F(\mu)$.
Further, if $x=(x_1, \ldots, x_n) \in \R^n$, let $x^{i,j}$ denote the element $x'\in \R^n$ with all coordinates $x'_k = x_k$ except $x'_j = x_i$. That is, it is $x$ but where $x_j$ is replaced with $x_i$.

\begin{theorem}\label{T:FVdef} As $N\to \infty$, $(\mu_N(t), t\ge 0)$ converges almost surely towards a measure-valued strong Markov process $(\mu_t, t\ge 0)$, called the Fleming-Viot diffusion. Initially, $\mu_0$ is the uniform measure on (0,1), but for any fixed $t>0$, the measure $\mu_t$ consists exactly of finitely many atoms. Moreover, it has a generator
$L$ defined by the following property: if $F$ is a function of the form $(\ref{formfunction})$, then
\begin{equation}\label{FVgenerator1}
LF(\mu) = \int \ldots \int \sum_{1 \le i < j \le n} (f(x^{i,j}) - f(x) ) \mu(dx_1) \ldots \mu(dx_n).
\end{equation}
This property characterises the Fleming-Viot process $(\mu_t, t\ge 0)$.
\end{theorem}

The almost sure convergence referred to in this theorem corresponds to the topology on measure-valued functions defined by saying $\mu_t(A) \to \mu(A)$ for all Borel sets $A$, uniformly on compact sets in $(0, \infty)$.

See, for instance (1.49) in \cite{etheridge} for the form of the generator of the Fleming-Viot process. (The construction which we have used here is closer in spirit to the ``almost sure construction" of the Chapter 5 in \cite{etheridge} and the Donnelly-Kurtz lookdown process\index{Donnelly-Kurtz}\index{Lookdown process} - more about that later, see Definition \ref{D:DK}). The fact that $\mu_t$ consists of only finitely many atoms for every $t>0$ is in fact the same phenomenon that Kingman's coalescent comes down from infinity\index{Coming down from infinity}.

We note that there exists numerous versions of Fleming-Viot processes. The version which we have considered here is the simplest possible: for instance there are no mutations in this description. Incorporating small mutations (with allelic type given by an element of the integer lattice $\Z^d$) leads (in the limit of small mutation steps) to the spatial Fleming-Viot process\index{Fleming-Viot!spatial} which is related to super Brownian motion\index{Super Brownian motion} (more about this later, too).

The generalisation which was considered by Bertoin and Le Gall in \cite{blg1} for $\Lambda$-coalescents was the following. Let $\Lambda$ be a fixed measure on (0,1) (without any atom at zero for simplicity). Consider a population model with infinitely many individuals $1, 2, \ldots,$ whose initial allelic types are, as above, i.i.d. uniform random variables $U_1,\ldots$. Let $(p_i, t_i)$ be a Poisson point process
$$
\cP= \sum_i \delta_{(p_i,t_i)}
$$
with intensity $ p^{-2} \Lambda(dp) \otimes  dt$, as in the Poissonian construction of Theorem \ref{T:Poissonconstruction}. The model is defined by saying that at each time $t_i$ such that $(p_i,t_i)$ is a point of $\cP$, we selected a proportion $p_i$ of levels, say $I_1, I_2,\ldots, $ by independent coin-toss. Then the allelic types of individuals $I_1, I_2,\ldots,$ are all modified and become equal to $A_{I_1}(t^-)$.  An example of the evolution of allelic types at one such time is given in Figure \ref{fig:FV}.

\begin{figure}
\begin{center}
\includegraphics[scale=.8]{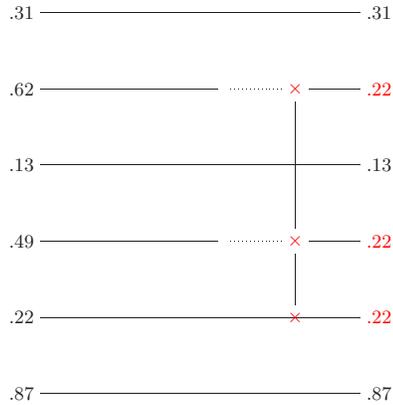}
\end{center}
\caption{Example of evolution of allelic types in the $\Lambda$-Fleming-Viot process. The red crosses indicate which levels were selected by coin tossing.}
\label{fig:FV}
\end{figure}

To see that this construction is well-defined, note that (as in the case of $\Lambda$-coalescents), the rate at which something happens in the first $n$ levels of the population (i.e., among the first $n$ individuals of this infinite population), is finite, and that the restrictions are consistent. Again, we can consider the empirical distribution of allelic types at time $t>0$:
$$
\mu_N(t) = \frac1n \sum_{i=1}^n \delta_{A_i(t)}
$$
and consider limits as $N \to \infty$. For a fixed $t>0$, it is easy to see that the sequence $\{A_i(t)\}_{i=1}^\infty$ is exchangeable: this is slightly counterintuitive as it seems a priori that lower levels play a more important role than upper ones, but is nevertheless true and is a consequence of the fact that the initial type sequence is i.i.d. and therefore exchangeable. By De Finetti's theorem, the limit of $\mu_N(t)$ thus exists almost surely, and one has the following. If $x = (x_1, \ldots, x_n) \in \R^n$, and if $I \subset [n] = \{1, \ldots, n\}$, let $x^I$ denote the element $x' \in \R^n$, with coordinates equal to those of $x$, except that all the coordinates $x'_j, j \in I$ have been replaced with $x_i$ and $i = \inf I$.

\begin{theorem}\label{T:FVLambdadef}
As $N\to \infty$, $(\mu_N(t), t\ge 0)$ converges almost surely towards a measure-valued strong Markov process $(\mu_t, t\ge 0)$, called the generalised Fleming-Viot or $\Lambda$-Fleming-Viot process\index{Fleming-Viot!generalised}. The generator
$L$ is defined by the following property: if $F$ is a function of the form (\ref{formfunction}), then
\begin{equation}\label{FVgenerator2}
LF(\mu) = \int \ldots \int \sum_{I \subset [n], |I| \ge 2} \lambda_{n, |I| }(f(x^{I}) - f(x)) \mu(dx_1) \ldots \mu(dx_n),
\end{equation}
where $\lambda_{n,k} = \int_0^1 x^{k-2}(1-x)^{b-k} \Lambda(dx)$ is the coalescence rate of any $k$-tuple of blocks among $n$ in a $\Lambda$-coalescent.
The property (\ref{FVgenerator2}) characterises the $\Lambda$-Fleming-Viot process $(\mu_t, t\ge 0)$.
\end{theorem}

The form of the generator and the fact that the martingale problem is well-posed can be seen from Section 5.2 in \cite{blg1}, and essentially boils down to the duality which we now discuss. It is more or less obvious that there is a relation of duality with the $\Lambda$-coalescent, which arises when time is running backwards (as usual!) The basic reason for this is that the time-reversal of a Poisson point process with a certain intensity $ d\nu \otimes dt $ is also a Poisson point process with the same intensity. Here is the corresponding statement. To state it, we start with a number $n\ge 1$ and a function $f(x_1, \ldots , x_n)$ on $[0,1]^n$. Let $\pi \in \cP_n$ be a partition of $[n]$, and assume that $\pi$ has $k$ blocks. Then for all $x \in \R^k$, we may define $x^\pi$ to be the element of $x'\in \R^n$ such that for all $1\le i \le k$, and for all $j \in B_i$ (the $i\th$ block of $\pi$), $x'_j = x_i$. In short, each of the $k$ coordinates of $x$ is copied along the blocks of $\pi$ to create an element of $\R^n$. Define the functional
\begin{equation}\label{formfunction2}
\Phi(\mu, \pi) = \int \ldots \int \mu(dx_1) \ldots \mu(dx_k) f(x^\pi).
\end{equation}

The duality relation states (see (18) in \cite{blg1})\index{Duality}:

\begin{theorem}\label{FVduality}
Let $\E^\rightarrow$ denote the expectation for the $\Lambda$-Fleming-Viot process $(\mu_t,t\ge 0)$ and let $\E^\leftarrow$ denote that for the $\Lambda$-coalescent $(\Pi_t, t\ge 0)$ restricted to $[n]$. Then we have, for all functions $\Phi$ of the form (\ref{formfunction2}):
\begin{equation}\label{DualityFV}
\E^\rightarrow_{\mu_0}(\Phi(\mu_t, \Pi_0)) = \E_{\pi_0}^\leftarrow(\Phi(\mu_0, \Pi_t))
\end{equation}
where $\pi_0$ is the trivial partition on $[n]$ into singletons.
\end{theorem}

We have already discussed how dual processes can be so useful (for instance, it is a crucial step in proving that the martingale problem is well-posed for the generalised Fleming-Viot process). In our case, we will see later that Fleming-Viot processes are an essential step to describe the connection of $\Lambda$-coalescents to continuous-state branching processes and continuum random trees. Further, this will be used below to describe natural stochastic differential equations and stochastic flows attached to $\Lambda$-coalescents.

\subsubsection{A stochastic flow of bridges}

We have already introduced the space of bridges $\cB$ above, but the random bridges we will discuss will have the following extra property:

\begin{definition}\label{D:bridge} A \emph{random bridge} $X$ is a random variable in $\cB$ such that the increments of $X$ are exchangeable.
\end{definition}
There is a natural operation on bridges, which is the composition $\circ$: if $X$ and $X'$ are two independent bridges, then so is $X \circ X'$.

What is the connection between bridges and $\Lambda$-coalescents? This connection is simple to explain if we think about a Cannings model whose genealogy is approximately a $\Lambda$-coalescent. (Recall that Sagitov \cite{sag99} showed this is always possible - see Remark \ref{R:Sagitov}). Thus, let $N\ge 1$, and let $\nu_1, \ldots, \nu_N$ be the exchangeable vector giving the respective progeny of individuals $1, \ldots, N$. Then $\sum_{i=1}^N \nu_i = N$, so what we can see is that the vector $(\nu_1, \ldots, \nu_N)$ encodes a discrete random bridge: more precisely, define a function
$$
\Delta:\{0, \ldots, N\} \to \{0, \ldots, N\}
$$
such that if $0 \le j \le N$:
\begin{equation}\label{discretebridge}
\Delta(j) = \sum_{i=1}^j \nu_i.
\end{equation}
Thus $\Delta(j)/N$ is the fraction of the population at the next generation which comes from the first $j$ individuals. (This interpretation will be crucial for what comes after). Note that $\Delta$ is a discrete bridge in the sense of Definition \ref{D:bridge}: it goes from 0 to $N$ between times 0 and $N$, and has exchangeable increments.

Now, introduce the time-dynamics of the Cannings model: thus we have an i.i.d. collection of exchangeable vectors $$\{\nu_i(t)\}_{i=1}^N, t \in \Z.$$
To this we can associate a discrete bridge $\Delta_t$ for each $t \in \Z$ as in (\ref{discretebridge}). Note the following property: consider two times $s< t \in \Z$. Then for each $0\le j \le N$, the fraction of individuals in the population at time $t$, that comes from the first $j$ individuals of the population at time $s$, is precisely:
\begin{equation}
\frac1N \Delta_{t-1} \circ \ldots \circ \Delta_s  (j)
\end{equation}
Thus let us define, for every $s\le t$ in $\Z$, the bridge $B^N_{s,t}$ as the linear interpolation on $[0,1]$ of
\begin{equation}
B^N_{s,t}(x) = \frac1N \Delta_{t-1} \circ \ldots \circ \Delta_{s} (j)
\end{equation}
if $x=j/N$. Bertoin and Le Gall call the collection of random variables $(B^N_{s,t})_{s\le t \in \Z}$ a discrete \emph{stochastic flow}\index{Stochastic flow} of bridges because:

\begin{enumerate}

\item $B_{s,s}$ is the identity map.

\item $B_{t,u} \circ B_{s,t} = B_{s,u}$ for all $s\le t \le u \in \Z$ (the \emph{cocycle property}\index{Cocycle property})

\item $B_{s,t}$ is stationary: its law depends only on $t-s$.

\item $B_{s,t}$ has independent increments: for every $s_1 < s_2 \ldots, s_n \in \Z$, then the bridges $B_{s_1, s_2}, \ldots, B_{s_{n-1},s_n}$ are independent.
\end{enumerate}

(Note however that \cite{blg1} and \cite{blg2} take for their definition $B^{s,t}$ what we call here $B_{-t,-s}$).
When $N\to \infty$ and time is sped up by a certain factor $c_N$ (the one which guarantees convergence of the genealogies to a $\Lambda$-coalescent process), it is to be expected that the flow
\begin{equation}\label{flowconvergence}
(B^N_{s/c_N, t/c_N}, - \infty < s \le t < \infty)  \Rightarrow (B_{s,t}, - \infty < s \le t < \infty)
\end{equation}
with respect to some topology. $(B_{s,t}, - \infty < s\le t < \infty)$ is then a (continuous) \emph{flow of bridges} because it satisfies properties 1--4 above, and furthermore in condition 1:
$$
B_{0,t} \to \text{Id}, \ \ t \to 0
$$
in probability, in the sense of the Skorokhod topology. This is thus a condition of continuity. We can now state Theorem 1 in \cite{blg1}, which states the correspondence between bridges and $\Lambda$-coalescents. For a random bridge $B$, let $s \in \cS_0$ be the tiling of (0,1) defined by the ranked sequences of jumps of $B$ (where continuous parts are associated with the dust component $s_0$). As usual, the correspondence arises by fixing a time (say $t=0$) and running $s$ backwards:

\begin{theorem}\label{T:BLG1} Let $(B_{s,t}, - \infty< s \le t < \infty)$ be the flow of bridges defined by (\ref{flowconvergence}). Let $S(t)$ be the tiling of (0,1) obtained from the bridge $B_{-t, 0}$, for all $t\ge 0$. Then
$(S(t), t\ge 0)$ has the same law as the ranked frequencies of a $\Lambda$-coalescent. Furthermore if $V_1, \ldots$ are i.i.d. uniform random variables on (0,1), we may define a partition $\Pi_t$ by saying $i \sim j$ if and only if $V_i$ and $V_j$ fall into the same jump of $B_{-t,0}$. Then
$$
(\Pi_t, t\ge 0) \text{ is a  $\Lambda$-coalescent}.
$$
\end{theorem}

The proof outlined above was not the route used by Bertoin and Le Gall to prove this theorem, and it might be worthwhile to turn this outline into a more precise argument.

\subsubsection{Stochastic Differential Equations}

The points of view developed above (that is, the flow of bridges on the one hand, and the Fleming-Viot process on the other hand), allow us to discuss analogues and generalisations of the Wright-Fisher\index{Wright-Fisher} stochastic differential equation, which was used to describe the proportion of individuals carrying a certain allele in a population whose genealogy is approximately Kingman's coalescent.

To explain the first result in this direction, we first introduce what may be called microscopic bridges or infinitesimal bridges: the is the bridge which describes the effect of one ``individual" at location $x \in [0,1]$ having a progeny of size $p \in (0,1)$. This bridge has the form:
\begin{equation} \label{elementarybridge}
b(u) = b_{x,p}(u) = u (1-p) +p \indic{u \ge x}.
\end{equation}
Let $(B_{s,t})_{-\infty<s\le t < \infty}$ be the flow of bridges constructed in (\ref{flowconvergence}) and let $F_t= B_{0,t}$. Thus $F_t(y)$ is the fraction of individuals at time $t$ descending from some individual in the interval $[0,y]$ of the population at time $0$ (this is indeed the equivalent of the quantity which we track when we study the Wright-Fisher diffusion, with initial fraction of alleles $a$ equal to $y$). When there is an atom $(x,p)$ at some time $-t$, then what is the change in $F_t(y)$? This atom means that an individual located at $x \in [0,1]$ is producing a macroscopic offspring, which represents a fraction $p$ of the population right after. Thus by the composition property, letting $F=F_{t^-}(y)$ and $F' = F_t(y)$, we see that $F' = b \circ F = F(1-p)$ when $F<x$, in which case the infinitesimal increment is $dF= -Fp$. If on the other hand, $F\ge x$, then we have $F'=p + F(1-p)$, and thus the infinitesimal increment is $dF= p(1-F)$. Thus define the function
$$
\psi(x,p,F) = \begin{cases}
-pF & \text{ if } F<x\\
p(1-F) & \text{ if } F \ge x.
\end{cases}
$$
If $p^{-2} \Lambda(dp)$ is a finite measure, there are only a finite rate of events in the stochastic flow of bridges of the previous paragraph, and if we label these events $(t_i, x_i,p_i)$, and let
$$
M= \sum_i \delta_{(t_i,x_i,p_i)}
$$
which is a Poisson point process with intensity $dt \otimes dx \otimes p^{-2} \Lambda(dp)$, we get immediately that $F_t=B_{0,t}$ may be written as a stochastic integral:
\begin{equation}\label{sde1}
F_t(y) = y + \int_{[0,t] \times [0,1] \times [0,1]} \!\!\! \!\!\!  M(\text{d} s, \text{d} x, \text{d} p) \psi(x,p,F_{s^-}(y)).
\end{equation}
It turns out that this stochastic integral still makes sense even if we don't assume that $x^{-2} \Lambda(dx)$ is a finite measure. This is, as usual, stated as a result of weak existence and uniqueness: any filtered probability space with a measure Poisson point process $M$ with intensity $dt \otimes dx \otimes p^{-2} \Lambda(dp)$ and cadlag process $X_t = (X_t(y), y \in [0,1])$, such that $X_t(y)$ satisfies the stochastic differential equation (\ref{sde1}) almost surely for all $y \in [0,1]$, is called a weak solution of (\ref{sde1}).

\begin{proposition}\label{P:SDEexistuniq} \emph{(Theorem 2 in \cite{blg2})} There exists a weak solution to $(\ref{sde1})$, with the additional property that a.s. for every $t\ge 0$, $X_t$ is a nondecreasing function on $[0,1]$. Furthermore, if $X$ is any solution to $(\ref{sde1})$, then $(X_t(y_1), \ldots, X_t(y_p))$ has the same distribution as the $p$-point motion $(F_t(y_1), \ldots, F_t(y_p))$. In particular there is weak uniqueness.
\end{proposition}

There is an associated martingale problem, which may be formulated as follows: given an atom $(x,p)$, we define an operator on $\cC^2$ functions $g:[0,1]^p \to \R$ by:
$$
\Delta_{x,p}g(y) = g(y+\psi(x,p,y)) - g(y) - \sum_{i=1}^p \psi(x,p,y_i) \frac{\partial g}{\partial y_i} (y)
$$
Then for every $y_1, \ldots, y_p$, the $p$-point motion $F_t(y_1), \ldots, F_t(y_p)$ satisfies
$$
g(F_t(y_1), \ldots, F_t(y_p)) - \int_{0}^t \cL g(F_s(y_1), \ldots, F_s(y_p))ds
$$
is a martingale, where
$$
\cL g(y) = \int_{0}^1 dx \int_0^1 p^{-2} \Lambda(dp)  \Delta_{x,p}g(y).
$$
This is well-defined as for any $g \in \cC^2$, by Taylor expansion one gets $\Delta_{x,p} g(y) \le Cp^2$ for some $C$ which does not depend on $p$ (or on $x$).

\subsubsection{Coalescing Brownian motions}

We end this statement with a surprising result of Bertoin and Le Gall, which links Kingman's coalescent to a certain flow of coalescing Brownian motions on the circle.

To this end, define an operator $T$ on $\cC^2$ functions defined on the torus $\mathbb{T} = \R /\Z$:
\begin{equation}\label{coalBM}
\cT g(y_1, \ldots, y_p)  = \frac12\sum_{i,j=1}^p b(y_i,y_j) \frac{\partial^2 g}{\partial y_1 \partial y_j}(y)
\end{equation}
where the covariance function $b(y,y')$ satisfies:
\begin{equation}\label{covariance}
b(y,y') = \frac1{12} - \frac12 d(y,y') (1-d(y,y')).
\end{equation}
The generator $\cT$ defines a martingale problem, and we note that if $X$ is a solution to this martingale problem (that is, if $g(X_t) - \int_0^t\cT g(X_s)ds$ is a martingale for every $g \in \cC^2(\mathbb{T})$), then each of the $p$ particles follows individually a Brownian motion on the torus with diffusion coefficient $\sqrt{1/12}$. However, these Brownian motions are not independent, and are correlated in a certain way. In particular, we will see that particles following this flow have the \emph{coalescing property}: if $X_t^i = X_t^j$ for some time $t>0$, this stays true ever after.

Let $X$ be a solution to the martingale problem defined by (\ref{coalBM}) with starting point $X_0=(V_1, \ldots, V_p)$ given by $p$ independent uniform random points on the torus, and define a partition $\Pi^p_t$ on $\cP_p$ by putting $i\sim j$ if and only if $X_t^i = X_t^j$, i.e., particles $i$ and $j$ have coalesced.
\begin{theorem}\label{T:flow}There is existence and uniqueness in law to the martingale problem defined by $(\ref{coalBM})$. For any solution $X$, the process $(\Pi_t, t\ge 0)$ is Kingman's $p$-coalescent.
\end{theorem}

\begin{proof} The uniqueness part of the result is a consequence of the fact that the generator is smooth away from the diagonal (i.e., $x_i \neq x_j$ for $i \neq j$) and of the fact that particles which hit each other coalesce, in the sense that they stay forever together. This can be seen as follows: consider for examples particles $1$ and 2, and let $T=\inf\{t\ge 0: X_t^1 = X_t^2\}$. Fix also a $z \in \mathbb{T}$, and consider the process
$$
Y_t = (z,X_t^1) - (z,X_t^2) \in [-1,1]
$$
where if $z,x \in \T$, then $(z,x)$ denotes the length of the counterclockwise arc from $z$ to $x$. This is a $\cC^2$ function of the trajectories $X_t^1, X_t^2, \ldots, X_t^p$ so long as neither $X_t^1 =z$ or $X_t^2=z$, so if $T' = \inf\{t\ge 0: X_t^1 = z \text{ or } X_t^2 = z\}$, and if $g$ is any $\cC^2$ real function, we get a local martingale
\begin{align*}
  M_t& =g(Y_{t\wedge T'}) - \int_0^{t\wedge T'} \cT g(X_s^1, \ldots, X_s^p)ds\\
  &= g(Y_t) -\int_0^t \frac12 |Y_s|(1-|Y_s|) g''(Y_s) ds
\end{align*}
for every $\cC^2$ function $g$ and any solution to (\ref{coalBM}). Thus on $[0,T']$, $Y_t$ is the diffusion on [-1,1] with generator
$$
\frac12|y|(1-|y|) \frac{d^2}{dy^2}
$$
for which zero is an absorbing boundary (this is, up to the sign, the same generator as in the Wright-Fisher diffusion, where the absorption property is easy to see). This guarantees that $Y_t\equiv 0$ for all $t \in[T,T']$, and from this the coalescence property follows easily.

To get the statement in the theorem which identifies the coalescent process $\Pi_t^p$ as Kingman's coalescent, the idea of the proof is as follows. Consider a measure $\nu(dp) = p^{-2} \Lambda(dp)$ on (0,1), and assume that $\nu$ is finite. Consider a Poisson point process of points
$$
M= \sum_{i} \delta_{(t_i,x_i, p_i)}
$$
with intensity $dt \otimes \lambda(dx) \otimes \nu(dp)$, where $\lambda$ is the Lebesgue measure on the torus $\T$. We use these points to create a stochastic flow of bridges just as above, except that now we consider functions not from $[0,1]$ to $[0,1]$ (bridges) but functions from $\T$ to $\T$. A coalescence occurring at individual $x \in \T$, with mass $p \in (0,1)$, corresponds to the composition by an elementary ``bridge" just as in (\ref{elementarybridge}), which sends an arc of size $p$ centered at $x$ to $x$ and sends the complement of this arc onto the full torus in a linear fashion. That is,
$$
\beta(y) = x \text{ if } d(y,x) \le p/2
$$
and, letting $\bar x$ be the point sitting opposite of $x$ in $\T$
$$
d(\bar x, \beta(y)) = \frac1{1-p} d(\bar x, y) \text{ otherwise. }
$$
We let $\beta_{x,p} = \beta$ be the above function, and we call
$$
\Phi_{s,t} = \beta_{x_k, p_k} \circ \ldots \circ \beta_{x_1, p_1}
$$
where $(x_i, p_i)$ are the list of atoms of $M$ between times $s$ and $t$ listed in increasing order (which is possible since $\nu$ is assumed to be finite). Taking $V_1, \ldots,$ a collection of i.i.d. uniform variables on $\T$, it is then trivial (see, e.g., Theorem \ref{FVduality}) to check that
\begin{equation}
(\Pi_t, t\ge 0) \text{ is a $\Lambda$-coalescent}
\end{equation}
where $\Pi_t$ is defined by saying $i \sim j$ if $\Phi_{0,t}(V_i) = \Phi_{0,t}(V_j)$. Specializing to the case where for $\eps>0$,
$$
\Lambda^\eps(dx) = \delta_{\eps}(dx)
$$
so that $\Lambda_\eps \Rightarrow \delta_0$ and the associated coalescent $\Pi^\eps$ converges in the Skorokhod topology towards Kingman's coalescent, it now suffices to study the limiting distribution as $\eps \to 0$ of the $p$-point motion $\Phi^\eps_t(V_1), \ldots, \Phi^\eps_t(V_p)$. Note for instance that in the case $p=1$, $\Phi^\eps(t)$ is a continuous-time random walk on $\T$ with mean zero and second moment which can be computed as
$$
\eps^2 \int_0^1 (x-\frac12)^2 dx = \frac{\eps^2}{12}.
$$
This is enough to characterize the limiting distribution of 1-point motions as Brownian motions with diffusion coefficients $\sqrt{1/12}$. The rest of the theorem follows by a similar Taylor expansion when $p\ge 2$.
\end{proof}

\newpage

\newpage

\section{Analysis of $\Lambda$-coalescents}

In this chapter we give some general asymptotic sampling formulae for $\Lambda$-coalescents that are the analogues of Ewens' sampling formula in these models.
The mathematics underneath these results relies heavily on the notion of continuous-state branching processes, which may be thought of as the scaling limits of critical Galton-Watson processes. After first stating these formulae, we give a basic exposition of the theory, and develop the connection to $\Lambda$-coalescents. This connection is then used to study in details the small-time behaviour of $\Lambda$-coalescents, i.e., close to the big-bang event of time $t=0$ when the process comes down from infinity. This raises many interesting mathematical questions, ranging from the typical number of blocks close to $t=0$, to fractal phenomena associated with variations in mass of these blocks. Surprisingly, while many questions concerning the limiting almost sure behaviour of $\Lambda$-coalescents have an answer, our understanding of limiting distributions is much more limited. We survey some recent related results at the end of this chapter.

\subsection{Sampling formulae for $\Lambda$-coalescents}

We start by stating some general asymptotic sampling formulae for $\Lambda$-coalescents, which are the analogues of the Ewens sampling formula\index{Ewens' sampling formula} for $\Lambda$-coalescents (see Theorem \ref{T:KingmanESF}). We focus in this expository section on the case of the infinite alleles model\index{Infinite alleles model}. To refresh the reader's memory, the problem we are interested in is the following. Consider a sample of $n$ individuals, and assume that the genealogical relationships between these individuals is given by a $\Lambda$-coalescent, where $\Lambda$ is an arbitrary finite measure on $(0,1)$. Conditionally given the coalescence tree, assume that mutations fall on the tree as a Poisson process with constant intensity $\rho$ on every branch. (In the case of Kingman's coalescent, it is customary to parameterize this intensity by $\theta =2 \rho$). Recall that in the infinite alleles model, every mutation generates a completely new allelic type. We are interested in describing the corresponding allelic partition\index{Allelic partition} of our sample, e.g., how many allelic types are we likely to observe in a sample of size $n$, how many allelic types have a given multiplicity $k\ge 1$ (i.e., are present in exactly $k$ individuals of this sample), etc. Recall also that in the case of Kingman's coalescent, a closed formula for the probability distribution of this partition is given by Ewens' sampling formula. In the general case of $\Lambda$-coalescents, the problem is noticeably harder, and in general no exact closed formula is known -- nor is it expected that such a formula exists. One possible approach, investigated by M\"ohle \cite{mohle}, is to derive a recursive formula for this distribution, which makes it possible to compute numerically some quantities associated with it. However it is hard to extract useful information from it. Instead, we follow here a different approach, which is to obtain asymptotic results when the sample size $n$ tends to infinity.

Our first result comes from \cite{bbl2} and gives the asymptotic number of allelic types $A_n$ for a general measure $\Lambda$, subject only to the assumption that the coalescent comes down from infinity.

\begin{theorem}\label{T:sfLambda}
For $\lambda>0$, let
\begin{equation}\label{brLambda}
\psi(\lambda) = \int_0^1 (e^{-\lambda x} - 1 + \lambda x) x^{-2} \Lambda(dx).
\end{equation}
As $n\to \infty$, we have the following asymptotics in probability:
\begin{equation}\label{sfLambdagen}
{A_n}\sim \rho \int_1^n \frac{\lambda }{\psi(\lambda)} d\lambda,
\end{equation}
by which we mean that the ration of the two sides converges to 1 in probability.
\end{theorem}

As the reader might have recognized, the function $\psi(\lambda)$ defined in (\ref{brLambda}) is the Laplace exponent of a L\'evy process whose
 L\'evy measure is $x^{-2} \Lambda(dx)$. There is in fact a connection between this L\'evy process and the $\Lambda$-coalescent, which goes through the notion of continuous-state branching process. This probabilistic connection is interesting in itself and is developed in the next sections. A necessary and sufficient condition for the $\Lambda$-coalescent to come down from infinity in terms of the function $\psi$ is obtained later in Theorem \ref{T:smalltimes} as a consequence of this connection.

There are certain cases where one can obtain much more precise information about the allelic partition, such as the entire asymptotic allele frequency spectrum\index{Allele frequency spectrum}. This is the case where $\Lambda$ has a property which we call (with a slight abuse of terminology) ``regular variation" near zero:

\begin{definition}\label{D:regvar}
  Let $\Lambda$ be a finite measure on $(0,1)$ and let $\alpha \in (1,2)$. We say that the $\Lambda$-coalescent has regular variation\index{Regular variation} with index $\alpha$ if there exists a function $f(x)$ such that $\Lambda(dx) = f(x)dx $ and a number $A>0$ such that
  \begin{equation}\label{regvar}
  f(x) \sim A x^{1-\alpha}
  \end{equation}
  as $x \to 0$.
\end{definition}

In that case, we obtain the following result. Let $\rho>0$ and assume that $\Lambda$ satisfies (\ref{regvar}) for some $1<\alpha<2$.
Let $\Pi$ be the random infinite allelic partition obtained by throwing a Poisson process of mutations on the infinite coalescent tree with constant mutation rate $\rho$.
For $n\ge 1$ and $1\le k \le n$, let $A_n(k)$ be the number of blocks of size $k$ of $\Pi|_{[n]}$.
Thus $A_n(k)$ is the number of allelic types of multiplicity $k$ in the first $n$ individuals of an infinite sample, under the infinite alleles model\index{Infinite alleles model}. Let also $A_n = \sum_{k=1}^n A_n(k)$ be the total number of allelic types, as above. Note in particular that the random variables $(A_n)_{n\ge 1}$ (resp. $(A_n(k))_{n\ge 1}$ for any $k\ge 1$) are now all simultaneously constructed on a common probability space, so that it now makes sense to talk about almost sure convergence in (\ref{sfLambdagen}). This coupling is natural in that it corresponds to revealing more and more data from a large sample, and is thus suitable for applications.

\begin{theorem}\label{T:Betaallelic}\emph{(\cite{bbs2}, \cite{bbl2})}
Under assumption \emph{(\ref{regvar})} we have, almost surely as $n \to \infty$:
  \begin{equation}\label{Numfamily}
  \frac{A_n}{n^{2-\alpha}} \longrightarrow \rho C
  \end{equation}
  where $C=\alpha (\alpha -1)/[A\Gamma(2-\alpha)\Gamma(\alpha)]$. Moreover, for every fixed $k\ge 1$:
  \begin{equation}\label{Numfamilyk}
  \frac{A_n(k)}{n^{2-\alpha}} \longrightarrow \rho C (2-\alpha) \frac{(\alpha-1) \ldots (\alpha + k -3)}{k!}
  \end{equation}
almost surely as $n \to \infty$. Moreover, if $P_1, P_2,\ldots$ denote the ordered allele frequencies in the population, then
\begin{equation}
  P_j \sim C' j^{\alpha -2},
\end{equation}
almost surely as $j \to \infty$, and $C' = (C/\Gamma(\alpha -1))^{1/(2-\alpha)}$.
\end{theorem}

\mn\textbf{Comments}. (1) The convergence in probability was first obtained in \cite{bbs2} in the case of Beta-coalescents with parameter $(2-\alpha, \alpha)$, using an exact embedding within the so-called stable Continuum Random Tree and an analysis of the mutation process using queues. However, these methods were limited to the case of Beta-coalescents and did not yield the almost sure limit. This extension was done in \cite{bbl1} and \cite{bbl2} using a different, martingale-based approach to the problem. The same result also holds for the frequency spectrum of the infinite sites model\index{Infinite sites model}\index{Site frequency spectrum}.

\mn (2) Taking $k=1$, we see that the fraction of singletons in the allelic partition is $A_{n}(1)/A_n \sim 2-\alpha$ almost surely as $n\to \infty$. Thus $\alpha$ can be measured in a straightforward fashion from a sample, because it is approximately 2 minus the proportion of alleles with multiplicity 1. Note that in Kingman's coalescent, this fraction is asymptotically $1/\log n$, and hence tends to 0 as $n \to \infty$. Thus if the fraction of singletons is not negligible in a particular data set, this is a good indication that Kingman's coalescent is not suitable for this data and that coalescent with multiple collisions are better approximations. Various data sets from pacific oysters suggest a value for $\alpha$ approximately around 1.4. See however the discussion in Example 4.2 of \cite{durrettbookDNA} and the work of Birkner and Blath \cite{BirknerBlath} for further investigation (but in the case of the infinite sites model\index{Infinite sites model}).

\mn (3) In the case where the $\Lambda$-coalescent does not come down from infinity, M\"ohle \cite{mohleAllelic} has obtained a limiting result for the number of allelic types if in addition one assumes that $\int_0^1 x^{-1} \Lambda(dx) < \infty$, i.e., by Theorem \ref{T:dustCNS}, if there is almost surely dust at any time in the coalescent. In that case, he was able to show that
$$
\frac{A_n}{n \rho} \longrightarrow A
$$
in distribution as $n \to \infty$, where $A$ has a distribution which can be described as follows: let $X_t = - \log S_t$, where $S_t$ is the mass of dust at time $t$ (and note that $X_t$ is then a subordinator). Then $A$ has the same distribution as $\int_0^\infty e^{- \rho t - X_t}dt$. (A similar result has been shown by Freund and M\"ohle \cite{MohleFreund} for coalescents with simultaneous multiple collisions that have dust). The condition that the coalescent has dust excludes cases such as the Bolthausen-Sznitman coalescent, but this particular example has been analysed in detail by Basdevant and Goldschmidt \cite{BasdevantGoldschmidt}  and with slightly less precise results by Drmota et al. \cite{Drmota+}.

\subsection{Continuous-state branching processes}

\subsubsection{Definition of CSBPs}

In this section, we backtrack a little to give an introduction to Continuous-State Branching Processes (or CSBP for short), which we will then use to give a flavour to some of the proofs in Theorem \ref{T:sfLambda} and Theorem \ref{T:Betaallelic}.

In a nutshell, CSBPs can be seen as generalisations and/or scaling limits of Galton-Watson processes. Our presentation departs from the classical one, in that we have chosen not the most elegant approach but the most effective one. In particular saves us the need to later introduce the technology of Continuous Random Trees, which would force us to get significantly more technical tan these notes are meant to be. However, this theory is extremely elegant and many ideas described below are more natural when seen through this particular angle, so we have included in the appendix some notions about these objects.

\medskip As mentioned above, a \emph{\csbp} is a continuous analogue of Galton-Watson processes. That is, the population size is now a continuous variable which takes its values in the set $\R_+$ (as opposed to the set $\Z_+$ for Galton-Watson processes). To define it properly, we first make the following observation in the discrete case. Let $(Z_n, n\ge 0)$ be a Galton-Watson process with offspring distribution $L$. Then, given $Z_n$, one may write $Z_{n+1}$ as $Z_{n+1} = \sum_{i=1}^{Z_n} L_i$, where $L_i$ are i.i.d. random variables distributed as $L$, so
\begin{equation}\label{GWrw}
Z_{n+1} - Z_n = \sum_{i=1}^{Z_n} X_i
\end{equation}
where $X_i = L_i-1$. If we view $X_i$ as the step of the random walk $S_n = \sum_{i=1}^n X_i$, then (\ref{GWrw}) tells us that we can view $Z$ as a time-change of the random walk $(S_n, n\ge 0)$, where to obtain $Z_{n+1}$ from $Z_n$ we run the random walk $S$ for exactly $Z_n$ steps. That is, one may write
\begin{equation}\label{Lampertidiscrete}
Z_n = S_{T_n}, n \ge 0
\end{equation}
where $T_n = Z_1+ Z_2+ \ldots +Z_{n-1}$. Similarly, if $(Z_t,t\ge 0)$ is a Galton-Watson process in continuous time (i.e., individuals branch at rate $a>0$ and give birth to i.i.d. offsprings with distribution $L$), then one can associate a continuous time random walk which jumps at rate $a>0$ to a position chosen according to the distribution of $X=L-1$. Thus, given $Z_t =z$, the rate at which $Z_t$ jumps to $z+x$ is simply $z$ times the rate at which the random walk $(S_t,t\ge 0)$ jumps to $z+x$.

It is this last representation which we want to copy in the continuous setting. The random walk $(S_t, t\ge 0)$ will be replaced by a process with independent and stationary increments $(Y_t,t\ge 0)$, i.e., a L\'evy process, which will have the property that all its jumps are nonnegative, since the jumps of the random walk are always greater or equal to -1. This $-1$ will vanish in the scaling limit and so we will only observe jumps. On the other hand, the positive jumps of $Y$ can be arbitrarily large. Thus let us fix $(Y_t,t\ge 0)$ a L\'evy process with no negative jumps (one says also \emph{spectrally positive}\index{Spectrally positive L\'evy process}), and let $\nu(dx)$ be the L\'evy measure of $Y$.\index{Levy measure@L\'evy measure} That is,
$$
\E(e^{-\lambda (Y_t-Y_0)}) = \exp (- t \psi (\lambda)),
$$
where
$$
\psi(\lambda) = a \frac{\lambda^2 }2  + b \lambda + \int_0^\infty (e^{-\lambda x} - 1 + \lambda x \indic{x \le 1}) \nu(dx).
$$
(Recall that $\nu$ is a L\'evy measure means that $\int_0^\infty (h^2 \wedge 1) \nu(dh) < \infty$.) The is the L\'evy-Khintchin formula\index{Levy-Khintchin@L\'evy-Khintchin formula} already discussed in the proof of Theorem \ref{T:Pitman}. To simplify our presentation, we assume in what follows that $a=b=0$. Then the corresponding L\'evy process may be characterized by its generator $G$ (see Sato \cite{sato}, pp. 205--212): if $f$ is rapidly decreasing function (an element of the Schwartz space, but there is no harm in thinking of a $\cC^\infty$ function with compact support, which forms a core for the generator), then
$$
Gf(x) = \int_{0}^\infty [f(x+h) - f(x) - h\indic{h \le 1} f'(x) ] \nu(dh).
$$
Essentially this formula means that jumps of size $h$ occur at rate $\nu(dh)$ (and through an ingenious system of compensation it is enough to require that the sum of the square of those jumps smaller than 1 up to a given time has a finite expectation).
We define the associated branching process as follows:

\begin{definition} \label{D:csbp} For $f \in \cC^\infty$ with compact support, let $Lf(z) = z Gf(z)$. We call continuous-state branching process associated with $\psi$, any process $(Z_t,t\ge 0)$ with values in $\R_+$ such that $$f(Z_t) - \int_0^t Lf(Z_s)ds$$ is a martingale with respect to the natural filtration of $Z$, for any $f \in \cC^\infty$ with compact support. This property determines uniquely the law of $(Z_t,t\ge 0)$, which is called the $\psi$-\csbp, or $\psi$-CSBP for short. The function $\psi$ is called the \emph{branching mechanism} of $Z$. \index{Continuous-state branching process (CSBP)}\index{Branching mechanism}
\end{definition}

It can be shown that $Z$ is a Markov process. This definition means that the transitions of a $\psi$-CSBP are the same as those of a L\'evy process with L\'evy exponent $\psi(\lambda)$, but they occur $z$ times as fast, where $z$ is the current size of the process. Note that CSBPs get absorbed at $Z_t = 0$, because then the rate of jumps is 0. The following result, which goes back to Lamperti, explains this further, by establishing the analogue to (\ref{Lampertidiscrete}) in continuous space.
\begin{theorem}\label{T:Lamperti2} Let $Z$ be a $\psi$-\csbp, and let $(Y_t,t\ge 0)$ be the associated L\'evy process. Then if $U(t) = \int_0^{t\wedge T} Y_s^{-1}ds$, where $T$ is the hitting time of 0 by $Y$, and if $U^{-1}(t)$ is the cadlag inverse of $U$, i.e.,
$$
U^{-1}(t) = \inf\{s \ge 0: U(s) > t\}.
$$
then
\begin{equation}\label{Lampertitransform}
(Z_t)_{t\ge 0} \overset{d}= (Y_{U^{-1}(t)})_{t\ge 0}.
\end{equation}
\end{theorem}

It is easy to check this result: indeed, everything is made so that the process defined by the right-hand side runs the clock at speed $z$ when $Z_t =z$, and has apart from this time-change the same transitions as the L\'evy process $(Y_t,t\ge 0)$. If we want to emphasize the starting point of the CSBP, we write $Z_t(x)$ to mean that the process was initially started at $Z_0 =x >0$. As a simple consequence of this definition, we get the following property:

\begin{theorem}\label{T:branch} Let $(Z_t(x),t\ge 0)$ be a $\psi$-CSBP. Then $Z$ enjoys the \emph{branching property}: that is, if $x,y>0$, and if $Z'(y)$ denotes an independent $\psi$-\csbp{} started from $y$ independent of $Z$, then we have the representation:
\begin{equation}\label{branching}
  Z(x+y) \overset{d}= Z(x) +Z'(y),
\end{equation}
where the equality is an identity in distribution for the processes.
\end{theorem}

\begin{proof}
It is obvious that the right-hand side is also a Markov process with the correct transition rates, so the right-hand side is indeed a $\psi$-CSBP, and its starting point is obviously $x+y$. Thus the law of the right-hand side is indeed identical to the law of $Z(x+y)$.
\end{proof}

The meaning of the branching property (\ref{branching}), is as follows: if the initial population is $x+y$, we can think of these two subpopulations evolving independently of one another, and their sum gives us the total population $Z(x+y)$. This is why it is often convenient to record the initial population $x$ as $Z(x)$. Theorem \ref{T:branch} is traditionally used as the definition of CSBPs: this is indeed the definition used by  Ji\v{r}ina in 1958 \cite{Jirina}, where these processes were first discussed: a \csbp{} is any Markov process on $\R_+$ which enjoys the branching property. That the two definitions are equivalent is a sequence of theorems due to Lamperti \cite{lamp, lamperti2}. We have preferred to use Definition \ref{D:csbp} because the role of the measure $\nu$ is more immediately transparent, and the properties of CSBPs can be established much more directly, as we will see in what follows. When using Theorem \ref{T:branch} as a definition, Lamperti's transformation theorem (Theorem \ref{T:Lamperti2}) is far from obvious, and in fact the proof in \cite{lamperti2} misses a few cases. A recent paper by Caballero, Lambert and Uribe Bravo \cite{cablam} contains several proofs and a thorough discussion.

\begin{theorem} \label{T:lamperti1}\emph{(Lamperti \cite{lamp})} Any \csbp{} $(Z_t,t\ge 0)$ is the scaling limit of Galton-Watson processes. That is, there exists a sequence of offspring distributions $L^{(N)}$ ($N\ge 1$), and a sequence of numbers $c_N$, such that, if $Z^{(N)}$ denotes the Galton-Watson process with offspring distribution $L^{(N)}$ started from $N$ individuals, then
$$
\left(\frac1NZ^{(N)}_{t/c_N}, t\ge 0\right) \underset{N \to \infty}\longrightarrow (Z_t(1), t\ge 0)
$$
weakly in the Skorokhod topology.
\end{theorem}

\begin{proof} (sketch). This is a rather simple consequence of the classical fact that any L\'evy process can be approximated by a suitable random walk: that is, there exists a sequence of step distribution $X^{(N)}$ and constants $c_N$ such that $(\frac1N S^{(N)}_{t/c_N},t\ge 0)$ converges weakly in the Skorokhod topology towards the L\'evy process $(Y_t,t\ge 0)$, where $(S^{(N)}_t,t\ge 0)$ is the random walk with step distribution $X^{(N)}$. Furthermore, $X^{(N)}$ can be chosen to be integer-valued and ``skip-free" in the sense that $X^{(N)} \ge -1$ almost surely. Then the offspring distribution $L^{(N)}$ is simply constructed as $L^{(N)} = X^{(N)}+1$. The representation (\ref{Lampertidiscrete}) then tells us that the relation (\ref{Lampertitransform}) must hold in the limit, and hence the result follows.
\end{proof}

\medskip An important example of \csbp{} is given by the class of $\alpha$-stable processes:
\begin{definition} The stable CSBP\index{Stable!CSBP} with index $\alpha \in (1,2)$ is the \csbp{} associated with the stable L\'evy measure
\begin{equation}\label{stableLevy}
\nu(dx) = \frac\alpha{\Gamma(1-\alpha)} x^{-\alpha-1}dx.
\end{equation}
In this case, the branching mechanism is $\psi(u) = Cu^\alpha$ for some $C>0$.
\end{definition}

In fact, if $\psi(u)=u^2$ (quadratic branching) it is still possible to define a corresponding CSBP. Naturally, in that case the process is related to Brownian motion and is nothing else but the Feller diffusion: see Theorem \ref{T:feller} in the appendix. In this case we still speak of the $2$-stable branching process.

We now come to an interesting property, which shows a relation between branching processes and a certain differential equation. It turns out that this differential equation lies at the heart of the analysis of $\Lambda$-coalescents.

\begin{theorem}\label{T:Lamperti3} Let $Z$ be a $\psi$-\csbp. Then for all $\lambda \ge 0$,
\begin{equation}\label{branchingmech}
\E(\exp(- \lambda Z_t(x))) = \exp(-x u_t(\lambda)),
\end{equation}
where the function $t\mapsto u_t(\lambda)$ is the solution of the differential equation
\begin{equation}\label{ODE3}
\begin{cases}
\displaystyle \frac{d u_t(\lambda)}{dt} = - \psi(u_t(\lambda))\\
u_0(\lambda) = \lambda.
\end{cases}
\end{equation}
\end{theorem}

\mn \textbf{Remark.} This connection is the prototype of some deeper links which arise when the branching process is endowed with some additional geometric structure, in which case the differential equation becomes a partial differential equation: see, e.g., Le Gall \cite{LeGallzurich}.

\begin{proof}
Define $F(t,x, \lambda)$ by saying $\E_x(e^{-\lambda Z_t}) = \exp( - F(t,x,\lambda))$. By the branching property, it is easy to see that $F(t,x,\lambda)$ must be multiplicative in $x$: that is, there exists $f_t(\lambda)$ such that $F(t,x,\lambda) = x f_t(\lambda)$. Then the Markov property shows that $f_t (f_s(\lambda)) = f_{t+s}(\lambda)$. So Theorem \ref{T:Lamperti3} can be rephrased as saying that any solution to this functional equation must in fact satisfy (\ref{ODE3}) for some Laplace exponent $\psi(\lambda)$ of some spectrally positive L\'evy process $(Y_t,t\ge 0)$.

To see why this is true, we go back to the discrete case, where the argument is somewhat more transparent. Thus consider a Galton-Watson process $(Z_t,t\ge 0)$ in continuous time: each individual branches at rate $a>0$ and leaves i.i.d. offsprings distributed according to $(p_k)_{k\ge 0}$.
Let $Q$ be the generator of the process: thus $Q= (q_{ij})_{i,j\ge 0}$ where $q_{ii} = -a + ap_1$ (since nothing happens when an individual branches and leaves 1 offspring), and
$q_{ij} = a p_{j-i+1} $ if $j \ge i-1$. The Kolmogorov backward equation $P'(t) = QP(t)$\index{Kolmogorov backward equation} shows that
$$
P'_{1j}(t) = \sum_{k \ge 0} Q_{1k} P_{kj}(t) = -a P_{1j}(t) + \sum_{k \ge 0}a p_k P_{kj}(t).
$$
Therefore, if we look at the moment generating function of $Z_t$, i.e., $F(s,t) = \E(s^{Z_t}|Z_0=1) = \E_1(s^{Z_t})$, we observe that
\begin{align*}
\frac{\partial F}{\partial t} &= \sum_{j \ge 0} s^j P'_{1j}(t) = - a F(s,t) + \sum_{j\ge 0} s^j \sum_{k\ge 0} a p_k P_{kj}(t)\\
& = - aF(s,t) +a\sum_{k\ge 0} p_k \E_k(s^{Z_t}) \text{ by Fubini's theorem}\\
& = - a F(s,t) + a \sum_{k\ge 0} p_k F(s,t)^k \text{ by the branching property}.\\
\end{align*}
Thus if $\phi(\lambda) = a u - g(u)$, where $g(u)$ is the moment generating function of $(p_k)$, we have
\begin{equation}\label{ode3discrete}
\frac{\partial F}{\partial t} = -\phi( F(s,t))
\end{equation}
which is the precursor of (\ref{ODE3}). Using the discrete approximation of Theorem \ref{T:lamperti1} and taking the limit in (\ref{ode3discrete}), we obtain (\ref{ODE3}).
\end{proof}

This explains further the role of the branching mechanism $\psi$: it may be interpreted as the Laplace exponent for the infinitesimal offspring distribution.\index{Branching mechanism}

When $Z$ is the $\alpha$-stable CSBP and $\alpha =2$ (quadratic branching), the differential equation\index{Differential equation} (\ref{ODE3}) is $\cdot u = - cu^2$, which should look familiar to you: it is precisely the same differential equation that was obtained for the heuristic analysis of Kingman's coalescent in (\ref{ode}). This is naturally not a coincidence: in fact, we will develop in next chapter a connection between $\Lambda$-coalescents and $\psi$-CSBP (for a certain branching function $\psi$ to be determined) which will finally make this connection rigorous, and from which many other properties will follow.

\medskip We end this section on the basic properties of branching processes with a statement about a necessary and sufficient condition for the process to become extinct, and, when the process does become extinct, what is the chance it has already gotten extinct by some time $t$. Here, becoming extinct means that there is a finite $T>0$ such that $Z_T=0$ (since 0 is absorbing, then $Z_t = 0$ for all $t\ge T$ automatically).

\begin{theorem}\label{T:Grey}\emph{(Grey's criterion \cite{grey})}\index{Grey's criterion}
Let $Z$ be a $\psi$-CSBP. Then $Z$ becomes extinct in finite time almost surely if and only if
\begin{equation}\label{Greycondition}
\int_1^\infty \frac{dq}{\psi(q)} <\infty.
\end{equation}
Let $p_z(t)$ denote the probability that $Z_t>0$, given $Z_0 = z$. Then $p_z(t) = 1- \exp( - z v(t))$, where $v(t)$ is defined by
\begin{equation}\label{D:v}
\int_{v(t)}^\infty \frac{dq}{\psi(q)} = t.
\end{equation}
\end{theorem}

\begin{proof} (sketch)
  Note that
  $$
  \P(Z_t = 0 ) = \lim_{\lambda \to \infty} \E(e^{-\lambda Z_t}) = \lim_{\lambda \to \infty} e^{-zu_t (\lambda)},
  $$ where $u_t(\lambda)$ is the solution to the differential equation (\ref{ODE3}). Observe that this differential equation can be solved explicitly:
  $$
  \frac{\dot u_s}{\psi(u_s)} = - 1
  $$
  so integrating between times $0$ and $s$, and making the change of variables $x =u_t$ we find (since $u_0(\lambda)= \lambda$):
  \begin{equation}\label{LapGrey}
  \int_{u_t(\lambda)}^\lambda \frac{dq}{\psi(q)} = t.
  \end{equation}
  Thus if (\ref{Greycondition}) does not hold, it must be that $\lim_{\lambda \to \infty } u_t(\lambda) = \infty$, since the right-hand side of (\ref{LapGrey}) does not depend on $\lambda$. Hence $\P(Z_t =0) =0,$ for all $t\ge 0$. On the other hand, if (\ref{Greycondition}) holds, then $\lim_{\lambda \to \infty} u_t(\lambda) = v(t)$ as defined by (\ref{D:v}). Thus there is positive probability of extinction by time $t>0$, which is equal to $\exp(-z v(t))$. One must work slightly harder to show that eventual extinction has probability 1.
\end{proof}

Note that the situation is slightly more complicated in the continuous world than in the discrete. For instance, (\ref{Greycondition}) may fail but the process still has $Z_t \to 0$ as $t\to \infty$ (for instance, $Z_t(x)=xe^{-t}$ is such a CSBP!) On the other side, there can also be explosions: this is investigated by Sliverstein \cite{silverstein}.

\subsubsection{The Donnelly-Kurtz lookdown process}

Let $(Z_t, t\ge 0)$ be a $\psi$-\csbp. We would like to have a notion of genealogy for $Z$, i.e., a way of making sense of the intuitive idea that in this evolving population, some individuals descend from a certain group of individuals at some earlier time. This turns out to be rather delicate in this continuous setting, and requires some additional technology. There are essentially two ways to proceed: one is to introduce continuum random trees, i.e., trees with the property that they branch continuously in time, such that the total population process has the law of $(Z_t,t\ge 0)$. A brief overview of this approach is presented in the appendix. The other possibility, which we have chosen to discuss here, is Donnelly and Kurtz's so-called
(generalised) lookdown process\index{Lookdown process}\index{Donnelly-Kurtz}. That these two notions of genealogy coincide is a theorem proved in \cite{bbs1} (see Theorem 12 in that paper). However, for the developments we have in mind (i.e., the analysis of $\Lambda$-coalescents through a connection with CSBPs), the approach which we propose here does not rely anymore on Continuum Random Trees, and so this aspect of things may be ignored by the reader. We stress however that initially (in \cite{bbs1, bbs2}, the connection relied on the continuous random tree approach rather than the Donnelly-Kurtz lookdown approach developed here, which owes much to the work of \cite{bbl2}.
Hopefully the latter approach makes these ideas more transparent.

\medskip We now describe the lookdown process. This was originally introduced in \cite{dk}, in order to provide a system of countable particles whose empirical measure would be a version of some predetermined measure-valued process (and CSBPs can precisely be viewed as one such example). This representation is very much in the spirit of the classical Fleming-Viot representation of Theorems \ref{T:FVdef} and \ref{T:FVLambdadef}. One way this process could be described would be to ask the following question: suppose a continuous population evolves according to the dynamics of a CSBP. Then we ask: \emph{what would samples from that population look like as time evolves?}

To ask the question in a more specific way, we may as usual endow each individual initially present with a unique allelic type (which, for us, will be just a uniform random variable on (0,1)). We sample from the population at time 0 and run the population dynamics  for some time. How has that sample evolved? Our only requirement is that each allele at time $t>0$ is represented with the correct frequency in our sample, but otherwise we may proceed as we wish to run the dynamics. The answer is as follows: assume for simplification that $Z$ has only jumps and no Brownian component. Suppose that at some time $t>0$, the population $Z$ has a jump of size $\Delta Z_t>0$. This jump is produced by ``one individual" who has a macroscopic offspring of size $\Delta Z_t$. In the population right after the jump, a fraction
$$
p=\frac{\Delta Z_t}{Z_t}
$$
has just descended from that individual and thus carries the type of this individual. The other individuals in the population haven't died but their relative frequency is now only $(1-p)$: thus if an allele occupied a fraction $a$ of the population, it now occupies a fraction $a(1-p)$. One can check that the following procedure produces exactly the desired change:

\begin{figure}[h]
\begin{center}
\includegraphics[height=5cm,width=6cm]{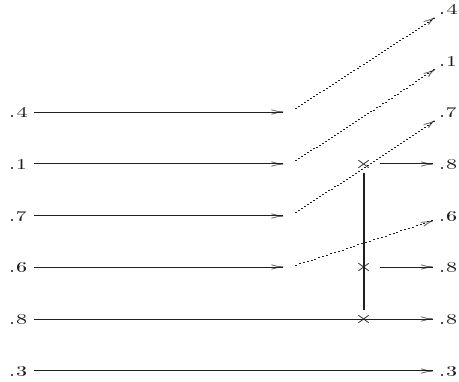}
\end{center}
\caption{Representation of the lookdown process. Levels 2,4 and 5
participate in a birth event. Other types get shifted upwards. The
numbers on the left and on the right indicate the types before and
after the birth event.}
\label{Fig:DK}
\end{figure}

\begin{enumerate}
\item A proportion $p$ of individuals is selected by independent coin-toss. (They are said to participate in the birth event). The type of those individuals is changed (if needs be) to the type of the lowest individual participating in the birth event.

\item The other types are shifted upwards to the first possible non-participating individual. (See figure \ref{Fig:DK})

\end{enumerate}

This procedure described what happens at any single jump time of the population $Z$. If this procedure is applied successively at all jump times of the population, then we obtain a countable system $(\xi_i(t), t\ge 0)$ which indicates the type of the individual at level $i$. Naturally, at any time $t\ge 0$, $\{\xi_i\}$ forms a subcollection of the initial types $\{U_i\}$. While it seems a priori that no type can ever be lost in this construction (things just get shifted upwards, and never die), this is not accurate: in fact, jumps may accumulate and push a given type ``to infinity", at which point it has disappeared from the population and does not exist anymore. In fact, we will see that under Grey's condition (\ref{Greycondition}), the number of types that survive up to any time $t>0$ is only finite almost surely. Since the initial collection of types is initially exchangeable and that the rest of the evolution is determined by i.i.d. coin-tosses, we immediately get that $(\xi_i(t))_{i=1}^\infty$ is an infinite exchangeable sequence for each $t\ge 0$.

\begin{definition}
\label{D:DK} The almost sure limit
$$
\Xi_t:= \lim_{N\to \infty}\frac1N \sum_{i=1}^N\delta_{\xi_i(t)}
$$
is called the Donnelly-Kurtz (generalised) lookdown process\index{Donnelly-Kurtz}\index{Lookdown process}.
\end{definition}

We have done everything so that $\Xi_t$ would accurately represent the composition of types in the population $Z$, so it remains to now state the theorem. First, we associate to the continuous-state branching process $(Z_t,t\ge 0)$ a measure-valued process $(M_t, t\ge 0)$ which is a measure on $(0, \infty)$. Informally speaking, for $r \in [0,1]$, $M_t([0,r])$ is the size of the population descended from the first $r$ individuals of the population, i.e., from individuals initially located in the interval $[0,r]$. One way to define it is to define a process $\{Z_t(x)\}_{t\ge 0, x \ge 0}$ simultaneously for all $t$ and for all $x \ge 0$. The way to do so is to use the branching property: for instance, to construct simultaneously $Z_t(x)$ and $Z_t(x+y)$, we start with a version of $Z(x)$ and add an independent version of $Z(y)$. The sum of these two processes is used to construct $Z(x+y)$. This construction defines uniquely a measure $M_t$ on $\R_+$ such that
$$
M_t([x,y]) = Z_t(y) - Z_t(x), \ \ 0 \le x \le y \le 1.
$$
From this measure $M_t$, there is another way to represent the composition of the population given this process $M_t$, which is simply to look at the ratio process:
$$
R_t[0,r] = \frac{M_t[0,r]}{Z_t}, \ \  r \le 1,
$$
where $Z_t = M_t[0,1]$ is the total mass. Thus $R_t$ is a measure with total mass equal to 1.

\begin{theorem}\label{T:DK}
  The ratio process $(R_t,t\ge 0)$ and the lookdown process $(\Xi_t,t\ge 0)$, have the same distribution.
\end{theorem}





A consequence of this representation is the following result about the Donnelly-Kurtz lookdown process, which tells us that the number of types which survive to time $t>0$ is finite if and only if the associated branching process $(Z_t,t\ge 0)$ dies out almost surely.

\begin{corollary}\label{C:numbertypes}
  The number of types surviving at time $t>0$ in the Donnelly-Kurtz lookdown process $(\Xi_t,t\ge 0)$ (i.e., the number of atoms of $\Xi_t$), is finite for all $t>0$ almost surely, if and only if Grey's condition (see Theorem \ref{T:Grey}) is fulfilled. Moreover, if Grey's condition holds, then the number of types which survive at time $t$ is Poisson with mean $v(t)<\infty$, where $v$ is the function defined in (\ref{D:v}).
\end{corollary}

\begin{proof}
Fix some $n\ge 1$, and separate the interval $[0,1]$ into $n$ subintervals $[i/n, (i+1)/n]$. In the Donnelly-Kurtz lookdown process, we identify all individuals whose types fall into the same subinterval. Thus by the branching property, we have $n$ copies of a $\psi$-CSBP, all started with mass $1/n$, which we view as $n$ different families. How many of those families have survived by time $t$? By Theorem \ref{T:Grey}, this is a Binomial random variable with parameters $n$ and $p_{z}(t) $ where $z=1/n$. Recall that $p_z(t) = 1- \exp( - z v(t)) \sim v(t) /n$ as $n \to \infty$. Thus by the Poisson approximation to binomial random variables, we find that as $n\to \infty$, if $K_n(t)$ is the number of families surviving by time $t$,
$$
K_n(t) \overset{d}\longrightarrow \Poi(v(t))
$$
as $n \to \infty$. On the other hand, $\lim_{n \to \infty} K_n(t)$ exists almost surely and is equal to the number of types in the Donnelly-Kurtz lookdown process. Thus this has the distribution of the random variable in the right-hand side of the above equation, and proves the corollary.
\end{proof}

Corollary \ref{C:numbertypes} will be a crucial step towards proving fine results on $\Lambda$-coalescents in next section. The Poisson distribution for the number of types is perhaps easier to understand in terms of continuous random trees, and should be compared with Theorem \ref{T:heightdescription} . For instance, this is the same reason why the number of excursion of Brownian motion up to time $\tau_1$ (the inverse local time) is a Poisson random variable.

\subsection{Coalescents and branching processes}

\index{Small-time behaviour}

Having given a brief overview of continuous-state branching processes and the lookdown process, we are now ready to describe the connection which relates $\Lambda$-coalescents to certain CSBPs. Intuitively, this connection states that, \emph{at small times}, the genealogy of any $\psi$-CSBP is given by a certain $\Lambda$-coalescent. It is essential to note that this description is valid only asymptotically as $t \to 0$ (however, an exact embedding exists for the $\alpha$-stable case, as will be discussed below). This connection allows us to state a general result about $\Lambda$-coalescents, which gives a second necessary and sufficient condition for coming down from infinity, and if they do, tells us the \emph{speed} at which they come down from infinity. That is, we find a deterministic function $v(t)$ such that $N_t / v(t) \to 1$ almost surely as $t \to 0$, where $N_t$ is the number of blocks at time $t$.

\subsubsection{Small-time behaviour}

We point out that, if $\Lambda$ does not satisfy Definition \ref{D:regvar}, some considerable difficulties arise, and will generally lead to $N_t$ oscillating between different powers of $t$ as $t \to 0$. (An instructive example of such a measure is analysed in a fair amount of details in \cite{bbl2}). The general solution that we now present is due to \cite{bbl1, bbl2}.

Let $\Lambda$ be any finite measure on $[0,1)$. Then define
\begin{equation}
  \label{D:psi}
  \psi(q) = \int_0^1 (e^{-qx}-1+qx )x^{-2}\Lambda(dx).
\end{equation}
Let $(Z_t,t\ge 0)$ be the CSBP with branching mechanism $\psi$.

\begin{theorem}
  \label{T:smalltimes}
  The $\Lambda$-coalescent comes down from infinity if and only if $Z$ becomes extinct in finite time, i.e.,
  \begin{equation}
  \int_1^\infty \frac{dq}{\psi(q)} < \infty
  \end{equation}
  If it does, then define $v(t)$ by: $\int_{v(t)}^\infty \frac{dq}{\psi(q)} = t$. Then as $t \to 0$,
  \begin{equation}\label{speedCDI}
  \frac{N_t}{v(t)} \longrightarrow 1,
  \end{equation}
  almost surely and in $L^p$ for every $p\ge 1$.
\end{theorem}

\begin{proof}
It is easy to understand that coming down from infinity might be related to the extinction in finite time of the CSBP. Indeed, by Theorem \ref{T:Grey}, extinction of a CSBP is equivalent to the fact that only finitely many individuals at time 0 have descendants alive at time $t>0$, for any $t>0$. When this occurs, the entire population at time $t>0$ comes from a finite number of ancestors and thus, running time backwards (assuming that the genealogy is approximately described by a $\Lambda$-coalescent), the coalescent has come down from infinity. The convergence (\ref{speedCDI}) also follows intuitively from a similar argument and Corollary \ref{C:numbertypes} once this connection is accepted.

We now explain why the genealogy of $(Z_t,t\ge 0)$ should be given by a $\Lambda$-coalescent asymptotically as $t \to 0$. Consider the CSBP $(Z_t,t\ge 0)$ defined by (\ref{D:psi}) and its genealogy defined in terms of the lookdown process. Then we know that at a time $t>0$ such that $\Delta Z_t >0$, a proportion $p= \Delta Z / Z$ of lineages is coalescing. Assume to simplify that $Z_0=1$ (this is of course unimportant). Recall that, by Lamperti's transform, $Z$ is a time-change of a L\'evy process $(Y_t,t\ge 0)$ with Laplace exponent $\psi(q)$. Now, both the time-change functional
$$
U^{-1}_t = \inf\Big\{s>0: \int_0^s \frac{ds}{Y_s} >t\Big\}
$$
and $Y$ are almost surely continuous at $t=0$. Therefore $Z$ is also continuous with probability 1 near $t=0$, and it follows that if $t$ is small, the proportion of lineages that coalesces at time $t$ is
\begin{equation}\label{approxlineages}
p =\frac{\Delta Z_t}{Z_t} \approx \Delta Z_t.
\end{equation}
But using Lamperti's transform one more time, we see that $\Delta Z_t = \Delta Y_{U^{-1}_t}$. Thus, we conclude that at time $U_t$, a fraction approximately $\Delta Y_t$ of lineages coalesces. Remembering that $U_t = \int_0^t (1/Y_s)ds$, we see that for $t$ small, $U_t \approx t$ as well.

To summarize, we have seen that with a negligible time-change the fraction of lineages that coalesces is equal to the jumps of the spectrally positive L\'evy process $(Y_t,t\ge 0)$. But by the L\'evy-It\^o decomposition,\index{Levy-Ito decomposition@L\'evy-It\^o decomposition} these jumps occur precisely as a Poisson point process with intensity
$$
dt \otimes x^{-2}\Lambda(dx)
$$
since, from (\ref{D:psi}), we see that the L\'evy measure of $Y$ is $x^{-2} \Lambda(dx)$. By the Poisson construction of $\Lambda$-coalescent (Theorem \ref{T:Poissonconstruction}), this gives precisely a $\Lambda$-coalescent.
\end{proof}

While this is a convincing argument for why the genealogy of $Z$ close to $t=0$ should be given by a $\Lambda$-coalescent approximately, it is much harder to turn it into a rigorous proof. The difficulty, of course, lies in the error made by the approximations. The main source of errors is due to the potentially wild fluctuations of the L\'evy process $Y$ around its starting point. These fluctuations are understood with a fair amount of details (see, e.g., \cite{pruitt}). It is for instance known that the increments $Y_t - Y_0$ may oscillate between two different powers of $t$ which are known as the inverse lower- and upper-index respectively. This helps controlling the error but ultimately this approach needs some assumptions on the regularity of the L\'evy process: namely, that the lower-index $\beta$ should be strictly greater than 1. The approach which we describe below bypasses these difficulties by giving a direct proof of Theorem \ref{T:smalltimes}.

\subsubsection{The martingale approach}

We introduce now a martingale discovered in \cite{bbl1} which is a crucial step for the proof of Theorem \ref{T:smalltimes}. Observe that the function $v(t)$ is the unique solution to the equation
\begin{equation}\label{vEquation}
  \log v(t) - \log v(z) +\int_z^t \frac{\psi(v(r))}{v(r)}dr = 0, \ \ \forall 0<z<t.
\end{equation}
Since we wish to prove that $N_t$ is approximately equal to $v(t)$, it makes sense to study the quantity:
\begin{equation}
\label{vEquation2}
X_z^t:=  \log N_t - \log N_z +\int_z^t \frac{\psi(N_r)}{N_r}dr, \ \ \forall 0<z<t.
\end{equation}

\begin{lemma}\label{L:martingale}
  There exists a bounded process $(H_t,t\ge 0)$ such that $$M_z^t: = X_z^t + \int_z^t H_r dr $$ is a martingale.
  Moreover, there exists $C>0$ such that the second moment process of $M_z$ satisfies $\E((M_z^t)^2) \le C (t-z)$ for all
  $0<z<t$.
\end{lemma}

Since the second moment process of $M_z$ is small, this implies by Doob's inequality that $M_z$ must be small, and hence (since the infinitesimal error $H$ is negligible) that $X_z^t$ itself is small. Since $v$ is the unique solution to (\ref{vEquation}), this and some further analysis yields Theorem \ref{T:smalltimes}.

\mn \textbf{Why is this a martingale?} Let $0<p<1$ and assume the number of blocks is currently $n$. Let $Y_{n,p}$ be a Binomial$(n,p)$ random variable. We may think of $Y$ as the number of blocks that participate in a $p$-merger. If $Y>0$, then the number of blocks after the $p$-merger is $n- Y_{n,p}+1$. Thus we see that, since $Y$ is typically small compared to $n$, and using the approximation $\log(1-x) \approx -x$,
\begin{align*}
  \E(d \log N_s | \cF_s) & = \int_0^1 \E\left[\log \frac{n- Y_{n,p} + \indic{Y_{n,p}>0}}n \right] p^{-2}\Lambda(dp)\\
  &\approx \int_0^1 \left[\frac{- \E(Y_{n,p}) + \P( Y_{n,p}>0)}n \right] p^{-2} \Lambda(dp)\\
  &=  {\int_0^1 \frac{-np + 1 - (1-p)^n }n p^{-2} \Lambda(dp)}\\
\end{align*}
Recalling that $n=N_t$ and noting that $(1-p)^n \approx e^{-np}$ in the integral above, we recognize in this integral the right-hand side of the equality which defines the function $\psi$ in (\ref{D:psi}). Hence we conclude that, up to some small errors,
\begin{equation}
\E(d \log N_s | \cF_s)  \approx - \frac{\psi(N_s)}{N_s}
\end{equation}
which is coherent with Lemma \ref{L:martingale}.

\mn \textbf{Interpretation.} With hindsight, the martingale $M_z^t$ has a simple interpretation. Recall that if $N_t$ is the number of blocks of a $\Lambda$-coalescent, then we always get a martingale $(\cM^z_t, t\ge z)$ for all $z>0$ by defining
$$
\cM^z_t = N_t + \int_z^t \gamma(N_s)ds, \ \ t \ge z
$$
where $\gamma(n) = \gamma_n = \sum_{k=2}^n (k-1) {n\choose k} \lambda_{n,k}$. It is not hard to see analytically that $\psi(\lambda) \sim \gamma(\lambda)$ as $\lambda \to \infty$. Thus the martingale $M_t^z$ of Lemma \ref{L:martingale} can be viewed as the martingale that one obtains from applying \emph{It\^o's formula}\index{It\^o's formula@Ito's formula} for discontinuous processes to $\log(N_t)$. From this point of view it is rather surprising that the method of \cite{bbl1} is conclusive: indeed, the logarithm function of $t$ is typically insensitive to small fluctuations and only picks up variations in the ``power" of $t$. However this turns out to be a strength as well, since the greatest challenge in this problem is to control wild fluctuations of the function $\psi(\lambda)$, which may oscillate between two different powers of $\lambda$.

\subsection{Applications: sampling formulae}

We now briefly explain how to obtain Theorem \ref{T:sfLambda} and Theorem \ref{T:Betaallelic} from the small-time behaviour of the number of blocks, i.e., Theorem \ref{T:smalltimes}.

\mn \emph{Sketch of proof of Theorem \ref{T:sfLambda}.} Consider the infinite alleles model, and make the following observation. Every mutation that appears on the tree is quite likely to have a corresponding representative in the allelic partition. Indeed, once a mutation arrives on the tree, it becomes quite difficult to fully disconnect it from the leaves: this is because a randomly chosen mutation is quite likely to be at the top of the tree. By analysing this process more carefully, a result of \cite{bbl2} shows:

\begin{lemma}\label{L:familyerrors}
  Assume that the $\Lambda$-coalescent comes down from infinity. Let $M_n$ be the total number of mutations on the coalescence tree restricted to $[n]$, and let $A_n$ denote the total number of families in the allelic partition restricted to $[n]$. Then
  $
  {A_n}/{M_n} \to 1,
  $
  in probability.
\end{lemma}

Given this lemma, our first task is thus to estimate the total number of mutations on the tree. (Note that this is identical to the total number of allelic types in the infinite sites model). Since mutations fall as a Poisson process with intensity $\rho$, we have that given the coalescence tree,
\begin{equation}
S_n = \text{Poisson}(\rho L_n),
\label{SegLength}
\end{equation}
where $L_n$ is the \emph{total length of the tree}\index{Total length of tree}, i.e., the sum of all the edge lengths in the coalescence tree of the first $n$ individuals in the sample. Thus the problem becomes that of finding asymptotics of $L_n$. But note that if initially the number of blocks is $N_0=n$, then the total length of the tree may be written as
\begin{equation}\label{Length-Numblocks}
L_n = \int_0^\zeta N_t dt
\end{equation}
where $\zeta$ is the coalescence time of all $n$ individuals. Indeed the contribution to $L_n$ of the time interval $(t, t+dt)$ of all branches in the tree is $N_tdt$, so integrating gives us the result (\ref{Length-Numblocks}). Using consistency of the $\Lambda$-coalescent, we can rewrite (\ref{Length-Numblocks}) in terms of a $\Lambda$-coalescent process started with infinitely many particles as
$L_n = \int_\eps^\zeta N_t dt$, where $\eps$ is chosen such that $N_\eps \approx n$. (It is not obvious at all that such a choice of $\eps$ is possible with positive probability, but it can be proved that this indeed the case at least in the regular-variation case, as we will discuss later. Other tricks need to be used in the general case, which we do not discuss here). Since $N_t \sim v(t)$, this means $\eps \sim  u(n)$ where $u(\lambda) = \int_\lambda^\infty \frac{dq}{\psi(q)}$. Plugging into the formula $L_n = \int_{\eps}^\zeta N_t dt$, and making the change of variable $\lambda = v(t)$, or $t = u(\lambda)$ and $dt = \frac1{\psi(\lambda)}$, we find:
\begin{equation}
L_n \approx \int_{v(\zeta)}^n \frac{\lambda}{\psi(\lambda)} d\lambda.
\label{lengthformula}
\end{equation}
The lower-bound of integration makes no difference and we may replace it by 1 if we wish. Recalling (\ref{SegLength}), we now obtain directly the result of Theorem \ref{T:sfLambda}. \qed

\mn \emph{Sketch of proof of Theorem \ref{T:Betaallelic}.}
A moment of thought and recalling (\ref{stableLevy}), we find that if $\Lambda(dx)$ satisfies the regular variation property of (\ref{regvar}), then $\psi(\lambda) \sim C \lambda^\alpha$ as $\lambda \to \infty$, for some constant $C>0$ whose value below may change from line to line. From this it follows by Theorem \ref{T:sfLambda} that, in probability:
\begin{equation}
A_n \sim \rho \int_1^n C\lambda^{1-\alpha} d\lambda \sim \rho C n^{2-\alpha}
\end{equation}
It turns out that the estimates in \cite{bbl1} and \cite{bbl2} are tight enough that this convergence holds almost surely. Since the allelic partition is obviously exchangeable, we may now apply results about the Tauberian theory (Theorem \ref{T:regvar}). Theorem \ref{T:Betaallelic} follows immediately.

\subsection{A paradox.}

The following consequence of Theorem \ref{T:smalltimes}
says that, among all the $\Lambda$-coalescents
such that $\Lambda[0,1]=1$,
Kingman's coalescent is extremal for
the speed of coming down from infinity. This is a priori surprising as in Kingman's coalescent only two blocks ever coalesce at a time, whereas in a coalescent with multiple mergers it is always a positive fraction of blocks that are merging. The assumption
\begin{equation}\label{unit mass}
  \Lambda[0,1]=1
\end{equation}
is a scale assumption, as multiplying $\Lambda$ by any number is equivalent to speeding up time by this factor.

\begin{corollary}
\label{C:Kfast}
Assume (\ref{unit mass}). Then with probability 1,
for any $\eps>0$, and for all $t$ sufficiently small,
$$
N_t \ge \frac2t (1-\eps).
$$
\end{corollary}
\begin{proof}
Without loss of generality assume that the $\Lambda$-coalescent comes down from infinity.
To see how the corollary follows from Theorem \ref{T:smalltimes},
observe that since $e^{-qx}\le 1-qx+ q^2x^2/2$ for
$x>0$,
\begin{equation}
\label{psi -g} \psi(q) \le \frac{q^2}{2}\int_{[0,1]} x^2\nu(dx) \le
\frac{q^2}{2} \ \ (\mbox{due to (\ref{unit mass})}).
\end{equation}
Hence if $u(s) = \int_s^\infty \frac{dq}{\psi(q)}$ (so that $v$ is the inverse of $u$):
\begin{equation}
u(s)\ge \int_{s}^{\infty} \frac{2}{q^2}\, dq=\frac{2}s
\mbox{ and } v(t)\geq \frac{2}t.
\label{E v nonint}
\end{equation}
Due to Theorem \ref{T:smalltimes},
$N_t \sim v(t)$ as $t\to 0$, implying that $N_t\ge 2(1-\eps)/t$ with probability 1 for
all $t$ small.
\end{proof}

\subsection{Further study of coalescents with regular variation}

The next section is devoted to a finer study of the case where the measure $\Lambda$ is the Beta$(2-\alpha, \alpha)$ distribution, with $1<\alpha<2$. In that case, an exact embedding of the coalescent in the corresponding continuous-state branching process (or the CRT) exists, and the special properties of this process (in particular, self-similarity) allows us to deduce several nontrivial properties of the Beta-coalescents. Some of these properties can be transferred by universality to the more general coalescents with regular variation.

\subsubsection{Beta-coalescents and stable CRT}

Let $(Z_t,t\ge 0)$ be an $\alpha$-stable CSBP, (i.e., with $\psi(\lambda) = \lambda^\alpha$: see (\ref{stableLevy}). Assume to simplify that $Z_0=1$. If we use the same reasoning as in the sketch of the proof of Theorem \ref{T:sfLambda}, we may ask: what is the rate at which we will observe a $p$-merger of the ancestral lineages, for any $0<p<1$? Let
$$
g(x) = \frac{\alpha(\alpha-1)}{\Gamma(2-\alpha)} x^{-1-\alpha}
$$
be the density of the L\'evy measure of the stable subordinator\index{Stable!subordinator}\index{Subordinator!stable} with index $\alpha$. Thus, if the current population size is $A$, the rate at which there is a jump of size $x$ in the population is $Ag(x)dx$. Reversing the direction of time, this means that a fraction $p$ of lineages coalesces, where
$$
p = \frac{x}{A+x} \text{ or } x = \frac{Ap}{1-p}
$$
since the new population size after birth is $A+x$. Thus:
\begin{align*}
\text{rate of $p$-merger} &= A g(x) \frac{dx}{dp} dp\\
& =  A\frac{\alpha(\alpha-1)}{\Gamma(2-\alpha)}\left(\frac{Ap}{1-p}\right)^{-1-\alpha} \frac{A}{(1-p)^2} dp\\
& = cA^{1-\alpha} p^{-2} \Lambda(dp)
\end{align*}
where $\Lambda$ is the Beta distribution with parameters $2-\alpha$ and $\alpha$ and $c=\alpha (\alpha -1) \Gamma(\alpha) $. Thus if time is sped up by a factor $Z_t^{1-\alpha}/c$, the rate is exactly the rate of $p$-mergers in a Beta-coalescents. We have thus proved the following result.

\begin{theorem}\label{T:Betaembedding} Let $1<\alpha<2$ and let $0\le s<t$. Define a random partition $\pi_s^t$ by saying $i \sim j$ if and only if individuals $i$ and $j$ have the same ancestor in the Donnelly-Kurtz lookdown process\index{Donnelly-Kurtz}\index{Lookdown process} of the $\alpha$-stable branching process. Let $R_t = c \int_0^t Z_s^{1-\alpha}ds$ and let $R^{-1}_t$ be the cadlag inverse of $R$. Then for all $0 \le s \le t$, if $\Pi_s = \pi_{R^{-1}(t-s)}^{R^{-1}(t)}$,
$$
(\Pi_s, 0 \le s \le t)
$$
\text{ is a Beta-coalescent run for time $t$}.
\end{theorem}

The version of this result quoted in the theorem was first proved by Birkner et al. in \cite{7}. There is an equivalent version on Continuum Random Trees, which was subsequently proved in \cite{bbs2} by showing that the two notions of genealogies defined by the lookdown process and by the continuum random tree must coincide. (It is the version on CRTs which is the most useful for capturing fine aspects of the small-time behaviour -- although see \cite{bbs1} for what you can do with just the lookdown process). However, the elementary approach which we give here, is based on yet unpublished work with J. Berestycki and V. Limic \cite{bbl2}, and this bypasses the rather complex calculations of \cite{7}. This result may be seen as a generalization to the stable case of an important result due to Perkins \cite{perkins} in the Brownian case: see also \cite{kingBM} for related results in this case.

\medskip Note that this result gives us a better understanding of Theorem \ref{T:schweinsbergGW}, where genealogies of population with offspring distribution in the domain of attraction of a stable random variable converge to a Beta-coalescent.

\subsubsection{Backward path to infinity}

It is also possible to get some information about the time-reversal\index{Time-reversal} of the process, a bit like in Aldous' construction and Corollary \ref{C:simplex}. However this is much more complicated in the case of Beta-coalescents: the first difficulty is that one doesn't know how many blocks were lost in the previous coalescence (unlike in Kingman's coalescent, where we know we have to make exactly one fragmentation).

A first result in this direction says that roughly speaking, if $N_t =n$ and $S^n_1 , S^n_2 , \ldots$ denote the previous number of blocks at times before $t$, then a result of \cite{bbs1} states that, after shifting everything by $n$, $(S^n_1 - n, S^n_2 - n, \ldots)$ converges in distribution towards a random walk with a nondegenerate step distribution $(S_1, S_2, \ldots)$. The limiting step distribution $S_{i+1}-S_i$ turns out to have an expected value of $1/(\alpha-1)$. This result, combined with a renewal argument, shows:

\begin{theorem}\emph{(\cite{bbs1})}
Let $V_n$ be the event that $N_t = n$ for
some $t$.  Then
\begin{equation}
\lim_{n \rightarrow \infty} \P(V_n) = \alpha - 1.
\end{equation}
\label{T:renewal}
\end{theorem}

We also note that there exists a formula for the one-dimensional distributions of the Beta-coalescent\index{Marginals!Beta-coalescent}, which can be found in Theorem 1.2 of \cite{bbs1}.

\subsubsection{Fractal aspects}
Changing the topic, recall that for random exchangeable partitions, we know that the number of blocks is
inversely related to the typical block size (see Theorem \ref{T:sbp bl}). Here, at least informally, since the number of blocks at small time is of order $t^{-1/(\alpha-1)}$,
we see
that the frequency of the block which contains 1 at time $t$ should be
of the order of $t^{1/(\alpha-1)}$ (this result was
proved rigorously in \cite{bbs1}). Put another way, this says that
almost all the fragments emerge from the original dust by growing
like $t^{1/(\alpha-1)}.$ We say that $1/(\alpha-1)$ is the
\emph{typical speed of emergence}.

\medskip However, some blocks clearly have a different behavior.
Consider for instance the largest block and denote by $W(t)$ its
frequency at time $t.$

\begin{theorem}\label{T:Betalargest} \emph{(\cite{bbs1})}
$$
(\alpha \Gamma(\alpha) \Gamma(2-\alpha))^{1/\alpha} t^{-1/\alpha}
W(t) \rightarrow_d X \hspace{.2in} \mbox{ as } t \downarrow 0
$$
where $X$ has the Fr\'echet distribution of index $\alpha$. \index{Largest block}
\end{theorem}

Hence
the size of the largest fragment is of the order of
$t^{1/\alpha},$ which is much bigger than the typical block size. Note that the randomness of the limiting variable $X$ captures the intuitive idea of a reinforcement phenomenon going on: the bigger a block is, the higher its chance of coalescing later on. Random limits in laws of large numbers are indeed typical of processes with reinforcement such as P\'olya's urn\index{P\'olya's urn}.

\medskip This suggests to study the existence of fragments that emerge with an
atypical rate $\gamma \neq 1/(\alpha-1)$. To do so, it is convenient
to consider a random metric space $(S,d)$ which encodes completely
the coalescent $\Pi$ (this space was introduced by Evans
\cite{evans} in the case of Kingman's coalescent).\index{Evans' random metric space} The space $(S,
d)$ is the completion of the space $(\N, d)$, where $d(i, j)$ is the
time at which the integers $i$ and $j$ coalesce. Informally speaking, completing
the space $\{1,2,\ldots\}$ with respect to this distance in particular
adds points that belong to blocks behaving atypically. In this
framework we are able to associate with each point $x\in S$ and each
$t>0$ a positive number $\eta(x,t)$ which is equal to the frequency
of the block at time $t$ corresponding to $x$. (This is formally
achieved by endowing $S$ with a mass measure $\eta$). In this
setting, we can reformulate the problem as follows: are there points
$x \in S$ such that the mass of the block $B_x(t)$ that contains $x$ at time $t$
behaves as $t^{\gamma}$ when $t \to 0$, or more formally such that
$\eta(x,t) \asymp  t^{\gamma}$? (Here $f(t)\asymp g(t)$ means that
$\log f(t)/\log g(t)\to 1$). Also, how many such points typically
exist?

We define for $\gamma \le 1/(\alpha -1)$
$$S_{\text{thick}}(\gamma)=  \{x\in S: \liminf_{t\to
0}\frac{\log(\eta (x,t))}{\log t} \le \gamma\}
$$
and similarly when $\gamma > 1/(\alpha -1)$
$$
S_{\text{thin}}(\gamma)=
 \{x\in S:
\limsup_{t\to 0}\frac{\log(\eta(x,t))}{\log t} \ge \gamma\}.
$$
When $\gamma  \le 1/(\alpha -1)$, $S_{\text{thick}}(\gamma)$ is the set of points
which correspond to {large} fragments. On the other hand when
$\gamma  \ge 1/(\alpha -1)$, $S_{\text{thin}}(\gamma)$ is the set of points
which correspond to {small} fragments. In the next result we
answer the question raised above by computing the Hausdorff
dimension (with respect to the metric of $S$) of the set
$S_{\text{thick}}(\gamma)$ or $S_{\text{thin}}(\gamma)$:

\begin{theorem}\label{T:Betafractal} \emph{(\cite{bbs2})}
\begin{enumerate}
\item[1.]
 If $ \frac1{\alpha}\le \gamma<\frac1{\alpha-1} $ then
$$
\dim_{\mathcal{H}} S_{\text{thick}}(\gamma)=\gamma\alpha-1.
$$
If $\gamma < 1/\alpha$ then $S_{\text{thick}}(\gamma)=\emptyset$ a.s. but
$S(1/\alpha)\neq \emptyset$ almost surely.

\item[2.] If $\frac1{\alpha-1}< \gamma \le
\frac{\alpha}{(\alpha-1)^2}$ then
$$
\dim_{\mathcal{H}}S_{\text{thin}}(\gamma)=\frac{\alpha}{\gamma(\alpha-1)^2}-1.
$$
If $\gamma > \frac{\alpha}{(\alpha-1)^2}$ then
$S_{\text{thin}}(\gamma)=\emptyset$ a.s. but $S(\frac{\alpha}{(\alpha-1)^2})\neq
\emptyset$ almost surely.
\end{enumerate}
\end{theorem}

\mn \textbf{Comment.}
The maximal value of $\dim_{\mathcal{H}} S(\gamma)$ is obtained
when $\gamma=1/(\alpha-1)$ in which case the dimension of
$S(\gamma)$ is also equal to $1/(\alpha-1)$. This was to be
expected since this is the typical exponent for the size of a
block. The value of the dimension then corresponds to the full
dimension of the space $S$, as was proved in \cite[Theorem
1.7]{bbs1}.\index{Multifractal spectrum}

\begin{center}
\begin{figure}[h]
\includegraphics{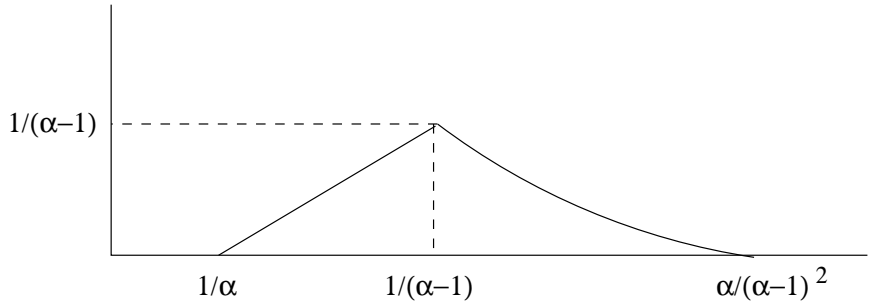}
\caption{Multifractal spectrum map $\gamma \mapsto
\dim_{\mathcal{H}}S(\gamma)$. The left-derivative at the critical
point is $\alpha$ while the right-derivative is $-\alpha$.}
\end{figure}
\end{center}


\subsubsection{Fluctuation theory}

The analysis of fluctuations, even for Beta-coalescents, seems considerably more complicated than any of the law-of-large number type of result described above. Only very partial results exist to this date. For instance, there isn't any result available concerning the fluctuations\index{Fluctuation theory} of the number of blocks at small time, or the biologically relevant total length of the tree. However for this last case, there is a partial result which is due to Delmas, Dhersein and Siri-Jegousse \cite{DelmasDhersinSJ}, which we describe below. We first need a result on the total number of collisions which was proved simultaneously and independently by Gnedin and Yakubovich \cite{GnedinYakubovich} on the one hand, and by Delmas, Dhersein and Siri-Jegousse \cite{DelmasDhersinSJ} on the other hand. Let $n\ge 1$, and let $\tau_n$ denote the total number of coalescence events of a Beta-coalescent started from $n$ blocks, before there is only one block left. That is, $\tau_n$ is the total number of jumps of the chain which counts the number of blocks (and decreases from $n$ to 1).

\begin{theorem}
  \label{T:collisions}As $n \to \infty$,
  $$
  n^{-1/\alpha}\left(n - \frac{\tau_n}{\alpha -1}\right) \overset{d}\to Y_{\alpha - 1}
  $$
where $(Y_s, s \ge 0)$ is a stable subordinator with index $\alpha$ started from 0.
\end{theorem}

This result also holds under a slightly more general form of regular variation than (\ref{regvar}), and the exact assumptions needed in \cite{GnedinYakubovich} and in \cite{DelmasDhersinSJ} are slightly different (note that the $\alpha$ of \cite{GnedinYakubovich} is what we here call $2-\alpha$). Note in particular that the order of magnitude of $\tau_n$ is $(\alpha -1)n$, to the first order.

Delmas, Dhersin and Siri-Jegousse then consider the length of a \emph{partial} tree of coalescence, i.e., the sum of the length of the branches from time 0 until a given number $k$ of coalescence events. In view of the above discussion, it is sensible to choose $k = \lfloor nt\rfloor $ with $t < \alpha -1$. Let $L_n(t)$ denote the corresponding length. The main result of \cite{DelmasDhersinSJ} (Theorem 6.1) is as follows, and shows a surprising phase transition at $\alpha = (1+\sqrt{5})/2$, the golden ratio. Let
$$
\ell(t) = \frac{y(t)}{\Gamma(\alpha) \Gamma(2-\alpha)}, \text{ where } y(t) = \int_0^t \left(1- \frac{r}{\alpha -1}\right)^{\alpha -1} dr.
$$

\begin{theorem}
  \label{T:partiallength} Let $\alpha_0 = (1+\sqrt{5})/2$. Let $(Y_s, s \ge 0)$ be a stable subordinator with index $\alpha$ started from 0, and let $\beta = - 1-(1/\alpha) +\alpha$.

  \begin{enumerate}

  \item[\emph{(1)}] Let $\alpha \in (1, \alpha_0)$. For all $t<\alpha -1$, we have the convergence in distribution:
  \begin{equation}\label{limitlengthdist}
  n^{\beta}\left(L_n(t) - n^{2-\alpha} \ell(t)\right) \overset{d}\to \int_0^t \left(1- \frac{r}{\alpha -1}\right)^{\alpha -1} Y_r dr.
  \end{equation}

  \item[\emph{(2)}] Let $\alpha \in [\alpha_0, 2)$. Then for any $\eps>0$,
  $$
  n^{-\eps}\left(L_n(t) - n^{2-\alpha} \ell(t)\right) \overset{d}\to 0
  $$
  in probability.
  \end{enumerate}
\end{theorem}

Intuitively, when we plug in $t=\alpha -1$ in the left-hand side, $L_n(t)$ is then almost the same thing as $L_n$ since there are approximately no more than $(\alpha -1)n$ coalescence events by Theorem \ref{T:collisions}. Note that by doing so, we recover the correct first order approximation of Theorem \ref{T:Betaallelic}. This strongly suggests that the result is true also for $L_n$ instead of $L_n(t)$ in the left-hand side and $t=\alpha -1$ in the right-hand side, but this has not been proved at the moment.

Theorem \ref{T:partiallength} allows the authors of \cite{DelmasDhersinSJ} to deduce a corollary about the fluctuations of the number of mutations that fall on the partial tree. For $\alpha$ sufficiently small, the fluctuations of the length of the tree dominate the Gaussian fluctuations induced by the Poissonization, hence obtaining the random variable on the right-hand side of (\ref{limitlengthdist}), while for larger $\alpha$ it is the opposite, and the limit is Gaussian. Surprisingly, the phase transition here occurs at $\alpha = \sqrt{2}$, rather than the golden ratio. See Corollary 6.2 in \cite{DelmasDhersinSJ} for further details.

\newpage
\section{Spatial interactions}

We now approach an area which seems to be expanding at a rapid pace, which consists in studying coalescing processes of particles with spatial interactions. The prototype of such a process is the model of (instantaneously) coalescing random walks, so we first describe a bit of classical work on this, such as the duality with the voter model, the result of Bramson and Griffeath \cite{bg80} on the long term density of the system as well as Arratia's Poisson point process limit \cite{arratia}. (This is done by appealing to the intuitive approach recently developed by van den Berg and Kesten \cite{BKcoalescence}, which we describe). We then move off to the stepping stone model, which are spatial versions of the Moran model, and describe some results of Cox and Durrett \cite{CoxDurrett} and of Z\"ahle, Cox and Durrett \cite{ZahleCoxDurrett}. Reversing the direction of time leads us naturally to the spatial $\Lambda$-coalescents, introduced by Limic and Sturm in \cite{LS}. We describe their result about the long-term behaviour of spatial coalescents on a torus. We then push the description of spatial coalescents further by briefly describing the result of \cite{abl} on the global divergence of these processes. Finally we draw a parallel with the work of Hammond and Rezkhanlou \cite{hr}, which leads us to a general discussion on coalescent processes in continuous space.

\subsection{Coalescing random walks}

The model of coalescing random walks is awfully simple to describe: let $d\ge 1$ and imagine that, initially, every site of $\Z^d$ is occupied by a single particle. As time evolves, particles perform independent simple random walks in continuous time with identical rates of jump (say 1). The interaction arises when a particle jumps onto a site which is already occupied: in that case, the two coalesce instantly. That is, their trajectories are identical after that time. (One must first ensure that the model is well-defined but that is not a big problem: see, e.g., Liggett \cite{liggett}). Obvious questions pertain to the density of this system and to the location of particles after rescaling so that the density is of order 1. Some more subtle ones ask what is the location of the set of ancestors of a given particle. This system may be easy to define but its analysis is far from simple, and has given birth to a rich theory.

\subsubsection{The asymptotic density}

Let
$p_t$ be the probability that there is a particle at the origin at time $t$. Of course, the system is translation-invariant so this is also the probability that there is a particle at any given site. Hence $p_t$ should be regarded as the \emph{density} of the system at time $t$. Bramson an Griffeath showed the following remarkable result on the asymptotic behaviour of $p_t$:
\begin{theorem}\label{T:BramsonGriffeath} As $t \to \infty$,
\begin{equation}\label{pt}
p_t \sim g_d(t):=
\begin{cases}
(\pi t)^{-1/2} & \text{ if $d=1$} \\
\log t /(\pi t ) & \text{ if $d=2$}\\
1/(\gamma_d t) & \text{ if $d\ge 3$}
\end{cases}
\end{equation}
Here $\gamma_d$ is the probability that simple random walk on
$\Z^d$ never returns to its starting point.
\end{theorem}

Bramson and Griffeath's original proof was rather complicated and based on a
moment calculation of Sawyer \cite{sawyer}, which, being essentially a big computation, did not shed much light on the subject. Later,
a simpler and more
probabilistic proof was discovered by Cox and Perkins
\cite{cox-perkins} using the super-Brownian\index{Super Brownian motion} invariance principle
for the voter model of Cox, Durrett and Perkins \cite{cdp}. As we will see, the voter model is indeed \emph{dual}\index{Duality} to coalescing random walks and hence it is not a surprise that this invariance principle would be of great help.

However, a completely elementary approach has been recently developed by van den Berg and Kesten \cite{BKcoalescence}, and this approach has the merit of being extremely robust to changes in the details of the model. The drawback is that one often has to work in higher dimensions, i.e., above $d=3$. However, since this approach is both elementary and elegant, we propose a brief exposition of the main idea.

\begin{proof}(sketch)

The idea is to try to compute the derivative of $p_t$. On first order approximation,
\begin{equation}\label{density1}
\frac{d}{dt} p_t \approx - p_t^2
\end{equation}
since the density decays when two particles meet, which happens at rate roughly $p_t^2$ if the location of the particles were independent. If (\ref{density1}) was an equality, we could solve this differential equation\index{Differential equation} and get that, as $t\to \infty$, $p_t \approx 1/t$. This gives us the right order of magnitude when $d\ge 3$, but even then note that the constant is off: we are missing a factor $\gamma_d$.

A more precise version of (\ref{density1}) is the following. Assume that $d\ge 3$. Note that, to compute the derivative of $p_t$ there is an exact expression based on the generator of the system. If $\eta \in \{0,1\}^{\Z^d}$ denotes the configuration of the system (where we identify $\eta$ with the set of occupied particles), and if $\eta^{x \to y}$ denotes the configuration where a particle at $x$ has been moved to $y$, and $\eta^{x\to\emptyset}$ the configuration where a particle at $x$ has been killed, then the generator $G$ of the system of particles is
\begin{align*}
Gf(\eta) &= \sum_{x \in \eta} \sum_{y \sim x ; y \notin \eta} [f(\eta^{x \to y}) - f(\eta)] \frac1{2d} \\
& \ \ \ + \sum_{x \in \eta}\sum_{y \sim x, y \in \eta} [f(\eta^{x\to \emptyset}) - f(\eta)] \frac1{2d}
\end{align*}
(where $y \sim x$ means that $y$ and $x$ are neighbours). Specializing to the function $f(\eta) = \indic{0 \in \eta}$, and denoting $I= \#\{y \sim 0: y \in \eta\}$ this implies
\begin{align*}
Gf(\eta) &= \indic{0 \notin \eta} \frac{I}{2d} - \indic{0 \in \eta}\\
&= \frac{I}{2d}(1- \indic{0 \in \eta}) - \indic{0 \in \eta}
\end{align*}
Therefore, by translation invariance, since $\E(I) = 2d \P(0 \in \eta)$, we get:
\begin{align}
\frac{d}{dt} p_t  & = \E[ Gf(\eta_t)] \nonumber \\
&= - \P(\text{both 0 and $e_1$ are occupied}) \label{density2}
\end{align}
where $e_1$ is any of the origin's $2d$ neighbours. If the occupation of both 0 and $e_1$ at time $t$ were independent events, we would thus immediately obtain equality in (\ref{density1}). However, this is far from being the case: indeed, if there is a particle at the origin, chances are that it killed (coalesced with) any particle around it! There is thus an effect of negative correlation here. Remarkably enough, it turns out that this effect can be evaluated.

Indeed, what is the chance that both 0 and $e_1$ are occupied? Let $E$ be this event, and fix some number $\Delta t >0$ such that $\Delta t \to \infty$ but $\Delta t = o(t)$ (think, for now, of $\Delta t = \sqrt{t}$). For the event $E$ to occur, there must be two ancestor particles at time $t-\Delta t$, located at some positions $x$ and $y \in \Z^d$, and the trajectories of two independent simple random walks started from $x$ and $y$ must find their way during time $\Delta t$ to 0 and $e_1$ without ever intersecting. If $\Delta t$ is large enough, then chances are that $x$ and $y$ must be far apart, in which case the events that $x$ and $y$ are occupied are indeed approximately independent. Hence the probability that $x$ and $y$ are both occupied is about $p_{t-\Delta t}^2$. However, since $\Delta t = o(t)$, $p_{t- \Delta t} \approx p_t$ and thus this probability is approximately $p_t^2$. Now, to compute the probability that the two random walks end up at 0 and $e_1$ respectively without intersecting, we use a time-reversal\index{Time-reversal} argument: it is the same as the probability that two random walks started at 0 and $e_1$ never intersect during $[0, \Delta t]$ and end up at $x$ and $y$. That is, letting $S$ and $S'$ be these walks:
\begin{align*}
  \P(E) & \approx \sum_{x,y \in \Z^d}\P( S_{\Delta t} = x, S'_{\Delta t} = y, S[0,\Delta_t] \cap S'[0,\Delta t] = \emptyset) p_{t}^2 \\
  &= p_t^2 P(S [0, \Delta t] \cap S'[0,\Delta t] = \emptyset)
\end{align*}
But since the difference of two independent rate 1 simple random walks is rate 2 simple random walk started from $e_1$, we see that the probability of the event in the right-hand side is the same as the probability that a random walk started from $e_1$ never returned to the origin up to time $\Delta t$. Since $d\ge 3$ and $\Delta t$ is large, it follows that simple random walk is transient and thus this probability is approximately $\gamma_d$. We conclude:
\begin{equation}
  \label{density3}
  \frac{d}{dt}p_t \sim  - \gamma_d p_t^2.
\end{equation}
Integrating this result gives Theorem \ref{T:BramsonGriffeath}.

Naturally, there are various sources of error in this approximation, all of which must be controlled. For instance, one error made in this calculation is that the ancestors $x$ and $y$ are not necessarily unique. They are however likely to be unique if $\Delta t$ is not too big. van den Berg and Kesten \cite{BKcoalescence} have used this approach to obtain a density result for a modified model of coalescing random walks, where the method of Bramson and Griffeath, relying on the duality with the voter model and the exact computation of Sawyer, completely collapses. However, they were able to obtain an asymptotic density result based on this heuristic (first in large enough dimension \cite{BKcoalescence}, and then for all $d\ge 3$ \cite{BKcoal2}).
\end{proof}

\subsubsection{Arratia's rescaling}

Soon after Bramson and Griffeath proved Theorem \ref{T:BramsonGriffeath}, Arratia considered the more precise question of what can be said about the location of the particles that have survived up to time $t$. In order to be able to see a limiting point process, one has to rescale space so that the average number of particles in a cube of volume 1 is one, say. That is, we shrink each edge to a length of
\begin{equation}\label{resclaingArratia}
\eps := g_d(t)^{1/d}
\end{equation}
and let
$$
\cP_t(dx) = \sum_{x \in \eps \Z^d} \delta(dx) \indic{\frac{x}\eps \in \eta_t}.
$$
Arratia's remarkable result \cite{arratia} is as follows:

\begin{theorem}
  Assume that $d\ge 2$. Then $\cP_t$ converges weakly to a Poisson point process with intensity $dx$, the Lebesgue measure on $\R^d$, as $t\to \infty$. If $d=1$ then there exists a nondegenerate limit which is non-Poissonian.
\end{theorem}

\begin{proof}(sketch) Arratia's proof is deceptively short, and we only sketch the idea of why this works: the reader is invited to consult \cite{arratia} for the real details of the proof. The reason we get a Poissonian limit only in dimension 2 and higher is because it is possible to find a $\Delta t$ such that $\Delta t =o(t)$ but $\Delta t$ is large enough that
$$
\sqrt{\Delta t} \gg 1/\eps.
$$
Since $\eps = g_d(t)^{1/d} \sim 1/(\sqrt{\pi t})$ in dimension $d =1$, it is not possible to find such a $\Delta t$ in dimension 1. However, taking for instance $\Delta t = t^{1/2+1/d}$ in dimension $d\ge 3$ and $\Delta t = t/\sqrt{\log t}$ for $d=2$ works. Once this is the case, the idea is to say that if $B$ is a fixed compact convex set of $\R^d$, with high probability there are no coalescences between times $t-\Delta t$ and $t$ within $B$. More precisely, if $\bar \cP_t$ denotes the same Point process as $\cP$ except that the coalescences are not allowed during this time interval, then
\begin{equation}\label{difference}
\P(\cP_t|_B \neq \bar \cP_t|_B) \to 0, \text{ as } t \to \infty.
\end{equation}
Indeed, note that, on the one hand, $\bar \eta$ always has more particles than $\eta$, and on the other hand, by applying the Markov property at time $s=t-\Delta t$, so that if $p_t(x,y)$ denotes the transition probabilities of continuous time simple random walk on $\Z^d$, and if we let $\eta_t(x) = \indic{x \in \eta_t}$,
\begin{align*}
\E(\bar \eta_t(x)) &= \sum_{y \in \Z^d} \E(\eta_{s}(y)) p_{\Delta t}(y,x)\\
&= p_s \sum_{y \in \Z^d} p_{\Delta t}(x,y) \\
&= p_s.
\end{align*}
Therefore, putting these two things together, we find, for $x \in \Z^d$:
\begin{align*}
\P(\eta_t (x) \neq \bar \eta_t(x))& \le \E(\bar \eta_t(x) - \eta_t(x))\\
& \le p_s - p_t,
\end{align*}
and it follows that
\begin{align*}
\P(\cP_t|_B \neq \bar \cP_t|_B) &\le \sum_{x \in \Z^d \cap (\frac1\eps)B} \P(\eta_t(x) \neq \bar \eta_t(x))\\
&\le \frac1{\eps^d} \lambda(B) (p_s - p_t)\\
&\le 2\lambda(B)\frac{p_s - p_t}{p_t} \to 0.
\end{align*}
It follows from this that we can pretend (with high probability) that no coalescence occurred, in which case particles behave as if they were independent simple random walks. However, since $\sqrt{\Delta t} \gg (1/\eps)$ (which is the typical size in the original lattice of the set $B$), it means that ``particles have enough time to mix" and thus their locations are i.i.d. uniform in $B$. Since the mean number of particles in $B$ is 1, this can only mean that particles are distributed as a Poisson point process with unit intensity.

This argument shows why the limit cannot be Poissonian in $d=1$: particles meet and coalesce too often for them to have the time to get back to some sort of equilibrium density. Indeed, in dimension $d=1$, the effect of negative correlations is so strong that it does not disappear even at large scales of space. The limiting object is a process known as Arratia's coalescing flow\index{Arratia flow}\index{Coalescing flow} which shares similar properties of the coalescing flow analysed in Theorem \ref{T:flow} for Kingman's coalescent. This is also intimately connected to an object called the \emph{Brownian web}\index{Brownian web}, which has been the subject of intense research recently.
\end{proof}

\subsubsection{Voter model and super-Brownian limit}

To describe more precise questions connected with the geometry of the set of individuals that have coalesced by some time, it is useful to introduce a system of particles called the (multitype) \emph{Voter model}\index{Voter model}. It turns out that this model is in duality with coalescing random walks on the one hand, and on the other hand, that there exists an invariance principle for this model (due to Cox, Durrett and Perkins \cite{cdp}). That is, this model is known to have super-Brownian motion as its scaling limit\index{Super Brownian motion}. This invariance principle has been further sharpened by Bramson, Cox and Le Gall \cite{bclg} who showed that the geometry of a single ``patch" (i.e., the set of individuals that coalesced to a single particle currently located at the origin, provided that there is such a particle), has the geometry of the Super-Brownian excursion measure.

We first explain the notion of \emph{duality} with the voter model. The (multitype) voter model is a system of particle on $\Z^d$ where each vertex is occupied by a certain opinion. This opinion may take two values (0 or 1) in the two-type case, but in the multitype case every individual initially has their own opinion. As time evolves, at rate 1, any site $x$ may infect a randomly chosen neighbour, say $y$: then $y$ adopts the opinion that $x$ currently holds. Thus it is convenient to label opinions by the vertices of $\Z^d$. To say that vertex $x$ at time $t$ has the opinion $y\in \Z^d$ means that there was a chain of infections from $y$ to $x$ in time $t$. (Thus $x$ carries the opinion that individual $y$ was carrying at time 0.) The duality between the two-type voter model and coalescing random walks $\eta$ states the following. Let $\E^{\downarrow}$ denote the expectation for coalescing random walks $\eta_t$ started from a set $B \subset\Z^d$; and let $\E^{\uparrow}$ denote the expectation for the two-type voter model started from a set $A$. (That is, initially $x \in A$ carries opinion 1, and everybody else carries opinion 0). Let $\xi_t$ denote the set of 1 opinions at time $t$ in this model:

\begin{theorem}
  \label{T:dualityCRWvoter}
  Let $A,B \subset \Z^d$ be two subsets. Then we have the duality relation:\index{Duality}
  \begin{equation}\label{dualityCRWvoter}
    \P^\uparrow(\xi_t \supseteq B | \xi_0 = A) = \P^\downarrow(\eta_t \subseteq A |\eta_0 = B).
  \end{equation}
\end{theorem}

\begin{proof}(sketch) The proof is most easily seen with a picture, as both processes can be constructed using a graphical representation. For each oriented edge $e=(x,y)$ linking two neighbouring vertices, associate an independent Poisson clock $(N^e(t), t \in \R)$ with rate 1. This clock has two interpretations, depending on whether we wish to use it to construct coalescing random walks or the voter model. For the former, a ring of the clock $N^e$ signifies that $x$ infects $y$, i.e., $y$ adopts the opinion of $x$. For coalescing random walks, a ring of the edge means that a particle which was at $x$ moves at $y$. This is shown in Figure \ref{Fig:dualityCRWvoter}.
\begin{figure}
\begin{center}
  \includegraphics[scale=.8]{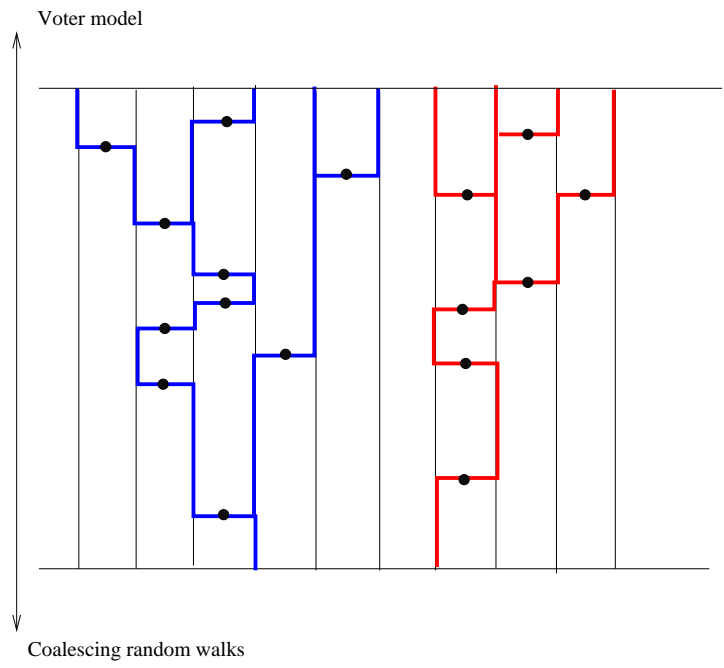}
\end{center}
\caption{The duality between coalescing random walks and the voter model. A dot indicates an edge which rings. Only the rings which affect the particles have been represented.}\label{Fig:dualityCRWvoter}
\end{figure}
Fix $0<t$. For $s<t$ and two vertices $x,y$ we say there is a path down from $(t,y)$ to $(s,x)$ if following the arrows emanating from $y$ at time $t$ leads to $x$ at time $s$. Define a process $(W_s^{t,y})_{0 \le s \le t}$ by putting $W_s^{t,y} =x$ if and only if there is a path down from $(t,y)$ to $(t-s,x)$. Then it is easy to see that $\{(W_s^{t,y})_{0\le s \le t}\}_{y \in \Z^d}$ is a system of coalescing random walks run for time $t$. On the other hand, if $A \subset \Z^d$ we may define for all $t\ge0$, $\xi_t = \{y \in \Z^d: W^{t,y}_t \in A\}$. Then $(\xi_t)_{t\ge 0}$ has the law of the two-type voter model started from $\xi_0=A$. Now note that, with this construction
$$
\{\xi_t \supseteq B\} = \{W_t^{t,y} \in A \text{ for all } y \in B \} = \{\eta_t \subseteq A\}
$$
where $\eta_t$ is the system of coalescing random walks defined by the processes $\{(W_s^{t,y})_{0\le s \le t}\}_{y \in B}$. This completes the proof.
\end{proof}

Observe that, with the voter model, we have a nice dynamics forward in time with a clear branching structure. The fact that the density of coalescing random walks is asymptotic to $1/(\gamma_d t)$ in dimension $d\ge 3$ implies that the a given opinion at time 0 in the multitype voter process has a probability about $1/(\gamma_d  t)$ to survive up to time $t$. It is well-known that there is a similar behaviour for critical Galton-Watson branching processes with finite variance (\cite{kolmogorov}, \cite{KestenNeySpitzer}): $\P(Z_t >0) \sim K/t$, where $K= 2/\sigma^2$.

This suggests that, in dimensions $d\ge 3$, the branching structure of the voter model is well-approximated by a critical Galton-Watson process. On top of this branching structure, particles are moving according to simple random walk: therefore, we do expect in the limit as $t\to \infty$ that this system can be rescaled to the super Brownian motion, as this is precisely defined as the scaling limit of critical branching Galton-Watson process with Brownian displacements (see Etheridge \cite{etheridge} for a wonderful introduction to the subject). The invariance principle of Cox, Durrett and Perkins states this result formally:

\begin{theorem}
  \label{T:voterSBM}
  Let $\xi_t^N$ denote a voter model started from a certain initial condition $\xi_0^N$. Let $d \ge 2$, and let $m_N = N$ if $d\ge 3$ or $N/\log N$ if $d=2$. If $X_t^N(dx) = \frac1{m_N}\sum_{y \in \eps \xi_t^N} \delta_y(dx)$, then, provided $X_0^N \Rightarrow X_0$ as $N\to \infty$, we have:
  \begin{equation}
    (X_{tN}^N,t\ge 0) \longrightarrow_d (X_t,t\ge 0)
  \end{equation}
the super-Brownian motion with branching rate $r= 2 \gamma_d$ and spatial variance $\sigma^2 =1$.
\end{theorem}

We finish this discussion by stating the sharpened result of Bramson, Cox and Le Gall \cite{bclg}, which describes the geometry of the patch $\cI_t$ of the origin, conditionally on the event that there is a particle at the origin.

\begin{theorem}
  \label{T:patchSBM}
  There is the following convergence in distribution as $t\to \infty$:
  $$
  \frac{\cI_t}{\sqrt{t}} \longrightarrow \text{\em Supp} (X)
  $$
where $X$ has the distribution of a super-Brownian excursion conditioned to reach level 1.
\end{theorem}

(This convergence holds with respect to the Hausdorff metric on compact sets.) We haven't defined the super-Brownian excursion\index{Super Brownian excursion} properly, nor super Brownian motion in fact. This is simply the rescaled limit of a single critical Galton-Watson tree with finite variance, conditioned to reach a high level $n$, and with independent Brownian displacements along the tree. See \cite{bclg} for details.

\subsubsection{Stepping stone and interacting diffusions}

We now provide a very short and partial presentation of the stepping stone model of population genetics. Essentially, this is a spatial version of the Moran model\index{Moran model} studied in Theorem \ref{T:moran}, and it may also be viewed as a generalization of the voter model of the previous section. The model is as follows. Fix a graph $G$ which for us will be always the $d$-dimensional Euclidean lattice or a large $d$-dimensional torus $(\Z / L)^d$, and let $N\ge 1$. We view each site as a colony or \emph{deme}\index{Deme}, and $N$ is the total number of individuals of a given population at each colony. Here we only define the model without mutations, but it is easy to add mutations if desired. The stepping stone model\index{Stepping stone model} tracks the evolution of various allelic types at each site $x \in \Z^d$, as they are passed along to descendants. Since this is a spatial version of the Moran model, individuals reproduce in continuous time at constant rate equal to 1, and when individual $i$ reproduces, some other individual $j$ adopts the type of individual $i$. We merely need to specify how $j$ is chosen. This is done as follows: fix $0<\nu <1$, which we think of being a small number. With probability $1- \nu$, $j$ is chosen uniformly at random different from $i$ but in the same colony as $i$, say $x \in \Z^d$. Otherwise, with probability $\nu$, $j$ is chosen from another colony, with colony $y$ being selected with probability $q(y,x)$, where $q(y,x)$ denotes the transition probabilities of a fixed random walk on $\Z^d$. This choice is made so that when we follow the genealogical lineages of this model, backward in time, we obtain a random walk in continuous time with transition probabilities $q(x, y)$. In what follows, the reader may think of the case where $q(x,y)$ is not only symmetric in $x,y \in (\Z/L)^d$ and a function only of $y-x$, but also that it has the same symmetries as $\Z^d$. (Naturally, it is fine to think of the simple random walk transition probabilities: $q(x,y) = (1/2d) \indic{x \sim y}$, where $x \sim y$ denote that $x$ and $y$ are neighbours in the torus $(\Z/ L)^d$).

The exact quantity which we track may depend on the context: for instance, by analogy with the Fleming-Viot model of Theorem \ref{T:FVdef}, it may be convenient to start the model with all individuals carrying allelic types given by independent uniform random variables on (0,1), and follow the process
$$
X_t = \left( \sum_{i=1}^N \delta_{\xi_i(t,x)}\right)_{x \in \Z^d}; \ \ \ t\ge 0.
$$
Here $\xi_i(t,x) \in (0,1)$ denotes the type of the $i\th$ individual at time $t$ in colony $x$. The stepping stone model has a long history, which it would take much too long to describe here. We simply mention the three papers which had huge impact on the subject, starting with the work of Kimura \cite{Kimura} and the subsequent analysis by Kimura and Weiss \cite{KimuraWeiss} and by Weiss and Kimura \cite{WeissKimura}. Durrett \cite{durrettbookDNA} devotes Chapter 5 of his book to this model. It contains a fine review of some recent results on this model due in particular to Cox and Durrett \cite{CoxDurrett} as well as Z\"ahle, Cox and Durrett \cite{ZahleCoxDurrett}. We describe some of the main results below. For those results, what matters is only the time of coalescence and genealogical properties of the model (which is why it is not too important what is the exact quantity tracked by the stepping stone model).

To start with, we describe some results regarding the time of coalescence of two lineages. We take the case $d=2$, which is not only the most biologically relevant but also the most interesting mathematically. The result depends rather sensitively on the relative order of magnitude of the starting locations of these lineages, the size $L$ of the torus, and the size $N$ of the population in each colony. To start with, assume that the two lineages are chosen uniformly at random from the torus. Let $T_0$ be the time which the lineages need to find themselves at the same location, and let $t_0$ be the total time they need to coalesce (thus $t_0 \ge T_0$, almost surely). The following result is due to Cox and Durrett \cite{CoxDurrett}.

\begin{theorem} \label{T:CoxDurrett}Let $d=2$. For all $t >0$, as $L \to \infty$, and $\nu \to 0$,
$$
\P\left(T_0 > \frac{L^2 \log L}{2 \pi \sigma^2 \nu} t\right) \to e^{-t}.
$$
Moreover, after $T_0$, the additional time needed to coalesce $t_0 -T_0$ satisfies:
$$
\E(t_0 - T_0) = NL^2.
$$
\end{theorem}

Here $\sigma$ denotes the variance of $q$ in some arbitrary coordinate (since $q$ has the same symmetries as $\Z^d$, it does not matter which one).

\begin{proof} (sketch) By considering the difference of the location of the two lineages, the question may be reformulated as follows: start from a location at random in the torus, and ask what is the hitting time of zero for a rate $2 \nu$ continuous time random walk $(X_t,t\ge 0)$ with kernel $q$. By stationarity, the expected amount of time $X_t$ has spent at 0 by time $L^2$ is exactly equal to 1. On the other hand, if $X_0 = 0$, the amount of time $X_t =0$ in the next $L^2$ units of time is, by the local central limit theorem,
$$
\int_0^{L^2} \P_0(X_t = 0) dt  \sim \frac{\log (L^2)}{2 \pi (2 \nu \sigma^2)} = \frac{\log L}{2 \pi \nu \sigma^2}
$$
since $\P_0(X_t = 0) \sim (2\pi \sigma^2(2\nu t))^{-1}$. Thus, on average the particle spends one unit of time at the origin by time $L$, but conditionally on hitting $0$ this becomes approximately $\log L/(2\pi \nu \sigma^2)$. It is not hard to deduce from these two facts that
$$
\P(T_0 < L^2) \sim \frac{2 \pi \nu \sigma^2}{\log L}.
$$
Using a mixing argument (after a large constant times $L^2$, the walk has mixed and forgotten its initial state, so there is a fresh chance to hit o in the next period of the same length) one can deduce the first result without too much difficulty. The second result is a much more general property of so-called symmetric matrix migration models, see Theorem 4.13 in \cite{durrettbookDNA}.
\end{proof}

There are thus two situations to consider for the asymptotics of $t_0$: either $t_0 - T_0$ dominates or $T_0$ dominates. The first one happens if
$$
\E(T_0) = O( \frac{L^2 \log L}{\nu}) \ll \E(t_0 - T_0) = NL^2
$$
i.e. if $N\nu / \log L \to \infty$. In this case we obtain:

\begin{corollary}
  Assume that $N\nu / \log L \overset{d}{\longrightarrow} \infty$. Then as $L \to \infty$,
  $$
  \frac{t_0}{NL^2} \to E,
  $$
  an exponential random variable.
\end{corollary}

The interesting case occurs of course if both contributions to $t_0$ are of a comparable order of magnitude. Thus
\begin{equation}\label{assumstepston}
\frac{2N\nu \pi \sigma^2 }{\log L} \to \alpha.
\end{equation}

\begin{theorem}
  \label{T:stepstonecrit} We have:
  \begin{equation}
  \P\left( t_0 > (1+\alpha) \frac{L^2 \log L}{2 \pi \sigma^2 \nu} t \right) \to e^{-t}.
  \end{equation}
\end{theorem}

This is complemented by a result, due to \cite{ZahleCoxDurrett}, which says that the genalogy of a random sample of $n$ individuals from the torus is approximately given by Kingman's coalescent, after a suitable time-change. Let $h_L = (1+\alpha) L^2 \log L/(2 \pi \sigma^2 \nu)$. The following result (Theorem 2 in \cite{ZahleCoxDurrett}), is originally formulated for the number of lineages backward in time of such a sample, but can be reformulated in terms of convergence to Kingman's coalescent, which we do here. Let $k\ge 1$ and let $(\Pi^{L,k}_t,t\ge 0)$ denote the ancestral partition process for these $k$ individuals.

\begin{theorem}
  \label{T:stepstoneK}As $L,N \to \infty$ and $\nu \to 0$ in such a way that $(\ref{assumstepston})$ holds, then
  $$
  (\Pi^{L,k}_{h_L t}, t\ge 0) \overset{d}{\longrightarrow} (\Pi^k_t,t\ge 0),
  $$
  Kingman's $k$-coalescent.
\end{theorem}

The proof of this result follows the lines of an argument due to Cox and Griffeath \cite{CoxGriffeath}, who proved a similar result in the context of the voter model, or coalescing random walks. Naturally, the idea is to exploit the fact that particles are very well mixed on the torus by the time they coalesce, leading to the asymptotic Markovian property of the ancestral partition process, and to the fact that every pair of coalescence is equally likely. (The fact that only pairwise mergers occurs is also a consequence of the relative difficulty to coalesce in more than 1 dimension: we rarely see three particles close enough to coalesce instantly on the time scale that we are looking at).

A particularly interesting case of this question arises when the two lineages are not selected just uniformly at random from the torus but from a subdomain of the torus which is a square of sidelength $L^\beta$ with $0\le \beta \le 1$. This reflects the fact that, in many biological studies, samples come from a fairly small portion of the space (see \cite{ZahleCoxDurrett} or section 5.3 of \cite{durrettbookDNA} for results). In that case Theorem \ref{T:stepstoneK} still holds but the time-change is slightly more complicated (and is not just linear, in particular). In the next section on spatial $\Lambda$-coalescents, an even more extreme view of individuals sampled from the exact same location is presented.

Note that mutations may be added to the stepping stone model without difficulty (instead of adopting the type of his parent, a newborn adopts a new and never seen before type). Also, if we imagine following forward in time the evolution of the densities at various sites of a certain subpopulation (say they are of type $a$, and there are no mutations), then we can establish a relation of {duality}\index{Duality} with certain \emph{interacting Wright-Fisher diffusions}. These diffusions $(p_x(t),t\ge 0)_{x \in \Z^d}$ are characterized by the infinite system of SDEs:
\begin{equation}\label{interactingdiff}
dp_x(t) =\sum_{y\in \Z^d}  q_{xy}(p_y(t) - p_x(t))dt  + \sqrt{p_x(t) (1-p_x(t))} dW_x(t)
\end{equation}
where $\{W_x\}_{x \in \Z^d}$ is a collection of independent Brownian motions. This duality is exactly the spatial analogue of Theorem \ref{T:dualityKWF} between Kingman's coalescent and the Wright-Fisher diffusion. (Note that existence and uniqueness of solutions to (\ref{interactingdiff}) is non trivial but follows from the duality method).

\subsection{Spatial $\Lambda$-coalescents}

\subsubsection{Definition}

When considering population models in which the geometric interaction of individuals is taken into account, we are led to studying a process which was introduced by Limic and Sturm in 2006 \cite{LS}, called the spatial $\Lambda$-coalescent\index{Spatial Lambda-coalescent@Spatial $\Lambda$-coalescent}. Loosely speaking, these models are obtained by considering the ancestral partition process associated with the stepping stone model of the previous section. However, there are two differences. On the one hand, the graph $G$ will be the $d$-dimensional lattice $\Z^d$ rather than a torus, where $d=1,2$ are the most relevant cases: for instance, one can think for $d=1$ of a species which lives on an essentially one-dimensional coastline. More importantly,
 there is an additional degree of generalisation compared to the stepping stone model. Strictly speaking, if we reverse the time in the stepping stone model (and speed up time by $(N-1)/2$, as in Theorem \ref{T:moran}), the process which we will obtain is the spatial Kingman coalescent: pairs of particles coalesce at rate 1 when they are on the same site, and otherwise perform independent random walks in continuous time. In the spatial $\Lambda$-coalescent, the mechanism of coalescence is such that it allows for multiple mergers at any site. More precisely, the model is defined as follows. Given a graph $G$ and a set of particles:

\begin{enumerate}

\item Particles follow the trajectory of independent simple random walks in continuous time on $\Z^d$, with a fixed jump rate $\rho>0$.

\item Particles that are on the same sites coalesce according to the dynamics of a $\Lambda$-coalescent.

\end{enumerate}

\begin{definition}\label{D:spatial}
  This model is the spatial $\Lambda$-coalescent of Limic and Sturm $\cite{LS}$.
\end{definition}

We content ourselves with this informal description of the process and refer the reader to \cite{LS} for a more rigorous one. It is nontrivial to check that the process is well-defined on an infinite graph, but this can be done using a graphical construction together with the Poissonian construcion of $\Lambda$-coalescents. We do not specify in this definition the initial configuration. In fact, it could be quite arbitrary: in particular, it is possible to start the process with an infinite number of particles on all sites of a given (possibly infinite) subset. This follows from the fact that there is a natural property of \emph{consistency} which is inherited directly from the concistency of $\Lambda$-coalescents.\index{Consistency}

A particular case of interest is, naturally, the spatial Kingman coalescent, where particles perform independent simple random walks with constant jump rate $\rho>0$ and each pair of particles on the same site coalesces at rate 1. This is related to the genealogical tree associated with the stepping stone model.

A first property of spatial $\Lambda$-coalescents is that if $\Lambda$ is such that the $\Lambda$-coalescent comes down from infinity (i.e., if Grey's condition is satisfied or the condition in Theorem \ref{T:cdiCNS} holds), then at every time $t>0$, there is only a finite number of particles on every site. Intuitively, this is because coming down from infinity is a phenomenon that happens so close to $t=0$ that particles don't have the time to jump before it happens. Limic and Sturm actually showed a stronger statement than this, showing that coming down from infinity, interpreted in the sense of a finite number of particles per site, always happens in a uniform way, independently of the graph structure: if $B$ is any finite subset of the graph, let $T_k$ be the the first time when the number of particles in $B$ is no more than $k$ per site on average (i.e., no more than $k|B|$ particles in $B$). Then the following estimate holds:

\begin{theorem}\label{T:uniformCDI}
  Assume that $\sum_{b=2}^\infty \gamma_b^{-1} < \infty$ so that the $\Lambda$-coalescent comes down from infinity. Then for any graph $G$,
  $$
  \E(T_k) \le \sum_{b=k}^\infty \frac1{\gamma_b} + \frac{k}{\gamma_k}<\infty.
  $$
\end{theorem}

\subsubsection{Asymptotic study on the torus}

Limic and Sturm considered the situation of a spatial $\Lambda$-coalescent on a large torus of $\Z^d$. They were able to show the following result: if $d\ge 3$, as the torus increases to $\Z^d$, the genealogy of an arbitrary number of samples from the torus is well-approximated by a mean-field Kingman's coalescent run at a certain speed. This is the analogue of Theorem \ref{T:stepstoneK}, except for the initial condition (as we now explain).

To state this result, let $G$ be the total expected number of visits at the origin by a simple random walk in $\Z^d$ started at the origin, that is, $G= 1/\gamma_d$ where $\gamma_d$ is the probability of not returning to the origin. Define
\begin{equation}
  \label{kappa}
  \kappa:= \frac{2}{G+ 2/\lambda_{2,2}}.
\end{equation}

Suppose that a fixed number of particles $n$ is sampled from the torus of side-length $2N+1$ and that the initial location $v_1, \ldots, v_n$ of these particles does not changed as $N\to \infty$ (here the torus is viewed as a subset of $\Z^d$ with periodic boundary conditions). Because simple random walk is transient, it is possible to define unambiguously a partition $\Pi^n$ which is the eventual partition formed by running the dynamics of the spatial $\Lambda$-coalescent on $\Z^d$. Thus $i\sim j$ in $\Pi^n$ if particles started from $v_i$ and $v_j$ did coalesce at some point. (This partition is typically nontrivial because of transience!)

\begin{theorem} \label{T:spatialmeanfield}Let $\Pi^{N,n}_t$ denote the partition obtained from running the dynamics of the spatial $\Lambda$-coalescent on the torus of sidelength $N$ for time $t$. Then if $(K_t, t\ge 0)$ is the (mean-field) Kingman coalescent of chapter 2, started from the partition $\Pi^n$, then we have:
\begin{equation}
  (\Pi^{N,n}_{(2N+1)^d t}, t\ge 0) \longrightarrow (K_{\kappa t},t\ge 0).
\end{equation}
This convergence holds as $N\to \infty$ and in the sense of the Skorokhod topology for $\cP_n$.
\end{theorem}

\begin{proof}(sketch)
The idea of the proof is the same as in Theorem \ref{T:stepstoneK} and in \cite{CoxGriffeath}. We are working in higher dimension than 2 here, but that can only help mixing. Note that the definition of the term $\kappa$ in (\ref{kappa}) precisely takes into account the effects of transience and mixing on the one hand, and coalesence on the other hand. This is why both $G$ and $\lambda_{2,2}$ come up in this definition. That $\kappa$ depends only on $\Lambda$ through $\lambda_{2,2}=\Lambda([0,1])$ is to be expected, since typically when particles meet, the density is so low that there are only two particles at this site, and the other are far away.
\end{proof}

We point out that a version of Theorem \ref{T:spatialmeanfield} was first proved by Cox and Griffeath \cite{CoxGriffeath} for instantaneously coalescing random walks and by Greven, Limic and Winter for the spatial Kingman coalescent \cite{GrevenWinterLimic}.

\subsubsection{Global divergence}

While Theorem \ref{T:uniformCDI} tells us that for a $\Lambda$-coalescent which comes down from infinity, there are only a finite number of particles per site, what it does not tell us is whether a finite number of particles are left in total at any time $t>0$. To consider an extreme case, imagine that the initial configuration at time 0 is an infinite number of particles on the same site but no particle anywhere else. What is the total number of particles at time $t>0$? It might happen that, even though there are only a finite number of particles per site, sufficiently many have escaped and they have instantly spread all over the lattice. The answer to these questions is provided in \cite{abl}, which shows that global divergence is a universal rule, no matter what graph and underlying measure $\Lambda$. We start with the case of Kingman's coalescent.

\begin{theorem}\label{T:spatialKingman}
  Let $N_t^n$ be the number of particles at time $t$ if initially there are $n$ particles at the origin in $\Z^d$. Then there exists $c_1, c_2>0$ such that with probability 1 as $n \to \infty$,
  \begin{equation}\label{log*n}
  c_1 (\log^* n)^d \le N_t \le c_2 (\log^* n)^d.
  \end{equation}
\end{theorem}

Here the function $\log^*n$ is simply defined as the inverse $\log^* n := \inf\{m\ge 1:
\text{Tow}(m)\ge n\}$ of the tower function:
  \begin{equation}\label{tower}
    \text{Tow}(n)= e^{\text{Tow}(n-1)} := \underbrace{e^{e^{\iddots^e}}}_{n
      \text{ times}}.
  \end{equation}
In words, start with a number $n$, and take the logarithm of this number iteratively until you reach a number smaller than 1. The number of iterations is $\log^* n$. In particular, $\log^* n$ has an incredibly slow growth to infinity: for instance, if $n=10^{78}$, (the number of particles in the universe) $\log^* n = 4$. Thus, whether or not you consider that the spatial Kingman coalescent comes down from infinity or not is essentially a matter of taste! (In particular, if you are a biologist, you might consider that $\log^*n =3$ for any practical $n$...)

Since Kingman's coalescent is the fastest to come down from infinity (Corollary \ref{C:Kfast}) it follows that any spatial $\Lambda$-coalescent on $\Z^d$ is globally divergent. Since moreover $\Z$ is the smallest bi-infinite graph, we get:

\begin{theorem}
  \label{T:Globaldiv}
  Any spatial $\Lambda$-coalescent on any infinite graph is globally divergent (i.e., does not come down from infinity).
\end{theorem}

This is in sharp contrast with the non-spatial case where coming down from infinity depends on the measure $\Lambda$.

\begin{proof} (of Theorem \ref{T:spatialKingman} - sketch).
The basic idea is to first focus on the vertex at the origin, and investigate how many particles ever make it out of the origin. We may thus represent the coalescence of particles there as a tree and we put a mark on the tree to indicate that the corresponding particle has jumped. For the moment, ignore the behaviour of particles after they have left the origin, so for instance you may think that a particle that jumps is frozen immediately. Note that, since jumps occur at constant rate $\rho$, the number of particles that ever leave the origin is \emph{exactly} equal to the number of families in the allelic partition with mutation rate $\rho$. In particular, this number is equal to $K_n$, the number of blocks in a $PD(\theta)$ random partition restricted to $[n]$ with $\theta/2 = \rho$. By Theorem \ref{T:PDnblocks} (which is a simple consequence of the Chinese Restaurant Process) this number is approximately $\theta \log n$.

Moreover, some more precise computations show that most of the action happens in a very short span of time: roughly, the vast majority of the particles who are going to leave the origin have already done so by time $1/(\log n)^2$. By this time, particles which have jumped once have not had the time to jump any further or come back. Due to the fact that there are about $2/t $ blocks in Kingman's coalescent at a small time $t$, there are about $2(\log n)^2$ particles left at the origin, and about $n_1=(\theta /(2d)) \log n$ particles at each neighbour (by the law of large numbers). Starting from here, we can replicate this argument: about $\theta \log n_1$ will ever leave all sites at distance 1 from the origin, so that in the next step, the sites at distance 2 from the origin receive about $n_2:=(\theta /2d) \log n_1 \approx (\theta/2d) \log \log n$ particles. The argument can be iterated, and each time one wants to colonize a new site, a $\log $ has to be taken. This may go on until we run out of particles, which happens after precisely $\log^*n$ iterations. At this point, a ball of radius $\log^*n$ has been colonized, which corresponds to a volume of about $(\log^*n)^d$. On each of those sites, using the work of Limic and Sturm for finite graphs (Theorem \ref{T:uniformCDI}), at time $t>0$ there will about $O(1)$ particles on every site. Theorem \ref{T:spatialKingman} follows after estimating the various errors made in this induction and showing they are negligible.
\end{proof}

Applying this reasoning to measures $\Lambda$ with the ``regular variation" property of Definition \ref{D:regvar} (such as the Beta distribution with parameters $(2-\alpha,\alpha)$), gives us:

\begin{theorem}\label{T:spatialBeta}
  Let $1<\alpha <2$ and consider the spatial $\Lambda$-coalescent with regular variation of index $\alpha$. Let $N_t^n$ be the number of particles at time $t$ if initially there are $n$ particles at the origin in $\Z^d$. Then there exists $c_1, c_2>0$ such that with probability 1 as $n \to \infty$,
  \begin{equation}\label{loglogn}
  c_1 \log \log n \le N^n_t \le c_2 \log \log n.
  \end{equation}
\end{theorem}

\subsubsection{Long-time behaviour}

Consider the spatial Kingman coalescent, started from a large number of particles at the same site. The results from the previous section tell us that the particles quickly settle to a situation where there is a bounded number of particles per site in a rather ``large" region of space (large in quotation marks, as we have seen that after all $\log^* n$ shouldn't be considered large for any practical purpose!). After this initial phase of decay, a different kind of behaviour kicks in, where the geometry of space plays a much more crucial role than before. Particles start diffusing and coalescence events become much more rare than before. The behaviour is then altogether rather similar to the case of instantaneously coalescing random walks: indeed, when two particles meet, they stand a decent chance of coalescing. Thus, starting from a situation where a ball of radius $m= \log^* n$ has about one particle per site (or $m= \log \log n$ for spatial beta-coalescents), we can expect the density to start decaying like $1/t$ as in Theorem \ref{T:BramsonGriffeath} for a while. When particles start realising that the initial condition wasn't infinite (i.e., there wasn't one particle per site on every site of the lattice but only a finite portion of size approximately $m$), the density essentially stops decaying. This takes place at a time of order $m^2$ because of the diffusivity of particles. At this time, the density is about $1/m^2$ and particles are distributed in a volume of about $m^d$. Thus, in dimensions $d\ge 3$, we expect about $m^{d-2}$ particles to survive forever. This heuristics is confirmed by the following result:

\begin{theorem}\label{T:spatialLongterm-d3}   Let $N_\infty$ be the number of particles that survive forever, if initially there are $n$ particles at the origin, and let $m= \log^* n$. There exist
  some constants $c>0$ and $C>0$ (depending only on $d$) such that,
  if $d\ge 3$:
  \begin{equation}
  \P\left( cm^{d-2} < N_\infty < Cm^{d-2} \right)
  \xrightarrow[n\to\infty]{} 1,
  \end{equation}
\end{theorem}

 In dimension 2, since simple random walk is almost surely recurrent, every pair of particle is bound to meet infinitely often and thus to coalesce. Hence only one particle may survive forever. However, the heuristics above can be adapted to show what is the asymptotic rate of decay of the number of particles. As anticipated, the correct time-scale is of order $m^2$:

\begin{theorem}
\label{T:spatialLongterm-d2}
  Let $d=2$ and let $N_t$ be the number of particles that survive up to time $t>0$. If $\delta >0$, there exist some universal constants $c_1$ and $c_2$ such that
  \begin{equation}\label{spatialLongterm-d2}
  \P\left( \frac{c_1}{\log(\delta+2)}\log m < N_{\delta m^2} < \frac{c_2}{\log(\delta +2)}\log m \right)
  \xrightarrow[n\to\infty]{} 1.
  \end{equation}
\end{theorem}

There are two statements hidden in (\ref{spatialLongterm-d2}). The first one says that the number of particles at time $\delta m^2$ is about $\log m$. The second says that if $\delta$ is large, then the constant term in front of $\log m$ is of the order of $1/\log (\delta)$.

\begin{proof} We are not going to offer any justification of Theorem \ref{T:spatialLongterm-d3} or \ref{T:spatialLongterm-d2}, but we will give a rigorous argument for a lower-bound on the expected number of survivors. This argument has the merit of making it clear that this is indeed the correct order of magnitude. The idea is the following: in the ball of radius $m$ initially, label all particles $1,2, \ldots, Km^d$ for some $K>0$, in some independent randomized way (for instance, uniformly at random). To every particle labelled $1\le i \le V$, with $V=Km^d$, associate an independent continuous time simple random walk $S^i$ started at $x_i$, the initial position of particle $i$. In the event of a coalescence between particles with labels $i<j$, we decree that the lower-labelled particle necessarily wins: after this event, both particles $i$ and $j$ follow the trajectory of the walk $S^i$.

Fix $\epsilon>0$ and let $P$ be the population of particles with labels $1 \le i \le \epsilon m^2$. It suffices to show that a typical member of $P$ stands a nonzero (asymptotically) chance to survive forever. However, a particle $i \in P$ may only disappear if it coalesces with a particle with a lower label than itself, and hence it can only disappear if it coalesces with another member $j$ from $P$ in particular. Let us assume for instance that this particle is sitting at the origin initially. If $j \in P$ is at position $x \in \Z^d$ initially, the probability that $S^j$ it ever is going to meet $S^i$ is smaller or equal to the expected number of visits to the origin by $S^j -S^i$, which is equal to the Green function\index{Green function} $G(x)$ of simple random walk. Now, it is standard that:
\begin{equation}\label{Green-asymp}
  G(x) \sim c |x|^{2-d}
\end{equation}
Hence, the expected number of coalescences $K(i)$ of $i$ with other members of $P$ is smaller than
\begin{align*}
\E(K(i)) & \le \sum_{x \in \Z^d} \P(\text{there is a particle from $P$ at $x$})\\
& \ \ \ \ \ \ \ \ \ \ \ \times \P(\text{the two particles coalesce})\\
& \le C\sum_{k=0}^m k^{d-1} \frac{\epsilon m^{d-2}}{m^d} k^{2-d}\\
& \le Cm^{-2} \sum_{k=0}^m \epsilon k \\
&= C\epsilon.
\end{align*}
The probability that $i$ disappears is smaller than the expected number of coalescences with members from $P$, thus the probability of survival is greater than $1-C\epsilon$. By making $\epsilon$ sufficiently small, this is greater than $1/2$, and we conclude that the expected number of particles that survive forever is thus at least greater or equal to $(1/2)\# P = (\epsilon /2) m^{d-2}$. Some more precise computations on the dependence between these events for various particles $i$ and a martingale argument are enough to conclude for the lower-bound in Theorem \ref{T:spatialLongterm-d3}.

The upper-bound is on the other hand much more delicate, and the argument of \cite{abl} uses a multiscale analysis to show that the density decays in a roughly inversely proportional way to time in $d \ge 3$. The key difficulty is to control particles that may escape to unpopulated regions and thereby slowing coalescence. The multiscale approach allows to bound the probability of such events at every stage.
\end{proof}

\subsection{Some continuous models}

At this stage the understanding of spatial $\Lambda$-coalescents is quite rough and it would be of considerable interest to establish more precise results about the distribution of the location of the particles in the manner of, say, Theorem \ref{T:voterSBM}. The first step is to identify the \emph{effective branching rate} in the time-reversed picture. There are several possible ways to do this. One of them is to get some inspiration from a remarkable study by Hammond and Rezkhanlou \cite{hr} of a model of coalescing Brownian motions.

\subsubsection{A model of coalescing Brownian motions}

In this model, $N$ particles are performing a Brownian motion in $\R^d$ with $d\ge 3$ say, although they have also studied the case $d=2$ in a separate paper \cite{hr2} (see also \cite{RezaRH} for a related model where masses may be continuous). Two particles may interact when they find themselves at a distance of order $\eps$ of one another, where $\eps>0$. One way to describe this interaction is that there is an exponential clock with rate $\eps^2$ for every pair of particles, such that when the total time spent by this pair at distances less than $\eps$ exceeds that clock, then the two particles coalesce, and are replaced by a single particle at a ``nearby location" (usually somewhere on the line segment which joins the two particles for simplicity, although this isn't so important). In this model, the diffusivity of particles depends also on the size of the block that it corresponds to: thus there is a function $d(n) \ge 0$ which is usually non-increasing such that a particle of mass $n$ (i.e., made up of $n$ particles having coalesced together) has a diffusivity of $d(n)$. This models the physically reasonable situation where larger particles don't diffuse as fast as light particles. We are also given a function $\alpha(n,m)$ which represents the microscopic propensity for particles of masses $n$ and $m$ to coalesce, thus the rate for the exponential clock is chosen to be $\eps^2 \alpha(n,m)$.

Hammond and Rezakhanlou are able to prove in \cite{hr}, subject to certain conditions on the functions $d(n)$, that the density of particles rescales to an infinite hierarchy of PDEs called \emph{Smoluchowski's equations}\index{Smoluchowski's equations}. The initial number of particles $N$ is chosen so that every particle coalesces with a bounded number of particles in a unit time-interval. A simple calculation based on the volume of the Wiener sausage of radius $\eps$ around the trajectory of a single Brownian motions shows that this happens if $N\eps^{2-d} = Z$ is constant. With these conventions, their result is as follows. Let $g^\eps_n(dx,t)$ denote the rescaled empirical distribution of particles of mass $n$ at time $t$ and at position $x$, that is, $g^\eps_n$ is the point measure:
$$
g^\eps_n(dx,t) = \eps^{2-d} \sum_{i \in \mathcal{P}(t)} \delta_{x_i}(dx)
$$
where $\mathcal{P}(t)$ is the set of particles alive at time $t$ and $x_i$ is the location of particle $i$. Then the main result of \cite{hr} is as follows.

\begin{theorem}\label{T:HR} For any test smooth test function $J(x)$, and any $n\ge 1$, and for all $t\ge 0$, as $\eps \to 0$:
$$
\int_{\R^d} J(x) g^\eps_n(dx,t) \overset{L^1}\longrightarrow \int_{\R^d} J(x) f_n(x,t)dx
$$
where the functions $f_n(x,t)$ satisfy:
\begin{equation}\label{smol}
  \frac{\partial f_n}{\partial t} = \frac{d(n)}2 \Delta f_n(x,t) + Q_1^n(f)(x,t) - Q_2^n(f)(x,t),
\end{equation}
where $Q_1$ and $Q_2$ are given by
$$
Q^n_1(f) = \sum_{k=1}^n \beta(k,n-k) f_k f_{n-k}
$$
and
$$
Q^n_2(f) = 2f_n \sum_{k=1}^n \beta(n,k) f_k .
$$
The numbers $\beta(n,m)$ are given in terms of a certain PDE which involves $\alpha(n,m)$: first find the solution to
$$
\Delta u_{n,m} = \frac{\alpha(n,m)}{d(n) + d(m)}(1+u_{n,m}), \text {in }B(0,1)
$$
with zero boundary condition on the unit sphere $\partial B(0,1)$, and then
\begin{equation}\label{betasmol}
\beta(n,m) =  \frac{\alpha(n,m)}{d(n) + d(m)} \int_{B(0,1)}(1+u_{n,m}).
\end{equation}
\end{theorem}

The equation (\ref{smol}) has an intuitive interpretation: roughly, $f_n$ changes because of motion (which gives the Laplacian term), or because of coalescence of two particles of size $k$ and $n-k$, where $1 \le k \le n$, (giving the first term $Q_1$) or because of coalescence of a particle of mass $n$ with a particle of possibly different mass $k\ge1$ (this is a loss in this case, and is responsible for the term $Q_2$). It should be emphasized that proving existence and uniqueness of solutions to (\ref{smol}) can be extremely difficult, essentially due to the fact that some nontrivial gel may form (i.e., creation of particles of infinite mass in finite time). This is an old problem, and one that Hammond and Rezkhanlou have also partly contributed to clarify in subsequent papers.

One of the remarkable features of this result is that the macroscopic coagulation rates $\beta(n,m)$ differ from the microscopic ones $\alpha(n,m)$. This reflects the fact that a kind of macroscopic averaging occurs and there is an ``effective rate of coalescence", which takes into account how much do particles effectively see each other when they diffuse and may coalesce with others.

The model of spatial $\Lambda$-coalescents may be viewed as a lattice approximation of the model of Hammond and Rezakhanlou in the particular case where the diffusivity $d(n)$ does not depend on $n$. In that case, the hierarchy of PDE's (\ref{smol}) simplifies greatly and becomes simply:
\begin{equation}\label{smolpde}
\frac{\partial f}{\partial t} = \frac12 \Delta f- \beta f^2
\end{equation}
for some $\beta >0$. Thus it is tempting to make the conjecture that the equation (\ref{smolpde}) also describes the hydrodynamic limit of the density of particles at time $t$ and at position $x$ in spatial $\Lambda$-coalescents. The number $\beta$ is the solution to a certain \emph{discrete difference equation} which is the discrete analogue of (\ref{betasmol}).

\subsubsection{A coalescent process in continuous space}

The fact that the number $\beta$ in (\ref{smolpde}) does not agree with (\ref{betasmol}) is rather disconcerting. It is an indication that, even after taking the hydrodynamic limit, the discrete nature of the interactions and the exact microscopic structure of the lattice on which these interactions take place, play an essential role in the macroscopic behaviour of the system. This makes it doubtful thats such models should be taken too seriously for modeling real populations. Instead, it is natural to ask for models that are more robust or \emph{universal}, just the same way as Brownian motion is a valid approximation of many discrete systems, irrespectively of their exact microscopic properties. Moreover, the kind of models discussed above (spatial $\Lambda$-coalescents and the model of coalescing Brownian motions of Hammond and Rezakhanlou) fall roughly in the same kind of models as that of coalescing random walks, and thus we also anticipate that the genealogical relationships they describe is similar in some way to that of super-Brownian motion\index{Super-Brownian motion}. It should however be pointed out that super-Brownian motion, although a rich source of mathematical problems in their own right, is rather inadequate as a model of populations living in a continuum. We refer the reader to the discussion in the introduction of \cite{beh} for such reasons. The difficulty is that they predict that if not extinct, at large
times, the population will form `clumps' of arbitrarily large density and extent, which goes against the intuition that some kind of equilibrium is settling in.

To circumvent this fact, Etheridge has recently introduced in \cite{etheridge:2008} a new model of coalescence in continuous space, which is based on a Poisson point process of events in a roughly analogous fashion to the Poissonian construction of $\Lambda$-coalescents. Suppose that we are given a measure $\mu(dr,du)$ on $\R_+ \times (0,1)$. The measure $dx \otimes \mu$ on $\R^d \times \R_+ \times (0,1)$ indicates the rate at which a proportion $u \in (0,1)$ of lineages in a ball of radius $r$ around any given point $x$ coalesces. The location of the newly formed particle is then chosen either uniformly in the ball of radius $r$ around $x$, or precisely at $x$. The analysis of this process is only at its very initial stage at the moment. To start with, even the existence of the model does not appear completely trivial: one needs some conditions on $\mu$ which guarantee that not too many events are happening in a given compact set: for instance, the trajectory of a given particle will be a L\'evy process and one set of conditions on $\mu$ comes from there. Another set is purely analogous to the $\Lambda$-coalescent condition. Conditions for the existence of the process and some of its properties are analysed in \cite{BartonEtheridgeVeber}, who rely on a modification of a result due to Evans \cite{Evans97}. Among other things, they study scaling limits of this process when space is no longer the full plane but rather a large two-dimensional torus, and the measure $\mu(dr,du)$ may be decomposed as a sum of two measures corresponding to ``big events" which involve a large portion of the space (such as a major ecological catastrophe) and a measure for small (local) events. They show that there is a large spectrum of coalescing processes which may be obtained in the limit to describe the genealogies of the population. This contrasts with the spatial $\Lambda$-coalescent model of Limic and Sturm where the scaling limit is always Kingman's coalescent, regardless of the measure $\Lambda$ (Theorem \ref{T:spatialmeanfield}).

We note that a model of a discrete population evolving in continuous space is described in the paper \cite{beh}, where the main result is that for certain parameter values the process is ergodic in any dimension. (As noted above, this contrasts sharply with population models based on super-Brownian motion). It is believed that the coalescent process associated with this model converges in the scaling limit to Etheridge's process. See \cite{etheridge:2008, beh, bev} for details.

\newpage \section{Spin glass models and coalescents}
\label{S:spin}

In this chapter we take a look at some developments which relate a certain description of spin glass models to a coalescent process known as the Bolthausen-Sznitman coalescent. This is the $\Lambda$-coalescent process which arises when one takes $\Lambda$ to be simply the uniform measure on $(0,1)$. We first introduce a beautiful representation of this process in terms of certain discrete random trees which was discovered by Goldschmidt and Martin \cite{GoldschmidtMartin}, and use this description to prove some properties of this process. We then introduce the famous Sherrington-Kirkpatrick\index{Sherrington-Kirkpatrick} model from spin glass theory, as well as the simplification suggested by Derrida known as the Generalized Random Energy Model\index{Generalized random energy model} (GREM). We describe Bovier and Kurkova's result \cite{bovkur} that the Bolthausen-Sznitman coalescent describes the statistics of the GREM as well as Bertoin and Le Gall's connection \cite{blg0} to Neveu's continuous branching process. Finally, we describe some recent outstanding conjectures by Brunet and Derrida \cite{bdmm1, bdmm2} which are related to this and to several other subjects such as random travelling waves and population models with selection, together with ongoing work in this direction.

\subsection{The Bolthausen-Sznitman coalescent}

\begin{definition}\index{Bolthausen-Sznitman coalescent}
  \label{D:bs} The Bolthausen-Sznitman $(\Pi_t,t\ge 0)$ is the $\cP$-valued $\Lambda$-coalescent process obtained by taking the measure $\Lambda$ to be the uniform measure $\Lambda(dx) =dx$.
\end{definition}

Thus the transition rates of the Bolthausen-Sznitman coalescent are computed as follows: for every $2 \le k \le b$, and for every $n\ge 2$, if the restriction of $\Pi$ to $[n]$ has $b$ blocks exactly, then any given $k$-tuple of blocks coalesces with rate
  \begin{align}
  \lambda_{b,k} &= \int_0^1 x^{k-2} (1-x)^{b-k}dx = \frac{(k-2)!(b-k)!}{(b-1)!}\nonumber\\
  & = \left[(b-1) {b-2 \choose k-2}\right]^{-1}. \label{transBS}
  \end{align}

\subsubsection{Random recursive trees}

We follow the approach of Goldschmidt and Martin \cite{GoldschmidtMartin} which shows a representation of the Bolthausen-Sznitman coalescent in terms of certain random trees called recursive trees.

\begin{definition}\index{Recursive trees}
  \label{D:recursivetrees}
  A recursive tree on $[n]$ is a labelled tree with $n$ vertices such that the label of the root is 1, and the label of vertices along any non-backtracking path starting from the root is monotone increasing.
\end{definition}

In other words, the label of a child is greater than the label of the parent. There are exactly $(n-1)!$ recursive trees on $[n]$: indeed, suppose that a tree of size $1 \le j \le n-1$ has been constructed. Then vertex with label $j+1$ can be added as the child of any of the $j$ vertices already present in the tree. It follows directly from this description that a \emph{randomly chosen} recursive tree (i.e., a recursive tree chosen uniformly at random among the $(n-1)!$ possibilities) can be obtained by a variation of the above procedure: namely, having chosen a randomly chosen recursive tree on $[j]$ with $1\le j \le n-1$, one obtains a random recursive tree on $j+1$ by attaching the vertex with label $j+1$ uniformly at random at any of the $j$ vertices of the tree.

There is a natural operation on recursive trees which is that of \emph{lifting}\index{Lifting of an edge} an edge of the tree. This means the following: assume that the edge $e=(i_1,i_2)$ which is being lifted connects two labels $i_1 < i_2$, so that $i_1$ is closer to the root than $i_2$. Let $i_2, i_3, \ldots, i_j$ be the collection of labels in the subtree below $i_1$. Then this subtree is deleted and the label of $i_1$ becomes $i_1, \ldots, i_j$. Graphically, all vertices below $i_1$ bubble up to $i_1$ and stay there (see Figure \ref{Fig:collapse} for an example).

\begin{figure}
\begin{center}
\includegraphics[scale=.7]{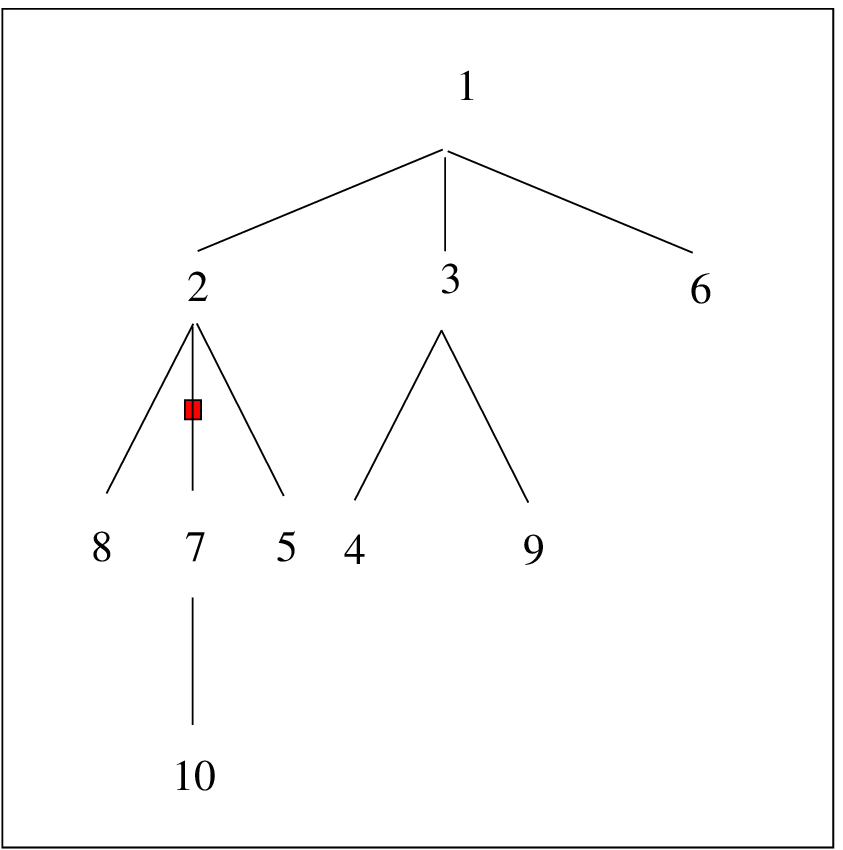}\includegraphics[scale=.7]{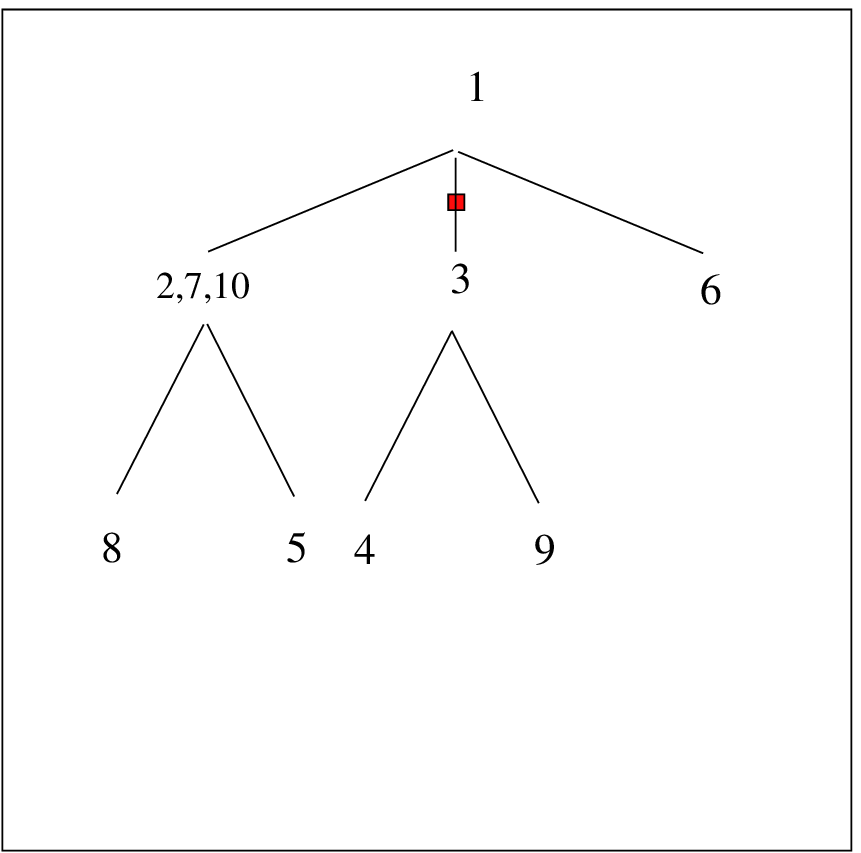}
\includegraphics[scale=.7]{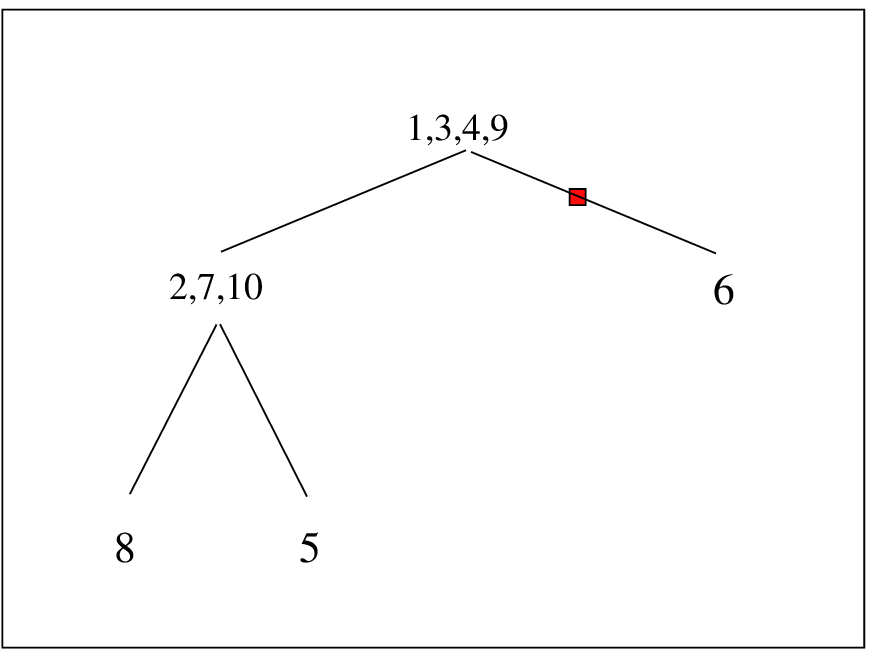}\includegraphics[scale=.7]{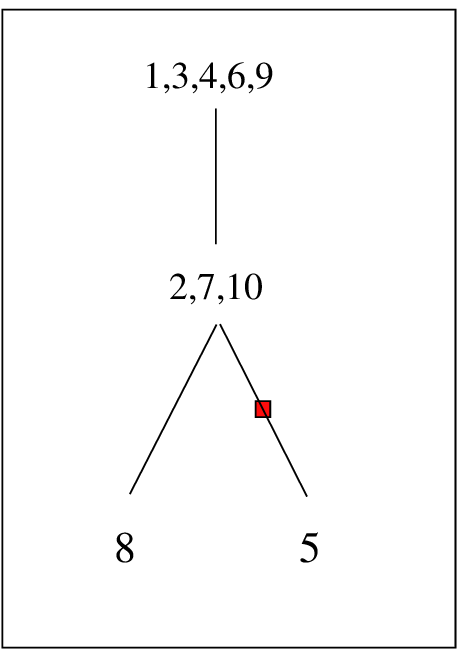}
\includegraphics[scale=.7]{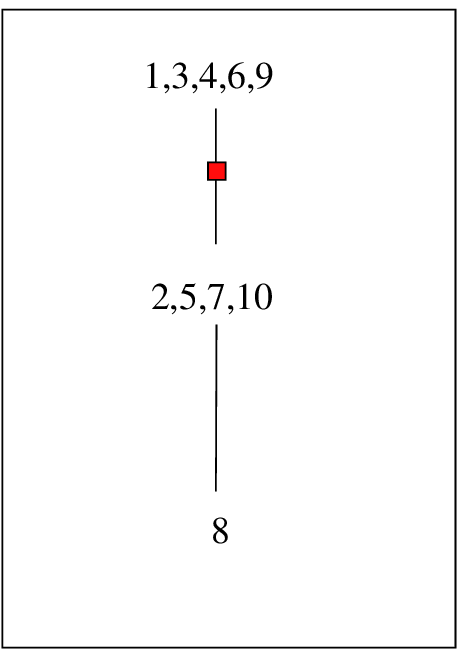}\includegraphics[scale=.7]{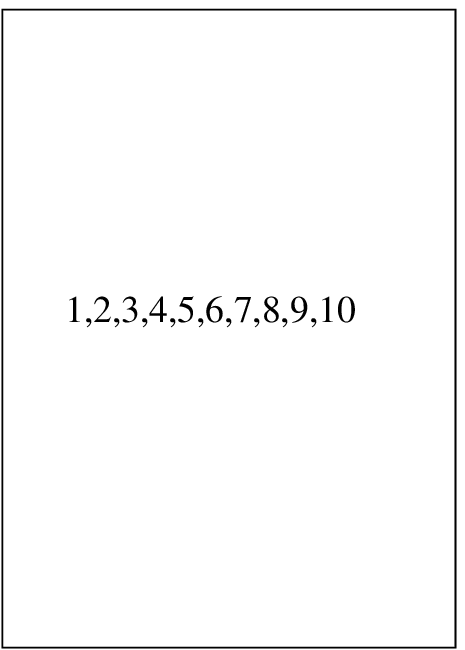}
\end{center}
\caption{Lifting of successive edges of a recursive tree on 10 vertices.}
\label{Fig:collapse}
\end{figure}

Goldschmidt and Martin \cite{GoldschmidtMartin} use the word \emph{cutting} for the operation just described, but we prefer the word lifting as it is more suggestive: we have in mind that the subtree below the edge is given a lift up to that vertex.

This leads us to the slightly more general definition of recursive tree. Let $\pi=(B_1, \ldots, B_k)$ be a partition of $[n]$, with the usual convention that blocks are ordered by their least element.

\begin{definition}\index{Recursive trees}
  \label{D:recursivetrees2}
  A recursive tree on $\pi$ is a labelled tree with $k$ vertices such that all $k$ vertices of the tree are labelled by a block of the partition. The label of the root is $B_1$, and the label of vertices along any non-backtracking path starting from the root is monotone increasing for the block order.
\end{definition}

Given a partition $\pi = (B_1, \ldots, B_k)$ there are naturally exactly $(k-1)!$ possible recursive trees on $\pi$, and the previous notion of recursive trees corresponds to the case where $\pi$ is the trivial partition into singletons.

With these definitions, we are able to state the following result, which is due to Goldschmidt and Martin \cite{GoldschmidtMartin}, and which gives a striking construction of the Bolthausen-Sznitman coalescent in terms of random recursive trees. Let $n\ge 1$, and let $T$ be a random uniform recursive tree on $[n]$. Endow each edge $e$ of $T$ with an independent exponential random variable $\tau_e$ with mean 1, and use this random variable $\tau_e$ to lift the edge $e$ at time $\tau_e$. The label set of the trees defined by these successive liftings define a random partition $(\Pi^n(t),t\ge 0)$.

\begin{theorem}
  \label{T:GoldschmistMartin}
  The process $(\Pi^n(t),t\ge 0)$ has the same distribution as the restriction to $[n]$ of the Bolthausen-Sznitman coalescent.
\end{theorem}

\begin{proof}
The proof is not very difficult. The main lemma, which in \cite{GoldschmidtMartin} is generously attributed to Meir and Moon \cite{MeirMoon1, MeirMoon2}, is the following: (note that it is not literally the same as the one given in \cite{GoldschmidtMartin}, where slightly more is proved).

\begin{lemma}
  \label{L:recunif}
  Let $L$ be a given label set with $b$ elements, and let $T$ be a random recursive tree on $L$. Let $e$ be an edge of $T$ picked uniformly at random, independently of $T$, and let $T'$ be the recursive tree obtained by lifting the edge $e$. Then, conditionally on the label set $L'$ of $T'$, $T'$ is a uniform random recursive tree on $L'$.
\end{lemma}

\begin{proof}
  Fix a label set $L'=\ell'$ such that $L$ is more refined than $L'$, i.e., $L'$ has been obtained from $L$ by coalescing certain blocks, and let $t'$ be a given recursive tree on $L'$. Let us compute $\P(T'=t'; L' = \ell')$. Since both $L$ and $L'$ are given, we know which blocks of $L$ exactly must have coalesced. Let us call $L = \{\ell_1, \ldots, \ell_b\}$ the labels of $T$ ordered naturally, and let $\{\ell_{i_1}, \ldots, \ell_{i_k}\}$ be those labels that coalesce. Thus let $M_1=\{\ell_{i_2}, \ldots, \ell_{i_k}\}$, and let $M_2 = L\setminus M_1$. Let us consider the various ways in which the event $\{T'=t'\}$ may occur. First build a recursive tree $t_1$ on $M_1$ rooted at $\ell_{i_2}$ (there are $(k-2)!$ ways of doing so), and consider the recursive tree $t_2$ on $M_2$ obtained by changing the label $\{\ell_{i_1}, \ldots \ell_{i_k}\}$ of $t'$ into $\ell_{i_1}$. Link $\ell_{i_2}$ to $\ell_{i_1}$ by an edge $e$. The tree $T$ must have been the one obtained by the junction of $t_1$ and $t_2$ (which has probability exactly $1/(b-1)!$), and the edge $e$ linking $\ell_{i_1}$ to $\ell_{i_2}$ must be lifted (which has probability $1/(b-1)$). Thus
  \begin{align}
  \P(T'=t'; L' = \ell') &= \frac{(k-2)!}{(b-1) (b-1)!} \label{transBStree0}
  \end{align}
In particular, after division by $\P(L'=\ell')$, the above does not depend on $t'$ and thus $T'$ is a uniform random recursive tree on $L'$.
\end{proof}

It is now a simple game to conclude that the process $\Pi^n(t)$ has the Markov property and the transition rates of the Bolthausen-Sznitman coalescent. Indeed, the lemma above shows that conditionally on $\Pi^n(t)$, the tree $T^n(t)$ is a uniform random recursive tree labelled by $\Pi^n(t)$, and moreover since there are $b-1$ edges in the tree, the total rate at which a merger of $k$ given blocks occurs (say $\ell_{i_1}, \ldots, \ell_{i_k}$ among the $b$ blocks $\ell_1, \ldots, \ell_b$) is exactly $b-1$ times $\P(L'= \ell')$ using the notations in the above lemma. This may be computed directly as in the lemma, as all that is left to do is choose one of the $(b-k)!$ recursive trees on $M_2$, so
\begin{align}
\P(L' = \ell') &= \frac{(k-2)!(b-k)!}{(b-1) (b-1)!} = \frac1{(b-1)^2}\frac{(k-2)!(b-k)!}{(b-2)!} \nonumber \\
  & = \frac1{(b-1)^2} {b-2 \choose k-2}^{-1}.\label{transBStree}
\end{align}
(Another way to obtain (\ref{transBStree}) is that since $T'$ is uniform conditionally given $L'$, and sice we already know that there are $(b-k)!$ recursive trees on $L'$, we conclude from (\ref{transBStree0}) that $\P(L'=\ell')$ is $(b-k)!$ times the right-hand side of (\ref{transBStree0}), which is precisely (\ref{transBStree}).)

Thus multiplying (\ref{transBStree}) by $(b-1)$, the rate at which $(\ell_{i_1}, \ldots, \ell_{i_k})$ is merging is exactly the same as $\lambda_{b,k}$ for the Bolthausen-Sznitman coalescent (\ref{transBS}). This finishes the proof.
\end{proof}

\subsubsection{Properties}

Theorem \ref{T:GoldschmistMartin} allows us to prove very simply a number of interesting properties about the Bolthausen-Sznitman coalescent, some of which were already discovered by Pitman \cite{pit99} although using more involved arguments. We start with the following result:

\begin{theorem}
  \label{T:BSPD}
  Let $(\Pi_t,t\ge 0)$ be the Bolthausen-Sznitman coalescent. Then for every $t>0$,
  $$
  \Pi_t \overset{d}= PD(e^{-t}, 0)
  $$
  has the Poisson-Dirichlet distribution\index{Poisson-Dirichlet} with $\alpha = e^{-t}$ and $\theta = 0$. In particular, $\Pi$ does not comes down from infinity.\index{Coming down from infinity}
\end{theorem}

\begin{proof}
  The proof is simple from the construction of Theorem \ref{T:GoldschmistMartin}: all we have to do is observe that the Chinese Restaurant Process\index{Chinese Restaurant Process} is embedded in the construction of random recursive trees. Indeed, let $E_1, \ldots, $ be independent exponential random variables with mean 1. Fix a time $t>0$ and imagine constructing a random recursive tree $T$ on $[n]$ by adding the vertices one at a time. We also put a mark on the edge $e_i$ which links vertex $i$ to the root if and only if $E_i<t$. We interpret a mark as saying that the edge has been lifted prior to time $t$, but rather than collapsing it we keep it in place and simply keep in mind that we have to do the lifting operation in order to obtain $\Pi^n(t)$. Suppose that after collapsing those edges, we would have $k$ blocks of respective sizes $n_1, \ldots, n_k$, with smallest elements $i_1, \ldots, i_k$. When vertex $n+1$ arrives, it forms a new block if it attaches to one of $i_1, \ldots, i_k$ and if its edge isn't marked, which has probability $ke^{-t}/n$. Otherwise, it becomes part of a block of size $n_j$ if it attaches to one of the $n_j-1$ vertices below $i_j$ (regardless of whether its edge has a mark) \emph{or} if it is attached to $i_j$ and its edge is marked. The probability this happens is
  $$
  \frac{n_j-1 + (1-e^{-t})}n = \frac{n_j - \alpha}{n}
  $$
  where $\alpha = e^{-t}$. Thus $\Pi_t$ has the Poisson-Dirichlet $PD(e^{-t},0)$ distribution. One can deduce from this and Theorem \ref{T:regvar} that the number of blocks at time $t$ in $\Pi^n(t)$ is approximately $n^\alpha$ with $\alpha =e^{-t}$. Since this tends to $\infty$, this proves Theorem \ref{T:BSPD}.
\end{proof}

Another striking application of Theorem \ref{T:GoldschmistMartin} is the following description for the frequency of a size-biased picked block\index{Size-biased! pick}. That is, consider $F(t)$ the asymptotic frequency of the block containing 1 in $\Pi_t$, where $(\Pi_t,t\ge 0)$ is the Bolthausen-Sznitman coalescent.

\begin{theorem}
  \label{T:BSfreq}
  The distribution of $F(t)$ is the Beta$(1-\alpha, \alpha)$ distribution, where $\alpha = e^{-t}$. Moreover, we have the following identity in distribution for the process $(F(t),t\ge 0)$:
  \begin{equation}\label{BSfreqGamma}
  \{F(t),t\ge 0\} \overset{d}= \left\{\frac{\gamma(1-e^{-t})}{\gamma(1)}, t\ge 0\right\},
  \end{equation}
  where $(\gamma(s),s \ge 0)$ is the Gamma subordinator\index{Subordinator}, i.e., the process with independent stationary increments such that
  $\P(\gamma(s) \in dx) = \Gamma(s)^{-1} x^{s-1} e^{-x} dx$.
  \end{theorem}

Note that the right-hand side of (\ref{BSfreqGamma}) is Markovian, hence so is the process $(F(t),t\ge 0)$. There is no obvious reason why this should be the case, and in fact we do not know of any other example where this is the case. Another consequence of this fact is that $-\log(1- F(t))$ has independent (but not stationary) increments.

\begin{proof}
  It is easier to think of a random recursive tree on the label set $\{0, \ldots, n\}$ with thus $n+1$ vertices. Then note that as we build the random recursive tree on this vertices, the partition $P_n$ of $\{1, \ldots, n\}$ obtained by looking which vertices are in the same component of the tree if we were to cut all edges connected to the root 0, is exactly a Chinese Restaurant Process\index{Chinese Restaurant Process} but this time with with parameters $\alpha = 0$ and $\theta =1$. Thus it has the same distribution as the one induced by random permutations\index{Random permutations}. It follows that in the limit, these normalized component sizes have precisely the $PD(0,1)$ distribution. Now, use the construction of Theorem \ref{T:GoldschmistMartin} on $n+1$ vertices with edges marked as in the proof of Theorem \ref{T:BSPD}, to see that the ranked jumps of $F(t)$, say $J_1 \ge J_2 \ge \ldots $, are precisely given by the ranked components of a $PD(0,1)$ random variable. This sequence of jumps is furthermore independent of the corresponding jump times $(T_1, \ldots, )$, which are by construction independent exponential random variables with mean 1. It is fairly simple to see that these three properties and the Poisson construction of $PD(0,1)$\index{Poisson-Dirichlet} partitions (see the remark after Theorem \ref{T:PDpoisson}) imply Theorem \ref{T:BSfreq}.
\end{proof}

Analysing in greater details the probabilistic structure of random recursive trees (which turns out to involve some  intriguing number theoretic expansions), Goldschmist and Martin are able to obtain some refined estimates on the limiting behaviour of the Bolthausen-Sznitman coalescent restricted to $[n]$ and close to the final coagulation time. We discuss a few of those results.

The following says that the sum of the masses $M_n$ of the blocks not containing 1 in the final coalescence of $(\Pi^n(t), t\ge 0)$, is approximately $n^U$, where $U$ is a uniform random variable. More precisely:

\begin{theorem}
  \label{T:finalcollision1}
  Let $M_n$ be as above and let $B_n$ be the number of blocks involved in the last coalescence event. Then
  $$
  \left(\frac{\log M_n}{\log n}, B_n\right) \overset{d}\to (U,1+ Y(UE))
  $$
  where $U$ is a uniform random variables, $(Y_t,t\ge 0)$ is a standard Yule process, $E$ is a standard exponential random variable and $Y,U,E$ are independent.
\end{theorem}

The Yule process is a discrete Galton-Watson process which branches in continuous time at rate 1 and leaves exactly two offsprings. The convergence of the second term in the left-hand side indicates that there is a nondegenerate limit for the number of blocks in the last coalescence event. One can similarly ask about the number of blocks involved in the next to last coalescence, and so on. Let $(M_n(1), \ldots, ) $ be this sequence of random variables, i.e., $M_n(i)$ is the number of blocks involved in the $i\th$ coalescence event from the end. Goldschmidt and Martin \cite{GoldschmidtMartin} show that all this sequence converges for finite-distributions towards a nondegenerate Markov chain. This Markov chain converges to infinity almost surely. They interpret this last result as a \emph{post-gelation phase} where most of the mass has already coagulated and the remaining small blocks are progressively being absorbed.

Along the same lines, they obtain a result concerning the time at which the last coalescence occurs. Naturally, this time diverges to $\infty$ since $\Pi$ does not come down from infinity, and Goldschmidt and Martin establish the following asymptotics:

\begin{theorem}
  \label{T:finalcollision2} Let $T_n$ be the time of the last coalescence event. Then
  $$
  T_n - \log \log n \overset{d}\to - \log E
  $$
  where $E$ is an exponential random variable with mean 1.
\end{theorem}

This means that the order of magnitude for the last coalescence time is about $\log \log n$. This could have been anticipated from the fact that, by Theorem \ref{T:BSPD} and Theorem \ref{T:regvar}, the number of blocks at time $t$ is about $n^\alpha$, with $\alpha = e^{-t}$. This becomes of order 1 when $t$ is of order $\log \log n$.

Finally, a result of Panholzer \cite{Panholzer} (see also Theorem 2.4 in \cite{GoldschmidtMartin}) about the number of cuts needed to isolate the root in a random recursive tree implies the following result.

 \begin{theorem}Let $\tau_n$ be the total number of coalescence events of a Bolthausen-Sznitman coalescent process started from $n$ particles. Then
\begin{equation}
  \frac{\log n}{n} \tau_n \to 1
\end{equation}
in probability, as $n \to \infty$.
\end{theorem}
\subsubsection{Further properties}

Many other properties of the Bolthausen-Sznitman have been studied intensively. The time reversal of the Bolthausen-Sznitman coalescent is studied by Basdevant \cite{basdevant} and is shown to be an inhomogeneous fragmentation process after an exponential time change. A similar idea was already present in the seminal paper of Pitman \cite{pit99}. In fact, this process is closely related to the ``Poisson cascade" introduced even earlier by Ruelle \cite{Ruelle},
and it was Bolthausen and Sznitman \cite{bosz98} who noticed that an exponential time change transformed the process into a remarkable coalescent process. Pitman later realised that this coalescent was an example of the coalescents with multiple collisions which he was considering.

The allelic partition of the Bolthausen-Sznitman coalescent was studied by Basdevant and Goldschmidt \cite{BasdevantGoldschmidt}, using an elegant martingale argument which fits in the theory of fluid limits developed by Darling and Norris \cite{DarlingNorris}. They were able to show that if there is a constant mutation rate $\rho>0$, then almost all types are singletons, meaning that they are represented in only one individual (or that their multiplicity is 1). More precisely, they showed:

\begin{theorem}
  \label{T:BSallelic}
  Let $M_k(n)$ denote the number of types with multiplicity $k$ in the Bolthausen-Sznitman coalescent, and let $M(n)$ be the total number of types. Then as $n \to \infty$,
  $$
  \frac{\log n}n M(n) \to \rho
  $$
  in probability, and for $k \ge 2$,
  $$
  \frac{(\log n)^2}n M_k(n) \to \frac{\rho}{k(k-1)},
  $$
  as $n \to \infty$ in probability.
\end{theorem}

A similar result was first proved by Drmota et al. \cite{Drmota+} for the total length of the coalescence tree, rather than the number of types. The biological interpretation of this result is less clear, since we do not have any evidence that this coalescent process is appropriate for modelling the genealogies of any species. However, it is believed that the Bolthausen-Sznitman coalescent describes a universal scaling limit for certain model with high selection, as will be discussed below.

\subsection{Spin glass models}\index{Spin glass}

\subsubsection{Derrida's GREM}
\index{Generalized random energy model}\index{Continuous Random Energy Model}

We start by a heuristic description of the model invented by Derrida known as the GREM (for generalized random energy model). The first version of the model was introduced in \cite{Derrida1}, and this was generalized in \cite{Derrida2}, to incorporate several energy levels. This idea was followed up by Bovier and Kurkova in the form of the Continuous Random Energy Model (CREM), which is the version we now discuss. We start by stating the problem and give the result of Bovier and Kurkova \cite{bovkur} about this model, which is followed by a brief description of some of the ingredients in the proof. We then explain the relation to the Sherrington-Kirkpatrick model.

The model is as follows. Let $N\ge 1$ and consider the $N$-dimensional hypercube $\cS_N=\{-1,1\}^N$. An element $\sigma \in \cS_N$ is a spin configuration, i.e., an assignment of $\pm 1$ spins to $1, \ldots, N$. We identify $\cS_N$ with the $N\th$ level $\mathbb{T}_N$ of the binary tree $\mathbb{T}$ as follows: if $\sigma \in \cS_N$, then $\sigma$ may be written as a sequence of $-1, +1$, say $\sigma = \sigma_1 \ldots \sigma_n$, and we interpret this sequence as describing the path from the root of the binary tree to the vertex $\sigma$ at the $N\th$ level of the tree: the first vertex is the root, the second is the left child of the root if $\sigma_1 = -1$, and the right child of the root if $\sigma_1 =+1$. The second vertex in this path is the left child of the preceding vertex if $\sigma_2=-1$, and its right child if $\sigma_2 =+1$, and so on.

Given two spin configurations $\sigma$ and $\tau$, there is a natural distance between them, which is the \emph{genealogical metric}:
\begin{equation}\label{GREMd}
d(\sigma, \tau) = 1- \frac1N\max\{1\le i \le N: \sigma_i = \tau_i\}.
\end{equation}
Thus for $0< \eps<1$, the distance between $\sigma$ and $\tau$ is less than $\eps$ if the paths from the root to $\sigma $ and $\tau$ are identical up to level $(1-\eps)N$. In other words, the distance $d(\sigma, \tau)$ is 1 minus the normalized level of the most recent common ancestor between $\sigma$ and $\tau$. We then assume that we are given a function $A:[0,1] \to [0,1]$ which is nondecreasing, such that $A(0) = 0$ and $A(1)=1$. Consider now a centered Gaussian field $(X_\sigma, \sigma \in \cS_N)$ which is specified by the following covariance structure:
\begin{equation}
  \label{GREMcov}
  \cov(X_\sigma, X_\tau) = A(1-d(\sigma, \tau)).
\end{equation}
Thus with this definition, note that spin configurations $\sigma$ and $\tau$ that are closely related genealogically are also highly correlated for the Gaussian field $X$. On the other hand, for spin configurations whose most recent common ancestor is close to the root of the tree, then the values of the field at these two configurations are nearly independent.

In the GREM, one fixes a parameter $\beta>0$ and consider the Gibbs distribution\index{Gibbs distribution} with inverse temperature $\beta$ defined as follows:
\begin{equation}
  \mu_\beta(\sigma) = \frac1Z e^{\beta X_\sigma},
\end{equation}
where $Z$ is a normalizing (random) constant chosen so that $\sum_{\sigma } \mu_\beta(\sigma) =1$ almost surely. Thus the Gibbs distribution favours the spin configurations such that $X_\sigma$ is large. Now, consider sampling $k$ spin configurations $\sigma_1, \ldots, \sigma_k$ independently according to the Gibbs distribution. A natural question is to ask what is the genealogical structure spanned by these spin configurations, i.e., what is the law of the subtree of $\mathbb{T}$ obtained by joining $\sigma_1, \ldots, \sigma_k$ to the root. The next result, which is due to Bovier and Kurkova \cite{bovkur}, shows that this is, up to a time-change, asymptotically the same as the Bolthausen-Sznitman coalescent.

More precisely, let $\Pi_N^k(t)$ be the partition of $[k]$ defined by: $i\sim j$ if and only $d(\sigma_i, \sigma_j) \le t$. Then we have the following result.

\begin{theorem}
  \label{T:bovkur} Let $(\Theta_t,t\ge 0)$ denote the restriction to $[k]$ of the Bolthausen-Sznitman coalescent. Then the process $(\Pi^k_N(t), t \geq 0)$ converges in the sense of finite-dimensional distributions as $N \rightarrow \infty$ to the process $(\Theta(-\log f(1-t)), 0 \leq t \leq 1)$, where for $0 < x < 1$,
$$f(x) = \min\bigg\{ \frac{1}{\beta} \sqrt{\frac{2 \log 2}{{\hat A}'(x)}}, 1 \bigg\}$$ and ${\hat A}$ denotes the the least concave majorant of $A$, and $\hat A'(x)$ indicates the right-derivative of $\hat A$.
\end{theorem}

$\hat A$ is also known as the convex hull of $A$, since it is the function such that the region under the graph of ${\hat A}$ is the convex hull of the region under the graph of $A$: see Figure \ref{Fig:bovkur}.
\begin{figure}[h]
  \begin{center}
   \includegraphics[scale=.8]{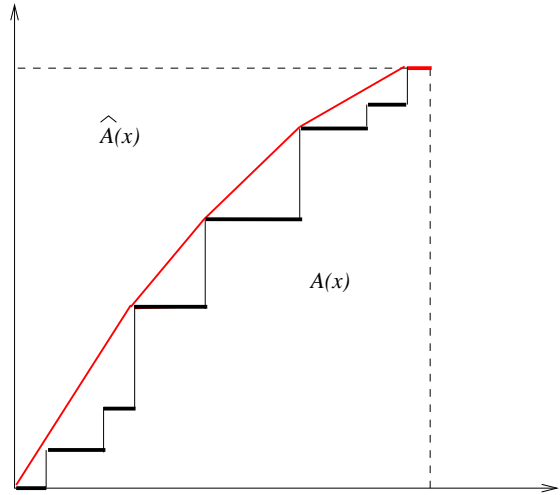}
  \end{center}
  \caption{The convex hull of the function $A$.}
  \label{Fig:bovkur}
\end{figure}
The case where $A$ takes only finitely many values effectively corresponds to the model discussed by Derrida (in the terminology of the spin glass literature, this represents finitely many energy levels), while the case where $A$ contains a continuous part is the ``Continuous Random Energy Model" analysed by Bovier and Kurkova. In what follows we will sketch a proof of this result in the case of a finite number of energy levels (in fact, with only two energy levels to simplify things).

\medskip Note that in Theorem \ref{T:bovkur}, the Bolthausen-Sznitman coalescent arises regardless of $\beta$ and $A$.  The dependence on $\beta$ and $A$ is only through the time change. However, there are some degenerate cases. For example, if $$\beta < \beta_c: = \sqrt{ \frac{2 \log 2}{\lim_{x \downarrow 0} {\hat A}'(x)}},$$ then $f(x) = 1$ for all $x$, and the Bolthausen-Sznitman coalescent gets evaluated at time zero, so there are no coalescence. In the physics language, $\beta$ is inverse temperature and $1/\beta_c$ is the critical temperature, above which there is no coalescence because we are not sampling enough from the values $\sigma$ for which $X_{\sigma}$ is large.

\subsubsection{Some extreme value theory}

The first ingredient is a basic result from ``extreme value theory". To begin with the simplest case, suppose $X_1, X_2, \dots$ are i.i.d. with a standard normal distribution, and let $M_n = \max\{X_1, \dots, X_n\}$.  Following Exercise 2.3 in \cite{durrett} or Exercise 4.2.1 in \cite{pitman stflour}, choose $b_n$ so that $\P(X_i > b_n) = 1/n$.  Then $b_n \sim \sqrt{2 \log n}$ and for all $x \in \R$,
\begin{equation}\label{gumbel}
\lim_{n \rightarrow \infty} \P(b_n (M_n - b_n) \leq x) = e^{-e^{-x}}.
\end{equation}
This is the famous result that the distribution of the maximum of $n$ normally distributed random variables has asymptotically a Gumbel (double exponential) distribution\index{Gumbel distribution}. Furthermore, because the random variables $X_i$ are independent, one can see from (\ref{gumbel}) that the expected number of the random variables $X_1, \dots, X_n$ that are greater than $b_n + x/b_n$ is $e^{-x}$, and that the distribution of the number of such random variables should be Poisson. More precisely, one can view the set of $X_i$ as a point process on the real line, and we can obtain a nontrivial limit by setting the origin to be where we roughly expect the maximum to be, i.e., $b_n$. Indeed, as $n \rightarrow \infty$, we have the convergence of point processes:
\begin{equation}\label{poisconv}
\sum_{i=1}^n \delta_{b_n(X_i - b_n)} \rightarrow_d {\cal P},
\end{equation}
where ${\cal P}$ is a Poisson process with intensity $e^{-x}$. A version of this result is stated as Theorem 9.2.3 in \cite{bovier}. We call $\cP$ the exponential Poisson process. The exponential Poisson process enjoys several remarkable and crucial properties which we now describe. Let $\{\tau_k\}_{k\ge 1}$ be the points of a uniform rate 1 Poisson process on $[0, \infty)$, and let $\Psi_k = \log(1/\tau_k)$.

\begin{enumerate}

\item $\{\Psi_k \}_{k\ge 1}$ forms an exponential Poisson process on $\R$.

\item For $\beta > 0$, and $c>0$, the points $\{\theta_k = ce^{\beta \Psi_k}\}_{k\ge 1}$ form a Poisson point process of intensity $\beta^{-1}c^{1/\beta} x^{-1-1/\beta}$ on $[0, \infty)$.

\item The points $\{\Psi_k + Y_k\}_{k\ge 1}$, where the $Y_k$ are i.i.d. with density $g$, form a Poisson process with intensity $h(x)$, where
    $$
    h(x) = \int_{-\infty}^{\infty} e^{-z} g(x-z) \: dx = \int_{-\infty}^{\infty} e^{y-x} \: g(y) \: dy = e^{-x} \E[e^{Y_k}].
    $$
    Thus if $V = \log \E[e^{Y_i}]$, they are a translated exponential Poisson process, with the new origin taken to be equal to $V$.

\item If we superimpose independent Poisson processes, where the $i$th has intensity $e^{-(x - x_i)}$, the resulting Poisson process has intensity $f(x)$, where $$f(x) = \sum_{i=1}^{\infty} e^{-(x-x_i)} = e^{-x} \sum_{i=1}^{\infty} e^{x_i} = e^{-(x-W)},$$ where $W = \log( \sum_{i=1}^{\infty} e^{x_i})$.

\end{enumerate}

\subsubsection{Sketch of proof}

\medskip We are now ready to discuss a sketch of the proof of Theorem \ref{T:bovkur}. We assume that $A$ has finitely many energy levels, i.e., $A$ is the distribution function of a probability measure with $n$ atoms at positions $0< x_1 , \ldots, <x_n$ say, with respective masses $a_1, \ldots, a_n$. Thus we assume that $\sum_i a_i = 1$ and that $a_i \ge 0$ for all $1\le i \le n$. Without loss of generality we may assume that $x_n =1$ and we let $x_0 = 0$.

We slightly change our notations for a spin configuration $\sigma \in \cS_N$: we now write it as $\sigma = \sigma_1 \ldots \sigma_n$, where $\sigma_i \in \cS_{N(x_i - x_{i-1})}$, i.e., $\sigma_i$ consists of $N(x_i - x_{i-1})$ spins. (Here we do not worry about the fact that $N(x_i - x_{i-1})$ is not necessarily an integer). Then it is easy to see that the random field $X_\sigma$ on $\cS_N$ may be explicitly constructed as follows: for all $1 \le k \le n$, and for all $\sigma_1, \ldots, \sigma_k$, let $X_{\sigma_1 \ldots \sigma_k}$ be i.i.d. standard Gaussian random variables. Then define $X_\sigma$ to be
\begin{equation}
  X_\sigma = \sqrt{a_1} X_{\sigma_1} + \sqrt{a_2} X_{\sigma_1 \sigma_2} + \ldots + \sqrt{a_n} X_{\sigma_1 \ldots \sigma_n}.
\end{equation}
Indeed one can check directly from the above formula that $X_\sigma$ has the correct covariance structure (and it is naturally a Gaussian field, being a linear combinations of i.i.d. standard Gaussian random variables).

Assume to simplify that $n=2$, so that $0<x_1 < x_2 =1$. For $\sigma_1 \in \cS_{Nx_1}$, let $\eta_{\sigma_1} = e^{\beta \sqrt{Na_1} X_{\sigma_1}}.$ By extreme value theory, there is a constant $b_{N}$ such that the points $b_{N} (X_{\sigma_1} - b_{N})$ converge to a Poisson process with intensity $e^{-x}$.  Also, because there are $2^{N x_1}$ of the random variables $X_{\sigma_1}$, we have $b_N \sim \sqrt{2 \log 2^{N x_1}} = \sqrt{(2 \log 2) N x_1}$. Therefore, the points
$$
\eta_{\sigma_1} = e^{\beta \sqrt{N a_1} b_N^{-1} b_N (X_{\sigma_1}-b_N)}e^{-b_n \beta \sqrt{Na_1}}
$$
form approximately a Poisson process with intensity $C_N x^{-1-1/\beta_1}$, where $C_N>0$ depends on $N$, $a_1$ and $x_1$ but not on $x$, and
\begin{equation}\label{beta1def}
\beta_1 = \frac{\beta \sqrt{N a_1}}{b_n} = \beta \sqrt{ \frac{a_1}{(2 \log 2) x_1}}.
\end{equation}
Similarly, for each fixed $\sigma_1 \in \cS_{Nx_1}$ and all $\sigma_2 \in \cS_{N(x_2-x_1)}$, let $\eta_{\sigma_1 \sigma_2} = e^{\beta \sqrt{N a_2} X_{\sigma^1 \sigma^2}}$. Then
$$
\eta_{\sigma_1} \eta_{\sigma^1 \sigma^2} = e^{\beta \sqrt{N} (\sqrt{a_1} X_{\sigma_1} + \sqrt{a_2} X_{\sigma_2})} = e^{\beta \sqrt{N} X_{\sigma}}.
$$
By extreme value theory again, the points $\eta_{\sigma_1 \sigma_2}$ form approximately a Poisson process with intensity $C'_N x^{-1-1/\beta_2}$, where $$\beta_2 = \beta \sqrt{\frac{a_2}{(2 \log 2) (x_2 - x_1)}}.$$  It follows that for each $\sigma_1$, the points $\eta_{\sigma_1} \eta_{\sigma_1 \sigma_2}$ form a Poisson process of intensity $C'_N \eta_{\sigma_1}^{1/\beta_2} x^{-1-1/\beta_2}$. Therefore, if we consider all points of the form $\eta_{\sigma_1} \eta_{\sigma_1 \sigma_2} = e^{\beta \sqrt{N} X_{\sigma}}$, they form a Poisson process with intensity $A x^{-1-1/\beta_2}$, where $A = C_NC'_N \sum_{\sigma^1} \eta^{1/\beta_2}_{\sigma^1}$, and once we condition on this entire Poisson process, the probability that a given point sampled from the Gibbs distribution $\mu_\beta (\sigma) = Z^{-1} e^{\beta X_\sigma}$ belongs to the ``family" associated with a particular $\sigma_1$ is proportional to $\eta_{\sigma_1}^{1/\beta_2}$.

\medskip This gives us the following picture for the genealogy of the process.  First, we sample $n$ of the values $X_{\sigma}$ with probability proportional to $e^{\beta \sqrt{N} X_{\sigma}}$.  We are likely to sample the same point more than once: in fact, as discussed above, sampling according to $e^{\beta \sqrt{N} X_\sigma}$ is approximately the same as sampling from a Poisson point process $\cP$ with intensity $Ax^{-1-1/\beta_2}$ with weight proportional to $x$. If we identify the samples which come from identical points, this gives us an exchangeable partition $\Pi_0$ where the frequency of the block corresponding to the point $x \in \cP$ is proportional to $x$. Thus the distribution of the ranked frequencies of this exchangeable partition is given by the ranked components of
\begin{equation}\label{gem}
\left(\frac{x_i}{\sum_{j \ge 1} x_j}, i \ge 1\right)
\end{equation}
and since $\cP$ has intensity proportional to $x^{-1-1/\beta_2}$, we conclude by the Poisson construction\index{Poissonian construction!Poisson-Dirichlet} of Poisson-Dirichlet $(\alpha, 0)$ partitions (Theorem \ref{T:PDpoisson})\index{Poisson-Dirichlet} that the vector (\ref{gem}) has the same distribution as the ranked coordinates as a Poisson-Dirichlet random variable with parameters $\alpha = 1/\beta_2$ and $\theta =0$. Thus $\Pi_0 \overset{d}=PD(1/\beta_2, 0)$, and thus $\Pi^k_N(0)$ has approximately the same distribution as the restriction to $[k]$ of a $PD(1/\beta_2)$ random variable.

Going back to the previous level, note that a given sample $x \in \cP$ chosen with weight proportional to its value $x$, comes from the ``family" generated by $\sigma_1$ with probability proportional to $\eta_{\sigma_1}^{1/\beta_2}$, which is a Poisson process with intensity proportional to $x^{-1-\beta_2/\beta_1}$. Thus if we sample from $\cP$ and identify the points that come from the same $\sigma_1$, we obtain an exchangeable partition $\Pi_1$ whose ranked frequencies have the same distribution as those of a Poisson-Dirichlet random variable with parameters $\alpha = \beta_2/\beta_1$ and $\theta=0$. Thus $\Pi_1 = PD(\beta_2/\beta_1, 0)$.

Thus taking $t = (x_2 - x_1)$, we obtain $\Pi_N^k(t)$ by taking every block of $\Pi_N^k(0)$ (which is a $PD(1/\beta_2, 0)$ random variable restricted to $[k]$), and coagulate them according to a $PD(\beta_2/\beta_1, 0)$ random variable. We claim that the resulting random partition is nothing but a $PD(1/\beta_1, 0)$ random variable. There are many ways to see this: one of them being precisely using the fact that the Bolthausen-Sznitman coalescent at time $t$ has the $PD(e^{-t},0)$ distribution (Theorem \ref{T:BSPD}). Indeed, by the Markov property for the Bolthausen-Sznitman coalescent at time $t$, we see that when we coagulate a $PD(e^{-t},0)$ partition with an independent $PD(e^{-s},0)$ partition, we must obtain a $PD(e^{-(t+s)},0)$ random partition.

\medskip Thus we can write for $t=t_2=0$, and $t=t_1=x_2-x_1$, $\Pi^k_N(t_i)\approx PD(1/\beta_i)$ with $i=1,2$ where
\begin{align*}
\frac{1}{\beta_i} &= \frac{1}{\beta} \sqrt{\frac{(2 \log 2) (x_i - x_{i-1})}{a_i}} \\
&= \frac{1}{\beta} \sqrt{ \frac{2 \log 2}{\hat A'(1-t_i)}} \\
&= f(1-t_i) = e^{-(- \log f(1-t_i))}.
\end{align*}
Thus for $i=1,2$, we have shown that $\Pi^k_N(t_i)$ has the same distribution as $\Theta_{- \log f(1-t_i)}$, as claimed in Theorem \ref{T:bovkur}. Note that this argument doesn't really explain how do lineages coalescence between the different energy levels, and this is why we only get convergence in the sense of finite-dimensional marginals in Theorem \ref{T:bovkur}.


\subsection{Complements}

\subsubsection{Neveu's branching process}

The intuitive picture presented here essentially goes back to the work of Ruelle \cite{Ruelle} who talks about probability cascades\index{Ruelle's probability cascades} for the properties of the exponential Poisson process. Bolthausen and Sznitman \cite{bosz98} then realised that reversing the direction of time defined the remarkable coalescent process which now bears their names. Bertoin and Le Gall \cite{blg0}, in their first joint paper on coalescence, showed that the Bolthausen-Sznitman coalescent process was embedded in the genealogy of a certain continuous-state branching process (CSBP)\index{Continuous-state branching process (CSBP)}, which is the CSBP associated with the branching mechanism
$$
\psi(u) = u \log u, \ \ u\ge 0.
$$
This CSBP is known as Neveu's branching process\index{Neveu's branching process}. This was the first paper showing a relation between the genealogy of a CSBP and a $\Lambda$-coalescent, and was a partial motivation to the papers \cite{7, bbs2, bbl2}. However, in the case of Neveu's branching process, the relation between the genealogy and the coalescent is trivial, in the sense that there is no time-change. Bertoin and Le Gall's original approach relied on a precursor to their flow of bridges discussed in Theorem \ref{T:BLG1}. The ideas outlined in Theorem \ref{T:smalltimes}, which come from \cite{bbl2}, provide a direct alternative route (more precisely, the approach of Theorem \ref{T:Betaembedding} shows that the point process $(t, \Delta Z/Z)$ arising from the genealogy of Neveu's branching process and the Bolthausen-Sznitman coalescent are identical). That Neveu's branching process was related to Derrida's GREM was first realized by Neveu in \cite{neveu}, in a paper which is unfortunately unpublished, even though in hindsight it inspired many subsequent developments in the field. The link with extreme value theory is also discussed in that paper.

\subsubsection{Sherrington-Kirkpatrick model}

The Generalized random energy model (GREM) was proposed by Derrida in \cite{Derrida1} and \cite{Derrida2} as a possible simplification of the celebrated Sherrington-Kirkpatrick model\index{Sherrington-Kirkpatrick}. The Sherrington-Kirkpatrick (SK) spin-glass\index{Spin glass} model is similar to the GREM, with the difference being that we use the Hamming distance $$d_N(\sigma, \tau) = \#\{i: \sigma_i \neq \tau_i\},$$
also known as the \emph{overlap}\index{Overlap} between $\sigma$ and $\tau$, where $\sigma_i$ and $\tau_i$ denote the $i$th coordinates of $\sigma$ and $\tau$ respectively.  (Also, the SK model typically refers to the case $A(x) = x$, but other covariance functions have also been studied.) Here $d_N$ is a metric but not an ultrametric\index{Ultrametric}.  Because $d_N$ is not an ultrametric, it is not clear that it even makes sense to define a coalescent process as was done for the GREM. However, it is widely conjectured that if we consider $k$ points $\sigma_1, \dots, \sigma_k$ chosen at random from ${\cal S}_N$ according to the Gibbs measure, the distances between them $d_N(\sigma_i, \sigma_j)$ have the ultrametric property in the limit as $N \rightarrow \infty$, which means that they can be viewed as points on the boundary of a tree equipped with the genealogical metric. Talagrand devotes section 4 of \cite{tal} to ``the ultrametricity conjecture" for the Sherrington-Kirkpatrick model and refers to ultrametricity as ``one of the most famous predictions about spin glasses." Derrida's insight consisted in imposing the ultrametricity directly in the model and analyzing what comes out of it. Remarkably enough, this simple addition makes the model much more tractable and fits the physicists' predictions about the SK model perfectly. See the monograph by Bovier \cite{bovier} for much material related to this field, and see also the lectures by Bolthausen in \cite{BolthausenSznitmanbook}. The ultrametric conjecture was first predicted by Parisi \cite{parisi}. We note however that an important prediction which follows from the ultrametric conjecture is a series of identities which have been proved rigorously by Ghirlando and Guerra \cite{GG} (in a slightly weaker form than predicted), known as the Ghirlando-Guerra identities.

\medskip Much of the magic of the emergence of the Bolthausen-Sznitman coalescent in these spin glass models boils down to the crucial stability properties of the exponential Poisson process (by superposition, addition of noise, etc.). It is natural to guess that this process is, in some sense, the only point process which enjoys these properties. While this is an attractive route to the ultrametric conjecture, we note that this seems a very difficult problem. We refer the reader to the recent work by Aizenmann and Arguin \cite{AizenmannArguin} as well as references therein.

\subsubsection{Natural selection and travelling waves}

As was discussed in the proof of the Bovier-Kurkova theorem, Derrida's GREM may be viewed as an assignment of Gaussian random variables on the leaves of the binary tree of depth $N$ with a covariance structure which depends on the genealogical metric between these leaves. There is one natural model where such correlation structures arise, which is the model of branching random walks\index{Branching random walks} where the step distribution is a standard Gaussian random variable, and where at each step, individuals branch in exactly two particles. That is, start with one particle at time 0. At each time step, particles divide in two and take i.i.d. jumps given by a prescribed distribution (which here is Gaussian). Thus at time $N$, there are $2^N$ particles, whose respective positions rescaled by $\sqrt{N}$ form a centered Gaussian field $X_\sigma$ with covariance $\cov(X_\sigma,X_\tau)$ given by the following formula: if the most recent common ancestor between particle labeled $\sigma$ and particle labeled $\tau$ is at generation $j$, corresponding to a position $S_j$, then there exists independent Gaussian variables $\cN$ and $\cN'$ such that $X_\sigma = N^{-1/2} (S_j + \cN)$ and $X_\tau = N^{-1/2} (S_j + \cN')$, so:
\begin{align*}
\cov(X_\sigma, X_\tau) &= \E(X_\sigma X_\tau) = \frac1N\E[(S_j + \cN)(S_j + \cN')]\\
& = \frac1N \E(S_j^2) = \frac{j}N = 1-d(\sigma, \tau).
\end{align*}
In particular, we may write $\cov(X_\sigma, X_\tau) = A(1-d(\sigma, \tau))$ with $A(x) = x$. Thus Gaussian branching random walks give a natural construction of a random energy landscape of the kind considered in the random energy model. Unfortunately, this is a degenerate case from the point of view of the application of Theorem \ref{T:bovkur}, as $\hat A'(x) =1$ for all $ x \in[0,1]$. Nevertheless, we get out of this simple calculation that the energy landscape defined in the GREM may be viewed as a form of perturbation of branching random walks, with a rather complex covariance structure. Theorem \ref{T:bovkur} then asks about the genealogy of this system of particles.

\begin{figure}[t]
\begin{center}
  \includegraphics[scale=.8]{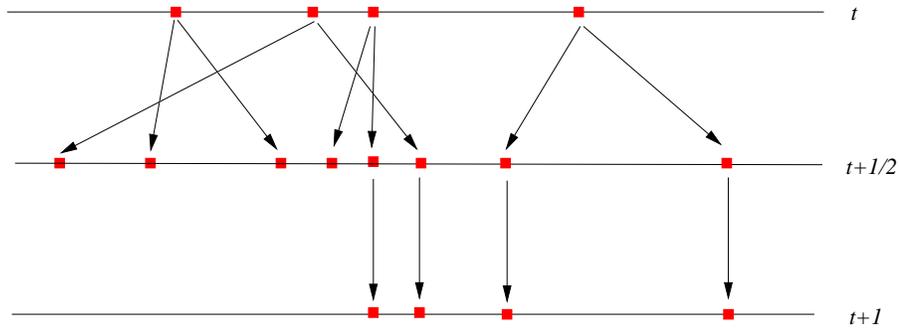}
\end{center}
\caption{The two steps of the Brunet Derrida model with $N=4$: at time $t+1/2$, state of the population after the branching step (there are $2N$ individuals). At time $t+1$, state of the population after the selection step. Only the $N$ largest particles survive.}
\label{Fig:branch}\end{figure}

\medskip Recently, Brunet, Derrida, Mueller and Munier \cite{bdmm1, bdmm2} have introduced a particle system of this kind and made fascinating predictions about its genealogy. Rather remarkably, this model also has an interpretation in terms of a population model with \emph{selection}\index{Selection}, which we now describe. As in the Moran model\index{Moran model}, the population size is kept constant equal to $N$. An individual is represented by her \emph{fitness}\index{Fitness}, which is a real number measuring the likelihood that this individual will produce offsprings surviving in the next generation. Thus, the population at time $t$ may be described by a cloud of $N$ points $X_1(t), \ldots, X_N(t)$ on the real line, ordered in some arbitrary fashion, say linearly. The model has discrete generations and is Markovian. The evolution from one generation $t$ to the next at time $t+1$ consists in two steps: branching and selection. Thus we have an intermediate state, which we may call $t+1/2$, where every individual gives a number of offsprings (let us fix this number to be equal to 2 for every individual, although one may think of a random rule as well). The position of the offsprings of individual $i$ are denoted by $X^1_i(t+1/2), X^2_i(t+1/2)$ and are obtained by:
\begin{align*}
X^1_i\left(t+\frac12\right) &= X_i(t) + \cN^1,\\
X^2_i\left(t+\frac12\right) &= X_i(t) + \cN^2,
\end{align*}
where $\cN^1, \cN^2$ are independent random variables with a fixed continuous distribution (say Gaussian). At this stage there are thus $2N$ individuals, and so the next step, which is the selection step, will reduce the population size to $N$ by \emph{keeping the largest $N$ particles from the population at time $t+1/2$}. Formally, for $1\le k \le N$ we put $X_k(t+1) = Y$ such that
$$
\#\{i:X_i^1(t+1/2) > Y\} + \#\{i:X_i^2(t+1/2) > Y\} = k-1.
$$
An illustration of the model is given in the accompanying Figure \ref{Fig:branch}: note the similarity with the Galton-Watson model of Schweinsberg in Theorem \ref{T:schweinsbergGW} (the difference being that here selection is based on fitness, whereas there selection was made at random).

The interest of \cite{bdmm1, bdmm2} is in the genealogy of an arbitrarily large but fixed sample of the population when its size $N$ tends to infinity, i.e., in the scaling limits of the \emph{ancestral partition process}\index{Ancestral partition} $(\Pi^{k,N}_t,t\ge 0)$ (to use the same terminology as in the first sections of these notes). Using convincing but not fully rigorous arguments, they are able to conjecture that the correct time scale for the ancestral partition process is roughly $(\log N)^3$. More precisely, they conjecture:

\begin{conj}\label{Conj:BD} The ancestral partition process, sped up by a factor $(\log N)^3$, converges to the Bolthausen-Sznitman coalescent. That is, for all $k\ge 1$,
$$
(\Pi_{t(\log N)^3}^{k,N}, t\ge ) \overset{d}\to (\Pi^k_t,t\ge 0)
$$
in the sense of finite-dimensional distributions, where $\Pi^k$ denotes the restriction to $[k]$ of a Bolthausen-Sznitman coalescent.
\end{conj}

This conjecture is accompanied with a very precise picture of what leads to this behaviour. Essentially, the cloud of particle is thought to travel to the right with a positive speed $v_N$ where $v_N \to 2$ as $N \to \infty$ (and there are some conjectures on the first and second correction terms). The particles stay fairly compact, with a width of no more than $O(\log N)$ at any time. Occasionally (every $(\log N)^3$ units of time), a particle travels far to the right, at distance approximately $3 \log \log N +O(1)$ away from the ``bulk" of the population. A particle which does so will will stand a good chance to keep all its offsprings in the next generation after the selection step, and so its descendants quickly generate a large fraction of the population, say $p>0$. This leads to a $p$-merger in the ancestral partition process. Thus the multiple collisions only arise when one takes the scaling limit, speeding up time by $(\log N)^3$.

The derivation of the characteristic time scale comes from an argument of comparison with a stochastic PDE called the stochastic Fisher-KPP equation (for Kolmogorov, Petrovsky and Piscunov), which has the following form:
\begin{equation}\label{sKPP}
\frac{\partial u}{\partial t} = \frac12 \Delta u + u(1-u) + \eps\sqrt{u(1-u)} \dot W
\end{equation}
where $W$ is a white noise. If one removes the noise from this equation (i.e. if $\eps=0$), one obtains the standard Fisher-KPP equation, which is at the heart of the theory of reaction-diffusion partial differential equations. This equation was first obtained independently by Kolmogorov et al. \cite{KPP} and Fisher \cite{Fisher37}, the latter to describe the spread of an advantageous gene in a population. It is known that this equation admits \emph{travelling wave solutions}\index{Travelling waves}, i.e., solution of the form $u(t,x) = F(x-vt)$ where $v>0$. For certain well-chosen initial conditions, the speed of this wave will always be equal to $v=2$. The idea of \cite{bdmm1} is that the distribution function for the population at time $t$ behaves approximately as a solution to (\ref{sKPP}) started from the state $u(0,x) = \indic{x\le 0}$. In the presence of noise, the equation (\ref{sKPP}) generates random travelling waves, which move to the right with a speed $v_\eps$ such that $v_\eps \to v= 2$ as $\eps \to 0$. The asymptotic correction $v_\eps - v$ was studied by Brunet and Derrida \cite{bd1,bd2,bd3} using non-rigorous methods. They conjectured:
\begin{equation}\label{kppspeed}
v_\eps - v \sim - \frac{\pi^2}{4\log^2 \eps}
\end{equation}
and \cite{bdmm1, bdmm2} predicted a second term
\begin{equation}\label{kppspeed2}
v_\eps - v + \frac{\pi^2}{4\log^2 \eps} \sim \frac{3\log |\log \eps|}{4|\log \eps|^3}
\end{equation}
Recently,  Mueller, Mytnik and Quastel \cite{mmq} managed to prove rigorously (\ref{kppspeed}) and give upper and lower bounds matching (\ref{kppspeed2}) up to constants. As the reader has surely guessed, it is this second term (with cubic exponent in $|\log \eps|$) which is the most relevant for Conjecture \ref{Conj:BD}.

We note that B\'erard \cite{bera08}, and B\'erard and Gou\'er\'e \cite{bego08}, have recently studied a discrete version of the Brunet and Derrida model (with particles' locations on $\Z$ rather than $\R$, and selection at random in case of a tie), and were able to show that for each $N$, the system of particles travels at a well-defined speed $v_N$. Furthermore, the second paper \cite{bego08} showed that $v_N - v_0 \sim - \alpha (\log N)^{-2}$ as $N \to \infty$, for some explicit $\alpha>0$ depending solely on the step distribution. This improved on the earlier paper \cite{bera08} which showed that $(\log N)^{-2}$ was the correct order of magnitude for the correction to the speed. This result relied crucially on some recent progress by Gantert, Hu and Shi \cite{GantertHuShi} on the near-critical behaviour of branching random walks.

\medskip Simon and Derrida \cite{SimonDerrida1} have considered a model branching Brownian motion with an absorbing wall and critical drift. They showed (using non-rigorous arguments) that when conditioned upon survival for a long time, this system has a genealogy which is also governed by the Bolthausen-Sznitman asymptotics as in Conjecture \ref{Conj:BD}. Brunet, Derrida and Simon \cite{BrunetDerridaSimon} have also used this theory to describe certain mean-field models of random polymers in (1+1) dimensions at zero temperature and found a similar behaviour, thus confirming further the universal\index{Universality} nature of the Bolthausen-Sznitman coalescent. The work in progress \cite{bbs3} partly confirms these findings for some related models.

\newpage
\appendix
\section{Appendix: Excursions and Random Trees}

What follows is a crash course on some deep ideas due essentially to Aldous, Le Gall and Le Jan in the 90's which relate excursions of random processes (above or below a fixed level) to some random trees which enjoy certain branching properties and in which branching occurs at a dense set of times (or levels). The archetypical example is Aldous' Continuum Random Tree and its relation to the Brownian excursion and the Ray-Knight theorem on the local times of reflecting Brownian motions. We start by recalling the fundamentals of It\^o's excursion theory for Brownian motion as this formalism is central to the study of continuum random trees. We then briefly explain the relation between random trees and random paths, and finally explain how these trees are related to the genealogy of CSBPs and the lookdown process.

\subsection{Excursion theory for Brownian motion}

Let $(B_t, t\ge 0)$ be a one-dimensional standard Brownian motion. The excursion theory of Brownian motion is one of the best tools to study fine properties of $B$. However, the basic idea behind the theory is extremely simple. We call an excursion $e$ of the Brownian motion $B$, a process $(e(t), t\ge 0)$ such that there exists $L<R$ with
$$
e(t) = B_{(L+t)\wedge R}
$$
and for $t \in [L,R]$, $B_t = 0$ if and only if $t = L$ or $R$. That is, $e$ is the piece of $B$ between times $L$ and $R$, which are two consecutive zeros of $B$. The state space of excursions is $\Omega^*$ the space of continuous functions from $\R$ to $\R$ such that there exists $\zeta>0$ satisfying:
\begin{enumerate}
  \item $e_t =0$ if $t\ge \zeta$
  \item For $t \in [0,\zeta]$, $e_t = 0 $ if and only if $t = 0$ or $\zeta$.
\end{enumerate}
$\zeta$ is called the lifetime of the excursion or its length. The basic idea behind the theory is that one can construct Brownian motion by \emph{``throwing down independent excursion"} and concatenating them. The result should be, indeed, a Brownian path.

Of course, it is a little tricky to make this intuition rigorous at first, but it turns out that we can use the language of Point process to express this idea: we will view the collection of excursions of Brownian motion as a Poisson point process on the set $\Omega^*$, and the intensity of the process is a measure called \emph{It\^o's excursion measure}. However to say this properly, we must look at excursions in the correct time-scale, that is, the time-scale at which we are ``adding a new excursion". This time-scale is that of the \emph{inverse local time} process, since local time increases precisely at times when the process hits zero, and thus begins a new excursion. We will refresh the reader's memory about these notions below.

\subsubsection{Local times}

It is well-known that Brownian motion spends an amount of time which has zero Lebesgue measure at any given point: for instance, if $T= \int_0^t \indic{B_s=0} ds$ then by Fubini's theorem
$$
\E(T) = \int_0^t \P(B_s=0)ds = 0
$$
and so $T=0$ almost surely. In fact this argument obviously generalizes to sets $A$ such that $A$ has zero Lebesgue measure: let $T(A)$ be the time spent by Brownian motion up to time $t$ in any given Borel subset of the real line, then if $|A|=0$ $T(A)=0$. Since $T(A)$ is easily seen to be a (random) measure, we get immediately, by the Radon-Nikodym theorem, that there exists almost surely a derivative $T(A)$ with respect to the Lebesgue measure $\text{d}x$:

\begin{definition}\label{D:localtimes} We set
\begin{equation}\label{RN}
L(t,x) = \frac{dT}{dx}
\end{equation}
almost surely. $L(t,x)$ is called the $\emph{local time}$ of $B$ at time $t$ and position $x$ (or level $x$).\index{Local time!(definition)}
\end{definition}

This definition is nice because it is quite intuitive, but is not very satisfactory because of the almost sure in this definition: this only defines $L(t,x)$ for fixed $t$ almost surely and almost everywhere in $x$. It turns out that

\begin{proposition} There exists almost surely a jointly continuous process $\{L(t,x)\}_{t\ge 0, x \in \R}$ for which (\ref{RN}) holds for all $t$ simultaneously.
\end{proposition}

This can be seen through Kolmogorov's continuity criterion. Because it is the Radon-Nikodym derivative of the occupation measure $T(A)$, and because it is continuous in $x$ and $t$, there are a couple of properties that follow immediately. The most useful is the approximation:
\begin{theorem}
\label{T:loctimesapprox}
For every $t\ge 0$, as $\eps \to 0$, we have the following almost sure convergence:
\begin{equation}\label{loctimesapproximation}
\frac1{2\eps}\int_0^t \indic{|X_s -x| \le \eps}ds \longrightarrow L(t,x).
\end{equation}
\end{theorem}

We generally focus on level $x=0$, in which case we almost always abbreviate $L_t = L(t,0).$ From the approximation (\ref{loctimesapproximation}), it follows that $L_t$ is a nondecreasing function, and may only increase at times $t$ such that $B_t = 0$. That is, let $dL_t$ be the Stieltjes measure defined by the nondecreasing function $t\mapsto L_t$, then
\begin{equation}\label{supp}
  \text{Supp}(dL_t) \subset \cZ
\end{equation}
where $\cZ$ is the zero set of $B$. Both sides of (\ref{supp}) are closed sets, so it is natural to conjecture that there is in fact equality. This turns out to be true but it requires some non-trivial arguments. In fact, the proof relies on a famous identity due to Paul L\'evy, which states the following:

\begin{theorem}\label{T:levyid} For every $t\ge 0$, denote by $S_t = \max_{s\le t} B_s$, the running maximum of Brownian motion. Then $(L_t,t\ge 0)$ and $(S_t,t\ge 0)$, have the same distribution as processes.
\end{theorem}

This identity is in fact more general than this: the identity stated above may be viewed as an identification (the maximum of $B$ is the local time of a different Brownian motion $B'$), and in this identification $S_t - B_t$ is equal to the reflected Brownian motion $|B'|$. Thus we have the bivariate identity:
\begin{equation}\label{levyid}
  \{(S_t, S_t - B_t), t\ge 0\} \overset{d}= \{(L_t, |B_t|), t\ge 0\}.
\end{equation}

To save time, we do not give the proof of this result even though it is in fact quite elementary. (Most proofs in textbooks such as \cite{revuz-yor} use the so-called Skorokhod equation, but in fact, the identity may already be seen at the discrete level of simple random walks approximating Brownian motion).

Armed with this result, it is easy to prove equality in (\ref{supp}). What is needed is to show that almost surely, at every time $t>0$ such that $B_t = 0$ then $L_t$ increases. Start with $t=0$: then since $L$ has the same law as $S$, which increases almost surely right after $t=0$, then so does $L$, and so this property holds for $t=0$. By the Markov property, it is not too hard to see it is also true at any time $t$ such that $t=d_s$ for some fixed $s>0$ ($d_s$ is the first zero after $s$, so $d_s = \inf \cZ \cap [s, \infty)$). Playing around with the fact that rational numbers are dense finishes the proof, and so we get
\begin{equation}\label{supp2}
  \text{Supp}(dL_t) = \cZ
\end{equation}
almost surely.

\medskip Our view of local times in these notes is purely utilitarian: even though they deserve much study in themselves, we will only stick to what we strictly need here. For our purpose the last thing to define is thus the \emph{inverse local time}\index{Inverse local time}: for any $\ell>0$, define
\begin{equation}
\label{inverseLT}
\tau_\ell : = \inf\{t>0: L_t >\ell\}.
\end{equation}
$\tau_\ell$ is thus the first time that $B$ accumulates more than $\ell$ units of local time at 0. Thus $\tau_\ell$ is a stopping time, and L\'evy's identity (Theorem \ref{T:levyid}) tells us that
\begin{equation}
\label{tau sub1}
(\tau_\ell, \ell\ge 0) \overset{d}= (T_x, x\ge 0)
\end{equation}
where $T_x$ is the hitting time of level $x$ by $B$. In particular, $(\tau_\ell, \ell\ge 0)$ has independent and stationary increments, and is nondecreasing: that is, $(\tau_\ell, \ell \ge 0)$ is a \emph{subordinator}.\index{Subordinator!} Moreover, it is not hard to see that in fact $\tau$ is the \emph{stable subordinator with index $\alpha = 1/2$}\index{Subordinator!stable}\index{Stable!subordinator} (this follows simply from the reflection principle and the law of $S_t$). That is, the L\'evy measure of $\tau$ has density
\begin{equation}\label{densitystable}
\frac{\alpha}{|\Gamma(1-\alpha)|} s^{-\alpha -1}.
\end{equation}

\subsubsection{Excursion theory}

We will now state It\^o's theorems about excursions of Brownian motion, which make rigorous the intuition explained above. First, a remark: by (\ref{supp2}), we see that if $e$ is an excursion of $B$, corresponding to the interval $[L,R]$, then the local time of $B$ is constant on that interval, since by definition there are no zeros during $(L,R)$. Thus if $(e_i){i\ge 1}$ is an enumeration of the Brownian excursions (something which it is possible to do since there are as many as jumps of a certain subordinator, and these are countable), then we call call $\ell_i$ the common local time of the excursion $e_i$, that is, the local time $L_t$ at any time $t \in (L_i, R_i)$ which is associated to $e_i$.

\begin{theorem} \label{T:Ito}There exists a $\sigma$-finite measure $\nu$ on the space of excursions $\Omega^*$, such that the point process:
$$
\cP(dx) = \sum_{i \ge 1} \delta_{(\ell_i, e_i)}
$$
is a Poisson point process, with intensity $d\ell \otimes \nu(de)$.
\end{theorem}

\begin{definition} $\nu$ is called It\^o's excursion measure.\index{It\^o's excursion measure}
\end{definition}

For instance, the number of excursions by time $\tau_\ell$ with length greater than some $\zeta_0$ is a Poisson random variable, with mean $\ell \nu(\zeta >\zeta_0)$. Another consequence is, for example, that the quantity of local time accumulated by time $T_x$ (the hitting time of $x>0$) is an exponential random variable, with parameter $\kappa(x):=\nu (\sup_{s\ge 0} e_s >x )$: indeed, in the local time scale, the number of points that fall in the set of excursions that hit level $x$, is a Poisson process with constant intensity equal to $\kappa(x)$. Thus the first point is exponentially distributed with parameter $\kappa(x)$ as well.

Thus, in order, to be useful, this theorem should be accompanied with some descriptions of It\^o's excursion measure. First of all, the It\^o measure of excursion of length greater than $\zeta_0$ can be identified through (\ref{densitystable}), since the jumps of $\tau_\ell$ are precisely the excursion lengths. Thus from (\ref{densitystable}) we get the first description in the result below:

\begin{theorem}
  \label{Itodescription}
  We have, for every $x>0$:
  \begin{equation}\label{nulength}
    \nu(\zeta >x) = \frac1{\sqrt{\pi x}}.
  \end{equation}
  Moreover, if $H= \sup_{s>0} e_s$, then
  \begin{equation}\label{nuheight}
    \nu(H>h) = \frac1{2h}.
  \end{equation}
\end{theorem}

\begin{proof}
It is easy to convince yourself that in (\ref{nuheight}), the right-hand side should be $\kappa/h$ for some $\kappa>0$. Indeed, fix some $0<x<h$. On the one hand, the number of excursions that reach $h$ by time $\tau_1$ is Poisson with mean say $\kappa(h)$. On the other hand, this is a thinning of the number of excursions that reach $x$, which is also Poisson but with mean $\kappa(x)$. The thinning probability is nothing but the probability that, given that an excursion reaches level $x$, it will also reach level $h$. However, it is plain to see that an excursion, given that it reaches $x$, behaves after $T_x$ as a Brownian motion killed at 0. Thus the thinning probability is
$$
p = \P_x(T_h < T_0) = \frac{x}{h}.
$$
Hence we deduce:
$$
\kappa(h) = \kappa(x) \frac{x}{h}
$$
for all $0<x<h$. Thus $\kappa(x)$ is equal to $\kappa/x$ for all $x \in (0,h)$ (for some $\kappa>0$). Since $h$ is arbitrary, $\kappa(x) = \kappa/x$ for all $x>0$. That $\kappa=1/2$ requires more work but is classical: see, e.g., (2.10) in Chapter XII of \cite{revuz-yor}. Note that the answer (with the correct value of $\kappa$) can also be guessed from a discrete argument: at each visit of 0, the probability that the next excursion will reach $Nx$ is $(1/2) 1/(Nx)$. (The first $1/2$ comes from asking for positive excursion, and the second term is the familiar ruin probability estimate). At the $N\th$ visit, the total number of excursion that reach $Nx$ is thus approximately a Poisson random variable with mean $\kappa(x)=1/(2x)$ as $N\to \infty$.
\end{proof}

One thing to pay attention to in (\ref{nuheight}) is that we do not count negative excursions in this random variable $H$. That is, $\nu(H>h)$ measures only those positive excursions that reach level $h$. There is an obvious symmetry property in $\nu$, so if instead we want to ask what is the measure of excursions that reach $h$ or $-h$ (which we often do when we think about reflecting Brownian motion), then this measure is now $1/h$ instead of $1/(2h)$.

\subsection{Continuum Random Trees}

After rushing through local times and excursion theory, we now propose another impressionistic rendering of the theory of Continuum Random Trees: that is, how to construct them, and how they are related to Brownian excursions.

\subsubsection{Galton-Watson Trees and Random Walks}

The theory starts with a well-known observation that a (not necessarily random) rooted labelled tree may be described by a certain path, sometimes called the Lukasiewicz path\index{Lukasiewicz path} of the tree. This path provides us with a convenient way of proving things about trees (as we will see that this path is a close cousin of random walk when the underlying tree is a Galton-Watson tree) but it is also very convenient from a purely practical point of view: this path is indeed a variant of the \emph{depth-first search process}\index{Depth-first search} which is used in any algorithm dealing with trees and graphs in general.

First a few definitions: given a finite rooted labelled planar tree $T$, there is a unique way of labelling the tree in ``lexicographical order". That is, the first vertex is the root, $u_0 = \emptyset$. We then list the children of the root, from left to right (this is why we require planarity). These children are called $u_1=1$, $u_2=2$, \ldots, $u_r=r$, say. We now go to the next generation, and attach to each vertex in the second generation a string of two characters (numbers) which is defined as follows: if that vertex is the $r_2\th$ child of the $r_1\th$ individual in the first generation, we attach the string $r_1 r_2$. More generally, to any vertex in the $n\th$ generation, we attach of a string of $n$ characters, $r_1 \ldots r_n$, which specify the path that leads to this vertex: hence, to find the vertex whose label is $u=r_1 \ldots r_n$, at generation 1, find the $r_1\th$ individual. In the next generation, find the $r_2\th$ child of that individual, and so on. This way of labelling all the vertices of the tree is called the canonical labelling, of a planar labelled rooted tree. We may moreover list these vertices in lexicographical order (i.e., as if placing them in a dictionary). This gives us a list of vertices $(u_0, u_1, \ldots, u_{p-1})$. Note that this list entirely specifies the tree; its length is the total size of the tree.

There is a natural way to encode this data into a path: simply, as you go through the list $(u_0, \ldots, u_{p-1})$ (in lexicographical order), record the \emph{height} of the vertex you're at. The height is just the generation or the level of the vertex: hence, a vertex in the second generation of the tree has a height equal to 2. The root has a height equal to 0. The \emph{height process}\index{Height process} of the tree $T$ is the discrete function
\begin{equation}\label{D:height}
h(n) = \text{height of vertex } u_n, 0 \le n \le p-1.
\end{equation}
See Figure \ref{Fig:coding} for an illustration.
\begin{figure}
  \begin{center}
  \includegraphics[scale=.6]{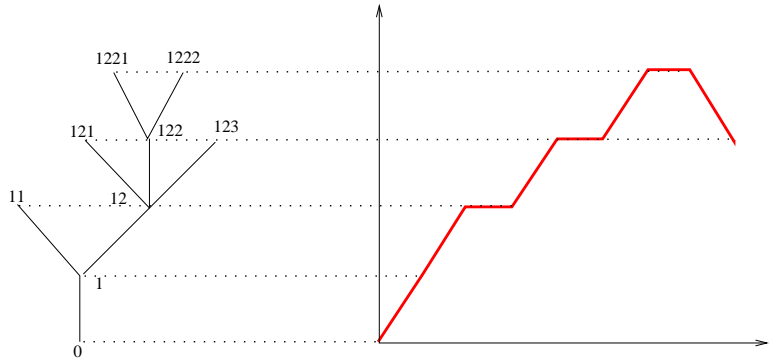}
  \end{center}
\caption{A labelled planar tree and its height process.}
  \label{Fig:coding}
\end{figure}

The Lukasiewicz path of $T$, however, is a different process. Suppose that as we move through the tree in lexicographical order, we recordd the number of children $k_u$ of each vertex $u$. For $0 \le i \le p-1$, define
\begin{equation}\label{dfs}
x_i = k_{u_{i}} -1
\end{equation}
Thus $x_i$ is the number of children of vertex $u_{i}$, minus 1. Define, for $0 \le n \le p-1$,
\begin{equation}\label{dfs2}
s_0 = 0, s_n = \sum_{i=1}^{n-1} x_i \text{ if $1 \le n \le p$}.
\end{equation}
The interpretation of this has to do with the \emph{depth-first search}\index{Depth-first search} of the tree. When we explore the tree, we can partition the tree into vertices that are \emph{active, dead}, and those not touched yet. Dead vertices are those which we have already examined. Active vertices are children of dead vertices, but we haven't explored yet their own children. Untouched vertices are all the rest: they are the descendants of active vertices. Then $s_n$ gives us the number of active vertices at stage $n$ of the lexicographical exploration of the tree: indeed, when we explore vertex $u_n$, there are $k_{u_n}$ new vertices to add to the list of active vertices, but since we are examining $u_n$ we need to subtract 1. The path
\begin{equation}\label{lukapath}
(s_0, \ldots, s_{p})
\end{equation}
is called the Lukasiewicz path\index{Lukasiewicz path}. What is the connection between the two processes?

\begin{lemma}\label{L:heightLuka} $h(n)$ is the number of times that, prior to time $n$, $s_j$ hit its infimum value between times $j$ and $n$:
\begin{equation}
h(n) = \text{\em Card}\left\{ 0 \le j \le n-1: s_j = \inf_{j \le k \le n} s_k\right\}.
\end{equation}
\end{lemma}
The reason this is true is because $s$ only decreases when we have reached a leaf, which is also when $h$ may decrease. Thus any point $u_j$ such that $s_j$ is the future infimum of its path, must be an ancestor of $u_n$. See Figure \ref{fig:Luka} for an illustration.

This a simple combinatorial lemma, but its consequences are hard to overstate: it tells us that $h_n$ may be seen as the \emph{local time at 0} of the process $s$ reflected at its infimum.

Now, consider an offspring distribution $\mu$ on $\N$, and consider the random Galton-Watson tree $T$ associated with the distribution $\mu$: that is, every individuals has an i.i.d. number of offsprings governed by the distribution $\mu$. We make the assumptions that
\begin{enumerate}
  \item $\mu$ is critical: $\E(L) = 1$, where $L \sim \mu$.

  \item $\mu$ has finite variance: $\E(L^2 ) < \infty$.
\end{enumerate}

Observe that the Lukasiewicz path $(S_0, S_1, \ldots , S_p)$ associated with $T$ is now a random walk on $\Z$ started at $S_0=0$ and ended where it first hits level -1:
$$
S_0 = 0, S_n = \sum_{i=1}^n X_i
$$
where the $X_i$ are i.i.d, random variables whose distribution is equal in law to $L-1$. In particular, by assumption 1 and 2 above,
$$
\E(X_i) = 0; \var(X_i) < \infty.
$$
We may consider an infinite sequence of such critical Galton-Watson trees and concatenate their Luckasiewicz paths. Every time the path goes below the starting level, this corresponds to exploring a new tree. The height process representation of Lemma \ref{L:heightLuka} still holds. We have thus encoded each tree in an infinite forest of Galton-Watson trees by the excursions above the infimum of a certain random on $\Z$ with mean 0, finite variance jump distribution. (See Figure \ref{fig:Luka}).

\subsubsection{Convergence to reflecting Brownian motion}

From the previous discussion, it is natural to consider a Brownian scaling of the height process, which is the most intuitive way of encoding the tree.

\begin{theorem}
  \label{T:heightCV}
  As $n \to \infty$, there is the following convergence in distribution, in the sense of the Skorokhod topology on $\mathbb{D}(\R_+, \R)$:
  \begin{equation}\label{heightCV}
  \left(\frac1{\sqrt{n}}H_{nt}, t \ge 0\right) \longrightarrow \left(\frac2{\sigma}|B_t|, t\ge 0\right),
  \end{equation}
  where $\sigma^2 = \var(L)$ is the offspring variance.
\end{theorem}

\begin{proof} (sketch)
This theorem is not hard to understand intuitively: indeed, the Lukasiewicz path, under this scaling, converges towards a Brownian motion with speed $\sigma^2$ (being a centered random walk with finite variance). The main observation is then to see that
\begin{equation}\label{heightcomp}
H_n \approx \frac2{\sigma^2}(S_n - I_n)
\end{equation}
where $I_n = \min\{S_i, i\le n\}$ is the running minimum. Thus it is natural to expect the convergence (\ref{heightCV}), since by the L\'evy identity (\ref{levyid}), $B_t - I_t$ is a reflected Brownian motion $|\beta_t|$.

Thus it suffices to explain (\ref{heightcomp}). Recall Lemma \ref{L:heightLuka}; note that $H_n$ is the number of jumps of the red curve in Figure \ref{fig:Luka}.

\begin{figure}
\begin{center}
\includegraphics[width=11cm]{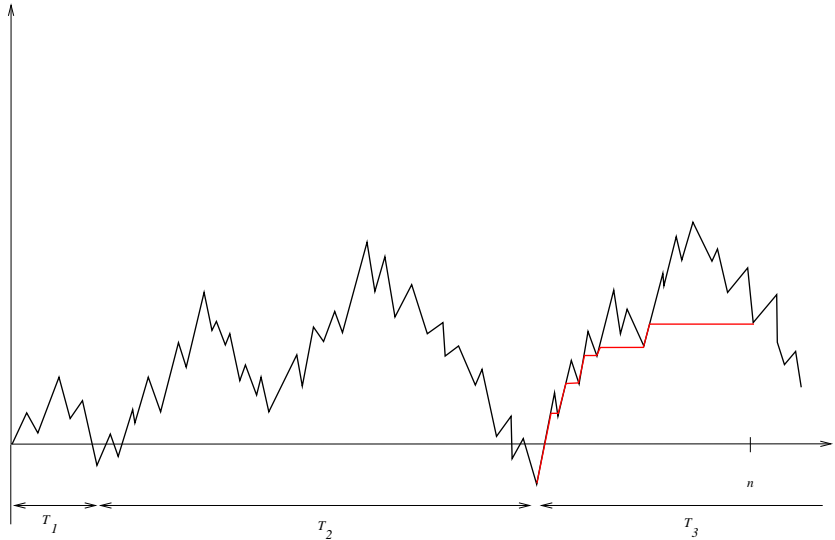}
\end{center}
\caption{Concatenation of the Lukasiewicz paths of three independent critical Galton-Watson trees, $T_1, T_2$ and $T_3$. $H_n$ is equal to the number of jumps of the red curve.}
\label{fig:Luka}
\end{figure}

By reversing the direction of time, we see that $H_n$ is also the number of times that the reverse path, $\hat S$ say, reaches a new lowest point (i.e., the number of jumps of the infimum process of the reverse walk $\hat S$). Now, each time the infimum process jumps, what is the distribution of the overshoot $Y$? Let us denote by $c>0$ the mean of this distribution.
That is, on average, every time the reverse process jumps downwards, it makes a jump of size $c$. It follows that, by the law of large numbers, after $H_n$ jumps, where $H_n$ is large, the total decrease in the initial position is approximately $c H_n$. But since this decrease in position must be equal to $S_n - I_n$, we see that
$$
S_n - I_n \approx c H_n
$$
from which we obtain $H_n \approx c^{-1} (S_n - I_n)$ after division by $c$. Putting things together we deduce that
$$
\left(\frac1{\sqrt{n}}H_{nt},t\ge 0\right) \longrightarrow\left(\frac{\sigma}c |B_t|,t\ge 0\right)
$$
as $n\to \infty$. It is not too difficult to compute the expectation of the overshoot distribution and find that $c = \sigma^2/2$. The result now follows. Further details can be found in Aldous \cite{aldous1}, but see also Marckert and Mokkadem \cite{MarckertMokkadem} and Le Gall and Le Jan \cite{LeGallLeJan}.
\end{proof}

Many corollaries follow rather easily from this asymptotic result. As a case in point, consider the following statement: if $T$ is a Galton Watson conditioned to reach a large level $p$, say,  then its height process satisfies
\begin{equation}\label{heightCV2}
\left(\frac1p H_{p^2 \sigma^2t/4}, t\ge 0\right) \longrightarrow (E_t,t\ge 0)
\end{equation}
where $E$ is a Brownian excursion conditioned to reach level $1$ (that is, a realisation of $\nu(\cdot | H>1)$, where $\nu$ is It\^o's excursion measure).

\subsubsection{The Continuum Random Tree}

We have seen in Theorem \ref{T:heightCV} that the height process, which encodes the genealogy of a critical Galton-Watson trees, conditioned to exceed a large height, converges in distribution towards a Brownian excursion conditioned to reach a corresponding height. It is natural to expect that, as a result, if we now view a finite tree $T$ as a metric space (as we may: we just think of each edge as a segment of length 1), then rescaling this tree suitably, the metric space $T$ converges (in a suitable sense) towards a limiting metric space $\Theta$. This is indeed the case, and the sense of this convergence is the Gromov-Hausdorff\index{Gromov-Hausdorff metric} metric. As this is pretty heavy machinery, we will not explain this construction. Instead, we will describe the limiting object $\Theta$ (the \emph{Continuum Random Tree} \index{Continuum Random Tree} of Aldous \cite{aldous}), and ask the reader to trust us that $\Theta$ is indeed the scaling limit of large critical Galton-Watson trees. More details concerning the Gromov-Hausdorff topology and this convergence can be found, for instance, in Evans' Saint Flour notes \cite{evansStFlour}.

\begin{figure}
\begin{center}
\includegraphics[scale=.75]{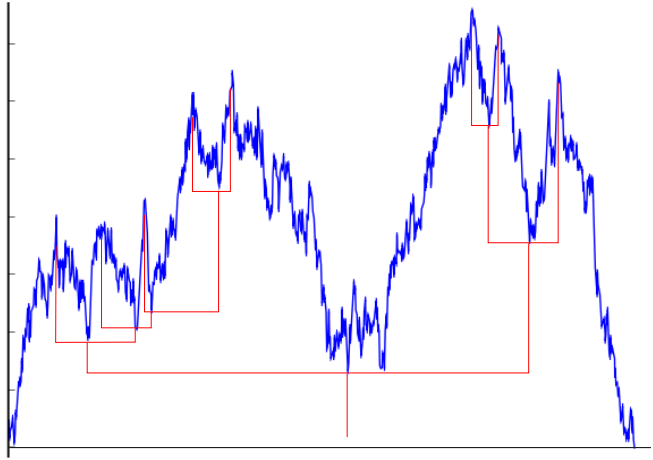}
\end{center}
\caption{A simplified representation of the Brownian Continuum Random Tree. In reality branching occurs continuously.}
\end{figure}

\medskip We now explain the definition of the Continuum Random Tree. Let $f: \R_+ \to \R_+$ be an excursion, i.e., an element of $\Omega^*$, and let $\zeta$ be the lifetime of this excursion.
We wish to think of $f$ as the height function of a certain continuous tree, and that will mean the following: the vertices of the tree can be identified to the interval $[0, \zeta]$ (the time at which we visit this vertex) provided that we make the identification between two times $s \le t$ such that
\begin{equation}\label{equiv}
f(t) = f(s) = \inf_{u \in [s,t]} f(u).
\end{equation}
Indeed, on a discrete tree $T$, if $s$ and $t$ are two times in $\{0, \ldots, |T|-1\}$, then the length of the geodesic between $u_s$ and $u_t$ is easily seen to be
\begin{equation}\label{genealogicalmetric}
h(s) +h(t)- 2 \inf_{ u \in [s,t]} h(u)
\end{equation}
since that distance is simply the sum of two terms, which are the numbers of generations between $u_s$ and $v$ and between $u_t$ and $v$ (where $v$ is the most recent common ancestor between $u_s$ and $u_t$). However, this most recent common ancestor is precisely at height $h(v) = \inf_{u \in [s,t]} h(u)$

Thus let $\sim$ be the equivalence relation on $[0,\zeta]$ defined by (\ref{equiv}), and let
\begin{equation}
\theta= [0,\zeta] / \sim
\end{equation}
be the quotient space obtained from that relation. On the quotient space $\theta$, we introduce the distance
$$
d(s,t) = f(s) + f(t) - 2 \inf_{u \in [s,t]} f(u)
$$
which is easily seen to be a distance on $\theta$.

\begin{definition} \label{D:crt}The metric space $(\theta, d)$ is the continuum tree derived from $f$. If $f(t) = 2 e_t, t \in [0,1]$, where $(e_t, t\ge 0)$ is the Brownian excursion conditioned so that $\zeta=1$, then the random metric space $(\Theta,d)$ derived from $f$ is called the (standard) Continuum Random Tree (or CRT for short).
\end{definition}

To help make sense of this definition, we note that if $T$ is a discrete tree, and if $C(t)$ is the \emph{Contour process} of $T$ (i.e., the linear interpolation of the process which navigates at speed 1 along the edges, exploring the tree in the order of depth-first search but backtracking rather than jumping when it has reached a leaf) then $T$ is isometric to the tree $\theta$ derived from $C(t)$ in Definition \ref{D:crt}.

Thus we have the following:

\begin{enumerate}
\item Any time $t$ such that $f(t)$ is a local minimum is a branching point of the tree.

\item Any time $t$ such that $f(t)$ is a local maximum is a leaf of the tree (there are in fact many other leaves).

\end{enumerate}

For us, any tree associated with a Brownian excursion (be it a Brownian excursion conditioned to be of duration 1 or be it an excursion conditioned to reach above level $x>0$ for some $x>0$, for instance) will be called a Brownian CRT: naturally, they are related by a simple scaling.

\subsection{Continuous-State Branching Processes}

\subsubsection{Feller diffusion and Ray-Knight theorem}

Come back for a moment to the critical Galton-Watson model that we have already introduced, with finite variance. Assume for instance that the population at some time $t>0$ is very large, say $Nx$ for some $x>0$. Then the population size at the next generation can be written as the sum of $Nx$ i.i.d. random variable with mean 1 (and finite variance). Thus we have,  we let $Nx=Z_t$:
$$
\E(dZ_t | \cF_t) = 0;
$$
and
$$
\var(dZ_t | \cF_t ) = \sigma^2 Z_t.
$$
This suggests the following diffusion approximation \index{Diffusion approximation}

\begin{theorem} \emph{(Feller 1951 \cite{Feller51})}
 \label{T:feller} Assume $Z_0=N$. After rescaling (speeding up time by a factor $N$ and dividing the total population size by $N$), the process $Z_{Nt}/N$ converges in the Skorokhod topology towards the unique in law solution of
 \begin{equation} \label{Feller}
 \begin{cases}
   dZ_t & = {\sigma} \sqrt{Z_t}dB_t \\
   Z_0&=1\\
 \end{cases}
 \end{equation}
 The diffusion (\ref{Feller}) with $\sigma^2=1$ is called the Feller diffusion.\index{Feller diffusion}
 \end{theorem}

The Feller diffusion is also sometimes known as the square Bessel process of dimension 0. However there isn't much intuition to gain from that connection (it is hard to imagine what a Brownian motion is in dimension 0). Being the scaling limit of critical Galton-Watson processes, $Z$ is a continuous-state branching process (CSBP), associated with the branching mechanism $\psi(u) = u^2/2$. Indeed, it is an easy exercise of stochastic calculus to check directly that the Feller diffusion enjoys the branching property:


\begin{proposition}
Let $Z(x)$ be the Feller diffusion started from $Z_0=x$. Then $Z$ has the branching property:
$$
Z(x+y) \overset{d}= Z(x) + Z(y)
$$
where the two processes on the right-hand side are independent.
\end{proposition}

\begin{proof} To see this, let $B$ and $B'$ be two independent Brownian motion, and consider two independent Feller diffusions $Z$ and $Z'$ driven by $B$ and $B'$ respectively. Then one has to show that $Z+Z'$ also satisfies (\ref{Feller}), as it is easy to check that uniqueness in distribution holds. However, if $Y=Z+Z'$, then note that
\begin{align*}
  dY_t & = dZ_t + dZ'_t \\
  & = \sqrt{Z_t} dB_t + \sqrt{Z'_t}dB'_t\\
  &= \sqrt{Y_t} dW_t
\end{align*}
where
$$
W_t = \int_0^t\frac{\sqrt{Z_s}}{\sqrt{Y_s}} dB_s+ \frac{\sqrt{Z'_s}}{\sqrt{Y_s}}dB'_s.
$$
Thus $W$ is a local martingale and, since $B$ and $B'$ are independent, $W$ has a quadratic variation equal to
$$
[W]_t = \int_0^t \frac{Z_s}{Y_s}ds + \frac{Z'_s}{Y_s}ds = t
$$
and hence, by L\'evy's characterisation, $W$ is a Brownian motion. Thus $Z$ has the branching property.
\end{proof}

\medskip We now explain the connection between this diffusion and the celebrated Ray-Knight theorem on the local times of Brownian motion. Recall the setup of Theorem \ref{T:heightCV}, where we have an infinite sequence of critical Galton-Watson trees, and we showed that the concatenated height processes converge towards the reflecting Brownian motion after rescaling.

Observe also that the number of visits of the height process at a certain level is precisely the total number of vertices at this generation. Thus,

\begin{center}
\emph{Local time of height process $\leftrightarrow$ Population size}
\end{center}

In particular, if we want to consider only the $N$ first trees $T_1, \ldots, T_N$, we simply stop the height process at the time of its $N\th$ visit to the origin, or equivalently, when it has accumulated a local time at the origin equal to $N$. The total population generated by these first $N$ trees in the next generation, evolves precisely like a Galton-Watson tree started with $N$ individuals (it doesn't matter that these individuals weren't connected to the same root earlier in time). Thus the Feller diffusion approximation of Theorem \ref{T:feller} holds, and given the above principle that the population size is the same as the local time of the height process, we obtain:

\begin{theorem}\label{T:RayKnight} \emph{(Ray-Knight theorem for Brownian motion.)\index{Ray-Knight theorem}} Let $(B_t, t\ge 0)$ be a reflecting Brownian motion at 0, and let $\tau_1 = \inf\{t\ge 0: L_t \ge 1\}$ where $L_t$ is the local time at 0 of $B$. If for $x>0$ we define
\begin{equation}
Z_x = L(\tau_1, x)
\end{equation}
be the total local time that $B$ accumulates at level $x$ before $\tau_1$, then $(Z_x, x \ge 0)$ is the Feller diffusion.
\end{theorem}

The Ray-Knight theorem (discovered simultaneously and independently by Ray and Knight) is actually more general than that, as there exists for instance a version of this result which describes the behaviour of $L(T_x,a)$ as a function of $a$, while $x>0$ is fixed. It is one of the most useful tools for studying one-dimensional Brownian motion, and has been for instance extensively used to describe polymer models (see, e.g., \cite{polreview} or \cite{bb}). Below we will see that this is actually a much more general statement about continuum random trees and continuous-state branching processes.

\subsubsection{Height process and the CRT}

We now show how the relation between the Feller diffusion (which here is seen as an example of CSBP) and the standard continuum random tree may be generalised to other CSBPs. This generalisation is along the lines of the work of Le Gall and Le Jan \cite{LeGallLeJan}. Essentially, we only touch on the surface of some fairly deep ideas that have been developed in the last 10 years, and about which the excellent monograph by Duquesne and Le Gall \cite{duleg} give a much broader overview.

Le Gall and Le Jan \cite{LeGallLeJan} proposed to study the rescaling of the height process of discrete Galton-Watson trees whose population size process converges towards a given continuous-state branching processes. They showed indeed the existence of a scaling limit for the height process, which takes the following form (the one which we quote is a variation on Theorem 2.2.1 in \cite{duleg}).

\begin{theorem}\label{T:LegLej}
Let $Z$ be a fixed CSBP. Let $L^{(N)}$ be the offspring distribution in Theorem \ref{T:lamperti1}, and let $c_N$ be the associated time-scale which guarantees convergence of the rescaled Galton-Watson process towards $Z$. Let $H^{(N)}$ be the height process associated with an infinite sequence of i.i.d. random trees with offspring distribution $L^{(N)}$. Then we also have convergence of the rescaled height process:
$$
(c_NH^{(N)}_{Nt/c_N}, t\ge 0) \to (H_t,t\ge 0)
$$
in the sense of finite-dimensional distribution.
\end{theorem}

See Theorem 2.3.1 in \cite{duleg} for a statement concerning the stronger convergence in the sense of the Skorokhod topology (basically, this convergence is proved under the condition that the CSBP becomes extinct (\ref{Greycondition}) and a technical, non-important condition).

It is now a good time to recall our principle ``one function, one tree": the limiting height process $(H_t, t\ge 0)$ encodes a certain Continuum Random Tree\index{Continuum Random Tree} $\Theta$, and the convergence in Theorem \ref{T:LegLej} ensures that the corresponding rescaled trees, converge in distribution (in the sense of the Gromov-Hausdorff metric) towards $\Theta$. This is a somewhat sloppy statement: for this to be true, we have to talk about the Galton-Watson tree conditioned to reach a large height, for instance, much as we did in (\ref{heightCV2}). However, for this to make sense, we need to make sure that $H$ is almost surely continuous. That turns out to be true if and only if the corresponding branching process becomes extinct, i.e., if and only if Grey's condition (\ref{Greycondition}) holds.

Naturally, having made this definition we want to see how the CSBP relates to the Continuum Random Tree $\Theta$, and the answer is again via a Ray-Knight theorem. Thus, define:
\begin{equation}
\label{LoctimesapproxHeight}
L(t,x):= \lim_{\eps\to 0}\frac1{2\eps} \int_0^t \indic{|H_s -x| \le \eps}ds
\end{equation}
which is the amount of local time\index{Local time!(for height process)} that $H$ spends near $x$. One thing to realise is that it is not obvious at all why the limit (\ref{LoctimesapproxHeight}) exists, as $H$ is neither a Markov process nor a semimartingale. This limit is in fact shown to exist in $L^1$ (uniformly in $t$) by Duquesne and Le Gall in Proposition 1.3.3 of \cite{duleg}. The Ray-Knight theorem in this setup states:

\begin{theorem}\label{T:RayKnight2}
Let $(H_t,t\ge0)$ be the height process of a $\psi$-CSBP. Let $L(t,x)$ be its joint local time process and let $(\tau_\ell, \ell \ge 0)$ be the inverse local time at $x=0$. Then
\begin{equation}
(Z_x = L(\tau_z,x), x\ge 0)
\end{equation}
is a $\psi$-CSBP started from $Z_0=z$.
\end{theorem}

The advantage of having introduced a tree to describe a CSBP is that it makes it possible to discuss issues related to the \emph{genealogy} of this continuous-state branching process. For instance, the number of individuals at time 0 who have descendants at time $x>0$ is equal to the number of excursions above 0 that reach $x>0$ (and is thus finite almost surely under Grey's condition (\ref{Greycondition}).

Much as in the case of the Brownian CRT, where It\^o's excursion measures can be used to describe the statistics of ``infinitesmial trees" above a given level, there is a valid generalisation of excursion theory to height processes. This generalisation can be stated in terms of excursions as in Theorem \ref{T:Ito} or in terms of trees. We will refrain from stating explicitly this result, except to say informally that the collection of trees generated by the height process above a certain level $x>0$ is a Poisson point process of trees with intensity $d\ell$ (the local time scale at level $x$) times a certain excursion measure, $\nu$. For instance, given $L(\tau_1, x)=\ell$, the number of excursions (or trees) that reach level $x+h$ is Poisson with mean $\ell \nu(\sup_{s\ge 0} H_s >h)$. Moreover this is ``independent from what happened at lower levels". (However, unlike in the Brownian case, it makes no sense to talk about excursions \emph{below} $x$ for which there is no excursion property).

We finish this section with a description of the excursion measure, which is the analogue to Theorem \ref{Itodescription}.

\begin{theorem}\label{T:heightdescription} Assume Grey's condition $(\ref{Greycondition})$.\\
 Let $v(x) = \nu(\sup_{s\ge 0} H_s >x)$ be the measure of excursions that reach above level $x$. Then $v(x)$ is uniquely determined by:
$$
\int_{v(x)}^\infty \frac{dq}{\psi(q)} =x.
$$
\end{theorem}

Naturally, this result says exactly the same thing as Theorem \ref{T:Grey} for the lookdown process.\index{Donnelly-Kurtz}\index{Lookdown process} Indeed, this can be proved directly using the same arguments, or can be deduced from it: it turns out that the notions of genealogy for $(Z_t,t\ge 0)$, as defined by the continuum random tree and by the lookdown process, are identical. Recall that in the world of CRT, an individual is identified with subtree below it, i.e., with an excursion above a certain level, and $u$ is an ancestor of $v$ if the excursion associated with $v$ is a piece of the excursion associated with $u$. However, in the lookdown process, individuals are seen as levels of a countable population, and individual $i$ at time $s$ is an ancestor of individual $j$ at time $t>s$ if $\xi_i(s) = \xi_j(t)$, where $(\xi_i(t),t\ge 0)_{i\ge 1}$ denotes the lookdown process.

The following result was proved in \cite{bbs2}, and shows that the two notions are identical, in the following sense. Let $(Z_t,t\ge 0)$ be a $\psi$-CSBP started from $Z_0=r>0$ satisfying (\ref{Greycondition}), and assume that $Z_t$ is obtained as
the local times of the height process $(H_t,t\le T_r)$ as in Theorem \ref{T:RayKnight2}. The key point is to order the excursions above a certain level $t$ suitably. We choose to rank them according to their
supremum. That is, we denote by $e_j{(t)}$ the $j\th$ highest excursion
above the level $t$. We draw a sequence of i.i.d. random
variables $(U_i)_{i \ge 1}$ uniform on $(0,1).$ For each $j\ge 1$, we associate to $e_j(0)$ the label $U_j$. As $t$ increases,
a given excursion may split: we decide that the children subexcursions each inherit the label of the parent excursion. We define a process $\xi_j(t)$, for all $j\ge 1$ and all $t\ge 0$ by saying that $\xi_j(t)$ the label of $e_j(t)$. Note that when an excursion splits, a fairly complex transition may occur for $(\xi_j(t), t\ge 0)$ as the excursions $e_j(t)$ are always ordered by their height. In fact, we have the following result (Theorem 14 in \cite{bbs2}):

\begin{theorem} The process $(\xi_j(t),t\ge 0)$ is the Donnelly-Kurtz lookdown process associated with $(Z_t,t\ge 0)$.
\end{theorem}

\newpage



\printindex[not]

\addcontentsline{toc}{section}{Index}
\printindex

\end{document}